\documentclass[10pt]{amsart}

\subjclass{60J60, 60H07, 60H10, 60F99, 34E10}
\keywords{Heteroclinic networks, exit problems, vanishing noise limit, rare events, polynomial decay, Malliavin calculus,  saddle points, metastability, invariant distributions, hierarchy of timescales, homogenization}
\usepackage{amsmath,amsfonts,latexsym,amssymb,amscd,
  graphicx,amsthm}
\usepackage{enumitem}
\usepackage{ulem}
\usepackage{cancel}
\usepackage[hidelinks]{hyperref} \usepackage{xcolor}
\pagecolor{white}
\usepackage{mathrsfs}

\usepackage{tikz}
\usetikzlibrary{arrows.meta} \usetikzlibrary{decorations.pathreplacing,decorations.markings} 

\tikzset{
on each segment/.style={
    decorate,
    decoration={
      show path construction,
      moveto code={},
      lineto code={
        \path [#1]
        (\tikzinputsegmentfirst) -- (\tikzinputsegmentlast);
      },
      curveto code={
        \path [#1] (\tikzinputsegmentfirst)
        .. controls
        (\tikzinputsegmentsupporta) and (\tikzinputsegmentsupportb)
        ..
        (\tikzinputsegmentlast);
      },
      closepath code={
        \path [#1]
        (\tikzinputsegmentfirst) -- (\tikzinputsegmentlast);
      },
    },
  },
mid arrow/.style={postaction={decorate,decoration={
        markings,
        mark=at position .5 with {\arrow[#1]{Stealth[]}}
      }}},
}

\usepackage{verbatim}   

\newif\ifshowtikz

\showtikztrue

\let\oldtikzpicture\tikzpicture
\let\oldendtikzpicture\endtikzpicture

\renewenvironment{tikzpicture}{\ifshowtikz\expandafter\oldtikzpicture \else\comment \fi
}{\ifshowtikz\oldendtikzpicture \else\endcomment \fi
}

\allowdisplaybreaks[1]

\newcommand{\bpf}[1][Proof]{{\noindent {\sc #1: }}}
\newcommand{\bpfm}[1][Heuristics for the model case]{{\noindent {\sc #1: }}}
\newcommand{\epf}{{{\hfill $\Box$ \smallskip}}}
\newcommand{\ONE}{{\mathbf{1}}}

\newcommand{\N}{{\mathbb N}}

\newcommand{\Fc}{\mathcal{F}}
\newcommand{\Nc}{\mathcal{N}}
\newcommand{\Pp}{\mathsf{P}}
\newcommand{\Psc}{\mathsf{Q}}
\newcommand{\Z}{{\mathbb Z}}

\newcommand{\dtime}{T} \newcommand{\E}{\mathsf{E}}
\newcommand{\Leb}{\mathrm{Leb}}

\newcommand{\R}{{\mathbb R}}
\newcommand{\TT}{{\mathbb T}}

\newcommand{\cc}{{\mathbf c}}

\newcommand{\fR}{\Pi}
\newcommand{\eqpm}{\asymp_\pm}

\newcommand{\eps}{{\varepsilon}}

\newcommand{\lam}{\lambda}

\newcommand{\Lf}{A}

\newtheorem{theorem}{Theorem}[section]
\newtheorem{remark}{Remark}[section]
\newtheorem{lemma}{Lemma}[section]
\newtheorem{proposition}{Proposition}[section]
\newtheorem{corollary}{Corollary}[section]

\numberwithin{equation}{section}

\renewenvironment{proof}[1][Proof]{{\noindent {\sc #1: }}
}{{{\hfill $\Box$ \smallskip}}
}
\makeatletter
\let\orgdescriptionlabel\descriptionlabel
\renewcommand*{\descriptionlabel}[1]{\let\orglabel\label
  \let\label\@gobble
  \phantomsection
  \edef\@currentlabel{#1}\let\label\orglabel
  \orgdescriptionlabel{#1}}
\makeatother

\newcommand{\indistr}{\stackrel{d}{\longrightarrow}}
\newcommand{\inprob}{\stackrel{P}{\longrightarrow}}

\newcommand{\e}{\varepsilon}
\newcommand{\sgn}{\mathop{\mathrm{sgn}}}
\newcommand{\Wu}{\mathcal{W}^{\mathrm{u}}}
\newcommand{\Ws}{\mathcal{W}^{\mathrm{s}}}

\newcommand{\vk}{\varkappa}

\newcommand{\stindex}{stability index }
\newcommand{\GoodMeasures}{\mathcal{M}}
\newcommand{\flow}{\varphi}
\newcommand{\kstar}{\kappa}

 \newcommand{\scx}{x}
\newcommand{\tscx}{{y}} \newcommand{\z}{z} \renewcommand{\P}{\Pp}
\newcommand{\Prob}[1]{\Pp\left\{ #1 \right\}}
\newcommand{\Probx}[2]{\Pp^{#1}\left\{ #2 \right\}}
\newcommand{\Probxy}[1]{\Pp^{\eps \scx}\left\{ #1 \right\}}
\newcommand{\Exy}[1]{\E^{\eps \scx} #1 }

\newcommand{\Ind}[1]{\ONE_{\{ #1\}}}

\newcommand{\li}{{\lambda^i}}
\newcommand{\lj}{{\lambda^j}}
\newcommand{\lk}{{\lambda^k}}
\newcommand{\blam}{\overline{\lambda}}
\newcommand{\ulam}{\underline{\lambda}}

\newcommand{\smallo}[1]{{o\left(#1\right)}}

\newcommand{\D}{\mathcal{D}} \newcommand{\Hil}{\mathcal{H}} \newcommand{\pbracket}[1]{\left[#1\right]_p}
\newcommand{\pbracketX}[2]{\left[#1\right]_{#2}}

\newcommand{\Tone}{{T_1}}

\newcommand{\cZZ}{\mathcal{Z}}

\newcommand{\htau}{{\zeta}}

\newcommand{\superpoly}{o_e(1)}

\newcommand{\cU}{\mathcal{U}}
\newcommand{\frU}{\mathcal{U}}
\newcommand{\Bepm}{B^\eps_\pm}
\newcommand{\Aepm}{A^\eps_\pm}

\newcommand{\Bc}{\mathcal{B}}
\newcommand{\Xc}{\mathcal{X}}

\newcommand{\CC}{\mathbf{C}}
\newcommand{\pp}{\mathtt{p}}
\newcommand{\NN}{\mathcal{N}}
\newcommand{\bA}{\overline{A}}
\newcommand{\bB}{\overline{B}}
\newcommand{\YY}{\mathbf{Y}}

\newcommand{\tYY}{\widetilde{\YY}}
\newcommand{\ZZ}{\mathbf{Z}}

\newcommand{\bZZ}{\overline{\mathbf{Z}}}
\newcommand{\hZZ}{\widehat{\mathbf{Z}}}

\newcommand{\bZ}{\overline{Z}}
\newcommand{\qd}[1]{\left\langle #1 \right\rangle}

\renewcommand{\for}{\text{for }}
\renewcommand{\and}{\text{ and }}
\newcommand{\bH}{\overline{H}}
\newcommand{\sigmaUN}{\sigma_{(U^{\leq\nu}_T,N^{>\nu}_T)}}

\newcommand{\trcurve}{\chi}
\newcommand{\NProj}{M}

\newcommand{\Sph}{\mathbb{S}}
\newcommand{\irc}{\textit{in rectified coordinates}}

\newcommand{\eqmodlp}{\stackrel{\text{\rm w.h.p.}}{=}}

\begin{document}
\title{Rare transitions in noisy heteroclinic networks}
\author{Yuri Bakhtin, Hong-Bin Chen,\and Zsolt Pajor-Gyulai}
\address{Courant Institute of Mathematical Sciences, New York University}
\email{bakhtin@cims.nyu.edu}
\email{hbchen@cims.nyu.edu}
\email{zsolt@gcims.nyu.edu}
\begin{abstract} We study small white noise perturbations of planar dynamical systems with heteroclinic networks in the limit of vanishing noise. We show that the probabilities of transitions between various cells that the network tessellates the plane into decay as powers of the noise magnitude. We show that the most likely scenario for the realization of these rare transition events involves spending atypically long times in the neighborhoods of certain saddle points of the network. We describe the hierarchy of time scales and clusters of accessibility associated with these rare transition events. We discuss applications of our results to homogenization problems and to the invariant distribution asymptotics. At the core of our results are local limit theorems for exit distributions obtained via methods of Malliavin calculus.
\end{abstract}
\maketitle

\tableofcontents
\section{Introduction}
\subsection{The setting, the motivation, and the goal of the paper.}
In this paper, we study the long-term behavior of smooth dynamical systems with heteroclinic networks under small white noise perturbations.

Solutions of It\^o SDEs  like
\begin{equation}
\label{eq:basic-sde}
dX_{\e,t}=b(X_{\e,t})dt+\e \sigma(X_{\e,t})dW_t,
\end{equation}
in $\R^d$ where $W$ is the Wiener process  with $d$ independent components, 
have very simple asymptotic behavior in the vanishing noise limit $\e\to0$ if considered on a finite time interval. Under very broad assumptions on the drift $b$ and diffusion~$\sigma$ coefficients, they converge, as $\e\to 0$, to solutions of the deterministic ODE
\begin{equation}
\label{eq:basic-ode}
\dot X_{0,t}=b(X_{0,t}).
\end{equation}
However, the behavior of solutions of~\eqref{eq:basic-sde} on infinite time intervals or intervals growing to infinity as  $\e\to 0$, may drastically differ from that of solutions of~\eqref{eq:basic-ode}. 

The asymptotic properties depend crucially on the geometry of the phase portrait generated by $b$ and typically do not depend much on $\sigma$ once an assumption of boundedness and uniform ellipticity (nondegeneracy) of $\sigma$  is made.

The most celebrated mathematical achievement in this area is the Freidlin--Wentzell theory of metastability and related concepts, studying the situation where the solution of~\eqref{eq:basic-sde} spends very long times near locally stable attractors making rare and rapid transitions between them.  In chemistry and physics, the exponential in $\e^{-2}$ growth of transition times is known as Kramers' asymptotics \cite{KRAMERS1940284}. The classical mathematical reference for these asymptotic results and other vanishing noise problems is~\cite{FW}.

In this paper, we continue the study of SDE~\eqref{eq:basic-sde} in the vanishing noise limit that we began in~\cite{Bakhtin2011} (see also an informal exposition in~\cite{Bakhtin2010:MR2731621}), under the assumption that $b$ generates a heteroclinic network. 

A heteroclinic network is a feature of the phase portrait of a dynamical system consisting of multiple hyperbolic critical points (saddles) and heteroclinic orbits connecting them, see Figure~\ref{fig:heteroc-example} for an example of a planar heteroclinic network. A heteroclinic orbit, also called a heteroclinic connection, belongs to, or coincides with, the unstable manifold of one saddle and the stable manifold of another saddle.  

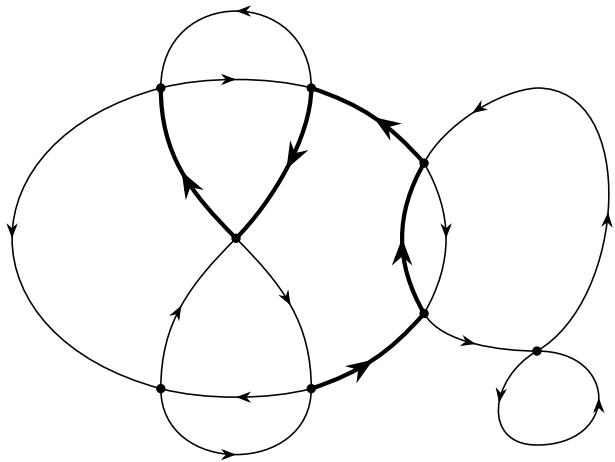
\begin{figure}[ht]

\centering
\begin{tikzpicture}

    \node  (0) at (0, 0)     {};
    \node  (1) at (-1, 2)    {};
    \node  (2) at (1, 2)     {};
    \node  (3) at (-1, -2)   {};
    \node  (4) at (1, -2)    {};
    \node  (5) at (2.5, 1)   {};
    \node  (6) at (2.5, -1)  {};
    \node  (7) at (4, -1.5)  {};
    \node  (8) at (4, 2)     {};
    \node  (9) at (4, -2.75) {};

\foreach \n in {0,1,2,3,4,5,6,7}
        \node at (\n)[circle,fill,inner sep=1.25pt]{};

    \path [draw=black, ultra thick, postaction={on each segment={mid arrow=black}}] 
    [in=45, out=-90, looseness=0.75] 
    (2.center) to (0.center);
    \path [draw=black, ultra thick, postaction={on each segment={mid arrow=black}}] 
    [in=-90, out=135] 
    (0.center) to (1.center);
    \path [draw=black, semithick, postaction={on each segment={mid arrow=black}}] 
    [bend right=90, looseness=1.75] 
    (2.center) to (1.center);
    \path [draw=black, semithick, postaction={on each segment={mid arrow=black}}] 
    [in=-135, out=90] 
    (3.center) to (0.center);
    \path [draw=black, semithick, postaction={on each segment={mid arrow=black}}] 
    [in=90, out=-45] 
    (0.center) to (4.center);
    \path [draw=black, semithick, postaction={on each segment={mid arrow=black}}] 
    [bend right=90, looseness=1.50] 
    (3.center) to (4.center);
    \path [draw=black, semithick, postaction={on each segment={mid arrow=black}}] 
    [bend left=15, looseness=0.75] 
    (1.center) to (2.center);
    \path [draw=black, semithick, postaction={on each segment={mid arrow=black}}] 
    [bend left=15, looseness=0.75] 
    (4.center) to (3.center);
    \path [draw=black, semithick, postaction={on each segment={mid arrow=black}}] 
    [bend right=75, looseness=1.75] 
    (1.center) to (3.center);
    \path [draw=black, ultra thick, postaction={on each segment={mid arrow=black}}] 
    [bend right=15] 
    (5.center) to (2.center);
    \path [draw=black, ultra thick, postaction={on each segment={mid arrow=black}}] 
    [bend right=15] 
    (4.center) to (6.center);
    \path [draw=black, semithick, postaction={on each segment={mid arrow=black}}] 
    [bend left] 
    (5.center) to (6.center);
    \path [draw=black, ultra thick, postaction={on each segment={mid arrow=black}}] 
    [bend left] 
    (6.center) to (5.center);
    \path [draw=black, semithick, postaction={on each segment={mid arrow=black}}] 
    [in=-180, out=-60, looseness=0.75] 
    (6.center) to (7.center);
    \path [draw=black, semithick, postaction={on each segment={mid arrow=black}}] 
    [bend right, looseness=0.75] 
    (8.center) to (5.center);
    \path [draw=black, semithick, postaction={on each segment={mid arrow=black}}] 
    [in=0, out=30] 
    (7.center) to (8.center);
    \path [draw=black, semithick, postaction={on each segment={mid arrow=black}}] 
    [in=180, out=-150, looseness=1.50] 
    (7.center) to (9.center);
    \path [draw=black, semithick, postaction={on each segment={mid arrow=black}}] 
    [in=0, out=0, looseness=2.25] 
    (9.center) to (7.center);

    \end{tikzpicture}

\caption{A planar heteroclinic network and a heteroclinic chain escaping a cell.}
\label{fig:heteroc-example}
\end{figure}

It is natural to presume that a diffusion near such a heteroclinic network mimics the process of sequential random decision making: it spends a lot of time in a small neighborhood of a critical point where the drift is very small, until eventually the noise pushes the solution in one of the unstable directions (thus, a decision on the exit direction is made). From here, the drift takes over, carrying the solution away from the equilibrium along a heteroclinic orbit towards the next critical point. This picture resembles a random walk on the directed graph with vertices representing saddles and directed edges corresponding to heteroclinic connections.

However, it turns out that diffusion near a heteroclinic network in vanishing noise limit may and often does look drastically different from a Markovian  random walk. In many instances, the outcome of the decision on the exit direction is influenced and even largely predetermined by the history of the process, thus exhibiting non-Markovian limiting behavior and departing from the random walk picture. A rigorous mathematical theory of this was given in~\cite{Bakhtin2011} although non-rigorous approaches had existed before
  \cite{Stone-Holmes:MR1050910}, \cite{Stone-Armbruster:doi:10.1063/1.166423},  \cite{Armbruster-Stone-Kirk:MR1964965}. 
 
More precisely, the typical behavior of diffusions with small noise near heteroclinic networks was described in~\cite{Bakhtin2011} for time scales logarithmic in the noise magnitude. In particular, that paper showed 
that the diffusion spends time of order $\log\e^{-1}$ near a saddle, travels along a heteroclinic orbit in time of order~$1$, then
spends time of order $\log\e^{-1}$ near the next saddle, etc. Moreover, for any finite sequence of saddles and heteroclinic connections between them, 
the limiting probability of evolution along those connections was computed, in the limit of vanishing noise. These limiting probabilities often equal~$0$ or~$1$, which means that traveling along certain pathways through the graph of heteroclinic connections is extremely unlikely on the logarithmic time scale. This  results in a limited vocabulary of observable pathways and, often, in heteroclinic cycling, where the process is trapped in a small region of the network and intermittently follows one of a few available cycles, occasionally switching between them.

The core of the analysis in~\cite{Bakhtin2011} is the study of exit problems for certain regions around the saddles and the connections, with scaling limit theorems for the exit time and location. It shows that certain transitions in noisy heteroclinic networks are unlikely and certain ones are typical. The typical ones completely define the limiting dynamics on timescales logarithmic in noise intensity. However, in order to study the behavior of the system over long or infinite time intervals, one must carry out a finer study of the unlikely transitions. This is exactly the goal of the present paper: to study the exit problems of~\cite{Bakhtin2011} in more detail and analyze the unlikely events responsible for the departure from the typical scenario described in this paper. This is the natural next step in the ambitious program to understand the limiting behavior of invariant distributions in the compact phase space case (on a torus) and homogenization and effective diffusivity for periodic heteroclinic networks.

\subsection{The main result: the polynomial rates of rare transitions and the underlying slowdown mechanism.}
\label{sec:main_result_descr}
We restrict ourselves to dynamics in the Euclidean plane~$\R^2$ or torus~$\TT^2$. Working with other $2$-dimensional manifolds, in charts, and with Stratonovich noise, is not much harder but would obscure our main points. We also  expect the picture to be similar in higher dimensions, especially for heteroclinic networks of saddles with 1-dimensional unstable manifolds.

In two dimensions, heteroclinic networks admit a relatively simple description: under fairly general regularity assumptions they all can be viewed as locally finite collections of closed curves with simple mutual intersections and self-intersections, see Figure~\ref{fig:heteroc-example}. They tessellate the plane into cells, the boundary of each cell being a union of several heteroclinic connections, which are either all oriented clockwise or all oriented counter-clockwise.
 
In this paper, we quantify rare transitions between neighboring cells and compute the asymptotic transition rates. More precisely, for each sequence of heteroclinic connections
 on the boundary of one cell, we compute the decay rate, as $\e\to0$, of the probability of escaping the cell immediately after following that sequence. An example of such a transition  is shown in Figure~\ref{fig:heteroc-example}, where a chain of
 heteroclinic connections almost entirely belongs to the boundary of one cell and the last heteroclinic connection escapes from this cell.

Our main result (see Theorem~\ref{th:multi-saddle-escape} for a precise statement and Figure~\ref{fig:main-heteroc} for a more detailed illustration of the setting) is that, depending on the contraction and expansion rates near each saddle of the sequence, and on the character of the scaling of the distance from the initial condition to the network, three situations are possible. As $\e\to0$, the probability of escape either 
\begin{enumerate}
\item \label{item:limit-probability} converges to a positive constant (as described in \cite{Bakhtin2011}), or
\item\label{item:maindescript-poly-decay} decays as   $h\e^\theta(1+o(1))$ for some numbers $\theta,h>0$, or
\item  \label{item:superpoly-decay}
decays faster than any power of $\e$.
\end{enumerate}
 
Under several technical assumptions, Theorem~\ref{th:multi-saddle-escape}  gives a detailed characterization of the conditions for each of these cases to occur, and in case \ref{item:maindescript-poly-decay} computes the scaling exponent~$\theta$, see~\eqref{eq:theta}. This exponent can also be defined as $\theta=0$ for case~\ref{item:limit-probability} and as  $\theta=\infty$ for case~\ref{item:superpoly-decay}. Moreover,
in case~\ref{item:superpoly-decay}, we actually prove a more precise estimate: the probability of escape is bounded by $\exp[-(\log \eps^{-1})^{1+\delta}]=\eps^{(\log \eps^{-1})^{\delta}}$ for some $\delta>0$. 
 
\medskip
 
The case~\ref{item:maindescript-poly-decay}  is the central, most interesting,  and hardest part of this paper. Compared to the results of~\cite{Bakhtin2011} where the analysis was performed at the level of weak convergence of appropriately scaled exit distributions, to obtain the  power asymptotics in part~\ref{item:maindescript-poly-decay}, we need to study the exit distributions zooming into finer scales and proving local limit theorems. We are able to prove local equidistribution results by studying the densities of the distributions involved with the help of estimates from~\cite{bally2014} based on  Malliavin calculus. The approach developed in~\cite{long_exit_time}, \cite{long_exit_time-1d-part-2}, \cite{YB-and-HBC:10.1214/20-AAP1599}, \cite{YB-and-HBC:10.1214/20-AOP1479}, \cite{B-PG:doi:10.1142/S0219493719500229} for exit problems near critical points of source type thus gets extended to the harder case of critical points of saddle type.

Our analysis also reveals the mechanism through which the rare transitions are realized. It turns out that imposing the condition on the process to leave the cell after passing a given saddle point effectively influences the behavior of the entire trajectory before the visit to that saddle point. The exit is prepared by getting atypically close to the network while visiting neighborhoods of preceding saddles. More precisely, there are certain slowdown saddles near which the process spends an abnormally long time thus extending the exposure to contraction towards the boundary of the cell in comparison
with the typical scenario.

The exponent $\theta$ in the power asymptotics of our main result is determined by the contraction and expansion rates near all the saddles involved. However, the definition of $\theta$ is not straightforward. One must find all the slowdown saddles via a special procedure and take into account that each of them contributes a factor of order of a power of $\e$, with the exponent being a nontrivial nonlocal function of the entire sequence of corresponding contraction/expansion rates.

\subsection{The hierarchy of time scales and clusters of accessibility.}
The polynomial decay rate of the escape probabilities in our main theorem suggests that the shortest time scale on which we can expect deviations from the typical behavior is of the order $\e^{-\theta}$ (up to a logarithmic factor) for some $\theta > 0$. Moreover, different transitions often have different associated exponents, implying an entire hierarchy of polynomially growing time scales on which more and more transitions become accessible for the dynamics and larger and larger clusters of points accessible at those time scales emerge. Under the requirement that the network is stable (exponentially attracting nearby initial conditions in the absence of noise), the noisy dynamics can be described as a multiscale process dominated by transitions between clusters at various levels.

This is akin to metastable cycling described in \cite{FW} where rare transitions between metastable states are described at the level of large deviations. They occur on time scales of order $e^{\theta\e^{-2}}$ with $\theta>0$ obtained by minimizing an appropriate action functional over paths connecting the metastable states involved.

The hierarchical structure of polynomial time scales and associated clusters emerging in our setting is discussed in Section~\ref{sec:hierarchy}. In that section, we also draw a connection to the general abstract picture of metastable cycling introduced in~\cite{FK-metastable:MR3652517}. We study a concrete  example of a heteroclinic network on the torus $\TT^2$ which, if lifted to a $\Z^2$-periodic cellular flow on $\R^2$, allows for Gaussian limit theorems for sufficiently large time scales. These can be viewed as homogenization results on effective diffusivity for second order parabolic PDEs, with the scaling limit given by the heat equation. Such a result would be hard to obtain via PDE methods (see, however, the Appendix in \cite{HIKNPG:MR3773377}).

In addition, for the torus case, we show how to compute the limit of the invariant distribution for the diffusion process as $\e\to0$. It is  always a mixture of Dirac masses at saddle points but the computation of the weights of individual atoms requires a multi-level iterative procedure based on the hierarchical structure.

We decided not to pursue rigorous exposition in Section~\ref{sec:hierarchy}, postponing that to a later publication.

\subsection{The structure of the paper.} In order to motivate and explain the new results, we have to start with recalling the results of~\cite{Bakhtin2011} in Section~\ref{sec:typical}. For our new results, we need to supplement the scaling limits of Section~\ref{sec:typical} with more detailed analysis. Some useful terminology and notation is introduced in Section~\ref{sec:notation}. In Section~\ref{sec:2saddles}, we study a relatively simple case where an N-shaped heteroclinic chain (see Figure~\ref{fig:2saddles}) is composed of two saddle points and three heteroclinic connections, the last one escaping from the cell making a ``wrong turn''. Section~\ref{sec:long-escape-chains} gives the main result for a
chain of arbitrary length.  After that, in Section~\ref{sec:hierarchy}, we give an informal discussion of the emerging hierarchy of clusters and time scales, and its implications.

\medskip

We must comment on the style of our exposition.
In Sections~\ref{sec:typical}--\ref{sec:long-escape-chains}, we give complete rigorous statements of results but not all explanations are
rigorous, some of them being heuristics for a simplified model case rather than complete proofs.
These sections should be read first in order to understand the whole picture.
The rigorous proofs of those statements in complete generality are given in Sections~\ref{sec:rectified}--\ref{section:density_est}, with the most technical part on local limit theorems for exit densities
being Sections~\ref{sec:Gaussian-approx}--\ref{section:density_est}.

\bigskip

{\bf Acknowledgments.} 
We thank Mark Freidlin  and Leonid Koralov for multiple stimulating discussions. Yuri Bakhtin thanks NSF for partial support via Award DMS-1811444. Zsolt Pajor-Gyulai is thankful to the Courant Institute where this work was initiated during his tenure as a Courant Instructor.

\section{The typical behavior}\label{sec:typical} 
The goal of this section is to  recall the results of \cite{Bakhtin2011}
(see also \cite{Bakhtin2010:MR2731621}, \cite{Almada-Bakhtin:MR2802310}, \cite{Almada-Bakhtin:MR2739004}) since they serve as an important starting point. We aim at a minimal description relevant for this paper, not a comprehensive one.

\subsection{Notation}
\label{sec:notation1}

We denote the Borel $\sigma$-algebra on $\R$ by $\Bc$. We call $\nu:\R\times \Bc\to [0,\infty)$ a transition kernel if for each $x\in\R$, $\nu(x,\cdot)$ is a Borel measure on $\R$, and for each $B\in\Bc$, $\nu(\cdot,B)$ is a Borel measurable function.

For any $m\in \N$, we use superscripts to denote the coordinates of points
$x=(x^1,x^2,\ldots,x^m)\in\R^m$. For  $a,b\in\R^m$, we write
\begin{align*}
    a\cdot b =\sum_{i=1}^m a^i b^i.
\end{align*}

For $m,n\in\N$ and $k\in\N\cup\{0\}$, an $\R^m$-valued function $f$ defined on an open set in $\R^n$ is said to belong to $C^k$ if $f$ is continuously differentiable up to the $k$-th order. If, in addition, the partial derivatives of $f$ of all orders up to $k$ are bounded, it is said to belong to $C^k_{\mathrm{b}}$. If, moreover, $f$ is a bijection and $f^{-1}\in C^k_\mathrm{b}$, then $f$ is called a $C^k_\mathrm{b}$-diffeomorphism.

Two vectors in $\R^m$ are called collinear if one of them is a multiple of another. In particular, the zero vector is collinear with any other vector. 

The Lebesgue measure on any Euclidean space is denoted by $\Leb$. 

The locally uniform convergence (i.e., uniform convergence on compact sets) is often abbreviated to convergence in LU-topology or simply in LU.  

The symmetric difference between two sets $A, B$ is denoted by $A\triangle B$.

For $a, b\in \R$, we write $a\wedge b = \min\{a,b\}$ and $a\vee b = \max\{a,b\}$.

Each statement involving signs $\pm$ and $\mp$ represents two statements: the first one where every $\pm$ is replaced by $+$, every  $\mp$ by $-$; and the second one where every $\pm$ is replaced by $-$, and every  $\mp$ by $+$.

 For
 $a_-,a_0,a_+\in\R$, we write
\begin{align}\label{eq:eqpm_notation}
    a_0\eqpm a_\pm 
\end{align}
if and only if $a_-\leq a_0\leq a_+$.

We usually work with a complete probability space $(\Omega,\Fc,\Pp)$ equipped with a filtration $(\Fc_t)_{t\ge0}$  satisfying the usual conditions. We assume that this probability space is rich enough to support all r.v.'s (random variables)  emerging in the paper. This is not necessary but makes notation a little lighter since we can use notation like $\Pp\{\xi\in[a,b]\}$ for a distributional limit $\xi$ of r.v.'s $\xi_\e$ defined on this probability space. At times it will be also convenient to use 
other probability spaces and measures. Irrespective of the details of the setting, we denote convergence of r.v.'s in distribution by $\indistr$ and in probability by $\inprob$. By $W=(W^1,W^2)$ we denote the standard two-dimensional Wiener process, i.e., $W^1$ and $W^2$ are independent
standard one-dimensional Wiener processes with respect to $(\Fc_t)$.  All stochastic integrals are understood in the It\^o sense.

We will denote by $g_c(x)$ the centered Gaussian density with variance $c>0$:
\begin{align}\label{eq:g_c_Gaussian}
    g_c(x)=\frac{1}{\sqrt{2\pi c}} e^{-\frac{x^2}{2c}},\quad x\in\R.
\end{align}
The associated  distribution function is denoted by $\psi_c$:
\begin{equation}
\label{eq:gaussian-cdf}
\psi_c(x)=\int_{-\infty}^x g_c(x')dx',\quad x\in\R.
\end{equation}  

We will often omit the dependence of $X_{\eps,t}$, a solution of~\eqref{eq:basic-sde} on the noise magnitude $\eps$. For example, the joint distribution of  $((X_t)_{t\ge0}, (W_t)_{t\ge 0})$ solving~\eqref{eq:basic-sde} conditioned on the initial value $X_0=x_0\in\R^2$ will be denoted by $\Pp^{x_0}$ with $\eps$ omitted.

When using the $o(\cdot)$ notation and its modifications, we mean taking limits as $\e\to 0$, unless stated otherwise.  

Throughout the paper we use $C$ to denote various constants whose values may differ from instance to instance.

More notation and terminology is collected in Section~\ref{sec:notation}.

\subsection{Exiting a neighborhood of a saddle}Of course, the main strategy is to surround each saddle by a neighborhood and study the exit problems in each neighborhood and transitions between those neighborhoods along heteroclinic connections. 

In this section, we consider a family of  diffusions $(X_\e)_{\e>0}$ near one saddle point in $\R^2$. 
An archetypal and relatively simple situation is where the drift is linear and the noise is additive and diagonal:
\begin{align}
 dX_{\e,t}^1&=\lambda X_{\e,t}^1dt+\e dW^1_t,\label{eq:linear-system1}\\
 dX_{\e,t}^2&=-\mu X_{\e,t}^2dt+\e dW^2_t\label{eq:linear-system2},
\end{align}
where $\lambda,\mu>0$. Here the origin is a hyperbolic fixed point for the associated deterministic linear 
dynamics. Its stable manifold coincides with the second coordinate axis, and the unstable one coincides with the
first coordinate axis.

Our goal is to show that if the initial condition has a distributional scaling limit, then the exit distribution also has a distributional scaling limit, with a new exponent and limiting distribution.

Let us equip the system~\eqref{eq:linear-system1}--\eqref{eq:linear-system2}
with the following initial condition:
\begin{align}
X_{\e,0}^1&=\e^{\alpha}\xi_\e,
\label{eq:entrance-distr-linear1}
\\
X_{\e,0}^2&=L,
\label{eq:entrance-distr-linear2}
\end{align}
where $L\in\R\setminus\{0\}$, $\alpha\in(0,1]$, and $(\xi_\e)_{\e>0}$ is a family of r.v.'s independent of the realization of the noise on $[0,\infty)$.
Let us assume that as $\e\to 0$, $\xi_\e\indistr\xi$ for some r.v. $\xi$. If $\alpha<1$, we will additionally assume that 
\begin{equation}
\label{eq:no_atom_at_zero}
\Pp\{\xi=0\}=0.
\end{equation}

Let us fix a threshold $R>0$ and  define the exit time from the domain 
\begin{equation}
\label{eq:strip}
D=\{(x^1,x^2): |x^1|<R\}.
\end{equation}
by
\begin{align}
\label{eq:def-exit-time}
\tau_\e&=\inf\{t\ge0:\ |X_{\e,t}^1|\ge R\} 
\\
\notag
&
=\inf\{t\ge0:\ X_{\e,t}\notin D\} = \inf\{t\ge0:\ X_{\e,t}\in \partial D\}
\end{align}
as the hitting time for $\partial D=\partial_+\cup\partial_-$, where 
\begin{equation}
\label{eq:boundary-of-strip}
\partial_\pm=\{x=(x^1,x^2)\in\R^2: x^1=\pm R\}.
\end{equation} 

The main result of \cite{Kifer1981} states that $\tau_\e/(\frac{1}{\lambda}\log{\eps^{-1}})\inprob1$ and the distribution of the exit location~$X_\e(\tau_\e)$ asymptotically concentrates near the points of intersection of the unstable manifold with the boundary, i.e. points
\begin{equation*}
 q_\pm=(\pm R,0).
\end{equation*}

Let us analyze the exit problem in more detail. We start by using Duhamel's principle: 
\begin{align}
\label{eq:Duhamel1}X_{\e,t}^1&=e^{\lambda t}(X_{\e,0}^1+\e U^{1}_t),\\
\label{eq:Duhamel2}X_{\e,t}^2&=e^{-\mu t}(X_{\e,0}^2+\e U^{2}_t)=e^{-\mu t}X_{\e,0}^2+\e N^{2}_t,
\end{align}
where
\begin{align}
\notag
U^{1}_t&=\int_0^te^{-\lambda s}dW^1_s,
\\
\label{eq:U-in-model-case}
 U^{2}_t&=\int_0^te^{\mu s}dW^2_s,
\\
\notag
N^2_t&=e^{-\mu t} U_t^2= \int_0^te^{-\mu (t-s)}dW^2_s.
\end{align}
The process $(U^{1}_t,N^{2}_t)$ is Gaussian, so it easy to check that
\begin{equation}
\label{eq:lim-of-U}
(U^{1}_t,N^{2}_t)\indistr (\frU,\NN),\quad t\to\infty,
\end{equation}
where $(\frU,\NN)$ is a centered Gaussian random vector with independent components and variances
\begin{align}
\label{eq:var-of-N}
\cc_1&=\int_0^{\infty}e^{-2\lambda s}ds=(2\lambda)^{-1},
\\
\cc_2&= \int_{-\infty}^0 e^{2\mu s}ds=(2\mu)^{-1}.  \notag
\end{align}
In fact, a.s.-convergence holds for the first component in~\eqref{eq:lim-of-U}.

The definition~\eqref{eq:def-exit-time} and~\eqref{eq:Duhamel1} imply
\begin{equation}
\label{eq:for_tau}
R=e^{\lambda \tau_{\e}}|X_{\e,0}^1+\e U^1_{\tau_{\e}}|.
\end{equation}
It is easy to check that $\tau_{\e}\inprob\infty$ as $\eps\to 0$. Together with~\eqref{eq:lim-of-U}, this suggests (although more arguments are required for a rigorous proof):
\begin{equation}
\label{eq:weak_convergence_to_Gauss}
(U^1_{\tau_{\e}}, N^2_{\tau_{\e}})\indistr (\frU,\NN).
\end{equation} 
Therefore, expressing $\tau_{\e}$ from~\eqref{eq:for_tau}, we obtain
\begin{equation}
\label{eq:expr-for-tau}
\tau_{\e}=\frac{1}{\lambda}\log\frac{R}{|\e^\alpha\xi_\e+\e U^1_{\tau_{\e}}|}=\frac{1}{\lambda}\log\frac{R}{\e^\alpha |Z_\eps|},
\end{equation}
where
\begin{equation}
\label{eq:Z_eps}
Z_\e=\xi_\e+\e^{1-\alpha} U^1_{\tau_{\e}}.
\end{equation}
Thus,
\begin{equation}
\label{eq:typical-exit-time-distr}
\tau_{\e}-\frac{\alpha}{\lambda}\log \e^{-1}=
\frac{1}{\lambda}\log\frac{R}{|Z_\e|} 
\indistr \frac{1}{\lambda}\log\frac{R}{|Z|},
\end{equation}
where
\begin{equation*}
Z =\xi+\ONE_{\alpha=1}\frU
\end{equation*}
is the distributional limit of $Z_\e$ as $\e\to0$. Hence, the exit typically happens around time $\frac{\alpha}{\lambda}\log \e^{-1}$.

We also notice that the direction of exit is given by
\begin{equation}
\sgn X_{\e,\tau_\e}^1=\sgn Z_\e,
\label{eq:exit-direction}
\end{equation}
with distributional limit~$\sgn Z$.
In particular, the limiting probabilities of exit  on the right (i.e., through $\{+R\}\times\R$) and on the left  (i.e., through $\{-R\}\times\R$)
are $p_+=\Pp\{Z>0\}$ and $p_-=\Pp\{Z<0\}$. Note that  $\Pp\{Z=0\}=0$: if  $\alpha<1$, this is a consequence of \eqref{eq:no_atom_at_zero}; and if  $\alpha=1$, then $Z$ is absolutely continuous being a Gaussian convolution.

In particular, if $\alpha<1$ and $\Pp\{\xi>0\}=1$, we have $p_+=1$ and $p_-=0$ and if  $\Pp\{\xi<0\}=1$, then $p_+=0$ and $p_-=1$, i.e., the direction of exit is asymptotically deterministic as $\eps\to0$.

If the distribution of $\xi$ is symmetric, then $p_+=p_-=1/2$.

To find out the asymptotics of the  exit location distribution, we use~\eqref{eq:typical-exit-time-distr} in~\eqref{eq:Duhamel2}:
\begin{align}
\label{eq:X_2-at-exit-time}
X_{\e,\tau_\e}^2&=e^{-\mu \tau_\e}X_{\e,0}^2+\e N^{2}_{\tau_\e}
=\e^{\alpha\rho}\frac{L}{R^{\rho}}|Z_\e|^{\rho}+\e N^{2}_{\tau_\e},
\end{align}
where the \stindex  $\rho$ measuring the strength of contraction relative to expansion near the saddle point is defined by
\begin{equation}
\label{eq:rho}
\rho=\mu/\lambda.
\end{equation}
It plays a crucial role throughout the paper. Introducing 
\begin{equation}
\label{eq:new-alpha}
\alpha'=\alpha\rho \wedge 1,
\end{equation}
we obtain
\begin{equation}
\label{eq:scaling-for-exit}
\frac{X_{\e,\tau_\e}^2}{\e^{\alpha'}}\indistr \xi',
\end{equation}
where
\begin{equation}
\label{eq:scaling-limit-distr}
\xi'= \frac{L}{R^{\rho}}|Z|^{\rho}\ONE_{\alpha\rho\le 1}+\NN\ONE_{\alpha\rho\ge 1}.
\end{equation}

It is important to distinguish between the cases where (i)~$\alpha\rho<1$, (ii)~$\alpha\rho>1$, and (iii)~$\alpha\rho=1$.
If $\alpha\rho>1$, then $\alpha'=1$, $\xi'=\NN$, and  we can rewrite~\eqref{eq:scaling-for-exit} informally as
\begin{equation*}
X_{\e,\tau_\e}^2\stackrel{d}{\sim}\e \NN.
\end{equation*}
We recall that $\NN$ is a symmetric Gaussian r.v. The scaling factor in front of $\NN$ is~$\eps^1$, which is the same order of magnitude as the noise.

Note that the limiting behavior in this case does not depend on the initial condition, neither on $L$ nor on $(\xi_\e)$, nor on $\alpha>\rho^{-1}$.

If $\alpha'=\alpha\rho<1$, then $\xi'= \frac{L}{R^{\rho}}|Z|^{\rho}$, and  we can rewrite~\eqref{eq:scaling-for-exit} informally as
\begin{equation*}
X_{\e,\tau_\e}^2\stackrel{d}{\sim}\e^{\alpha'} \frac{L}{R^{\rho}}|Z|^{\rho}.
\end{equation*}
The distribution of $\xi'$ in this case is one-sided, i.e., it is concentrated on $(0,\infty)$ if $L>0$ and on $-(0,\infty)$ if $L<0$.
Moreover, the noise magnitude $\eps$ is smaller than the scaling~$\eps^{\alpha'}$.

In the intermediate case, $\alpha\rho=1$,  both terms in~\eqref{eq:scaling-limit-distr} are nonzero, so we obtain an asymmetric  distribution supported on the entire~$\R$.

The drastic difference in the asymptotic behavior may be explained as follows. If $\alpha\rho>1$, i.e., $\alpha/\lambda>1/\mu$ , the exit time of order $\frac{\alpha}{\lambda}\log \e^{-1}$ is long enough to allow the contraction (happening at exponential rate~$\mu$) to eliminate the dependence on the initial condition, whereas  if $\rho\alpha<1$, i.e., $\alpha/\lambda<1/\mu$, this exit time is so short that the exit typically happens sooner than the contraction along the stable manifold towards
the unstable one has taken place, hence the exit happens on the same side of the unstable manifold as the starting point. Let us also note that if the distribution of $\xi$ is absolutely continuous (has a Lebesgue density), then so is the distribution of  $\xi'$. Also, if $\alpha\rho \ge 1$, then the distribution of $\xi'$ is absolutely continuous, being either  Gaussian or  a Gaussian convolution.

The analysis above is done for a simplified system at a heuristic level. 
A rigorous general version of the reasoning and results above may be found in~\cite{Bakhtin2011} and~\cite{Almada-Bakhtin:MR2802310}. Let us give a summary, in the form of a theorem, of what we need to move~on. 
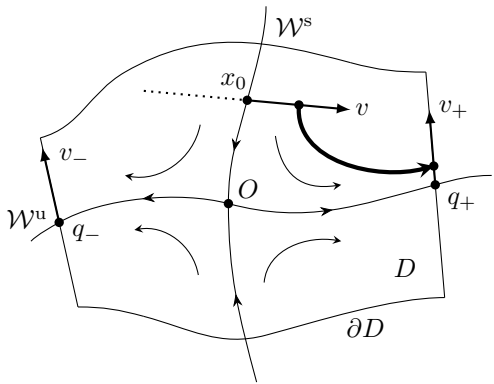
\begin{figure}[ht]
\centering
\begin{tikzpicture}

    \node (0) at (-0.625, 0.25) {};
    \node (1) at (2.875, 0.625) {};
    \node (2) at (-3.25, -0.25) {};
    \node (3) at (-3.125, 1.125) {};
    \node (4) at (-2.625, -1.125) {};
    \node (5) at (2, 2) {};
    \node (6) at (2.25, -1) {};
    \node (7) at (-2, 2) {};
    \node (8) at (1, 2.125) {};
    \node (9) at (-1.75, -1.25) {};
    \node (10) at (-0.125, -1.5) {};
    \node (11) at (-0.125, 2.875) {};
    \node (12) at (-0.375, 1.625) {};
    \node (14) at (-1, 1.3) {};
    \node (15) at (-2, 0.625) {};
    \node (16) at (0, 1.15) {};
    \node (17) at (0.875, 0.5) {};
    \node (18) at (-1.875, -0.1) {};
    \node (19) at (-1.025, -0.7) {};
    \node (20) at (-0.15, -0.8) {};
    \node (21) at (0.875, -0.25) {};
    \node (22) at (-0.25, -2.125) {};
    \node (23) at (-1.75, 1.75) {};
    \node (24) at (1, 1.5) {};
    \node (25) at (-2.875, 0) {};
    \node (26) at (2.125, 0.5) {};

    \node (27) at (0.3125, 1.5625) {}; \node (28) at (2.105, 0.75) {}; 

    \node [right] at (-0.625,0.45) {$O$};
    \node [left] at (2.7, 1.5) {$v_+$};
    \node [left] at (2.8, 0.28) {$q_+$};
    \node [right] at (-3, 0.86) {$v_-$};
    \node [right] at (-2.825, -0.15) {$q_-$};
    \node [right] at (0.94, 1.5) {$v$};
    \node [above] at (-0.55, 1.65) {$x_0$};
    \node [left] at (2, -0.6) {$D$};
    \node [below] at (1.2, -1.15) {$\partial D$};
    \node [above] at (-3.3, -0.2) {$\Wu$};
    \node [right] at (-0.125, 2.6) {$\Ws$};

\foreach \n in {0,12,25,26,27,28}
        \node at (\n)[circle,fill,inner sep=1.25pt]{};

    \draw [] [in=210, out=15] (3.center) to (7.center);
    \draw [] [in=165, out=30] (7.center) to (8.center);
    \draw [] [in=-180, out=-15] (8.center) to (5.center);
    \draw [] [in=165, out=0] (4.center) to (9.center);
    \draw [] [in=-165, out=-15] (9.center) to (10.center);
    \draw [] [in=180, out=15, looseness=0.75] (10.center) to (6.center);
    \draw [] [in=75, out=-90, looseness=0.75] (11.center) to (12.center);
    \draw [-stealth] [bend left=45] (14.center) to (15.center);
    \draw [-stealth] [bend right=45, looseness=1.25] (16.center) to (17.center);
    \draw [stealth-] [bend left=45] (18.center) to (19.center);

    \draw [ -stealth] [in=170, out=85] (20.center) to (21.center);

    \draw [postaction={on each segment={mid arrow=black}}] [in=90, out=-105] (12.center) to (0.center);
    \draw [postaction={on each segment={mid arrow=black}}] [in=-90, out=105] (22.center) to (0.center);
    \draw [thick, dotted] (23.center) to (12.center);
    \draw [thick, -latex] (12.center) to (24.center);
    \draw [] (3.center) to (25.center);
    \draw [] (25.center) to (4.center);
    \draw [postaction={on each segment={mid arrow=black}}] [in=30, out=165, looseness=0.75] (0.center) to (25.center);
    \draw [] [in=45, out=-150, looseness=0.75] (25.center) to (2.center);
    \draw [] (5.center) to (26.center);
    \draw [] (26.center) to (6.center);
    \draw [postaction={on each segment={mid arrow=black}}] [in=-165, out=-15] (0.center) to (26.center);
    \draw [] [in=-180, out=15] (26.center) to (1.center);

    \draw [thick, -latex] (25.center) to (-3.1, 1);
    \draw [thick, -latex] (26.center) to (2.045, 1.5);

    \draw [ultra thick, -stealth] [in=195, out=-95] (27.center) to (28.center);

\end{tikzpicture}
\caption{Dynamics near a saddle point}
{\label{fig:1saddle}}
\end{figure}

Let us first describe the setting and notation.

\begin{enumerate}[label={\bf(\Alph*)}, series=setting]
\item
\label{setting:general}
Let $X_\e$ solve equation~\eqref{eq:basic-sde} in $\R^2$.  
We assume that $b\in C^2_{\mathrm{b}}$. We assume that $\sigma\in C^3_\mathrm{b}$ and that it is  uniformly elliptic: the eigenvalues of
$\sigma(x)\sigma^*(x)$ are bounded away from zero. In particular, the flow $(\flow^t)_{t\in\R}$ generated by the vector field~$b$ is well defined by
\begin{align}
\begin{cases}\label{eq:flow}
\frac{d}{dt}\flow^t x&= b(\flow^t x),\quad t\in\R,\\
\flow^0x&=x.
\end{cases}
\end{align} 

\item\label{setting:geometry-domain}  
Suppose a simply connected domain $D$ with a simple closed boundary $\partial D$  contains~$O$,\
a hyperbolic critical point of $b$ with eigenvalues of the linearization of $b$ at~$O$ being $\lambda>0$ and $-\mu<0$. (We refer to Sections~2.7 and~2.8 of~\cite{Perko:MR1801796} for the basics of local theory near hyperbolic critical points: the Hadamard--Perron theorem,  invariant stable/unstable manifolds, Hartman--Grobman theorem.) Let $x_0\in D$ belong to the  stable manifold~$\Ws$ of~$O$
\[
\Ws=\left\{x\in\R^2: \lim_{t\to+\infty}\flow^tx=O\right\}.
\]
Let $v$ be a vector not collinear with  $b(x_0)$ and such that $x_0+[-1,1]v\subset D$. Let  $\Wu$ be the unstable manifold 
\[
\Wu=\left\{x\in\R^2: \lim_{t\to-\infty}\flow^tx=O\right\},
\]
and assume that on both sides of~$O$, it intersects $\partial D$ at points $q_\pm$ and there are no other points of intersection between $q_+$ and $q_-$.

Let us assume that there are vectors $v_\pm$  such that  $q_\pm+[-1,1]v_\pm\subset \partial  D$, and $v_\pm$ is not collinear with $b(q_\pm)$ (i.e., $\Wu$ is transversal to $\partial D$ at $q_\pm$). We also need to specify orientations for $v$ and $v_{\pm}$. We choose $v$ to point towards $q_+$
and $v_\pm$ to point towards $x_0$, see Figure~\ref{fig:1saddle}. We also require that if $y\in [-1,0)\cup (0,1]$,  then the trajectory $(\flow^t (x_0+yv))_{t\ge 0}$ exits $D$ transversally to $\partial D$ at $q_{+}+\pi(y)v_{+}$ (if $y>0$) or $q_{-}+\pi(y)v_{-}$ (if~$y<0$) for some $\pi(y)\in(-1,1)$.

\item \label{setting:initial-cond}
The initial condition satisfies
\begin{equation}
 X_{\e,0}=x_0+\e^\alpha \xi_\eps v,\quad \e>0,
\label{eq:entrance-scaling} 
\end{equation}
for some $\alpha\in(0,1]$ and a family $(\xi_\e)_{\e>0}$
of r.v.'s 
satisfying
$\eps^\alpha|\xi_\e|\le 1$ and
measurable with respect to~$\Fc_0$ (and thus independent of the noise realizations).
\item \label{setting:entrance-scaling-limit}  As $\e\to0$,  $\xi_\e$ converge in 
distribution to  a r.v.~$\xi$. If $\alpha<1$, then $\xi$ has no atom at~$0$, i.e., $\Pp\{\xi=0\}=0$.
\end{enumerate}

Conditions \ref{setting:initial-cond} and \ref{setting:entrance-scaling-limit} are tightly related to one another but 
in the coming sections it will be convenient to use them separately.

For each $\e>0$, we define
\[\tau_\e=\inf\{t\ge 0: X_{\e,t}\in\partial D\}.\]
and
\begin{equation}
\label{eq:def-Apm}
A_{\pm,\e}=\{X_{\e,\tau_\e}\in q_\pm+[-1,1]v_\pm\}.
\end{equation}

\begin{theorem}[\cite{Bakhtin2011},\cite{Almada-Bakhtin:MR2802310}]\label{th:poincare-saddle} 
Under assumptions~\ref{setting:general}, \ref{setting:geometry-domain}, \ref{setting:initial-cond}, \ref{setting:entrance-scaling-limit}, let us 
introduce  $\alpha'$ by~\eqref{eq:rho}--\eqref{eq:new-alpha} and define 
 r.v.'s $(\xi'_{\e})_{\e>0}$
 on $A_{\pm,\e}$ by
\begin{equation}
\label{eq:scaling-at-exit}
 X_{\e,\tau_\eps}=q_\pm+\e^{\alpha'}\xi'_{\eps}v_{\pm},
\end{equation}
and arbitrarily outside of $A_{+,\e}\cup A_{-,\e}$.

Then  there is
a r.v.~$\xi'$ with no atom at~$0$ and a partition into events $A_\pm$ (i.e., $A_+$ and $A_-$ are disjoint and $p_\pm=\Pp(A_\pm)$ satisfy
$p_++p_-=1$), such that 
\begin{enumerate}
\item 
As $\e\to 0$, 
\[
(\ONE_{A_{\e,+}},\ONE_{A_{\e,-}},\xi'_{\e})\indistr (\ONE_{A_+},\ONE_{A_-}, \xi').
\] 

\item \label{item:distr-of-xi-prime}
\begin{enumerate}
 \item
 If $\alpha'=\alpha\rho<1$, then  $\Pp\{\xi'>0\}=1$.
 \item If $\alpha\rho>1$,
then the distribution of $\xi'$ is symmetric Gaussian. 
\end{enumerate}

\item\label{item:prob-exit} 
 \begin{enumerate}
  \item If $\alpha<1$ and $\Pp\{\xi>0\}=1$, then $p_+=1$, $p_-=0$. 
  \item  If $\alpha<1$ and $\Pp\{\xi<0\}=1$, then $p_+=0$, $p_-=1$. \item
If the distribution of $\xi$ is symmetric, then $p_+=p_-=1/2$.
\end{enumerate}
\item\label{item:exit_time_is_log}  As $\e\to 0$, 
\begin{equation}
\label{eq:exit-time-rough-asymp}
\frac{\tau_\e}{\frac{\alpha}{\lambda}\log \e^{-1}}\inprob 1. 
\end{equation}
\end{enumerate}
\end{theorem}
In fact, more precise asymptotics for the exit time similar to~\eqref{eq:typical-exit-time-distr} is available but for our purposes,~\eqref{eq:exit-time-rough-asymp}~is sufficient.

One can say that this lemma describes the asymptotics of the random Poincar\'e map defined by the system~\eqref{eq:linear-system1},\eqref{eq:linear-system2} in the neighborhood of the origin. It claims that if the entrance point to the neighborhood satisfies the scaling relation~\eqref{eq:entrance-scaling}, then as $\e\to 0$, the probabilities of exiting along the branches of the invariant manifold associated to the main eigenvalue of the linearization stabilize to limiting values $p_+$ and $p_-=1-p_+$. For each of the two points of concentration of the exit distribution, the random exit point in its vicinity satisfies a scaling relation of the same
type~ \eqref{eq:scaling-at-exit}, with new scaling $\e^{\alpha'}$ in front of a random vector $\xi'v_+$ on $A_+$ and  $\xi'v_-$ on $A_-$. 

To prove this theorem, one must apply a simplifying conjugacy. According to the Hartman--Grobman theorem, for every hyperbolic critical point, there is a continuous change of coordinates in a sufficiently small neighborhood conjugating the dynamics generated by~\eqref{eq:basic-ode} to linear dynamics. Typically, this conjugacy possesses more smoothness, so one can apply the It\^o calculus and obtain, in new coordinates, a system similar to 
\eqref{eq:linear-system1}--\eqref{eq:linear-system2} but with small corrections and possibly non-diagonal diffusion matrix. This was studied in~\cite{Bakhtin2011}. In special resonant cases, conjugacy to a linear system is impossible, the appropriate normal
form contains resonant monomials of higher order but they also can be controlled and that was done in~\cite{Almada-Bakhtin:MR2802310}.

\subsection{Dynamics along heteroclinic connections}

In principle, one can take the domain $D$ to contain an arbitrarily large piece of the unstable manifold, and that is how we are going to proceed studying saddle after saddle. However, it is useful to remind that this is due to the fact that it typically takes nearly constant time to travel between neighborhoods of saddle points, and the character of the scaling does not change during this period. This is a special case of a more general and detailed theorem from~\cite{Almada-Bakhtin:MR2739004}: 
\begin{theorem}\label{th:along-hetero} 
 Let $D\subset \R^2$ be a domain with simple closed boundary $\partial D$. 
Suppose that $q\in D$
and  assume that the solution of the deterministic  equation~\eqref{eq:basic-ode} started at $q$ 
reaches  $\partial D$  in finite time~$T$ at a point ~$y$. We assume that there is  a vector $u$ not collinear with~$b(y)$ 
such that $y+[-1,1]u\subset \partial D$. 

Let $X_\e$, $\e>0$, solve the SDE~\eqref{eq:basic-sde} with initial condition
\[
X_{\e,0}=q+ \e^{\alpha}\xi_\e v,
\]
where $\alpha\in(0,1]$,  $v$ is a  vector not collinear with $b(q)$ and 
 r.v.'s $\xi_\e$ converge in distribution to some $\xi$ with no atom at $0$. 
 We assume that $u$ and $v$ point to the same side of the orbit of $q$, see Figure~\ref{fig:levinson}.
 
\begin{figure}[ht]
\centering
\begin{tikzpicture}

    \node (2) at (-4, 0) {};
    \node (3) at (4, 0) {};
    \node (4) at (2.75, 1.5) {};
    \node (5) at (2.75, -1.25) {};
    \node (6) at (-1.75, 1) {};
    \node (7) at (-1.95, -0.25) {};
    \node (q) at (-1.86, 0.34) {};
    \node (y) at (2.75, 0.045) {};
    \node (u) at (2.75, 1) {};
    \node (9) at (-1.5, 1.75) {};
    \node (10) at (-1.75, -1.25) {};

\foreach \n in {q,y}
        \node at (\n)[circle,fill,inner sep=1.25pt]{};

    \draw [thick, postaction={on each segment={mid arrow=black}}] [in=180, out=15] (2.center) to (3.center);
    \draw [thick][in=15, out=150, looseness=0.75] (4.center) to (9.center);
    \draw [thick][in=165, out=-165, looseness=1.50] (9.center) to (10.center);
    \draw [thick][in=-165, out=-15, looseness=0.75] (10.center) to (5.center);
    \draw [thick] (4.center) to (5.center);
    \draw [thick, dotted] (7.center) to (q.center);
    \draw [ultra thick, -latex] (q.center) to (6.center);
    \draw [ultra thick, -latex] (y.center) to (u.center);

    \node [left] at (2.75, 0.26) {$y$};
    \node [left] at (2.75, 1) {$u$};
    \node [left] at (1, 1) {$D$};
    \node [right] at (2.75, -0.8) {$\partial D$};
    \node [left] at (-1.86, 0.55) {$q$};
    \node [above] at (-1.75, 1) {$v$};
\end{tikzpicture}

\caption{Exit problem along a heteroclinic connection away from saddle points.}
\label{fig:levinson}
\end{figure}
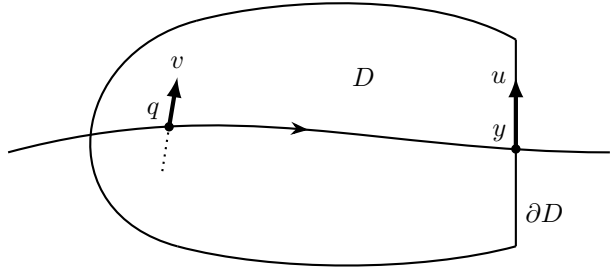

 Let \[\tau_\e=\inf\{t\ge 0:\ X_{\e,t}\in\partial D\}.\]

Then 
\[\tau_\e\inprob T,\]
and 
there  are numbers $a,b>0$ and r.v.'s $(\eta_\e)_{\e>0}$ such that
\[
\Pp\{X_{\e,\tau_\e}=y+\e^{\alpha}\eta_\e u\}\to 1,
\]
and
\[
\eta_\e\indistr a\xi  + b N \ONE_{\alpha=1},\quad \e \to 0,
\]
where   $N$ is a standard Gaussian r.v.\   independent of $\xi$. 
\end{theorem}

The main ingredient in the proof of this theorem is the linearization of the stochastic dynamics along the deterministic orbit of $x_0$. We will use the same approach to prove  Lemmas~\ref{lem:expansion} and~\ref{lem:scal_lim} extending Theorem~\ref{th:along-hetero}.
 
\subsection{A heteroclinic chain}We continue rewriting the results of~\cite{Bakhtin2011} in a convenient way, also preparing the ground for the new results. Now we can apply Theorem~\ref{th:poincare-saddle} iteratively and compute the asymptotic probability of traveling
along each finite pathway through the graph of heteroclinic connections. 

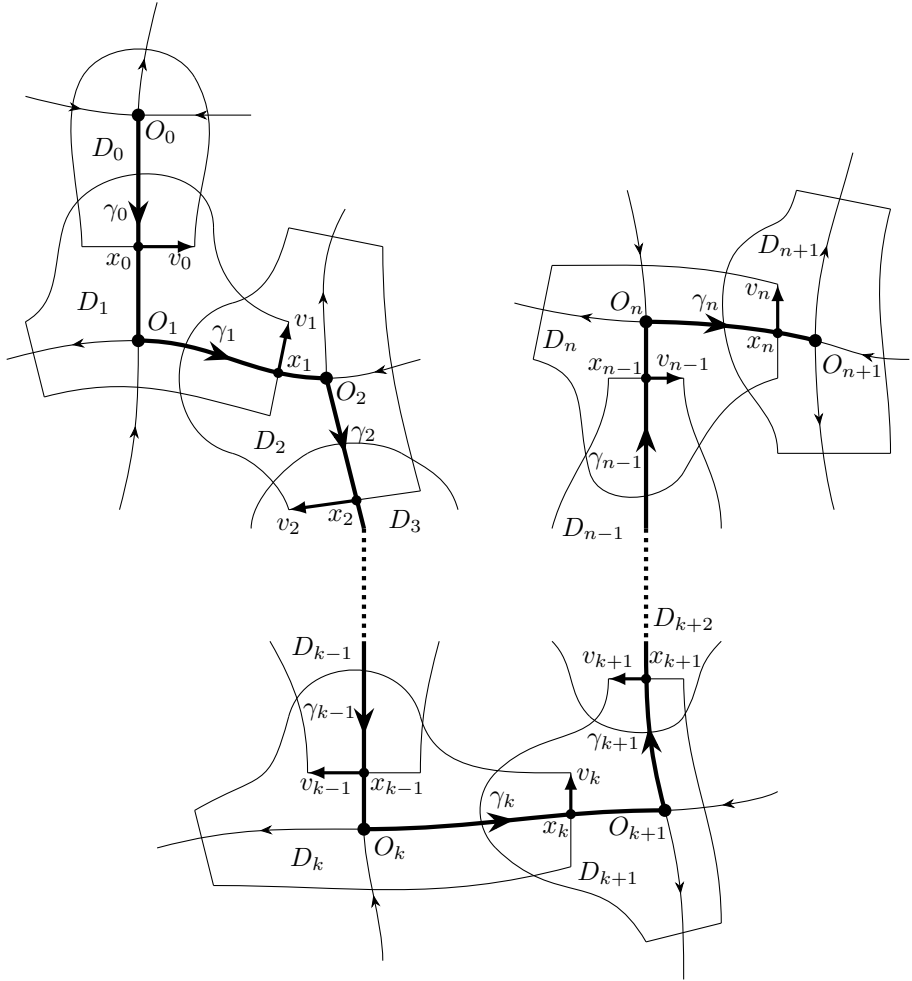
\begin{figure}[ht]
\centering
\begin{tikzpicture}

    \node (0) at (-5.75, 6.25) {};
    \node (1) at (-4.25, 6) {};
    \node (2) at (-4, 7.5) {};
    \node (3) at (-2.75, 6) {};
    \node (4) at (-4.25, 3) {};
    \node (5) at (-6, 2.75) {};
    \node (6) at (-4.5, 0.75) {};
    \node (7) at (-1.75, 2.5) {};
    \node (8) at (-1.5, 4.75) {};
    \node (9) at (-0.5, 2.75) {};
    \node (10) at (-1.25, 0.5) {};
    \node (11) at (0.75, 3.5) {};
    \node (12) at (2.5, 3.25) {};
    \node (13) at (2.25, 5) {};
    \node (14) at (4.75, 3) {};
    \node (15) at (5.25, 5.5) {};
    \node (16) at (6, 2.75) {};
    \node (17) at (5, 0.75) {};
    \node (18) at (2.5, 0.5) {};
    \node (19) at (-1.25, -1) {};
    \node (20) at (-1.25, -3.5) {};
    \node (21) at (-4, -3.75) {};
    \node (22) at (-1, -5.25) {};
    \node (23) at (2.75, -3.25) {};
    \node (24) at (2.5, -1) {};
    \node (25) at (4.25, -3) {};
    \node (26) at (3, -5.5) {};
    \node (27) at (-5, 4.25) {};
    \node (28) at (-5, 6.5) {};
    \node (29) at (-3.5, 6.5) {};
    \node (30) at (-3.5, 4.25) {};
    \node (31) at (-5.75, 3.25) {};
    \node (32) at (-5.5, 2.25) {};
    \node (33) at (-5.25, 4.5) {};
    \node (34) at (-3.25, 4.5) {};
    \node (35) at (-2.25, 3.25) {};
    \node (36) at (-2.5, 2) {};
    \node (37) at (-2.25, 4.5) {};
    \node (38) at (-1, 4.25) {};
    \node (39) at (-3, 3.5) {};
    \node (40) at (-3, 1.5) {};
    \node (41) at (-2.25, 0.75) {};
    \node (42) at (-0.5, 1) {};
    \node (43) at (-2.75, 0.5) {};
    \node (44) at (-2, 1.5) {};
    \node (45) at (-0.75, 1.5) {};
    \node (46) at (0, 0.75) {};
    \node (47) at (1.25, 4) {};
    \node (48) at (1, 2.75) {};
    \node (49) at (4.25, 3.75) {};
    \node (50) at (4.25, 2.5) {};
    \node (51) at (1.75, 1.25) {};
    \node (52) at (3, 1.25) {};
    \node (53) at (2, 2.5) {};
    \node (54) at (3, 2.5) {};
    \node (55) at (1.25, 0.5) {};
    \node (56) at (3.5, 0.5) {};
    \node (57) at (4.5, 5) {};
    \node (58) at (5.75, 4.75) {};
    \node (59) at (5.75, 1.5) {};
    \node (60) at (4.25, 1.5) {};
    \node (61) at (3.75, 4.25) {};
    \node (62) at (3.75, 2.5) {};
    \node (63) at (-2.5, -1) {};
    \node (64) at (-2, -2.75) {};
    \node (65) at (-0.5, -2.75) {};
    \node (66) at (-0.25, -1) {};
    \node (67) at (-3.5, -3.25) {};
    \node (68) at (-3.25, -4.25) {};
    \node (69) at (-2.25, -2) {};
    \node (70) at (-0.5, -2) {};
    \node (71) at (1.5, -2.75) {};
    \node (72) at (1.5, -4) {};
    \node (73) at (0.5, -2.75) {};
    \node (74) at (0.5, -3.75) {};
    \node (75) at (2.5, -5) {};
    \node (76) at (3.5, -4.75) {};
    \node (77) at (3, -1.5) {};
    \node (78) at (2, -1.5) {};
    \node (79) at (1.25, -1) {};
    \node (80) at (1.75, -2) {};
    \node (81) at (3, -2) {};
    \node (82) at (3.5, -1) {};

    \node (x0) at (-4.25, 4.25) {};
    \node (v0) at (-3.5, 4.25) {};
    \node (x1) at (-2.39, 2.57) {};
    \node (v1) at (-2.25, 3.25) {};
    \node (x2) at (-1.35, 0.875) {};
    \node (v2) at (-2.25, 0.75) {};
    \node (xk-1) at (-1.25, -2.75) {};
    \node (vk-1) at (-2, -2.75) {};
    \node (xk) at (1.5, -3.3) {};
    \node (vk) at (1.5, -2.75) {};
    \node (xk+1) at (2.5, -1.5) {};
    \node (vk+1) at (2, -1.5) {};
    \node (xn-1) at (2.5, 2.5) {};
    \node (vn-1) at (3, 2.5) {};
    \node (xn) at (4.25, 3.1) {};
    \node (vn) at (4.25, 3.75) {};

\foreach \n in {1, 4, 7, 12, 14, 20, 23}
        \node at (\n)[circle,fill,inner sep=1.75pt]{};

        \foreach \n in {x0, x1, x2, xk-1, xk, xk+1, xn-1,xn}
        \node at (\n)[circle,fill,inner sep=1.35pt]{};

\draw [very thick, -latex] (x0.center) to (v0.center);
    \draw [very thick, -latex] (x1.center) to (v1.center);
    \draw [very thick, -latex] (x2.center) to (v2.center);
    \draw [very thick, -latex] (xk-1.center) to (vk-1.center);
    \draw [very thick, -latex] (xk.center) to (vk.center);
    \draw [very thick, -latex] (xk+1.center) to (vk+1.center);
    \draw [very thick, -latex] (xn-1.center) to (vn-1.center);
    \draw [very thick, -latex] (xn.center) to (vn.center);

    \draw [postaction={on each segment={mid arrow=black}}] [in=180, out=-15, looseness=0.75] (0.center) to (1.center);
    \draw [postaction={on each segment={mid arrow=black}}] [in=360, out=180] (3.center) to (1.center);
    \draw [postaction={on each segment={mid arrow=black}}] [in=-105, out=90, looseness=0.75] (1.center) to (2.center);
    \draw [ultra thick, postaction={on each segment={mid arrow=black}}] (1.center) to (4.center);
    \draw [postaction={on each segment={mid arrow=black}}] [in=15, out=180] (4.center) to (5.center);
    \draw [ultra thick, postaction={on each segment={mid arrow=black}}] [in=180, out=0] (4.center) to (7.center);
    \draw [postaction={on each segment={mid arrow=black}}] [in=360, out=-165] (9.center) to (7.center);
    \draw [postaction={on each segment={mid arrow=black}}] [in=270, out=75] (6.center) to (4.center);
    \draw [postaction={on each segment={mid arrow=black}}] [in=-120, out=90] (7.center) to (8.center);
    \draw [ultra thick, postaction={on each segment={mid arrow=black}}] (7.center) to (10.center);
    \draw [postaction={on each segment={mid arrow=black}}] [in=-15, out=180, looseness=0.75] (12.center) to (11.center);
    \draw [ultra thick, postaction={on each segment={mid arrow=black}}] [in=165, out=0, looseness=0.75] (12.center) to (14.center);
    \draw [postaction={on each segment={mid arrow=black}}] [in=90, out=-75, looseness=0.75] (13.center) to (12.center);
    \draw [ultra thick, postaction={on each segment={mid arrow=black}}] [in=270, out=90] (18.center) to (12.center);
    \draw [postaction={on each segment={mid arrow=black}}] [in=-15, out=180, looseness=1.25] (16.center) to (14.center);
    \draw [postaction={on each segment={mid arrow=black}}] [in=255, out=90] (14.center) to (15.center);
    \draw [postaction={on each segment={mid arrow=black}}] [in=105, out=-90, looseness=0.75] (14.center) to (17.center);
    \draw [postaction={on each segment={mid arrow=black}}] [in=15, out=180] (20.center) to (21.center);
    \draw [ultra thick, postaction={on each segment={mid arrow=black}}] [in=90, out=-90] (19.center) to (20.center);
    \draw [postaction={on each segment={mid arrow=black}}] [in=-90, out=90, looseness=0.75] (22.center) to (20.center);
    \draw [ultra thick, postaction={on each segment={mid arrow=black}}] [in=180, out=0] (20.center) to (23.center);
    \draw [postaction={on each segment={mid arrow=black}}] [in=0, out=-150, looseness=0.50] (25.center) to (23.center);
    \draw [ultra thick, postaction={on each segment={mid arrow=black}}] [in=270, out=105] (23.center) to (24.center);
    \draw [postaction={on each segment={mid arrow=black}}] [in=90, out=-75] (23.center) to (26.center);

    \draw [in=-120, out=90, looseness=0.75] (27.center) to (28.center);
    \draw [in=120, out=60] (28.center) to (29.center);
    \draw [in=90, out=-60] (29.center) to (30.center);
    \draw (27.center) to (30.center);
    \draw (31.center) to (32.center);
    \draw [in=-105, out=30] (31.center) to (33.center);
    \draw [in=90, out=75, looseness=1.25] (33.center) to (34.center);
    \draw [in=165, out=-90] (34.center) to (35.center);
    \draw (35.center) to (36.center);
    \draw [in=15, out=165] (36.center) to (32.center);
    \draw [in=15, out=-105] (37.center) to (39.center);
    \draw [in=165, out=-165, looseness=1.25] (39.center) to (40.center);
    \draw [in=105, out=-15, looseness=0.75] (40.center) to (41.center);
    \draw (41.center) to (42.center);
    \draw [in=270, out=105] (42.center) to (38.center);
    \draw (38.center) to (37.center);
    \draw [in=-150, out=90, looseness=0.75] (43.center) to (44.center);
    \draw [in=150, out=30, looseness=0.75] (44.center) to (45.center);
    \draw [in=105, out=-30] (45.center) to (46.center);
    \draw (47.center) to (48.center);
    \draw [in=105, out=-15, looseness=1.25] (48.center) to (51.center);
    \draw [in=-120, out=-75] (51.center) to (52.center);
    \draw [in=-165, out=60] (52.center) to (50.center);
    \draw (49.center) to (50.center);
    \draw [in=165, out=0] (47.center) to (49.center);
    \draw (53.center) to (54.center);
    \draw [in=-90, out=75] (55.center) to (53.center);
    \draw [in=90, out=-90] (54.center) to (56.center);
    \draw (57.center) to (58.center);
    \draw [in=90, out=-105] (58.center) to (59.center);
    \draw (59.center) to (60.center);
    \draw [in=-60, out=90] (60.center) to (62.center);
    \draw [in=-135, out=120, looseness=0.75] (62.center) to (61.center);
    \draw [in=-105, out=45, looseness=1.50] (61.center) to (57.center);
    \draw [in=90, out=-60] (63.center) to (64.center);
    \draw (64.center) to (65.center);
    \draw [in=255, out=90] (65.center) to (66.center);
    \draw (67.center) to (68.center);
    \draw [in=-165, out=0] (68.center) to (72.center);
    \draw (72.center) to (71.center);
    \draw [in=-105, out=30, looseness=1.50] (67.center) to (69.center);
    \draw [in=105, out=75, looseness=1.25] (69.center) to (70.center);
    \draw [in=180, out=-75] (70.center) to (71.center);
    \draw [in=45, out=-90, looseness=1.25] (78.center) to (73.center);
    \draw [in=135, out=-135] (73.center) to (74.center);
    \draw [in=120, out=-45, looseness=1.25] (74.center) to (75.center);
    \draw (75.center) to (76.center);
    \draw [in=-90, out=90] (76.center) to (77.center);
    \draw (77.center) to (78.center);
    \draw [in=135, out=-45] (79.center) to (80.center);
    \draw [in=-120, out=-45, looseness=0.75] (80.center) to (81.center);
    \draw [in=-135, out=60, looseness=1.25] (81.center) to (82.center);

    \draw [ultra thick, dotted] (10.center) to (19.center);
    \draw [ultra thick, dotted] (18.center) to (24.center);

    \node [below] at (-3.95, 6.05) {$O_0$};
    \node at (-4.65, 5.55) {$D_0$};
    \node [left] at (-4.25, 4.7) {$\gamma_0$};
    \node [left] at (-4.2, 4.05) {$x_0$};
    \node [left] at (-3.4, 4.05) {$v_0$};

    \node [right] at (-2.4, 2.7) {$x_1$};
    \node [right] at (-2.3, 3.25) {$v_1$};
    \node [right] at (-4.25, 3.23) {$O_1$};
    \node [left] at (-4.5, 3.5) {$D_1$};
    \node [above] at (-3.1, 2.85) {$\gamma_1$};

    \node [below] at (-1.55, 0.875) {$x_2$};
    \node [below] at (-2.25, 0.75) {$v_2$};
    \node [right] at (-1.75, 2.3) {$O_2$};
    \node at (-2.5, 1.65) {$D_2$};
    \node at (-1.25, 1.75) {$\gamma_2$};

    \node [below] at (-0.7, 0.875) {$D_3$};

    \node [below] at (-0.8, -2.7) {$x_{k-1}$};
    \node [below] at (-1.75, -2.7) {$v_{k-1}$};
    \node at (-1.8, -1.1) {$D_{k-1}$};
    \node [below] at (-1.7, -1.7) {$\gamma_{k-1}$};

    \node [right] at (1, -3.5) {$x_k$};
    \node [right] at (1.45, -2.8) {$v_k$};
    \node  at (-0.9, -3.75) {$O_k$};
    \node  at (-2, -3.9) {$D_k$};
    \node [above] at (0.6, -3.35) {$\gamma_k$};

    \node [above] at (2.9, -1.5) {$x_{k+1}$};
    \node [above] at (2, -1.5) {$v_{k+1}$};
    \node at (2.4, -3.5) {$O_{k+1}$};
    \node at (2, -4.1) {$D_{k+1}$};
    \node [left] at (2.55, -2.35) {$\gamma_{k+1}$};

    \node [above] at (3, -1) {$D_{k+2}$};

    \node [above] at (2.1, 2.4) {$x_{n-1}$};
    \node [above] at (3, 2.45) {$v_{n-1}$};
    \node [above] at (1.8, 0.2) {$D_{n-1}$};
    \node [above] at (2.1, 1.15) {$\gamma_{n-1}$};

    \node [left] at (4.35, 2.95) {$x_n$};
    \node [left] at (4.3, 3.65) {$v_n$};
    \node at (2.25, 3.5) {$O_n$};
    \node at (1.36, 3) {$D_n$};
    \node at (3.3, 3.5) {$\gamma_n$};

    \node at (5.25, 2.65) {$O_{n+1}$};
    \node at (4.4, 4.25) {$D_{n+1}$};
\end{tikzpicture}

\caption{A heteroclinic chain.}
{\label{fig:heteroc-chain}}
\end{figure}

Let us describe the geometric setup first, see Figure~\ref{fig:heteroc-chain}.
\begin{enumerate}[label={\bf(\Alph*)}, resume*=setting]
\item\label{setting:het-chain}
There is a sequence of saddle points $O_0,O_1,O_2,\ldots,O_n,O_{n+1}$ (in principle, one does not really need points $O_0$ and $O_{n+1}$ but we include them for notational convenience)
with eigenvalues of linearization at $O_k$ being $\lambda_k>0$ and $-\mu_k<0$,  \stindex $\rho_k=\mu_k/\lambda_k$,
and heteroclinic orbits $\gamma_0,\gamma_1,\ldots,\gamma_{n}$ between them as on Figure~\ref{fig:heteroc-chain}, so
that $\gamma_k$ connects $O_k$ to $O_{k+1}$ for $k=0,\ldots,n$. These heteroclinic connections are said to form a
{\it heteroclinic chain}.

For each $k=0,\ldots,n+1$, we plot a domain $D_k$ containing $O_k$, so that
for all $k=0,\ldots,n$, the following holds:
$D_k\cap \gamma_k$ and $D_{k+1}\cap\gamma_{k}$ are connected sets;
$D_{k}\cap D_{k+1} \cap \gamma_k\ne\emptyset$;  $\partial D_k$ intersects $\gamma_k$ at a point $x_k$ transversally, moreover,
there is a vector $v_k$ not collinear with $b(x_k)$ such that $x_k+[-1,1]v_k\subset \partial D_k$. For $k\ge 1$ out of two possible directions for $v_k$ we choose~$v_k$ to point towards $\gamma_{k-1}$. For $k=0$, out of two possible directions for $v_0$ we choose $v_0$ to point towards~$\gamma_1$.

\item \label{setting:scaling-at-saddle-0} The diffusion starts near $x_0$ and, for some $\alpha_0\in(0,1]$ and r.v. $\xi_{0,\e}$,  satisfies 
\begin{equation}
\label{eq:initial-condition-near-chain}
X_{\e,0}=x_0+\e^{\alpha_0}\xi_{0,\e}v_0,\quad \e>0.
\end{equation}
\item \label{setting:scaling-limit-at-saddle-0}
As $\e\to0$, $\xi_{0,\e}$ converges in distribution to a r.v.\  $\xi_0$.  If  $\alpha_0<1$, then we require  that  $\Pp\{\xi_0=0\}=0$.
\end{enumerate}

Let us define a sequence of stopping times $(\tau^k_\e)_{k=0,\dots,n}$ iteratively:
we set $\tau^0_\e=0$ and then, sequentially, for  $k=1,\ldots,n$, we  set
\begin{gather}
    \nu^k_\e=\inf\{t\ge \tau^{k-1}_\e: X_\e(t)\in\partial D_k\}, \notag\\
    \tau^k_\e=\begin{cases}
    \nu^k_\e,&  \text{\rm \  if\ } \nu^k_\e <\infty \text{\rm \  and\ } X_\e(\nu^k_\eps)\in x_k+[-1,1]v_k,\\
    \infty,& \text{\rm otherwise}.
    \end{cases}\label{eq:tau^k_eps}
\end{gather}

Using the strong Markov property and Theorem~\ref{th:poincare-saddle} iteratively, we obtain sequences 
  $(\alpha_k)_{k=0,\ldots,n}$ of exponents, r.v.'s $(\xi_{k,\e})_{k=1,\ldots,n;\ \eps>0}$, 
 $(\xi_k)_{k=1,\ldots,n}$ and events
  $(A_k)_{k=1,\ldots,n}$ such that 
\begin{equation}
\label{eq:seq-scaling}
X_{\e,\tau^k_\e}=x_k+\e^{\alpha_k}\xi_{k,\e}v_k,\quad \e>0,  
\end{equation}
holds on
\begin{align}\label{eq:A_k}
    A_{k,\e}=\{\tau^k_\eps<\infty\}
\end{align}
for $k=0,1,\ldots,n$, 
and
\begin{equation}
\label{eq:distr-conv-in-scaling}
(\ONE_{A_{k,\e}},\ONE_{A_{k,\e}}\xi_{k,\e})\indistr(\ONE_{A_k},\ONE_{A_k}\xi_k),\quad \e\to 0.
\end{equation}

Due to ~\eqref{eq:new-alpha},the sequence $(\alpha_k)_{k=0,\ldots,n}$ of exponents satisfies a recursive relation
\begin{equation}
\label{eq:recursion_alpha}
\alpha_k=(\alpha_{k-1}\rho_k) \wedge 1.
\end{equation}

The relations~\eqref{eq:seq-scaling},~\eqref{eq:distr-conv-in-scaling} and the definition~\eqref{eq:recursion_alpha} are really  meaningful only if $p_k>0$,
where
\begin{equation}
\label{eq:prob-of-sequence}
p_k=\Pp(A_k)=\lim_{\e \to 0}\Pp(A_{k,\e}),
\end{equation}
is always well-defined.

However there are multiple situations where $p_k=0$. In fact, it follows from Theorem~\ref{th:poincare-saddle}~\eqref{item:distr-of-xi-prime} that if for some $k$, 
$\alpha_k<1$, then, conditioned on $A_k$,  the distribution of $\xi_k$ is concentrated on $(0,+\infty)$. Theorem~\ref{th:poincare-saddle}~\eqref{item:prob-exit} implies now that if  $\gamma_{k-1}$ and $\gamma_{k+1}$ are on the opposite sides of $\gamma_k$, i.e.,
  the union of curves $\gamma_{k-1},\gamma_{k},\gamma_{k+1}$ is N-shaped (see, e.g., curves $\gamma_0,\gamma_1,\gamma_2$ in Figure~\ref{fig:heteroc-chain})
then $p_k=0$. 

In other words, due to insufficient contraction near $O_{k}$ (and the preceding saddles of the heteroclinic chain), the probability of crossing the  heteroclinic connection~$\gamma_k$
while traveling along it from $O_k$ to $O_{k+1}$ is asymptotically zero, so while the 
diffusion near the
heteroclinic chain experiences insufficient contraction
(i.e., $\alpha_k<1$ for exponents $\alpha_k$ defined via~\eqref{eq:recursion_alpha}) it will typically stay on one side of the heteroclinic chain.  However, once a value $\alpha_k=1$ is reached due to the presence of strong contraction 
(the stability index $\rho_k$ is large enough to ensure $\alpha_{k-1}\rho_{k}\ge 1$), there is a nonvanishing positive  chance to cross $\gamma_k$.

We can summarize the above as a theorem:
\begin{theorem}\label{thm:typical}
 Under the conditions~\ref{setting:general}, \ref{setting:het-chain}, \ref{setting:scaling-at-saddle-0}, \ref{setting:scaling-limit-at-saddle-0}, the following holds true:
 \begin{enumerate} 
 \item
The numbers $p_k, k=1,\ldots,n$, are well-defined by~\eqref{eq:prob-of-sequence}. 
\item If $p_{k}=0$ for some $k\in\{1,2,\ldots,n-1\}$, then $p_{k+1}=p_{k+2}=\ldots=p_n=0$.
\item Suppose $p_{k}>0$ for some $k\in\{1,2,\ldots,n-1\}$.
\begin{enumerate}
\item  \label{item:if-exp=1-bifurcation} If  $\alpha_k=1$, then $0<p_{k+1}<1$.
\item  If $\alpha_k<1$ and if $\gamma_{k-1}$ and $\gamma_{k+1}$ are on the same side of $\gamma_k$, then $p_{k+1}=p_k$.
\item \label{item:improbable-trans} If  $\alpha_k<1$, and if $\gamma_{k-1}$ and $\gamma_{k+1}$ are on the opposite sides of $\gamma_k$, then
$p_{k+1}=0$.
\item \label{item:time-heteroc-chain} If $p_n>0$, then, conditioned on $A_{n,\e}$,  
\[
\frac{\tau^n_\e}{\chi \log \e^{-1}} \inprob 1,\quad \e\to0,
\]
where
\[
\chi=\sum_{i=1}^n \frac{\alpha_{i-1}}{\lambda_i}.
\]
\end{enumerate} 
\end{enumerate}
\end{theorem}
Part \ref{item:time-heteroc-chain} of Theorem~\ref{thm:typical} means that this theorem is relevant for  time scales logarithmic in $\e^{-1}$.
 It describes 
typical and unlikely sequences of heteroclinic connections followed by the diffusion over those times.
However, it does not describe the rate of the improbable transitions or the mechanism of their emergence and thus implies little for longer time scales. The quantitative analysis of asymptotically improbable transitions described in 
part~\ref{item:improbable-trans} of the theorem is the main goal of this paper.

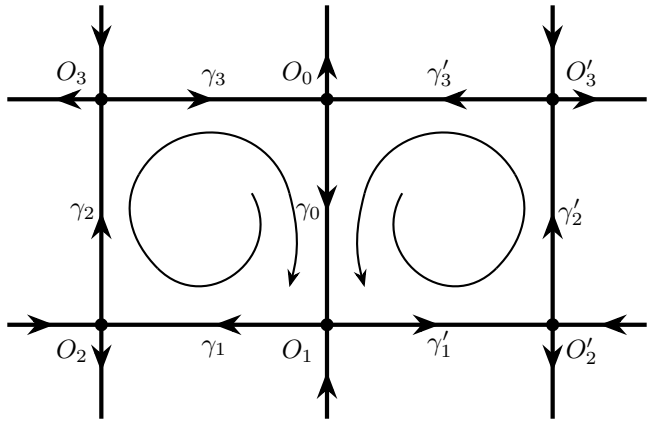
\begin{figure}[ht]
\centering
\begin{tikzpicture}

    \node (0) at (0, 0) {};
    \node (1) at (3, 0) {};
    \node (2) at (3, 3) {};
    \node (3) at (0, 3) {};
    \node (4) at (-3, 3) {};
    \node (5) at (-3, 0) {};
    \node (6) at (-3, -1.25) {};
    \node (7) at (0, -1.25) {};
    \node (8) at (3, -1.25) {};
    \node (9) at (4.25, 0) {};
    \node (10) at (4.25, 3) {};
    \node (11) at (3, 4.25) {};
    \node (12) at (0, 4.25) {};
    \node (13) at (-3, 4.25) {};
    \node (14) at (-4.25, 3) {};
    \node (15) at (-4.25, 0) {};
    \node (16) at (-1, 1.75) {};
    \node (18) at (-2.25, 0.75) {};
    \node (19) at (-2.5, 2) {};
    \node (21) at (-0.5, 1.75) {};
    \node (22) at (-0.5, 0.5) {};
    \node (23) at (1, 1.75) {};
    \node (24) at (2.25, 0.75) {};
    \node (25) at (2.5, 2) {};
    \node (26) at (0.5, 1.75) {};
    \node (27) at (0.5, 0.5) {};

\foreach \n in {0, 1, 2, 3, 4, 5}
        \node at (\n)[circle,fill,inner sep=1.75pt]{};

    \draw [ultra thick, postaction={on each segment={mid arrow=black}}] (3.center) to (0.center);
    \draw [ultra thick, postaction={on each segment={mid arrow=black}}] (0.center) to (1.center);
    \draw [ultra thick, postaction={on each segment={mid arrow=black}}] (0.center) to (5.center);
    \draw [ultra thick, postaction={on each segment={mid arrow=black}}] (5.center) to (4.center);
    \draw [ultra thick, postaction={on each segment={mid arrow=black}}] (1.center) to (2.center);
    \draw [ultra thick, postaction={on each segment={mid arrow=black}}] (4.center) to (3.center);
    \draw [ultra thick, postaction={on each segment={mid arrow=black}}] (2.center) to (3.center);
    \draw [ultra thick, postaction={on each segment={mid arrow=black}}] (2.center) to (10.center);
    \draw [ultra thick, postaction={on each segment={mid arrow=black}}] (4.center) to (14.center);
    \draw [ultra thick, postaction={on each segment={mid arrow=black}}] (9.center) to (1.center);
    \draw [ultra thick, postaction={on each segment={mid arrow=black}}] (15.center) to (5.center);
    \draw [ultra thick, postaction={on each segment={mid arrow=black}}] (5.center) to (6.center);
    \draw [ultra thick, postaction={on each segment={mid arrow=black}}] (13.center) to (4.center);
    \draw [ultra thick, postaction={on each segment={mid arrow=black}}] (7.center) to (0.center);
    \draw [ultra thick, postaction={on each segment={mid arrow=black}}] (3.center) to (12.center);
    \draw [ultra thick, postaction={on each segment={mid arrow=black}}] (11.center) to (2.center);
    \draw [ultra thick, postaction={on each segment={mid arrow=black}}] (1.center) to (8.center);
    \draw [thick] [in=-120, out=135] (18.center) to (19.center);
    \draw [thick] [in=-45, out=-60, looseness=1.50] (16.center) to (18.center);
    \draw [thick] [in=105, out=60, looseness=1.25] (19.center) to (21.center);
    \draw [thick, -Stealth] [in=75, out=-75] (21.center) to (22.center);
    \draw [thick] [in=-135, out=-120, looseness=1.50] (23.center) to (24.center);
    \draw [thick] [in=-60, out=45] (24.center) to (25.center);
    \draw [thick] [in=75, out=120, looseness=1.25] (25.center) to (26.center);
    \draw [thick, -Stealth] [in=105, out=-105] (26.center) to (27.center);

    \node [left] at (-0.05, 3.35) {$O_0$};
    \node [left] at (-3.05, 3.35) {$O_3$};
    \node [right] at (3.05, 3.35) {$O'_3$};
    \node [left] at (-0.05, -0.35) {$O_1$};
    \node [left] at (-3.05, -0.35) {$O_2$};
    \node [right] at (3.05, -0.35) {$O'_2$};

    \node [left] at (0.05, 1.5) {$\gamma_0$};
    \node [left] at (-2.95, 1.5) {$\gamma_2$};
    \node [right] at (2.95, 1.5) {$\gamma'_2$};
    \node [above] at (1.5, 3.05) {$\gamma'_3$};
    \node [above] at (-1.5, 3.05) {$\gamma_3$};
    \node [above] at (1.5, -0.55) {$\gamma'_1$};
    \node [above] at (-1.5, -0.55) {$\gamma_1$};
\end{tikzpicture}

\caption{In this example, the \stindex at saddle $O_k$ is $\rho_k$.  The left and right cells are mirror images of each other, so $\rho_2$ is also the \stindex at $O'_2$  and $\rho_3$ at $O'_3$. Additionally, we assume $\rho_1,\rho_2,\rho_3<1$ but $\rho_0\rho_1\rho_2\rho_3\ge 1$.}
{\label{fig:two-cells}}
\end{figure}

Let us briefly discuss an example depicted in Figure~\ref{fig:two-cells}, two neighboring cells of a certain cellular flow. Assuming that
$\rho_1,\rho_2,\rho_3<1$ but $\rho_0\rho_1\rho_2\rho_3\ge 1$ and starting with $\alpha_0=1$, we obtain that the scaling  exponents $\alpha_k$ after passing the neighborhood of a saddle point $O_k$ are given by $\alpha_0=1$, $\alpha_1=\rho_1$, $\alpha_2=\rho_2\rho_1$, $\alpha_3=\rho_3\rho_2\rho_1$ and then $\alpha_0=1$ again. As a result, on logarithmic time scales, it is unlikely for the diffusion to escape the union of two cells. Near~$O_1$, the diffusion may choose one of the outgoing connections, to $O_2$ or $O'_2$, but once this choice is made, the next choices are predetermined with high probability, and  diffusion visits sequentially either $O_1,O_2,O_3,O_0$ or $O_1,O'_2,O'_3,O_0$. However at $O_0$ the contraction is strong enough to result in the exit exponent being $1$ again and the scaling limit
is symmetric Gaussian, so the process of making a choice of the exit direction at~$O_1$ and then cycling through one of the sequences 
$O_1,O_2,O_3,O_0$ or $O_1,O'_2,O'_3,O_0$, repeats, etc. This behavior, with the boundary $\gamma_0$ between these two cells being permeable and the boundary of the union of these cells impenetrable remains typical on the logarithmic time scales. The results that we obtain in this paper apply to this specific example, so we will be able to quantify the decay (as $\e\to0$) of probabilities of exiting the union of these cells through  connections $\gamma_1,\gamma_2,\gamma_3,\gamma'_1,\gamma'_2,\gamma'_3$, (it turns out they decay as a power of $\e$), find the most likely exit scenarios  and estimate the exit times. We will also be able to describe exit times and typical exit scenarios for pretty general cell complexes with boundaries composed of heteroclinic connections.

\section{Some notation and terminology}\label{sec:notation}

As we have seen, crossing the  heteroclinic network may be a rare event. Which scenarios  lead to those rare events
and what can be said about the decay of their probabilities as $\e\to0$?
To answer this question, we need to distinguish between various degrees of unlikeliness. We will mostly be interested in the events with probabilities that decay to 0 as a power of $\e$ as $\e \to 0$. Some events are even more unlikely, 
with probabilities decaying faster than any power of~$\e$. To describe these events, we will need to make sense of the claim that a r.v.\ is essentially of order~$\eps^\alpha$, up to logarithmic corrections. In this short section, we introduce appropriate definitions and notations (see Section~\ref{sec:notation1}  for more notational agreements).

\bigskip

For $\e>0$ and $\vk>0$, $\alpha,\beta\in\R$ satisfying $\alpha \ge\beta $, we denote
\begin{align}
l_\e&=\log\e^{-1},\notag
\\
K_\vk(\e)&
=[- l^\vk_\e,  l^\vk_\e]\subset \R. \label{eq:def_K(eps)}
\end{align}

If there is $\varkappa>1$ such that $f(\e)=o(e^{-l_\e^\varkappa})$, $\e\to 0$, we write $f(\e)=o_e(1)$. If $f(\e)=o_e(1)$, then $f(\e)$ converges to zero, as $\eps\to 0$, faster than any power of $\eps$  because for all $\varkappa>1$ and $p>0$,
\[
\frac{e^{-l_\e^\varkappa}}{\e^p}= e^{-l_\e^{\varkappa}+p l_\e}\to 0,\quad \e\to0.
\]

The following definitions describe certain properties of {\it families} of events indexed by $\e>0$ but, for brevity, we abuse the terminology slightly and speak of events themselves.

We say that events $(A_\e)_{\e>0}$  happen with high probability (w.h.p.) if $\Pp(A_\e(x))=1-o_e(1)$. We say that events $(A_\e)_{\e>0}$  happen with low probability (w.l.p.) if $\Pp(A_\e(x))=o_e(1)$. We also call them high (respectively, low) probability events. 

Suppose we have a family of events $(A_\e(x))_{x\in I_\e, \e>0,}$ and probability measures $(\Pp_\e^x)_{x\in I_\e,\e>0}$ depending on $\e>0$ and $x$ ranging through some some set  $I_\e$ which in turn depends on $\e$. 
We say that  $A_\e(x)$  happen w.l.p.\ under $\Pp_\e^x$ uniformly over $I_\e$ if \[\sup_{x\in I_\e} \Pp_\e^x(A_\e(x))=o_e(1).\] The complements of $A_\e(x)$ are then said to happen  w.h.p.\ under $\Pp_\e^x$ uniformly over~$I_\e$.

We say that $B_\e$ happens on $A_\e$ w.h.p.\ if $A_\e\setminus B_\e$ happens w.l.p. 

We say that $(\xi_\e)_{\e>0}$ are of order $\e^\alpha$ if for some~$\vk>0$, $\xi_\e\in \e^\alpha [ l^{-\vk}_\e, l^\vk_\e]$ w.h.p.

If for some $\alpha$ and all sufficiently large~$\vk$, $\xi_\e>  \e^\alpha l^{-\vk}_\e$ w.h.p., we say that $\xi_\e$ is of order at least $\eps^\alpha$. 
 
If for some $\alpha$ and some~$\vk>0$, $\xi_\e\in[0,\e^\alpha l^\vk_\e)$ w.h.p., we say that $\xi_\e$ is of order at most $\eps^\alpha$.

If for some $\alpha$ and some $\vk>0$,  $\xi_\e\in [0,\e^\alpha l^{-\vk}_\e)$ w.h.p., we say that the order of~$\xi_\e$ is below $\eps^\alpha$. 
 
If for some $\alpha$ and all sufficiently large~$\vk$, $\xi_\e>\e^\alpha l^\vk_\e$ w.h.p., we say that the order of $\xi_\e$ is above $\eps^\alpha$.

If in the definitions above ``w.h.p.'' is replaced by ``a.s.'', the  r.v.'s  $\xi_\e$ are said to be {\it strictly} of order $\e^\alpha$,   {\it strictly} of order below $\e^\alpha$, etc.  

Instead of ``order $\e^0$'', we often say ``order $1$''.\

If  for some $\vk>0$,  $|\xi_\e|< l^\vk_\e$ w.h.p., we call   r.v.'s $\xi_\e$ tame.

If $(\xi_\e^x)_{x\in I_\e, \e>0}$ is a family of r.v.'s also indexed by some parameter $x$ and for some $\vk>0$,  $|\xi_\e^x|< l^\vk_\e$ w.h.p.\ uniformly in $x\in I_\e$, then we say that $\xi_\e^x$ are uniformly tame  in $x\in I_\e$.

We write $\xi_\e\eqmodlp \xi'_\e$ if $\{\xi_\e\ne \xi'_\e\}$ is a low probability event. We write $A_\e\eqmodlp A'_\e$  if $A_\e\triangle A'_\e$ is a low probability event. In addition, for events $A,B$, we write $A\stackrel{\Pp}{=}B$ if $\Pp(A\triangle B)=0$.

\section{Two saddles}\label{sec:2saddles}

The results of Section~\ref{sec:typical} imply that the 2-dimensional diffusion near a heteroclinic network often tends to stay on one side of the network mostly
exploring a single cell.
In this section we consider a very short 
$N$-shaped heteroclinic chain composed of heteroclinic connections $\gamma_0,\gamma_1,\gamma_2$, 
see Figure~\ref{fig:2saddles}, and compute the asymptotics of the probabilities 
of $A_2=\{\tau^2_\e<\infty\}$, i.e., the probability that the diffusion 
starting  near $x_0$ first exits from $D_1$ into $D_2\setminus D_1$ through  $x_1+[-1,1]v_1$ and then exits from~$D_2$ into $D_3\setminus D_2$ through $x_2+[-1,1]v_2$. For the latter, it needs to make a ``wrong turn'' near $O_2$, so this may be a small probability event.  Longer heteroclinic chains will be considered in the next section.

\begin{figure}[ht]
\centering
\begin{tikzpicture}

    \node (0) at (1.75, -0.25) {};
    \node (1) at (-2.5, 0) {};
    \node (2) at (-2.25, 3.25) {};
    \node (3) at (-5, -1) {};
    \node (4) at (4.75, 0.5) {};
    \node (5) at (-2.25, -2.5) {};
    \node (6) at (3, 3) {};
    \node (7) at (2.25, -3.25) {};
    \node (8) at (-4.5, 0) {};
    \node (9) at (-4.25, -1.25) {};
    \node (10) at (-3, 2.25) {};
    \node (11) at (-1.75, 2.25) {};
    \node (12) at (0.25, 0.5) {};
    \node (13) at (0.25, -1) {};
    \node (14) at (2, 2.5) {};
    \node (15) at (3.5, 2.25) {};
    \node (16) at (3, -2.5) {};
    \node (17) at (1.25, -2.5) {};
    \node (18) at (-0.5, -1.25) {};
    \node (19) at (-0.75, 0.75) {};
    \node (20) at (0.5, -3) {};
    \node (21) at (1.25, -1.25) {};
    \node (22) at (2.25, -1.25) {};
    \node (23) at (3.75, -3) {};
    
    \node (x0) at (-2.505, 1) {};
    \node (v0) at (-1.75, 1.1) {};
    \node (v0-) at (-3.15, 0.9) {};
    \node (x1) at (0.25, -0.185) {};
    \node (v1) at (0.25, 0.4) {};
    \node (x2) at (2.13, -2.5) {};
    \node (v2) at (1.4, -2.5) {};

\foreach \n in {x0,x1,x2,1,0}
        \node at (\n)[circle,fill,inner sep=1.25pt]{};

    \draw [postaction={on each segment={mid arrow=black}}][in=30, out=-180, looseness=0.75] (1.center) to (3.center);
    \draw [ultra thick, postaction={on each segment={mid arrow=black}}][in=-180, out=0] (1.center) to (0.center);
    \draw [postaction={on each segment={mid arrow=black}}][in=0, out=195] (4.center) to (0.center);
    \draw [ultra thick, postaction={on each segment={mid arrow=black}}][in=90, out=-105] (2.center) to (1.center);
    \draw [postaction={on each segment={mid arrow=black}}][in=-90, out=105] (5.center) to (1.center);
    \draw [postaction={on each segment={mid arrow=black}}][in=-120, out=75] (0.center) to (6.center);
    \draw [ultra thick, postaction={on each segment={mid arrow=black}}][in=90, out=-105, looseness=0.75] (0.center) to (7.center);
    \draw (8.center) to (9.center);
    \draw [in=-135, out=30, looseness=1.50] (8.center) to (10.center);
    \draw [in=120, out=45] (10.center) to (11.center);
    \draw [in=180, out=15] (9.center) to (13.center);
    \draw (13.center) to (12.center);
    \draw [in=165, out=-60, looseness=1.25] (11.center) to (12.center);
    \draw [in=165, out=-135] (19.center) to (18.center);
    \draw [in=120, out=-15, looseness=1.25] (18.center) to (17.center);
    \draw (17.center) to (16.center);
    \draw [in=-120, out=105] (16.center) to (15.center);
    \draw (15.center) to (14.center);
    \draw [in=45, out=-120, looseness=1.50] (14.center) to (19.center);
    \draw [in=-165, out=90, looseness=0.75] (20.center) to (21.center);
    \draw [in=165, out=15] (21.center) to (22.center);
    \draw [in=105, out=-15] (22.center) to (23.center);

    \draw [thick, -latex] (x0.center) to (v0.center);
    \draw [thick, dotted] (v0-.center) to (x0.center);
    \draw [thick, -latex] (x1.center) to (v1.center);
    \draw [thick, -latex] (x2.center) to (v2.center);

    \node [left] at (-2.5, 1.15) {$x_0$};
    \node [above] at (-1.87, 1.05) {$v_0$};
    \node [right] at (0.22, -0.05) {$x_1$};
    \node [right] at (0.22, 0.4) {$v_1$};
    \node [below] at (1.95, -2.5) {$x_2$};
    \node [below] at (1.5, -2.5) {$v_2$};
    \node [below] at (-2.75, 0) {$O_1$};
    \node [above] at (2.13, -0.25) {$O_2$};
    \node [above] at (-3.3, 0.1) {$D_1$};
    \node [above] at (1.2, 0.5) {$D_2$};
    \node [right] at (2.67, -2.9) {$D_3$};
    \node [right] at (-2.5, 1.8) {$\gamma_0$};
    \node [above] at (-0.48, -0.1) {$\gamma_1$};
    \node [right] at (1.87, -1.6) {$\gamma_2$};
\end{tikzpicture}
\caption{A short N-shaped heteroclinic chain: $\gamma_0,\gamma_1,\gamma_2$.}
{\label{fig:2saddles}}
\end{figure}
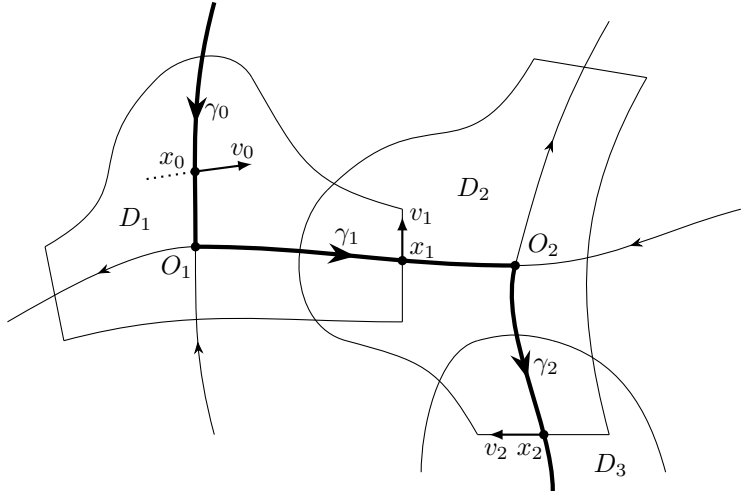

Our analysis below shows that there are three main cases:  (1) $\alpha_1= 1$; (2) $\alpha_1<1$ and $\alpha_0<1$;  (3)  $\alpha_1<1$ and $\alpha_0= 1$.

In the first case,   the  probability of following the  connections $\gamma_0,\gamma_1,\gamma_2$ has a positive limit, as we know from Section~\ref{sec:typical}.

In the other two cases, $\alpha_1<1$, so the scaling limit along $\gamma_1$ is one-sided and the probability of following  connections $\gamma_0,\gamma_1,\gamma_2$ converges to $0$. The typical exits from~$D_1$ are too far from $\gamma_1$ which typically results in exits from $D_2$ in the direction opposite to $\gamma_2$.  The main question then is: how improbable are the exits from~$D_1$ that are $\e$-close to $\gamma_1$? It turns out that if $\alpha_0<1$, this probability decays superpolynomially while if $\alpha_0=1$, it decays as a power of $\e$.

Before stating the main result of this section we must introduce an extra assumption that  we will use.
\begin{enumerate}[label={\bf(\Alph*)}, resume*=setting]
\item \label{setting:conjugacy} For a saddle point with eigenvalues of linearization $\lambda>0$ and $-\mu<0$ there is an (open) neighborhood  $U$ of the saddle and an orientation-preserving $C^5_\mathrm{b}$-diffeomorphism $f$ mapping~$U$ onto a neighborhood of $0\in\R^2$ and conjugating the vector field $b$ to the linear vector field $\bar b(x)=Ax$, where $A=\mathop{\mathrm{diag}}(\lambda,-\mu)$ is a 2-by-2 diagonal matrix:
 \[
 Df(x) b(x)=A f(x).   
 \]
Here, $Df$ denotes the derivative (Jacobian) of the map $f$. Shrinking $U$ if necessary, we may assume  that $x_0+[-1,1]v$ does not intersect the closure of $U$ and that the trajectory $(\flow^tx)_{t\ge 0}$ of every point $x\in U\setminus \Ws$ intersects $q_{+}+(-1,1)v_+$ or  $q_{-}+(-1,1)v_-$ transversally. We also note that if $f$ is a linearizing conjugacy then so is $cf$ for any $c>0$. Thus, we can choose $f$  to make $f(U)$ as large as needed.
\end{enumerate}

We need the $C^5_\mathrm{b}$ assumption on the linearizing change of coordinates to ensure that the second derivatives emerging in the It\^o formula for this linearization are~$C^3_\mathrm{b}$, which is required for our Malliavin calculus techniques to work in Section~\ref{section:density_est}. A~sufficient condition for existence of such a smooth linearization is that $b\in C^\infty$ and there are no resonances between $\lambda$ and $-\mu$, i.e.,  neither of them can be represented as  
$\alpha \lambda-\beta\mu$ with some $(\alpha,\beta)\in\Z_+^2$  satisfying $\alpha+\beta\ge 2$ (see the discussion in~\cite[Section~4]{Almada-Bakhtin:MR2802310}). We believe that our conclusions still hold true even without this restriction, say, for area-preserving flows, where $\lambda=\mu$. When such resonances are present, one has to control the emerging resonant monomial terms in the normal form and extend the results of~\cite{Almada-Bakhtin:MR2802310}. We think that this is possible in our setting but we do not pursue this generality here. We also note that~\ref{setting:conjugacy} implies $b\in C^4_\mathrm{b}$ because $b(x)=Df^{-1}(y)\bar b(y)$ for $y=f(x)$. This is a stronger smoothness requirement on $b$ than stated in \ref{setting:general}.

In the remaining part of this  paper, we will require that  a smooth linearization exists locally near each saddle point of the heteroclinic chain:
\begin{enumerate}[label={\bf(\Alph*)}, resume*=setting]
\item \label{setting:conjugacy-all} 
The conjugacy condition \ref{setting:conjugacy} holds near each saddle point $O_1,\ldots, O_n$.
\end{enumerate}

Let us state the main result of this section. It will be  generalized to longer heteroclinic chains in the next one.
\begin{theorem}\label{th:2saddles}
Assume the setting described by conditions \ref{setting:general},  \ref{setting:het-chain},~\ref{setting:scaling-at-saddle-0},~\ref{setting:scaling-limit-at-saddle-0},~\ref{setting:conjugacy-all}  with $n=2$ and with $\gamma_0$ and $\gamma_2$ on different sides of $\gamma_1$, see Figure~\ref{fig:2saddles}. 
 Assume that $p_1>0$ (in the case of $\alpha<1$, this means that $\Pp\{\xi_{0}>0\}>0$.)

Recall that $\alpha_1=\alpha_0\rho_1 \wedge 1$ according to~\eqref{eq:recursion_alpha}.

\begin{enumerate}

\item\label{item:typical-case}
Suppose $\alpha_1=1\le\alpha_0 \rho_1$. Then  $p_2>0$.

\item \label{item:power-asymp-for-2saddles} Suppose $\alpha_1=\rho_0<1$ and $\alpha_0=1$. In addition, we assume that $\xi_{0,\e}$ is tame. Then there is a number $h>0$ such that
\begin{equation}
\label{eq:polynomial-asymp-2saddles}
\Pp(A_{2,\eps})=h \e^{\frac{1}{\rho_1}-1}\left(1+o\left(1\right)\right),\quad \e\to 0.
\end{equation}
In this case, conditioned on $A_{2,\e}$,
\begin{equation}
\label{eq:exit-time-asymp-2saddles}
\frac{\tau_{2,\e}}{(\frac{1}{\mu_1}+\frac{1}{\lambda_2}) \log \e^{-1}}\inprob 1,\quad \e\to0.
\end{equation}
\item \label{item:too-far-from-network}
Suppose $\alpha_1=\alpha_0\rho_1<1$ and $\alpha_0<1$. In addition, assume that $\xi_{0,\e}$ is of order~$1$. Then $A_{2,\e}$ happen w.l.p.
\end{enumerate}
\end{theorem}

Without making an assumption on the order of $\xi_{0,\e}$ in part~\ref{item:too-far-from-network}, we may end up with a situation where other orders of magnitude are present with small probabilities that may still dominate the picture. 

\medskip

We will derive this theorem from a sequence of lemmas studying both, the exit from $D_1$ and then the exit from $D_2$, in more detail than in Section~\ref{sec:typical}. However, within this section, we only give heuristic arguments for these lemmas and only for the case of  the simpler case of the  linear system~\eqref{eq:linear-system1}--\eqref{eq:linear-system2} in the strip $D$ given by~\eqref{eq:strip},
with initial conditions~\eqref{eq:entrance-distr-linear1}--\eqref{eq:entrance-distr-linear2}. We will refer to this as {\it the model case.} 

The full generality needs rigorous proofs taking into account the nonlinear geometry and  correction terms, some of which present massive technical
difficulties and
will be given in Sections~\ref{sec:rectified}--\ref{section:density_est}. 

For the model case, we will need the following auxiliary result on processes defined in~\eqref{eq:U-in-model-case} and the exit time $\tau_\e$ defined in~\eqref{eq:def-exit-time}, which is an easy consequence of the exponential martingale inequality (see Lemma~\ref{lem:exp-marting-ineq}). More general results with rigorous proofs are Lemmas~\ref{lem:estimates-for-terms} and \ref{lem:second-coord-at-exit}.

\begin{lemma}\label{lem:tail-control-for-basic-rvs-in-Duhamel}
Consider the diffusion in a neighborhood of a saddle in the model case. Then the r.v.'s 
$\sup_{t\in[0,\tau_\e]}|U^1_t|$,  $\sup_{t\in[0,\tau_\e]}|N^2_t|$, and~$\tau_\e$, are
 uniformly tame over all initial conditions
and there is a constant $C$ such that, for every $\vk>0$ and every $\alpha \in (0,1]$,
\begin{align*}
\sup_{x \in K_\vk(\eps)}\Pp^{x_0 + \eps^\alpha x v}\left\{\sup_{[0,\tau_\e]} \left|U^1_t\right|> z  \right\}\le C e^{-z^2/C},\quad z>0.\\
\sup_{x \in K_\vk(\eps)}\Pp^{x_0 + \eps^\alpha x v}\left\{ |N^2_{\tau_\e}|> z  \right\}\le C e^{-z^2/C},\quad z>0.
\end{align*}
\end{lemma}

\smallskip

We begin with the following general statement which is a more precise version of one of the claims of  Theorem~\ref{th:poincare-saddle}.
We recall that events $A_{\pm,\e}$ were defined in~\eqref{eq:def-Apm}.

\begin{lemma}\label{lem:exit-straight} 
Under assumptions~\ref{setting:general}, \ref{setting:geometry-domain}, \ref{setting:initial-cond}, and  \ref{setting:conjugacy}, if $\xi_\eps$ is tame, then $A_{+,\e}\cup A_{-,\e}$ happen w.h.p.\ and $\xi'_\e$ defined on this event  uniquely via~\eqref{eq:scaling-at-exit} is tame.
 
Moreover, for every $\vk>0$, 
\begin{align}\label{eq:stronger_exit-straight1}
    \sup_{x\in K_\vk(\eps)}\Pp^{x_0+ \eps^\alpha x v}((A_{+,\e}\cup A_{-,\e})^c)=o_e(1),
\end{align}
and there is $\vk'>0$ such that
\begin{align}\label{eq:stronger_exit-straight}
    \sup_{x\in K_\vk(\eps)}\Pp^{x_0+ \eps^\alpha x v}\left(\left\{X_{\eps,\tau_\eps}\not\in q_{\pm}+\eps^{\alpha'}K_{\vk'}(\eps)v_{\pm}\right\}\cap\ A_{\pm,\eps}\right)=o_e(1).
\end{align}
\end{lemma}
\begin{remark}\rm In fact, a stronger claim holds under the conditions of this lemma: the order of the maximum (over times $t\le \tau_\e$) distance from $X_{\e,t}$ to the heteroclinic network is at most $\e^\beta$ for some positive $\beta$.
\end{remark}

\bpfm
Here, we consider only the case of the linear system~\eqref{eq:linear-system1}--\eqref{eq:linear-system2} with initial conditions~\eqref{eq:entrance-distr-linear1}--\eqref{eq:entrance-distr-linear2}. Lemma~\ref{lem:tail-control-for-basic-rvs-in-Duhamel} implies that $|Z_\e|$ given in~\eqref{eq:Z_eps} is of order at most~$1$. Therefore we obtain that the absolute values of both terms in \eqref{eq:X_2-at-exit-time} are of order at most~$\e^{\alpha'}$,  which implies our first two claims. Relation~\eqref{eq:stronger_exit-straight} follows from a similar argument with $\xi_{\e}$ replaced by $x\in K_\vk(\eps)$.
\epf

As we know, the exit from $D_1$ happens near $\gamma_1$. Which exit locations contribute most to  $\Pp(A_{2,\e})$? The next lemma applied to diffusion in $D_2$ shows that the contribution from the exits that are not $\e$-close to $\gamma_1$, decays superpolynomially.

In the rest of this section, under assumptions~\ref{setting:general},~\ref{setting:geometry-domain}, for $x\in\R$, we denote by~$\Psc^x=\Pp^{x_0+\e x v}$ the distribution of the diffusion with initial condition
\begin{equation}
 X_{\e,0}=x_0+\e x v.
\label{eq:entrance-scaling-conditioned} 
\end{equation}
We recall that, according to our convention from Section~\ref{sec:notation1}, we still denote a generic probability measure by $\Pp$ when working with r.v.'s whose distribution is unambiguously clear from the context.
\begin{lemma}
\label{lem:exit-on-the-same-whp}
Let us assume conditions~\ref{setting:general},~\ref{setting:geometry-domain}, and~\ref{setting:conjugacy}. Then  $A_{-,\e}$ happen w.l.p.\ under $\Psc^x$, uniformly in $x \in [ l^\vk_\eps,\eps^{-1}]$, for sufficiently large~$\vk$.
\end{lemma}

\bpfm In the model case~\eqref{eq:linear-system1}--\eqref{eq:linear-system2}, the proof is straightforward. Using~\eqref{eq:exit-direction} and~\eqref{eq:Z_eps} with $\alpha=1$, we obtain
\begin{align*}
\Psc^x(A_{-,\e})=
\Psc^x\left\{ \sgn X_{\e,\tau_\e}^1=-1\right\}&\le
 \Pp\left\{\sup_{t\in[0,\tau_\eps]}\left|U^1_t\right|>l_\e^{\vk}\right\}
=o_e(1)
\end{align*}
for sufficiently large $\vk$ due to Lemma~\ref{lem:tail-control-for-basic-rvs-in-Duhamel}.\epf

This lemma means that, conditionally on the exit from $D_1$ at distance from $\gamma_1$ being of order above $\e^1$, the probability of $A_{\e,-}$ decays to zero superpolynomially.

The next lemma  means that conditionally on the exit from $D_1$ at distance from~$\gamma_1$ being of order 
at most $\e^1$, those probabilities converge to a positive limit. This gives slightly more detail than Theorem~\ref{th:poincare-saddle}.

\begin{lemma}\label{lem:postive-limit-prob-if-close-to-manifold}
Assume conditions~\ref{setting:general},~\ref{setting:geometry-domain}, and~\ref{setting:conjugacy}. Then, there is $s>0$ such that for every $\vk>0$
\begin{align*}
    \sup_{x\in K_\vk(\eps)}\left|\Psc^x(A_{-,\eps}) - \psi_s(-x) \right|=\smallo{\eps^\delta},
\end{align*}
for some $\delta>0$, where $K_\vk(\eps)$ is defined in~\eqref{eq:def_K(eps)} and $\psi_s$ is defined in~\eqref{eq:gaussian-cdf}.
\end{lemma}

\bpfm For the system~\eqref{eq:linear-system1}--\eqref{eq:linear-system2}, we recall that the direction of exit is determined by the sign of $Z_\e=x+U^1_{\tau_\e}\indistr x+\frU$. Defining $s$ to be $\cc_1$, the variance of $\frU$, see~\eqref{eq:var-of-N}, we obtain
\[
\Psc^x(A_{-,\e})\to \Pp\{\frU<-x\}=\psi_s(-x).
\]
\epf

These lemmas, especially Lemma~\ref{lem:exit-on-the-same-whp}, show that 
in the case where $\alpha_1<1$, we need to study how the diffusion is set up to be at distance of order at most~$\eps^1$ from~$\gamma_1$ when exiting $D_1$,  even if this means an atypical scenario near $O_1$. 

It turns out that the probability of such a scenario differs drastically between the situations where $\alpha_0<1$ and $\alpha_0=1$. We address them in the following two lemmas.

First, we address the situation where the entrance distribution is concentrated at scale $\e^\alpha$ with $\alpha<1$ and $\alpha\rho<1$.

\begin{lemma} \label{lem:alpha<1-and-rho-alpha<1-- concentration at exit} 
Under conditions~\ref{setting:general},~\ref{setting:geometry-domain},~\ref{setting:initial-cond}, and~\ref{setting:conjugacy}, let us assume that
\begin{equation}
\label{eq:a<1-and-ar<1}
\alpha<1\quad \text{\rm and}\quad  \alpha\rho<1,
\end{equation}
and that $\xi_\eps$ is of order $1$. Then
\[
X_{\e,\tau_\e}=q_++\eps^{\alpha\rho} \xi'_\e v_+, 
\]
where $\xi'_\e$ is of order $1$.
\end{lemma}
 
\bpfm 
First, due to Lemma~\ref{lem:exit-on-the-same-whp}, the exit happens through $\partial_+$ (see the definition~\eqref{eq:boundary-of-strip}), w.h.p., uniformly over values of $x$  of order above $1$.

We can rewrite ~\eqref{eq:X_2-at-exit-time} as
\begin{align}
\label{eq:X_2-at-exit-time-alpha-rho}
X_{\e,\tau_\e}^2
= \e^{\alpha\rho}\left(\frac{L}{R^{\rho}}|Z_\eps|^{\rho}+\e^{1-\alpha\rho} N^2_{\tau_\e}\right).
\end{align}
Recalling~\eqref{eq:Z_eps} and using Lemma
\ref{lem:tail-control-for-basic-rvs-in-Duhamel}, we obtain that $Z_\eps$ is of 
order~1.
Applying Lemma
\ref{lem:tail-control-for-basic-rvs-in-Duhamel} to  the right-hand side of~\eqref{eq:X_2-at-exit-time-alpha-rho}, we now obtain the statement of the lemma.\epf

Thus, under~\eqref{eq:a<1-and-ar<1}, the exit at scale at most $\e^1$ is extremely unlikely. Let us consider the remaining case where $\alpha=1$,
which is actually the most interesting and technical part of our program. The lemma we are about to state describes exits at scale $\e^\beta$, where $\beta\in(\rho,1]$. In this section, we are mostly interested in $\beta=1$ but we will need this lemma in full generality in the next section when considering longer heteroclinic chains. 

We denote by $\GoodMeasures$ the set of all nonzero absolutely continuous measures $\nu$ on $\R$ satisfying
\begin{align}
&\nu((0,\infty)) >0, \label{eq:GoodMeasures>0}
\\
&\nu((-\infty,z])\le C\left(1+z^C\right),\quad z\ge 0, \label{eq:GoodMeasures}
\\
& \frac{d\nu}{d\Leb}(z) \leq C\left(1+|z|^C\right),\quad z\in \R, \label{eq:GoodMeasure_density}
\end{align}
for some $C>0$. The elements of $\GoodMeasures$ are called (absolutely continuous) measures of polynomial growth.

\begin{lemma} \label{lem:local-limit-theorem}
Under conditions~\ref{setting:general},~\ref{setting:geometry-domain}, and~\ref{setting:conjugacy}, suppose 
\begin{equation*}
\alpha=1, \qquad  \alpha\rho=\rho<1, \quad \text{\rm and}\quad \beta\in(\rho,1].
\end{equation*}
Then the following holds:
\begin{enumerate}\item \label{item:power-asymp}
There are constants $c,\delta >0$, and $\nu\in\GoodMeasures$, such that for all~$\vk,\vk'>0$,
\begin{align*}
    \sup_{\substack{x\in K_\vk(\e), \\ [a,b]\subset K_{\vk'}(\e)}}\left|\eps^{-(\frac{\beta}{\rho}-1)}\Psc^x\left\{X_{\e,\tau_\e}\in q_++\e^\beta [a,b]v_+ \right\}- g_c(x)\nu([a,b])\right|= \smallo{\eps^\delta}.
\end{align*}

\item\label{item:large-negative-values-w.l.p.}
For every $\vk,\vk'>0$, 
\begin{align}\label{eq:tau_cvg_in_prob_quantified}
    \sup_{x\in K_\vk(\eps)}\Psc^x\left\{
\left|\frac{\tau_\e}{\frac{\beta}{\mu} l_\eps }- 1\right|>\delta,\  A_{+,\eps},\  X_{\e,\tau_\e}\in q_++\e^\beta (-\infty,l_\e^{\vk'}]v_+
\right\}=O\left(\e^{{\frac{\beta(1+\delta)}{\rho}-1}}\right).
\end{align}

\item \label{item:exit-below-unlikely}
For any $\vk>0$ and any $\vk'>\frac{1}{2}$,
\begin{align*}
    \sup_{x\in K_\vk(\e)}\Psc^x\left(A_{+,\e}\cap \left\{X_{\e,\tau_\e}\notin q_++\e [-l_\e^{\vk'},+\infty)v_+\right\}\right)= o_e(1).
\end{align*}
 
\end{enumerate}
\end{lemma}

Note that, in~\eqref{item:power-asymp}, due to $\eps^\beta l_\eps^{\vk'}\leq 1$ for small $\eps$, we have $\{X_{\e,\tau_\e}\in q_++\e^\beta [a,b]v_+\}\subset A_{+,\eps}$ uniformly in $[a,b]\subset K_{\vk'}(\eps)$ for sufficiently small $\eps$.

In comparison with other results given in this section, a complete proof of this local limit theorem requires a lot of technical work involving multiple approximations, techniques based on Malliavin calculus, an iteration scheme similar to those of~\cite{long_exit_time},\cite{long_exit_time-1d-part-2}, \cite{YB-and-HBC:10.1214/20-AAP1599}, \cite{YB-and-HBC:10.1214/20-AOP1479} helping to gradually extend the analysis of the diffusion to longer and longer times, and detailed analysis of tails of exit times.

Let us stress that although the natural scale for $X_\e(\tau_\e)$ is $\e^\alpha$ with $\alpha<1$,  Lemma~\ref{lem:local-limit-theorem} shows 
that the distribution of $X_\e(\tau_\e)$  has local regularity (approximate equidistribution) at smaller scales down to order $\e^1$ and thus can be viewed as a local limit theorem.

Note that the limit $g_c(x) \nu([a,b]) $  in the local limit theorem (part \ref{item:power-asymp} of Lemma~\ref{lem:local-limit-theorem}) is a product 
of two factors depending only on the initial condition $x$ and the exit location $[a,b]$ respectively.
 This indicates an asymptotic loss of memory that will be useful in the proof of Theorem~\ref{th:2saddles} and in the analysis of longer heteroclinic chains.

Together, Lemma~\ref{lem:local-limit-theorem}~\eqref{item:power-asymp} and \eqref{item:large-negative-values-w.l.p.} imply that for every fixed $x$, under $\Psc^x$ conditioned on $A_{+,\eps}\cap\{ X_{\e,\tau_\e}\in 
q_++\e^\beta (-\infty,l_\e^{\vk'}]v_+\},$
we have
\begin{equation*}
\frac{\tau_\e}{\frac{\beta}{\mu} l_\eps }\inprob 1,\quad \e\to 0.
\end{equation*}

A rigorous proof of Lemma~\ref{lem:local-limit-theorem} is given in Section~\ref{sec:original-coords}. It requires a lot of
preparatory work in Sections~\ref{sec:rectified}--\ref{section:density_est}.
\medskip

\bpfm 
Using~\eqref{eq:X_2-at-exit-time}, we can write
\begin{equation}
\label{eq:converting-to-Z}
\Psc^{x}\left\{X_{\e,\tau_\e}\in q_++\e^\beta [a,b]v_+\right \}=\Psc^{x}\left\{\e^{\rho}\frac{L}{R^{\rho}}|Z_\e|^{\rho}+\e N^{2}_{\tau_\e}\in\e^\beta[a,b]\right\},
\end{equation}
where, similarly to~\eqref{eq:Z_eps},
\begin{equation*}
Z_\e=x+U^1_{\tau_{\e}}.
\end{equation*}
Since the exit happens near $q_+$, i.e., through $\partial_+$, we have $Z_\e>0$ on our event.

Due to~\eqref{eq:lim-of-U}, we only make a small error computing instead
\begin{align}
\label{eq:local-limit-computation}
&\Pp\left\{\e^{\rho}\frac{L}{R^{\rho}}(x+\frU)^{\rho}+\e \NN\in\e^{\beta}[a,b],\ x+\frU>0\right\}
\\
\notag
=\ &
\Pp\left\{(x+\frU)^\rho\in \frac{R^\rho}{L}\e^{\beta-\rho}[a-\e^{1-\beta}\NN,b-\e^{1-\beta}\NN],\ x+\frU>0 \right\}.
\end{align} 
If $\beta=1$, the right-hand side equals
\[
\Pp\left\{\frU\in \frac{R}{L^{1/\rho}}\e^{\frac{1}{\rho}-1}[((a-\NN)\vee 0)^{1/\rho},((b-\NN)\vee 0)^{1/\rho}] -x \right\},
\]
and, using the independence and Gaussianity of $\frU$ and $\NN$,  for small~$\eps$, due to $\frac{1}{\rho}-1>0$, we can approximate this probability by
\begin{align*}
&
g_{\cc_1}(x)\frac{R}{L^{1/\rho}}\e^{\frac{1}{\rho}-1} \E \left(((b-\NN)\vee 0)^{1/\rho}-((a-\NN)\vee 0)^{1/\rho}\right),
\end{align*} 
where $\cc_1$ is the variance of $\frU$ given in~\eqref{eq:var-of-N}.
Defining~$\nu\in\GoodMeasures$ by 
\[
\nu((-\infty,z])=  \frac{R}{L^{1/\rho}} \E ((z-\NN)\vee 0)^{1/\rho},\quad z\in\R.
\]
we complete the proof of part~\eqref{item:power-asymp} for $\beta=1$. 

In the case of $\beta\in(\rho,1)$, the right-hand side of \eqref{eq:local-limit-computation} can be approximated by
\begin{align}
\label{eq:local-limit-computation-beta>rho}
\approx\ & 
\Pp\left\{\frU\in \frac{R}{L^{1/\rho}}\e^{\frac{\beta}{\rho}-1}[(a\vee 0)^{1/\rho},(b\vee 0)^{1/\rho}] -x \right\}.
\end{align} 
Since $\frac{\beta}{\rho}-1>0$, this probability can be approximated for small $\e$ by
\begin{align*}
&
 g_{\sigma^2_{\frU}}(x)\frac{R}{L^{1/\rho}}\e^{\frac{\beta}{\rho}-1} ((b\vee 0)^{1/\rho}-(a\vee 0)^{1/\rho}).
\end{align*} 
Now it remains to define~$\nu\in\GoodMeasures$ by 
\[
\nu((-\infty,z])=  \frac{R}{L^{1/\rho}} (z\vee 0)^{1/\rho},\quad z\in\R,
\]
and part~\eqref{item:power-asymp} for $\beta\in(\rho,1)$ follows.

To prove part~\eqref{item:exit-below-unlikely}, we similarly compute for large $\vk'$:
\begin{align*}
&\Pp\left\{\e^{\rho}\frac{L}{R^{\rho}}(x+\frU)^{\rho}+\e \NN\notin\e[-l_\e^{\vk'},+\infty),\ x+\frU>0\right\}
\\
=&
\Pp\left\{(x+\frU)^\rho\le \frac{R^\rho}{L}\e(-l_\e^{\vk'}-\NN),\ x+\frU>0 
\right\}
\\
\le&
\Pp\left\{-l_\e^{\vk'}-\NN>0\right\}=\Pp\{\NN<-l_\e^{\vk'}\}=o_e(1).
\end{align*}

To prove part~\eqref{item:large-negative-values-w.l.p.}, we note that up to small errors, similarly to~\eqref{eq:converting-to-Z} and~\eqref{eq:local-limit-computation-beta>rho},
\begin{align*}
\left\{X_{\e,\tau_\e}\in q_++\e^\beta (-\infty,l_\e^{\vk'}]v_+ \right\}\approx \left\{0< x+\frU\le \frac{R}{L^{1/\rho}}\e^{\frac{\beta}{\rho}-1}((l_\e^{\vk'}-\eps^{1-\beta}\NN)\vee 0)^{1/\rho} \right\},
\end{align*}
and on the latter event
\begin{align}
\tau_{\e}&=\frac{1}{\lambda}\log\frac{R}{\e |Z_\eps|}\approx \frac{1}{\lambda}\log\frac{R}{\e |x+\frU|} \label{eq:tau_formula_heuristic}
\\ &\ge\frac{\beta}{\mu} l_\eps  + \frac{1}{\mu} \log\frac{L}{l_\e^{\vk'}-\e^{1-\beta}\NN},
\notag
\end{align}
and $l_\e^{\vk'}-\e^{1-\beta}\NN>0$. Thus, on this event,
$\tau_\e <(1-\delta)\frac{\beta}{\mu} l_\eps $ implies
\begin{align*}
\e^{1-\beta}\NN-l_\e^{\vk'}<-L\eps^{-\beta\delta},
\end{align*}
which is a low probability event. For a matching upper bound on $\tau_\e$, we note that
\begin{align*}
\Pp\left\{\tau_\e>(1+\delta)\frac{\beta}{\mu} l_\eps  \right\}\approx \Pp\left\{|x+\frU|<R\e^{\frac{\beta(1+\delta)}{\rho}-1}\right\}=O\left(\e^{\frac{\beta(1+\delta)}{\rho}-1}\right).
\end{align*}
These estimates imply~\eqref{eq:tau_cvg_in_prob_quantified}.
\epf

This lemma providing the power asymptotics $\e^{\frac{\beta}{\rho}-1}$ for the probability of the unlikely event of approaching the outgoing heteroclinic connection at distance of order below $\e^\beta$, also  describes the mechanism responsible for creating these events.

We see that the exit time needed to realize the rare event is about $\frac{\beta}{\mu} l_\eps $ which, due to $\rho=\mu/\lambda<1$, is much longer than the typical exit times concentrating near $\frac{1}{\lambda} l_\eps $, see the limit theorem in~\eqref{eq:typical-exit-time-distr} or the more general claim~\eqref{item:exit_time_is_log} of Theorem~\ref{th:poincare-saddle}. We saw before that those typical exit times are not long enough for the contraction to bring the diffusion close enough to the unstable manifold. However, if the diffusion happens to be exposed to contraction while withstanding the repulsion out of a neighborhood of a saddle for a longer period~$\sim \frac{\beta}{\mu} l_\eps $ (this is a rare event with probability of order $\e^{\frac{\beta}{\rho}-1}$ as we just computed),  then this is enough for the diffusion to approach the unstable heteroclinic connection at a distance of order at most~$\eps^\beta$. 

We give a more precise study of tails of the exit times in~Section~\ref{sec:LLT}. For the proof of Theorem~\ref{th:2saddles} we only need one more estimate on the exit time, which can be viewed as a stronger version of Theorem~\ref{th:poincare-saddle}~\eqref{item:exit_time_is_log}.
\begin{lemma}\label{lem:exit_time_is_log}
Under conditions~\ref{setting:general},~\ref{setting:geometry-domain}, and~\ref{setting:conjugacy}, for each $\delta>0$,
\begin{align*}
    \sup_{x\in K_\vk(\eps)}\Psc^x\left\{\left|\frac{\tau_\e}{\frac{1}{\lambda} l_\eps }-1\right|>\delta\right\}=O\left(\eps^\delta\right).
\end{align*}
\end{lemma}
\bpfm 
Using the expression~\eqref{eq:tau_formula_heuristic}, we obtain
\begin{align*}
    \Psc^x\left\{\tau_\e<\frac{1-\delta}{\lambda} l_\eps \right\} \approx \Psc^x\{ |x+\frU|>\eps^{-\delta}R\} = o_e(1),\\
    \Psc^x\left\{\tau_\e>\frac{1+\delta}{\lambda} l_\eps \right\} \approx \Psc^x\{ |x+\frU|<\eps^{\delta}R\} = O\left(\eps^\delta\right),
\end{align*}
uniformly in $x\in K_\vk(\eps)$.
\epf

\medskip
With Lemmas~\ref{lem:exit-straight}--\ref{lem:exit_time_is_log} at hand, we can  give a rigorous proof of Theorem~\ref{th:2saddles}.

\medskip

\bpf[Proof of Theorem~\ref{th:2saddles}] Part~\ref{item:typical-case} follows from part~\ref{item:if-exp=1-bifurcation} of Theorem~\ref{thm:typical}. Part~\ref{item:too-far-from-network} is a consequence of Lemmas~\ref{lem:exit-on-the-same-whp} and \ref{lem:alpha<1-and-rho-alpha<1-- concentration at exit}.

To prove part \ref{item:power-asymp-for-2saddles}, we will combine Lemmas~\ref{lem:local-limit-theorem} and~\ref{lem:postive-limit-prob-if-close-to-manifold}. First, we write
\begin{align*}
\Pp(A_{2,\e})&=\E [\Pp (A_{2,\e}| X_{\e,\tau^1_\e})]=I_1+I_2,
\end{align*}
where
\begin{align*}
I_1&=\E [\Pp (A_{2,\e}| X_{\e,\tau^1_\e}) \ONE_{\{X_{\e,\tau^1_\e}\in x_1+\e K_{\vk'}(\e)v_1 \}}],\\
I_2&=\E [\Pp (A_{2,\e}| X_{\e,\tau^1_\e}) \ONE_{\{X_{\e,\tau^1_\e}\notin  x_1+\e K_{\vk'}(\e)v_1\}}].
\end{align*}
We can also write
\begin{align*}
I_1=I_{11}+I_{12}
\end{align*}
where
\begin{align*}
I_{11}&= \E [\Pp (A_{2,\e}| X_{\e,\tau^1_\e }) \ONE_{\{X_{\e,0}\in x_0+\e K_{\vk}(\eps)v;\  X_{\e,\tau^1_\e}\in x_1+\e K_{\vk'}(\e)v_1\}}],\\
I_{12}&= \E [\Pp (A_{2,\e}| X_{\e,\tau^1_\e }) \ONE_{\{X_{\e,0}\notin x_0+\e K_{\vk}(\eps)v;\  X_{\e,\tau^1_\e}\in x_1+\e K_{\vk'}(\e)v_1\}}].
\end{align*}
Let us prove that the leading term $I_{11}$ satisfies
\begin{align}
\label{eq:I_11}
I_{11}= \e^{\frac{1}{\rho}-1} (h +o(1)),
\end{align}
where
\begin{align}\label{eq:h_2_saddles}
    h = \E [g_c(\xi_{0})]  \int_{\R} \psi_s(-z)\nu(dz).
\end{align}
Since $g_c$ (given in \eqref{eq:g_c_Gaussian}) and $\psi_s$ (given in \eqref{eq:gaussian-cdf}) are positive everywhere, and since $\nu\in\GoodMeasures$, we have $h\in(0,\infty)$.

Let us introduce additional notation. Extending the definition of $\Psc^x$ as the distribution associated with the initial condition given by~\eqref{eq:entrance-scaling-conditioned}, we will denote by~$\Psc^x_i,$ $i=1,2$, the distribution of the diffusion with initial condition
\begin{equation*}
X_{\e,0}=x_{i-1}+\e x v_{i-1}.
\end{equation*}

For brevity, we write $\nu_x(dz) = g_{c}(x)\nu(dz)$, where the measure $\nu\in\GoodMeasures$ and constant $c>0$ are introduced in Lemma~\ref{lem:local-limit-theorem} applied to the 
diffusion near the saddle point $O_1$, and
\begin{gather*}
    \nu_{x,\eps}(dz) = \eps^{-(\frac{1}{\rho}-1)}\Psc^x_1\{X_{\e,\tau^1_\e} \in x_1 +\eps  (dz)\, v_1 \} .
\end{gather*}
Using this notation, we can rewrite
\begin{align}\label{eq:I_11_rewrite}
    I_{11} = \eps^{\frac{1}{\rho}-1} \E\ONE_{\xi_{0,\eps}\in K_\vk(\eps)}\int_{K_{\vk'}(\eps)}\Psc^z_2(A_{2,\eps})\nu_{\xi_{0,\eps},\eps}(dz).
\end{align}
Let us show that the following is $o(1)$ uniformly over $x\in K_\vk(\eps)$:
\begin{align}
    &\left|\int_{K_{\vk'}(\eps)}\Psc^z_2(A_{2,\e})\nu_{x, \eps}(dz) - \int_{K_{\vk'}(\eps)}\psi_s(-z)\nu_{x}(dz)\right| \notag\\
   & \leq \left| \int_{K_{\vk'}(\eps)}\psi_s(-z)\big(\nu_{x, \eps}(dz)-\nu_x(dz)\big)\right| + \left| \int_{K_{\vk'}(\eps)}\big(\Psc^z_2(A_{2,\e})-\psi_s(-z)\big)\nu_{x,\eps}(dz)\right|\notag\\
    & =II_{1} + II_{2}. \label{eq:II_1+II_2}
\end{align}
To estimate $II_1$, we first note that $z\mapsto \psi_s(-z)$ is decreasing and takes values in~$[0,1]$. Hence, setting $n_\eps = \lfloor \eps^{-\delta/2}\rfloor+1$ with $\delta$ given in Lemma~\ref{lem:local-limit-theorem}~\eqref{item:power-asymp}, we can find, for $1\leq i\leq n_\eps$,
closed intervals $E_i\subset K_{\vk'}(\eps)$ with disjoint interiors whose union is $K_{\vk'}(\eps)$ such that $\psi_s(-z)\in [\frac{i-1}{n_\eps},\frac{i}{n_\eps}]$ for $z\in E_i$. Then, we have 
\begin{align*}
    &\left|\int_{E_i}\psi_s(-z)\left(\nu_{x, \eps}(dz)-\nu_x(dz)\right)\right|\\
    &\leq \max\left\{\left|\frac{i}{n_\eps}\nu_{x,\eps}(E_i) - \frac{i-1}{n_\eps}\nu_x(E_i)\right|,\ \left|\frac{i}{n_\eps}\nu_x(E_i) - \frac{i-1}{n_\eps}\nu_{x,\eps}(E_i)\right|\right\}\\
    &\leq \max\left\{\left|\frac{i}{n_\eps}\nu_{x,\eps}(E_i) - \frac{i}{n_\eps}\nu_x(E_i)\right|,\ \left|\frac{i-1}{n_\eps}\nu_x(E_i) - \frac{i-1}{n_\eps}\nu_{x,\eps}(E_i)\right|\right\}+\frac{1}{n_\eps}\nu_x(E_i)
    \\
    &\leq \left|\nu_{x,\eps}(E_i)-\nu_x(E_i)\right| + \frac{1}{n_\eps}\nu_x(E_i).
\end{align*}
Then,
\begin{align*}
    II_1 \leq  \sum_{i=1}^{n_\eps}\left|\nu_{x,\eps}(E_i)-\nu_x(E_i)\right| + \frac{1}{n_\eps}\nu_x(K_{\vk'}(\eps)).
\end{align*}
Now, using Lemma~\ref{lem:local-limit-theorem}~\eqref{item:power-asymp}, the boundedness of $g_c$, and the fact that $\nu\in\GoodMeasures$, we conclude that $II_1 = o(1)$ uniformly in $x\in K_\vk(x)$. To estimate $II_2$, we note that  Lemma~\ref{lem:postive-limit-prob-if-close-to-manifold} and Lemma~\ref{lem:local-limit-theorem}~\eqref{item:power-asymp} imply that, for some $\delta,\delta'>0$,
\begin{align*}
    II_2 \leq o(\eps^{\delta}) \nu_{x,\eps}(K_{\vk'}(\eps))\leq o(\eps^{\delta})\left(\nu_x(K_{\vk'}(\eps))+\smallo{\eps^{\delta'}}\right)=o(1),
\end{align*}
where the last equality is due to $\nu_x = g_c(x)\nu$, the boundedness of $g_c$ and the fact that $\nu\in\GoodMeasures$. Hence, the last term in \eqref{eq:II_1+II_2} is $o(1)$. Using this, \eqref{eq:I_11_rewrite}, and our definition of $\nu_x$, we obtain
\begin{align*}
    I_{11}= \e^{\frac{1}{\rho}-1} (h_\eps +o(1)),
\end{align*}
where
\begin{align*}
    h_\eps = \E \left [ g_c(\xi_{0,\eps})\ONE_{\xi_{0,\eps}\in K_\vk(\eps)}\right] \int_{K_{\vk'}(\eps)}\psi_s(-z)\nu(dz).
\end{align*}
Using the tameness of $\xi_{0,\eps}$ and the boundedness of $g_c$, we have 
\begin{align*}
     \E \left [ g_c(\xi_{0,\eps})\ONE_{\xi_{0,\eps}\in K_\vk(\eps)}\right] = \E [ g_c(\xi_{0,\eps})]+o_e(1)
\end{align*}
for sufficiently large $\kappa$. Using the convergence of $\xi_{0,\eps}$ in distribution to $\xi_0$ (see \ref{setting:scaling-limit-at-saddle-0}), we have $\lim_{\eps\to0}\E [g_c(\xi_{0,\eps})]  = \E[g_c(\xi_0)]$. Using the exponential decay of $\psi_s(-z)$ as $z\to\infty$ (see the definition of $\psi_s$ in \eqref{eq:gaussian-cdf}), and the polynomial growth of $\nu$ given in~\eqref{eq:GoodMeasures}, we can see that the integral in the definition of $h_\eps$ converges to the integral in the definition of $h$ as $\eps\to0$. Hence, we have $\lim_{\eps\to 0}h_\eps = h$, and thus \eqref{eq:I_11} follows.

\medskip
The proof of~\eqref{eq:polynomial-asymp-2saddles} will be complete if we show that $I_{12}$ and~$I_2$ are~$o_e(1)$. 
Since $|\xi_{0,\e}|$ is of order at most $1$, we obtain $I_{12}=o_e(1)$ for sufficiently large $\vk$.
To estimate~$I_2$, we write
\begin{align*}
I_2\le I_{21}+I_{22},
\end{align*}
where
\begin{align*}
I_{21}&=\Pp\left(A_{2,\e}\cap\{ X_{\e,\tau^1_\e}\in x_1+\e (-\infty,-l_\e^\vk] v_1\}\right),
\\
I_{22}&=\Pp\left(A_{2,\e}\cap \{ X_{\e,\tau^1_\e}\in x_1+\e (l_\e^\vk,\infty) v_1\}\right).
\end{align*}
For sufficiently large $\vk$, $I_{21}=o_e(1)$  due to Lemma~\ref{lem:local-limit-theorem}\eqref{item:exit-below-unlikely}, and $I_{22}=o_e(1)$ due to Lemma~\ref{lem:exit-on-the-same-whp}, so~\eqref{eq:polynomial-asymp-2saddles} follows, with $h$ given in \eqref{eq:h_2_saddles}. 

Now, we turn to~\eqref{eq:exit-time-asymp-2saddles}.  For $\delta>0$, we write 
\begin{align}
    &\Pp\left\{\left|\tau^2_\e - (\mu_1^{-1}+\lambda_2^{-1}) l_\eps \right|>2\delta l_\eps ,\ A_{2,\eps}\right\}\notag
    \\
    \leq&\, \Pp\left\{\left|\tau^1_\e - \mu_1^{-1} l_\eps \right|>\delta l_\eps ,\ A_{2,\eps}\right\} +  \Pp\left\{\left|\tau^2_\e -\tau^1_\e- \lambda_2^{-1} l_\eps \right|>\delta l_\eps ,\ A_{2,\eps}\right\}.\label{eq:tau_2,tau_1}
\end{align}
To bound the first term in~\eqref{eq:tau_2,tau_1}, we use the assumption that $\xi_{0,\eps}$ is tame, and apply Lemma~\ref{lem:exit-on-the-same-whp} 
to the diffusion near $O_2$ and Lemma~\ref{lem:local-limit-theorem}~\eqref{item:exit-below-unlikely} to the diffusion near the saddle point $O_1$ to see that for all sufficiently large $\vk>0$,
\begin{align}\label{eq:X_tau_1_A_2}
    X_{\e,\tau_1}\in x_1+\eps [-l^\vk_\eps, l^\vk_\eps]  v_1
\end{align}
on $A_{2,\eps}$ except for an exceptional set of low probability. Then, due to~\eqref{eq:tau_cvg_in_prob_quantified} and the tameness of $\xi_{0,\eps}$, the first term in~\eqref{eq:tau_2,tau_1} is $o(\eps^{\frac{1}{\rho_1}-1})$. Due to~\eqref{eq:X_tau_1_A_2}, the second term can be bounded from above by
\begin{align*}
    o_e(1)+ \E\left[\Pp\left(\left|\tau^2_\e -\tau^1_\e- \lambda_2^{-1} l_\eps \right|>\delta l_\eps ,\ A_{2,\eps}\Big| X_{\e,\tau_1}\right)\ONE_{X_{\e,\tau_1}\in q_++\eps [-l^\vk_\eps, l^\vk_\eps]v_1}\right].
\end{align*}
Now combining~ Lemma~\ref{lem:exit_time_is_log}  and Lemma~\ref{lem:local-limit-theorem}~\eqref{item:power-asymp}, we can see that the second term in~\eqref{eq:tau_2,tau_1} is $o(\eps^{\frac{1}{\rho_1}-1})$. Hence,~\eqref{eq:exit-time-asymp-2saddles} follows from these and~\eqref{eq:polynomial-asymp-2saddles}.
\epf

\section{Heteroclinic chains of arbitrary length}\label{sec:long-escape-chains}
\subsection{Introduction}\label{sec:intro-to-long-escape-chains}

The goal of this section is to give a rigorous statement of our main result described briefly in Section~\ref{sec:main_result_descr}, give some intuition behind it, and a combination of heuristic and rigorous arguments. Sections~\ref{sec:rectified}--\ref{section:density_est} contain the proofs  adding rigor to the heuristic arguments.

Our main result concerns the decay rates (as $\e\to0$) for probabilities to follow arbitrarily long heteroclinic chains of the kind shown in Figure~\ref{fig:main-heteroc}, where curves $\gamma_0,\gamma_1,\ldots,\gamma_{n-1}$ belong to the boundary of one cell (of arbitrary orientation, clockwise or counterclockwise, see  Figure~\ref{fig:heteroc-example}), and $\gamma_n$ does not belong to that boundary, ``making a wrong turn''. We call such a heteroclinic chain a cell escape chain.

The setting for this section is described by conditions \ref{setting:general},~\ref{setting:het-chain},~\ref{setting:scaling-at-saddle-0},~\ref{setting:scaling-limit-at-saddle-0},~\ref{setting:conjugacy-all}, and the requirement that $(\gamma_0,\gamma_1,\ldots,\gamma_{n})$ is a cell escape chain.
We recall  \eqref{eq:A_k}, the definition of the event $A_{n,\eps}$ describing sequential exits from domains $D_1,D_2,\ldots,D_n$.
In the case of a cell escape chain, it is natural to say that  on $A_{n,\eps}$ the diffusion escapes from the cell along the sequence $(\gamma_0,\gamma_1,\ldots,\gamma_{n})$. However, it is important to distinguish between the escape and the first exit. In principle, it is possible for the diffusion to cross $\gamma_k$ for some $k<n$ (thus exiting the cell) and still follow the remaining heteroclinic connections of the chain closely.

Similarly to the case of short heteroclinic chains considered in the previous section, we will describe conditions under which, in the limit $\eps\to0$,
the
probability 
of $A_{n,\eps}$ 
either converges to a positive number, or decays to~$0$, either as a power of~$\e$ or faster than any power of~$\e$, see Theorem~\ref{th:multi-saddle-escape}.

\begin{figure}[ht]
\centering
\begin{tikzpicture}

    \node (0) at (-5, 0) {};
    \node (1) at (-2.25, 0) {};
    \node (2) at (2.25, 0) {};
    \node (3) at (4.75, 0) {};
    \node (4) at (-5.875, -1.55) {};
    \node (5) at (-4.25, -1.5) {};
    \node (6) at (-2.925, -1.6) {};
    \node (7) at (-1.25, -1.5) {};
    \node (8) at (1.375, -1.675) {};
    \node (9) at (2.75, -1.25) {};
    \node (10) at (3.9, -2.25) {};
    \node (11) at (6, -1.5) {};
    \node (12) at (-5.25, 2.25) {};
    \node (13) at (-0.25, 0.95) {};
    \node (14) at (0.75, 0.95) {};
    \node (15) at (6.25, 1.5) {};
    \node (16) at (-6, -0.75) {};
    \node (17) at (-5.25, -1.25) {};
    \node (18) at (-3, -0.25) {};
    \node (19) at (-2.5, 1) {};
    \node (20) at (-2.25, -1.25) {};
    \node (21) at (-3.25, -1) {};
    \node (22) at (-0.625, 0.325) {};
    \node (23) at (-0.475, 1.575) {};
    \node (24) at (1.25, -1) {};
    \node (25) at (2, -1.25) {};
    \node (26) at (3.9, 0.075) {};
    \node (27) at (4.25, 1) {};
    \node (28) at (0.575, 0.275) {};
    \node (29) at (1.375, 0) {};
    \node (30) at (1.825, 1.15) {};
    \node (31) at (0.75, 1.775) {};
    \node (32) at (3.5, -1.25) {};
    \node (33) at (4.75, -1.5) {};
    \node (34) at (5.475, 1.475) {};
    \node (35) at (6, 0.7) {};
    \node (36) at (0, 1.925) {};
    \node (37) at (-0.5, -0.5) {};

    \node (x0) at (-5.2, 1) {};
    \node (x1) at (-2.75, 0.365) {};
    \node (x2) at (-0.557, 0.942) {};
    \node (xn-2) at (1.63, 0.625) {};
    \node (xn-1) at (4.07, 0.523) {};
    \node (xn) at (4.142, -1.37) {};
    \node (v0) at (-4.75, 1.125) {};
    \node (v1) at (-2.5, 1) {};
    \node (v2) at (-0.475, 1.575) {};
    \node (vn-2) at (1.825, 1.15) {};
    \node (vn-1) at (4.25, 1) {};
    \node (vn) at (3.5, -1.25) {};

\foreach \n in {0, 1, 2, 3}
        \node at (\n)[circle,fill,inner sep=1.75pt]{};

    \foreach \n in {x0, x1, x2, xn-2, xn-1, xn}
        \node at (\n)[circle,fill,inner sep=1.35pt]{};

\draw [thick, -latex] (x0.center) to (v0.center);
    \draw [thick, -latex] (x1.center) to (v1.center);
    \draw [thick, -latex] (x2.center) to (v2.center);
    \draw [thick, -latex] (xn-2.center) to (vn-2.center);
    \draw [thick, -latex] (xn-1.center) to (vn-1.center);
    \draw [thick, -latex] (xn.center) to (vn.center);

    \draw [ultra thick, dotted] (13.center) to (14.center);

    \draw [ultra thick, postaction={on each segment={mid arrow=black}}] [in=105, out=-90] (12.center) to (0.center);
    \draw [postaction={on each segment={mid arrow=black}}] [in=75, out=-135, looseness=0.75] (0.center) to (4.center);
    \draw [postaction={on each segment={mid arrow=black}}] [in=-75, out=120, looseness=0.75] (5.center) to (0.center);
    \draw [ultra thick, postaction={on each segment={mid arrow=black}}] [in=135, out=45] (0.center) to (1.center);
    \draw [postaction={on each segment={mid arrow=black}}] [in=75, out=-120, looseness=0.75] (1.center) to (6.center);
    \draw [postaction={on each segment={mid arrow=black}}] [in=315, out=120, looseness=0.75] (7.center) to (1.center);
    \draw [ultra thick, postaction={on each segment={mid arrow=black}}] [in=180, out=60] (1.center) to (13.center);
    \draw [ultra thick, postaction={on each segment={mid arrow=black}}] [in=120, out=0, looseness=0.75] (14.center) to (2.center);
    \draw [postaction={on each segment={mid arrow=black}}] [in=75, out=-120, looseness=0.50] (2.center) to (8.center);
    \draw [postaction={on each segment={mid arrow=black}}] [in=-60, out=105, looseness=0.75] (9.center) to (2.center);
    \draw [ultra thick, postaction={on each segment={mid arrow=black}}] [in=135, out=60, looseness=1.25] (2.center) to (3.center);
    \draw [ultra thick, postaction={on each segment={mid arrow=black}}] [in=75, out=-120] (3.center) to (10.center);
    \draw [postaction={on each segment={mid arrow=black}}] [in=315, out=120] (11.center) to (3.center);
    \draw [postaction={on each segment={mid arrow=black}}] [in=-150, out=60, looseness=0.75] (3.center) to (15.center);
    \draw (16.center) to (17.center);
    \draw [in=-180, out=60, looseness=0.75] (17.center) to (18.center);
    \draw (19.center) to (18.center);
    \draw [in=90, out=150, looseness=1.75] (19.center) to (16.center);
    \draw (20.center) to (21.center);
    \draw [in=165, out=120, looseness=1.75] (21.center) to (23.center);
    \draw (23.center) to (22.center);
    \draw [in=60, out=-165] (22.center) to (20.center);
    \draw (25.center) to (24.center);
    \draw [in=150, out=105, looseness=1.75] (24.center) to (27.center);
    \draw (27.center) to (26.center);
    \draw [in=60, out=-180, looseness=0.75] (26.center) to (25.center);
    \draw [in=135, out=0, looseness=0.75] (28.center) to (29.center);
    \draw (29.center) to (30.center);
    \draw [in=0, out=120] (30.center) to (31.center);
    \draw (33.center) to (32.center);
    \draw [in=180, out=105, looseness=1.75] (32.center) to (34.center);
    \draw (34.center) to (35.center);
    \draw [in=75, out=-135] (35.center) to (33.center);
    \draw [in=150, out=180, looseness=1.75] (36.center) to (37.center);

    \node at (-5.4, 0.9) {$x_0$};
    \node at (-2.87, 0.6) {$x_1$};
    \node at (-0.75, 1.15) {$x_2$};
    \node at (2.1, 0.635) {$x_{n-2}$};
    \node at (4.55, 0.55) {$x_{n-1}$};
    \node at (4.4, -1.6) {$x_{n}$};

    \node at (-4.59, 1.125) {$v_0$};
    \node at (-2.76, 0.97) {$v_1$};
    \node at (-0.475, 1.7) {$v_2$};
    \node at (2.15, 1.26) {$v_{n-2}$};
    \node at (4.4, 1.17) {$v_{n-1}$};
    \node at (3.4, -1.35) {$v_{n}$};

    \node at (-5.3, 0.1) {$O_1$};
    \node at (-2.6, 0) {$O_2$};
    \node at (2.75, 0) {$O_{n-1}$};
    \node at (4.4, 0) {$O_{n}$};

    \node at (-5, 1.455) {$\gamma_0$};
    \node at (-3.9, 0.85) {$\gamma_1$};
    \node at (-1.7, 0.95) {$\gamma_2$};
    \node at (0.95, 1.1) {$\gamma_{n-2}$};
    \node at (3, 1) {$\gamma_{n-1}$};
    \node at (4.65, -0.75) {$\gamma_{n}$};

    \node at (-5.2, -0.75) {$D_1$};
    \node at (-2.3, -0.75) {$D_2$};
    \node at (-0.5, -0.06) {$D_3$};
    \node at (0.7, 0.5) {$D_{n-2}$};
    \node at (1.6, -0.25) {$D_{n-1}$};
    \node at (3.85, -0.8) {$D_{n}$};

\end{tikzpicture}

\caption{A cell escape heteroclinic chain is almost entirely, except the last heteroclinic connection, a part of the boundary of one cell.}
\label{fig:main-heteroc}
\end{figure}
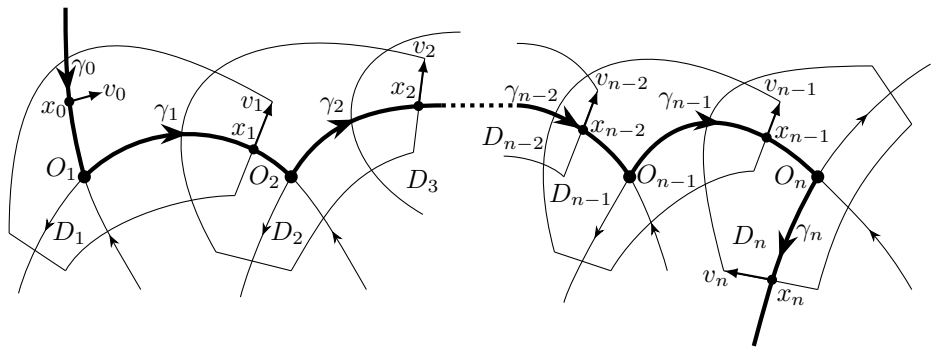

Let us discuss the ideas behind our approach first.
To study the decay of $\Pp(A_{n,\e})$ we need to supplement results of Sections~\ref{sec:typical},~\ref{sec:2saddles} with more precise analysis of how the distance from the diffusing particle to the heteroclinic chain changes upon passing near a saddle point. More precisely, we need to quantify how probable or improbable transitions between various orders of magnitude are.  

We already know that some transitions are typical, some are unlikely, and probabilities of some transitions decay as a power of $\e$. 

In addition to this, we will also prove (see Lemma~\ref{lem:if-entrance-far-exit-far}) that if $\alpha\le 1$, $\alpha\rho\le 1$ (so that $\alpha'=\alpha\rho$), and the distance from the starting point of diffusion to the heteroclinic chain is of order above $\eps^\alpha$, then the exit from the neighborhood of the saddle happens w.h.p.\ on the same side of the chain, at distance of order above~$\eps^{\alpha\rho}$. Iterating this statement, we can work with a  sequence of exponents~$(\bar\alpha_i)_{i\ge\kstar}$ for some $\kstar\in\{0,1,\dots,n-1\}$, such that if the distance from the exit point from a neighborhood of a saddle point~$O_i$,  to the heteroclinic connection is above order~$\eps^{\bar\alpha_i}$, then w.h.p.\ the same holds for exit locations for all saddle points $O_i'$ with $i'>i$ (see Lemma~\ref{lem:new-scaling}). Requiring that $\bar\alpha_{n-1}=1$ (this fixes a concrete sequence $(\bar\alpha_i)$) and using the fact that  if the diffusion enters the neighborhood of $O_n$ at distance of order above~$\eps^1$, then w.h.p.\ it exits on the same side of the heteroclinic chain (see Lemma~\ref{lem:exit-on-the-same-whp}) and thus the cell escape does not happen, we can conclude that conditioned on one of those events of exiting too far from the network, the escape event happens with low probability. This allows to conclude that by restricting the diffusion to exit all saddles $O_i$ through a window of size of order $\eps^{\bar\alpha_i}$, we only make a tiny (``low probability'') error when computing $\Pp(A_{n,\e})$ (see Lemma~\ref{lem:restricting-to-a-narrow-corridor}).

We will see that once the diffusion exits are restricted to those windows, all transitions after the saddle point $\kstar$ can be classified into two types: transitions from scale $\e^\alpha$ to 
scale~$\e^{\alpha\rho}$ for $\alpha<1$ and $\alpha\rho\le1$; transitions from scale~$\e^1$ to scale~$\e^\beta$ with $\beta\in(\rho,1]$. A transition of the former type is typical, i.e., it gets realized with probability converging to~$1$. The probability of a transition of the latter type decays as a power of $\eps$.  Thus it is plausible that the probability to realize all of those transitions behaves as the product of these powers of $\eps$, i.e., it is a power of $\eps$ itself. However, in order to make this argument rigorous and prove that the escape probability equals $h\e^\theta(1+o(1))$ for some constant $h>0$ (see~\eqref{eq:main_polynomial_asymptotics}) we have to study scaling limits of transition kernels between those windows and obtain results in a form that allows for iterative analysis of convolutions of those kernels (see Lemmas~\ref{lem:trans_ker},  \ref{lem:typical_chain_exit}, \ref{lem:local-limit-for-H}), with limiting measures defined as certain nonlinear transformations of Gaussian distributions.

\subsection{The new sequence of effective exponents and the main result}
To state the main result (Theorem~\ref{th:multi-saddle-escape} below) we need to define a new sequence of exponents $(\bar\alpha_i)$ agreeing with the original sequence $(\alpha_i)$ up to a certain index~$\kstar$ and  describing the 
scales $\e^{\bar\alpha_k}$ on which the distributions of $X_{\e,\tau^k_\e}$ concentrate in order to realize the ``wrong turn'', i.e., the event $A_{n,\e}$. 
The definition of the new sequence $(\bar \alpha_i)$ may seem unintuitive at the first sight but 
it follows the logic described in Section~\ref{sec:intro-to-long-escape-chains}, and 
in Lemma~\ref{lem:restricting-to-a-narrow-corridor} 
we will give an approximation to the escape or ``wrong turn'' event of interest $A_{n,\e}$ in terms of $(\bar \alpha_i)$. Namely, we will show that  on~$A_{n,\e}$  (up to an exceptional  low probability event),  for all~$i\ge\kstar$, 
the diffusion exits the saddle $O_i$ at distance of order at most $\e^{\bar \alpha_i}$ from its unstable manifold. Then the main task will be to analyze the convolutions of  the transition kernels from scale $\e^{\bar\alpha_i}$  
to scale $\e^{\bar\alpha_{i+1}}$ over all $i\ge\kstar$.

Let us recall that given $\alpha_0\in(0,1]$, the exponents $(\alpha_k)_{k=0}^n$ are computed iteratively using~\eqref{eq:recursion_alpha}. Let
\begin{equation}\label{eq:def_kstar}
\kstar =\max\{k:\  0\le k \le n-1,\ \alpha_k=1 \}.
\end{equation}
If $\alpha_i<1$ for all $i=0,\ldots,n-1$, then $\kstar$ is not defined.

If $0\le k < j\le n-1$, we define
\begin{equation*}
\rho_{kj}=\prod_{i=k+1}^j\rho_{i}.
\end{equation*}
Agreeing that a product over an empty set equals $1$, we also set $\rho_{kk}=1$ for all $k$. 
We call $k\in\{\kstar+1,\ldots, n-1\}$ and its  associated saddle $O_k$ {\it binding} if  $\rho_{kj}<1$ for all $j\in\{k+1,\ldots,n-1\}$.
For $k=n-1$,
the latter set is empty and the condition is trivially true,  so $n-1$ is always binding. Let us denote the set of all binding indices by $H$.

We will also need the sets 
\begin{align}
H'&=H\cup \{\kstar\}\setminus \{n-1\},\label{eq:H'}\\
J&=H' + 1=\{k+1:k\in H'\}. \label{eq:def_J}
\end{align}
As we will see, in order to realize $A_{n,\e}$, up to a low probability event, the diffusion must stay $\e^1$-close to the heteroclinic connection upon passing near each binding saddle, and  near each saddle $i\in J$ it must spend abnormally long time getting from scale $\e^1$ to scale $\e^{\bar\alpha_i}$ with $\bar\alpha_i\in(\rho_i,1]$.  
These are the saddles that Lemma~\ref{lem:local-limit-theorem} will be applied to. They may be called  {\it the slowdown saddle points}.

For $i\in\{\kstar+1,\ldots,n-1\}$, let $k(i)=\min\{k\in H:\ k\ge i\}$. We define the new exponents $\bar\alpha_i$ by 
\begin{equation}
\bar\alpha_i=
\begin{cases}\rho_{i,k(i)}^{-1},& \kstar+1\le i\le n-1,\\
\alpha_i, &i\le \kstar.
\end{cases}
\label{eq:bar-alpha}
\end{equation}

It is not explicit in the definition but the sequence $(\bar\alpha_i)$ is uniquely defined by the sequence $(\alpha_i)$.
Lemma~\ref{lem:structure-of-reverse-exponents} contains this claim and other properties of $(\alpha_i)$ and~$H$.
Figure~\ref{fig:alphas} gives an example
of  $(\alpha_i)$ and the associated $(\bar \alpha_i)$.

Having defined $H$ and $(\bar\alpha_i)$, we are ready to state the main result of
the paper.

\begin{theorem}
\label{th:multi-saddle-escape} Let us assume the setting defined by conditions~\ref{setting:general}, \ref{setting:het-chain},~\ref{setting:scaling-at-saddle-0}, ~\ref{setting:scaling-limit-at-saddle-0}, and~\ref{setting:conjugacy-all}. Let us also assume  that the heteroclinic chain is a cell escape chain shown on Figure~\ref{fig:main-heteroc}.

Let $\xi_{0,\e}$ be tame and, if $\alpha_0<1$,  let  $\Pp\{\xi_0>0\}>0$
(we recall that $\xi_{0,\e}$ is used to define
the initial condition in 
~\eqref{eq:initial-condition-near-chain} of~\ref{setting:scaling-at-saddle-0} and $\xi_0$ is the distributional limit of $\xi_{0,\e}$ from
\ref{setting:scaling-limit-at-saddle-0}).

Then the following holds true:
\begin{enumerate}
\item\label{item:k-star=n-1}
If $\kstar=n-1$, i.e., $\alpha_{n-1}=1$, then there is $p_n>0$ such that $\Pp(A_{n,\e})\to p_n$.

\item\label{item:cell-escape-probability}
If  $\kstar<n-1$, then there is a constant $h>0$ such that
\begin{equation}
\label{eq:main_polynomial_asymptotics}
\Pp(A_{n,\e})= h\e^\theta(1+o(1)),\quad \e\to0,
\end{equation}
where
\begin{equation}
\label{eq:theta}
\theta=\sum_{i\in J}  \left(\frac{\bar\alpha_{i}}{\rho_{i}}-1\right)>0.
\end{equation}
In this case, conditioned on $A_{n,\e}$,
\begin{align}\label{eq:tau_main_th}
    \frac{\tau^n_\e}{\bar\chi\log \e^{-1}} 
\inprob 1,\quad \e\to0,
\end{align}
where
\begin{align}\label{eq:bar_chi_def}
\bar\chi
=\sum_{i\notin J}\frac{\bar\alpha_{i-1}}{\lambda_i}+\sum_{i\in J} \frac{\bar\alpha_i }{\mu_i}.
\end{align}
\item\label{item:k-star-not-defined}
If $\alpha_i<1$ for all $i=0,1,\ldots,n-1$ (i.e., $\kstar$ is not defined) and $\xi_{0,\e}$ is of order~1, then $A_{n,\e}$ happens with low probability.
\end{enumerate}
\end{theorem}

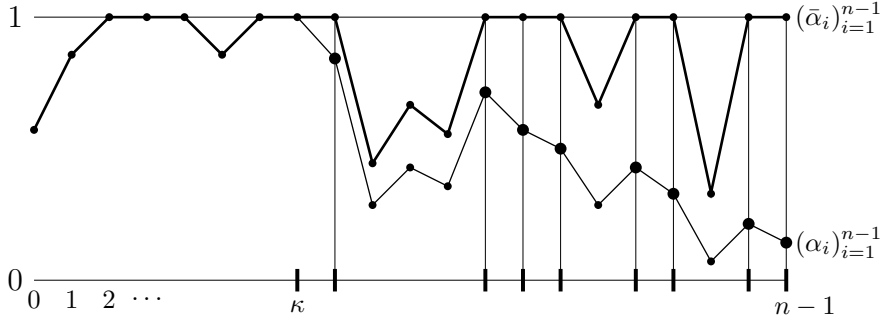
\begin{figure}[ht]

\centering

\begin{tikzpicture}[scale=0.5]

\draw [line width=0.3] (4,3) --(24,3);
\draw [line width=0.3] (4,10) --(24,10);
\draw [line width=0.5] (4,7)--(5,9)--(6,10)--(7,10)--(8,10)--(9,9)--(10,10)
--(11,10)--(12,8.9)--(13,5)--(14,6)--(15, 5.5)--(16,8)--(17,7)--(18,6.5)--(19,5)--(20,6.0)--(21,5.3)
--(22,3.5)--(23,4.5)--(24,4);
\draw [line width=1] (4,7)--(5,9)--(6,10)--(7,10)--(8,10)--(9,9)--(10,10)
--(11,10)--(12,10)--(13,6.111)--(14,7.666)--(15, 6.888)--(16,10)--(17,10)--(18,10)--(19,7.666)--(20,10)--(21,10)
--(22,5.333)--(23,10)--(24,10);

\fill[black] (4,7) circle (3pt);
\fill[black] (5,9) circle (3pt);
\fill[black] (6,10) circle (3pt);
\fill[black] (7,10) circle (3pt);
\fill[black] (8,10) circle (3pt); 
\fill[black] (9,9) circle (3pt);
\fill[black] (10,10) circle (3pt);
\fill[black] (11,10) circle (3pt);  
\fill[black] (12,8.9) circle (4.5pt); 
\draw [line width=0.3] (12,3) --(12,10);
\fill[black] (13,5) circle (3pt); 
\fill[black] (14,6) circle (3pt);
\fill[black] (15, 5.5) circle (3pt);

\fill[black] (16,8) circle (4.5pt); 
\draw [line width=0.3] (16,3) --(16,10);

\fill[black] (17,7) circle (4.5pt); 
\draw [line width=0.3] (17,3) --(17,10);

\fill[black] (18,6.5) circle (4.5pt); 
\draw [line width=0.3] (18,3) --(18,10);

\fill[black] (19,5) circle (3pt);

\fill[black] (20,6.0) circle (4.5pt);  
\draw [line width=0.3] (20,3) --(20,10);

\fill[black] (21,5.3) circle (4.5pt);  
\draw [line width=0.3] (21,3) --(21,10);

\fill[black] (22,3.5) circle (3pt);

\fill[black] (23,4.5) circle (4.5pt); 
\draw [line width=0.3] (23,3) --(23,10);   

\fill[black] (24,4) circle (4.5pt);   \draw [line width=0.3] (24,3) --(24,10); 

 \fill[black] (12,10) circle (3pt);
 \fill[black] (13,6.111) circle (3pt);
 \fill[black] (14,7.666) circle (3pt);
 \fill[black] (15,6.888) circle (3pt);
 \fill[black] (16,10) circle (3pt);
 \fill[black] (17,10) circle (3pt);
 \fill[black] (18,10) circle (3pt);
 \fill[black] (19,7.666) circle (3pt);
 \fill[black] (20,10) circle (3pt);
 \fill[black] (21,10) circle (3pt);
 \fill[black] (22,5.3) circle (3pt);
 \fill[black] (23,10) circle (3pt);
  \fill[black] (24,10) circle (3pt);

\node at (3.5,3){\Large $0$};

\node at (3.5,10){\Large $1$};

\node at (4,2.5){$0$};
\node at (5,2.5){$1$};
\node at (6,2.5){$2$};
\node at (7,2.5){$\ldots$};

\node at (11,2.3){$\kstar$}; 
\draw [line width=1.5] (11,2.7) --(11,3.3);
\draw [line width=1.5] (12,2.7) --(12,3.3);
\draw [line width=1.5] (16,2.7) --(16,3.3);
\draw [line width=1.5] (17,2.7) --(17,3.3);
\draw [line width=1.5] (18,2.7) --(18,3.3);
\draw [line width=1.5] (20,2.7) --(20,3.3);
\draw [line width=1.5] (21,2.7) --(21,3.3);
\draw [line width=1.5] (23,2.7) --(23,3.3);
\draw [line width=1.5] (24,2.7) --(24,3.3);

\node at (24.5,2.3){$n-1$};

\node at (25.4,4){$(\alpha_i)_{i=1}^{n-1}$};
\node at (25.4,10){$(\bar\alpha_i)_{i=1}^{n-1}$};

\end{tikzpicture}
\caption{The lower graph is an example of the sequence $(\alpha_i)_{i=1}^{n-1}$. The upper one is the corresponding $(\bar\alpha_i)_{i=1}^{n-1}$. The two sequences coincide up to $\kstar$. The tickmarks on the horizontal axis show the binding points (elements of $H$) and $\kstar$. The bold dots show the record values of the 
sequence~$(\alpha_i)$ inspected from $n-1$ down to $\kstar+1$. The vertical lines pass through binding points and the associated record values. Note that the values of $\bar\alpha_i$ corresponding to those records are equal to 1.  
Various properties of the set $H$ and the sequence $(\bar\alpha_i)_{i=1}^{n-1}$ are discussed in Lemma~\ref{lem:structure-of-reverse-exponents}.}
\label{fig:alphas}
\end{figure}

\begin{remark}\rm
Let us compare this with Theorem~\ref{th:2saddles} on two saddles, where $n=2$. If $\alpha_1=1$, then $\kstar=1=n-1$, and we obtain the equivalence between part \ref{item:k-star=n-1}  of Theorem~\ref{th:multi-saddle-escape} and part \ref{item:typical-case} of Theorem~\ref{th:2saddles}.
If both $\alpha_0<1$ and $\alpha_1<1$, then $\kstar$ is not defined, and part \ref{item:k-star-not-defined} of Theorem~\ref{th:multi-saddle-escape}
coincides with part ~\ref{item:too-far-from-network} of Theorem~\ref{th:2saddles}. Finally, if $\alpha_0=1$ and $\alpha_1<1$, then $\kstar=0<n-1$,  
$H=\{1\}$, $J=\{1\}$, $\bar\alpha_1=1$, $\theta=1/\rho_1-1$, so Theorem~\ref{th:multi-saddle-escape}~\eqref{item:cell-escape-probability}
coincides with Theorem~\ref{th:2saddles}~\eqref{item:power-asymp-for-2saddles}.
\end{remark}

\begin{remark}\rm
Theorem~\ref{th:multi-saddle-escape}~\eqref{item:k-star=n-1}
is a specific case of Theorem~\ref{thm:typical}.  
Part~\ref{item:k-star-not-defined}
follows by induction from Lemma~\ref{lem:alpha<1-and-rho-alpha<1-- concentration at exit}. 
\end{remark}

\begin{remark} \rm
The requirement that $\Pp\{\xi_0=0\}=0$ in condition \ref{setting:scaling-limit-at-saddle-0}  for the case where $\alpha_0<1$
allows to avoid the situation where the distribution of the initial condition for the diffusion has a macroscopic component concentrated at scales smaller than $\e^{\alpha_0}$. However, one can modify our proof and show that even in that situation, in the case $\kstar<n-1$, under natural additional assumptions,~\eqref{eq:main_polynomial_asymptotics} holds with the same $\theta$ given by~\eqref{eq:theta}.
A step in that direction is Lemma~\ref{lem:rare_trans_upp_bdd}.
\end{remark}

\subsection{Proof of (\ref{eq:main_polynomial_asymptotics}) of Theorem~~\ref{th:multi-saddle-escape}~(\ref{item:cell-escape-probability})}

Here, we give a  proof of Theorem~\ref{th:multi-saddle-escape}~\eqref{item:cell-escape-probability} that is partially rigorous and partially heuristic. The heuristic arguments will be replaced by rigorous ones in Sections~\ref{sec:rectified}--\ref{section:density_est}. Also, the proof of Lemma~\ref{lem:typical_loc_lim} is postponed to Section~\ref{sec:original-coords}.

The main point in the proof is to show that the event $A_{n,\e}$ is realized, up to small probability events, by transitions described by local limit theorems of Lemma~\ref{lem:local-limit-theorem} that involve spending abnormal time near 
{\it slowdown} saddle points and typical transitions (described by Theorem~\ref{th:poincare-saddle}) near all other saddle points.

\subsubsection{Properties of the new exponent sequence} First, we
 collect various properties of the set $H$ of binding points and the exponents 
$(\bar\alpha_i)$ in
Lemma~\ref{lem:structure-of-reverse-exponents} below,
illustrated by Figure~\ref{fig:alphas}
Parts~\ref{item:record-setting1},~\ref{item:record-setting-algorithm}, and~\ref{item:record-setting2} of the lemma can be viewed as alternative definitions of~$H$  
describing it as the set of record points of the sequence $(\rho_{i,n-1})_{i=n-1,n-2,\ldots,\kstar+1}$.
Parts \ref{item:backward-recursion1} and \ref{item:backward-recursion2}
can be viewed as alternative definitions of $(\bar\alpha_i)_{i=0,\ldots,n-1}$.
Part~\ref{item:slowdown-saddles} will allow us to apply Lemma~\ref{lem:local-limit-theorem} to the diffusion near slowdown saddle points. 

\begin{lemma}
\label{lem:structure-of-reverse-exponents} Let us assume that $\kstar$ is well-defined.
\begin{enumerate}\item \label{item:old-alphas-after-k-star} For all $i\in\{\kstar+1,\ldots,n-1\}$, $\alpha_i=\alpha_{i-1}\rho_i<1$.

 \item\label{item:record-setting1}  For $i\in\{\kstar+1,\ldots,n-1\}$, $i\in H$ iff
 \begin{equation}
 \label{eq:def-of-record}
 \rho_{i,n-1}< \rho_{j,n-1},\quad j\in\{i+1,\ldots,n-1\}.
 \end{equation}
 
 \item \label{item:record-setting-algorithm}
  The set $H$ can be constructed via the following algorithm:
  \begin{itemize}
  \item
  initialize  $H:=\{n-1\}$, $j:=n-1$; 
   \item repeat the following cycle until the stop condition is met: 
   \begin{itemize}
   \item $A:=\left\{i\in\{\kstar+1,\ldots,j-1\}:\ \rho_{i,n-1}<\rho_{j,n-1}\right\}$;
   \item if $A=\emptyset$, then stop; \\else  redefine $j:=\max A$ and $H:=H\cup\{j\}$.
  \end{itemize}
\end{itemize}

 \item\label{item:record-setting2} For $i\in\{\kstar+1,\ldots,n-2\}$,  $i\in H$ iff 
 \begin{equation}
 \label{eq:record}
 \rho_{i,k(i+1)}<1.
 \end{equation}

 \item\label{item:alpha_bar_div_rho} For $i\in\{\kstar+1,\ldots,n-2\}$, $\bar\alpha_{i+1}/\rho_{i+1}>1$ iff $i\in H$.
 
  \item \label{item:backward-recursion1} For $i>\kstar$, one can define $\bar\alpha_i$  recursively via $\bar\alpha_{n-1}=1$ and 
 then for $i=n-2,n-3,\ldots,\kstar+1$ setting
 \begin{equation}
 \label{eq:backward-iteration1}
 \bar\alpha_i=\begin{cases} 1,& i\in H,\\
 \bar\alpha_{i+1}/\rho_{i+1},& i\notin H.
 \end{cases}
 \end{equation} 
 \item \label{item:backward-recursion2} One can define $\bar\alpha_i$ for $i>\kstar$ recursively via $\bar\alpha_{n-1}=1$ and 
  then for $i=n-2,n-3,\ldots,\kstar+1$ setting
\begin{equation*}
\bar\alpha_i=(\bar\alpha_{i+1}/\rho_{i+1} )\wedge 1.
\end{equation*}
In particular, for all $i$, we have 
$\bar\alpha_i\le 1.$
 
\item \label{item:bar-alpha-rho-less-1} For all $i\in\{\kstar,\ldots, n-2\}$, $\bar \alpha_{i}\rho_{i+1}\le 1$.

\item \label{item:between-binding}
For all $i\in\{\kstar+1,\ldots,n-1\}$, $\bar\alpha_i=\alpha_i/\alpha_{k(i)}> \alpha_i$.

\item \label{item:slowdown-saddles} If $i\in H'$, then $\bar\alpha_{i+1}\in(\rho_{i+1},1]$.

\item \label{item:alpha_bar_defined_by_alpha} The sequence $(\alpha_i)$ defines the set $H$ and 
the sequence $(\bar\alpha_i)$ uniquely.

\end{enumerate}
\end{lemma}

\bpf Part~\ref{item:old-alphas-after-k-star} follows from the definition of $\kstar$ and~\eqref{eq:recursion_alpha}.
Part~\ref{item:record-setting1} holds since $\rho_{ij}=\rho_{i,n-1}/\rho_{j,n-1}$, so~\eqref{eq:def-of-record} is equivalent to  $\rho_{ij}<1$. This part describes
$H$ as the set of record points of the the sequence $(\rho_{i,n-1})$ explored from $n-1$ down to $\kstar+1$.

This is made precise in the obvious algorithm described in part~\ref{item:record-setting-algorithm}. It discovers the record points one by one. Let us just comment that at any point of execution of this algorithm, $j$ is the latest found record point of $(\rho_{i,n-1})_{i=n-1,n-2,\ldots,\kstar+1}$ and~$H$ is the set of all already discovered record points. The algorithm stops when no new record points can be found.

Part~\ref{item:record-setting2} follows, since~\eqref{eq:record} means that $i$ is the next  record point discovered by the algorithm after discovering~$k(i+1)$.

Part~\ref{item:alpha_bar_div_rho} follows from part~\ref{item:record-setting2} since
\[
\bar\alpha_{i+1}/\rho_{i+1}=\rho_{i+1,k(i+1)}^{-1}/\rho_{i+1}=\rho_{i,k(i+1)}^{-1}.
\]

Part~\ref{item:backward-recursion1} directly follows from~\eqref{eq:bar-alpha} and the last display. Indeed, if $i\in H$, then $k(i)=i$ and thus $\bar \alpha_i=1$ due to~\eqref{eq:bar-alpha}. If $i\not\in H$, then $k(i)>i$ and $k(i)=k(i+1)$. Hence, $\bar\alpha_i =\rho^{-1}_{i,k(i+1)}=\bar\alpha_{i+1}/\rho_{i+1} $ due to~\eqref{eq:bar-alpha}.

Part~\ref{item:backward-recursion2} follows from parts~\ref{item:backward-recursion1},~\ref{item:alpha_bar_div_rho} and the definition~\eqref{eq:bar-alpha}

For  $i>\kstar$, part~\ref{item:bar-alpha-rho-less-1}  follows directly from part~\ref{item:backward-recursion2}. To prove it for $i=\kstar<n-1$, it suffices
to notice that $\alpha_{\kstar}=\bar\alpha_{\kstar}=1$ and $1>\alpha_{\kstar+1}=\alpha_{\kstar}\rho_{\kstar+1}=\rho_{\kstar+1}$.

To prove the identity in part~\ref{item:between-binding}, we note that it is trivially true for $i\in H$ and then parts~\ref{item:old-alphas-after-k-star} and~\ref{item:backward-recursion1} allow to extend it by induction to the remaining values of~$i$. The inequality then also follows since due to part~\ref{item:old-alphas-after-k-star}, $\alpha_{k(i)}<1$.

According to part~\ref{item:backward-recursion2},  $\bar\alpha_i\leq 1$ for all $i$. This and part~\ref{item:alpha_bar_div_rho} imply  part~\ref{item:slowdown-saddles} holds for $i\in H\setminus\{n-1\}$. For $i=\kstar$, since $\alpha_\kstar=\bar\alpha_\kstar=1$ and $\rho_{\kstar+1}<1$, part~\ref{item:between-binding} implies
$
\bar\alpha_{\kstar+1}> \alpha_{\kstar+1}=\rho_{\kstar+1},$ completing the proof of part~~\ref{item:slowdown-saddles}.

Let us prove part~\ref{item:alpha_bar_defined_by_alpha}.
Since the condition $\rho_{kj}<1$ for all $j>k$ is equivalent to $\alpha_j<\alpha_k$ for all $j>k$, we conclude that $H$ is uniquely
defined by $(\alpha_i)$. Therefore, the values $(k(i))_{i>\kstar}$ are also uniquely defined
by $(\alpha_i)$.
Now our claim follows from the identity in part~\ref{item:between-binding}.
\epf

\bigskip

\subsubsection{Preliminaries for analysis of transitions}

To make our proof of Theorem~\ref{th:multi-saddle-escape} work, we actually need a better understanding of the typical case discussed in Theorem~\ref{th:poincare-saddle}. In particular, we need
to control the tails of the distributions involved and to deal with measures from $\GoodMeasures$ instead of probability distributions.

\begin{lemma} 
\label{lem:if-entrance-far-exit-far}
Let us assume~\ref{setting:general},~\ref{setting:geometry-domain},~\ref{setting:initial-cond},~\ref{setting:conjugacy},  and
suppose $\alpha\le 1$ and $\alpha\rho\le1$. Suppose that 
in~\eqref{eq:entrance-scaling}, $\xi_\e$ is of order above 1. 
On $A_{+,\e}$, we define $\xi'_\e$
by~\eqref{eq:scaling-at-exit}.  
Then $\xi'_\e$ is of order above $1$. 
Moreover, this is true uniformly in initial conditions. More precisely, for every $\vk'>0$ and sufficiently large $\vk>0$,
\begin{align*}
    \sup_{x\in (l^{\vk}_\eps,\eps^{-\alpha}] }\Pp^{x_0+\eps^\alpha x v}\left\{ \xi_\eps'\leq l_\eps^{\vk'},\ A_{+,\eps}\right\} = o_e(1).
\end{align*}

\end{lemma}
\bpfm 
Recalling~\eqref{eq:X_2-at-exit-time}, we obtain
\begin{align}
\label{eq:X_2-at-exit-time-with-Z-plugged in}
X^2_{\e,\tau_\e}&=\e^{\alpha\rho}\frac{L}{R^{\rho}}|\xi_\e+\e^{1-\alpha}U^1_{\e,\tau_\e}|^{\rho}+\e U^2_{\e,\tau_\e}.
\end{align}
Using Lemma~\ref{lem:tail-control-for-basic-rvs-in-Duhamel},
we obtain that the first term is of 
order above $\e^{\alpha\rho}$.
Using this lemma once again,  we obtain that the entire expression is of order above $\e^{\alpha\rho}$,
so $\xi'_\e$ is of order above $1$. 
\epf

On $A_{k,\e}$ defined in~\eqref{eq:A_k}, $k\in{1,\ldots,n-1},$ we can define $\eta_k$ via 
\begin{equation*}
X_{\e,\tau^k_\e}=x_k+\eta_{k,\e}v_k,\quad \e>0.  
\end{equation*}
Using the definition of $\tau^k_\eps$ in~\eqref{eq:tau^k_eps}, we have $\eta_{k,\e}\in [-1,1]$.
The difference with~\eqref{eq:seq-scaling} is that there is no scaling factor in front of $\eta_{k,\e}$.

\begin{lemma}\label{lem:new-scaling}
Let~\ref{setting:general}, \ref{setting:het-chain},~\ref{setting:scaling-at-saddle-0},~\ref{setting:scaling-limit-at-saddle-0}, \ref{setting:conjugacy-all} hold. Let us assume that $i\geq\kstar$.  
If $\eta_{i,\e}$ is of order above $\e^{\bar\alpha_i}$, then  $\eta_{n-1,\eps}$ is of order above $\e^1$.
Moreover, this  is true uniformly in initial conditions.
More precisely, for every $\vk_{n-1}>0$ and sufficiently large $\vk_{i}>0$, it holds for every $\widetilde\vk_0$ that
\begin{align*}
    \sup_{x\in K_{\widetilde \vk_0}(\eps)}\Pp^{x_0+\eps^\alpha x v}\{\eta_{i,\eps}>\eps^{\bar\alpha_i}l^{\vk_i}_\eps,\ \eta_{n-1,\eps}\leq \eps l^{\vk_{n-1}}_\eps,\ A_{n-1,\eps}\}=o_e(1).
\end{align*}
\end{lemma}
We will also introduce $\vk_0$ later. It will be useful to distinguish between $\widetilde\vk_0$ and~$\vk_0$ in case $\kstar=0$.

\bpf First we use part~\ref{item:backward-recursion1} of Lemma~\ref{lem:structure-of-reverse-exponents}
in order to apply 
Lemma~\ref{lem:if-entrance-far-exit-far} iteratively to saddles $k(i), k(i)-1,\ldots,i+2,i+1$ concluding that, for every $\vk_{k(i)}$ and sufficiently large $\vk_i$, we have $\eta_{k(i),\eps}>\eps^1 l_\eps^{\vk_{k(i)}}$ w.h.p.\ on the event $\{\eta_{i,\eps}>\eps^{\bar\alpha_i}l^{\vk_i}_\eps,\ A_{n-1,\eps}\}$. Then, applying Lemma~\ref{lem:if-entrance-far-exit-far} iteratively to saddles $n-1,n-2,\ldots,k(i)+2,k(i)+1$ and adjusting $\vk_k$'s iteratively if necessary, we derive $\eta_{n-1,\eps}>\e^{ \rho_{k(i),n-1}}l^{\vk_n}_\eps$ w.h.p.\ on the same event for every $\vk_n>0$ and sufficiently large $\vk_i$. The proof is completed by taking into account that $k(i)\in H$ so that  $\rho_{k(i),n-1}<1$.
\epf

\subsubsection{Restricting the cell escape event to scales  defined by exponents $(\bar\alpha_i)$.}
Our next local goal is Lemma~\ref{lem:restricting-to-a-narrow-corridor} allowing to restrict further analysis to scales defined by exponents $(\bar\alpha_i)$.

\begin{lemma}\label{lem:typical_scale}
Under conditions~\ref{setting:general}, \ref{setting:het-chain}, ~\ref{setting:conjugacy-all} , we have that,
for all $k\in \{1,2,\dots,n-1\}$ and every $\vk_0>0$, there is $\vk_k> 0$ such that
\begin{align*}
    \sup_{x\in K_{\vk_0}(\eps)}\Pp^{x_0 + \eps^{\alpha_0}xv_0}\{X_\eps(\tau^k_\eps)\not\in x_k +\eps^{\alpha_k}K_{\vk_k}(\eps)v_k,\ A_{k,\eps}\}=o_e(1).
\end{align*}

\end{lemma}

\begin{proof}
This follows from an iterative application of Lemma~\ref{lem:exit-straight}. 
\end{proof}

\begin{lemma}\label{lem:one-step_lower_bound}
Under assumptions~\ref{setting:general}, \ref{setting:geometry-domain}, and \ref{setting:conjugacy},
for every $\vk>0$, for every $\beta\in[0,1]$, and for every $\vk'>\frac{1}{2}\ONE_{\beta=1}$, we have 
\begin{align*}
    \sup_{x\in K_{\vk}(\eps)}\Pp^{x_0 + \eps^{\alpha}xv}\{X_\eps(\tau_\eps)\in q_+ +\eps^{\beta}(-\infty,- l_\eps^{\vk'})v_+,\ A_{+,\eps}\}=o_e(1).
\end{align*}
\end{lemma}

\bpfm
Using the notation~\eqref{eq:seq-scaling}, the formula for one-step transition~\eqref{eq:X_2-at-exit-time-with-Z-plugged in} in the model case, and \eqref{eq:weak_convergence_to_Gauss}, we can approximate the probability above by 

\begin{multline*}
    \Pp^{x_0 + \eps^{\alpha_0}xv_0}\left\{\eps^{\alpha_0\rho_1}LR^{-\rho_1}|x+\eps^{1-\alpha_0}\frU|^{\rho_1}+\eps \NN< - \eps^{\beta}l_\eps^{\vk'} \right\} \\ \leq \Pp^{x_0 + \eps^{\alpha_0}xv_0}\left\{|\NN|\geq \eps^{\beta-1}l_\eps^{\vk'}\right\}=o_e(1)
\end{multline*} 
uniformly in $x\in K_{\vk_0}(\eps)$.
\epf

\begin{lemma}\label{lem:transition_lower-bound}
Under conditions~\ref{setting:general}, \ref{setting:het-chain},~\ref{setting:conjugacy-all}
for every $k\in \{1,2,\dots,n-1\}$, every $\vk_0>0$,  every $\beta\in[0,1]$, and every $\vk>\frac{1}{2}\ONE_{\beta=1}$, we have
\begin{align*}
    \sup_{x\in K_{\vk_0}(\eps)}\Pp^{x_0 + \eps^{\alpha_0}xv_0}\{X_\eps(\tau^k_\eps)\in x_k +\eps^{\beta}(-\infty,- l_\eps^\vk)v_k,\ A_{k,\eps}\}=o_e(1).
\end{align*}
\end{lemma}
\begin{proof}
The claim follows from  Lemma~\ref{lem:one-step_lower_bound} and an iterative application of Lemma~\ref{lem:typical_scale}. 
\end{proof}

Combining Lemmas~\ref{lem:exit-on-the-same-whp}, \ref{lem:new-scaling} and~\ref{lem:transition_lower-bound}, we obtain the following claim:
\begin{lemma} \label{lem:restricting-to-a-narrow-corridor} Let us denote, for $k\ge \kstar$, and $\vk,\e>0$, 
\begin{equation}
\label{eq:def-of-corridor}
\bar A_{k,\vk,\e} = A_{k,\e}\cap \left\{X_{\e,\tau^k_\e}\in x_k+\e^{\bar \alpha_k}K_\vk(\e) v_k\right\}.
\end{equation}
There is sequence of positive constants $(\vk'_k)_{k\in \{\kstar \}\cup H\cup J }$ such that, for every sequence $(\vk_k)_{k\in \{\kstar \}\cup H\cup J }$ satisfying $\vk_k\geq \vk'_k$ for every $k$, it holds for every $\widetilde \vk_0>0$ that, uniformly in $x\in K_{\widetilde\vk_0}(\eps)$,
\begin{equation}\label{eq:approximate_A_n}
\Pp^{x_0+\eps^{\alpha_0} x v}(A_{n,\e})=\Pp^{x_0+\eps^{\alpha_0} x v}\left( A_{n,\e}\cap \bigcap_{k\in \{\kstar\} \cup H \cup J} \bar A_{k,\vk_k,\e} \right)+o_e(1).
\end{equation}

\end{lemma}
\bpf 
Lemma~\ref{lem:exit-on-the-same-whp} implies that $\eta_{n-1,\e}\leq\e l^{\vk'_{n-1}}_\e$ w.h.p.\ on $A_{n,\eps}$ for some $\vk'_{n-1}>0$. This and Lemma~\ref{lem:new-scaling} yield that,
on $A_{n,\eps}\cap\{\eta_{n-1,\e}\leq\e l^{\vk'_{n-1}}_\e\}$, we have $\eta_{i,\e}\leq\eps^{\bar \alpha_i} l^{\vk'_i}_\eps$ w.h.p.\ for some $\vk'_i>0$. 
We can make $\vk'_i$ larger to ensure $\vk'_i>\frac{1}{2}$.
Lemma~\ref{lem:transition_lower-bound} implies $\eta_{i,\e}\geq -\eps l^{\vk'_i}_\eps$ w.h.p. These all hold uniformly in $x\in K_{\widetilde \vk_0}$. Combining these estimates, we obtain \eqref{eq:approximate_A_n} for $(\vk'_k)$. Since the main term on the right of \eqref{eq:approximate_A_n} is also smaller than the left-hand side, we conclude that \eqref{eq:approximate_A_n} holds for larger $\vk_k$.
\epf
\begin{remark}\rm
The proof shows that the lemma would still hold if we changed the definition of~$\bar A_{k,\vk,\e}$ to
$
\bar A_{k,\vk,\e} = A_{k,\e}\cap \left\{X_\e (\tau^k_\e)\in x_k+[-\e l_\e^\vk, \e^{\bar \alpha_k}l_\e^\vk] v_k\right\}$. 
We use a symmetric set $K_\vk(\e)$ in \eqref{eq:def-of-corridor} for brevity, which will eventually be useful in lengthy estimates
involving iterated
integration.
\end{remark}

The significance of this lemma is that one can compute the probability on the r.h.s.\ approximately, using the strong Markov property
and the fact that for $k>\kstar$,
\begin{equation*}
\Pp( \bar A_{k,\vk,\e}| \bar A_{k-1,\vk,\e})=\begin{cases} c_k+o(1),& k\notin J,\\
c_k \e^{\frac{\bar\alpha_{k}}{\rho_{k}}-1}(1+o(1)),& k\in J,  
\end{cases}
\end{equation*}
for constants $c_k>0$. This leads to polynomial decay rates.

To make this plan rigorous, we need a detailed study of appropriately rescaled kernels describing sequential transitions that the process undergoes between times $\tau_\eps^{k-1}$ and $\tau_\eps^{k}$ for all $k$, and convolutions of those kernels. This is the material of the next three subsections.

\subsubsection{A basic estimate on transition kernel convolutions}

\begin{lemma}\label{lem:trans_ker}
Let $\nu_\eps,\bar\nu_\eps,\mu_\eps,\bar \mu_\eps$ be transition kernels indexed by $\eps\in(0,1)$ and let $\vk,\vk',\vk''>0$. Suppose
\begin{enumerate}[label={\rm(h\arabic*)}]
    \item \label{item:nu-bar_nu} there is $\delta>0$ such that
\begin{gather*}
    \sup_{\substack{x\in K_\vk(\eps), \\ [a,b]\subset K_{\vk'}(\eps)}}\left|\nu_\eps(x,[a,b])- \bar \nu_\eps(x,[a,b])\right|=\smallo{\eps^\delta};
\end{gather*}

    \item \label{item:mu-bar_mu} there is $\delta'>0$ such that
\begin{gather*}
    \sup_{\substack{y\in K_{\vk'}(\eps), \\ [a,b]\subset K_{\vk''}(\eps)}}\left|\mu_\eps(y,[a,b])- \bar\mu_\eps(y,[a,b])\right|=\smallo{\eps^{\delta'}};
\end{gather*}

    \item \label{item:nu(K)} there is $p\geq0$ such that
    \begin{align*}
        \sup_{x\in K_\vk(\eps)}\bar\nu_\eps(x,K_{\vk'}(\eps))\leq l_\eps^p;
    \end{align*}
    \item \label{item:mu_decomp} there is $p'\geq0$ such that for all sufficiently small $\eps$ and every $[a,b]\subset K_{\vk''}(\eps)$, there are two monotone real-valued functions $\phi_{+,\eps},\phi_{-,\eps}$ bounded by $l_\eps^{p'}$ in absolute value and 
     such that
\begin{align*}
    \bar\mu_\eps(\cdot, [a,b]) = \phi_{+,\eps}(\cdot) + \phi_{-,\eps}(\cdot).
\end{align*}
\end{enumerate}
Then, there is $\delta''>0$ such that
\begin{align}\label{eq:nu_mu_K}
    \sup_{\substack{x\in K_\vk(\eps), \\ [a,b]\subset K_{\vk''}(\eps)}}\left|\int_{K_{\vk'}(\eps)} \nu_\eps(x, dy)\mu_\eps(y,[a,b])- \int_{K_{\vk'}(\eps)} \bar \nu_\eps(x, dy)\bar\mu_\eps(y,[a,b])\right| = o(\eps^{\delta''}).
\end{align}

\end{lemma}

\begin{remark} \rm Condition~\ref{item:mu_decomp} is very close to a total variation bound but it is convenient for us to work with the decomposition into monotone functions directly.
\end{remark}

\begin{proof}
For $[a,b] \subset K_{\vk''}(\eps)$, we write
\begin{align*}
    \phi_\eps(y)  = \mu_\eps(y,[a,b]),\qquad 
    \bar \phi_\eps(y) =\bar\mu_\eps(y,[a,b]).
\end{align*}
We want to estimate     
\begin{align*}
    &\left|\int_{K_{\vk'}(\eps)}\phi_\eps(y)\nu_\eps(x,d y) - \int_{K_{\vk'}(\eps)}\bar \phi_\eps(y)\bar\nu_\eps(x,d y)\right|\notag
    \\
    &\leq \left|\int_{K_{\vk'}(\eps)}(\phi_\eps(y)-\bar\phi_\eps(y))\nu_\eps(x,d y)\right| + \left|\int_{K_{\vk'}(\eps)}\bar\phi_\eps(y)(\nu_\eps(x,dy)-\bar\nu_\eps(x,d y))\right|    \notag
    \\
    &=\mathtt{I}+\mathtt{II}
\end{align*}
uniformly over $x\in K_\vk(\eps)$ and $[a,b]\in K_{\vk'}(\eps)$. Using~\ref{item:nu-bar_nu} and~\ref{item:nu(K)}, we have
\begin{align*}
    \sup_{x\in K_\vk(\eps)}\nu_\eps(x,K_{\vk'}(\eps))\leq 2 l_\eps^p
\end{align*}
for sufficiently small $\e$.
This and~\ref{item:mu-bar_mu} yield that, for some $\delta_1>0$, $\mathtt{I}=o(\eps^{\delta_1})$  uniformly in $x$ and $[a,b]$.

It remains to estimate $\mathtt{II}$. We use~\ref{item:mu_decomp}  to decompose $\bar\phi_\eps$ into a sum of two monotone functions $\phi_{\pm,\eps}$ with values bounded in absolute value by  $l^{p'}_\eps$. 
For $\delta$ from  condition~\ref{item:nu-bar_nu}, setting
\begin{align}\label{eq:n_eps}
    n_\eps = \lfloor \e^{-\delta/2}\rfloor +1,
\end{align}
we can decompose $K_{\vk'}(\eps)$ into a union of closed intervals
\begin{align*}
    E_{\eps,i}^\pm,\qquad i=-n_\eps+1,\,-n_\eps+2,\,\dots,\,n_\eps,
\end{align*}
with disjoint interiors such that $\phi_{\pm,\eps}\in [\frac{i-1}{n_\eps}l^{p'}_\eps,\frac{i}{n_\eps}l^{p'}_\eps]$ on $E_{\eps,i}^\pm$. Then, using the monotonicity of $\phi_{\pm,\eps}$, we have
\begin{align*}
    &\left|\int_{E^\pm_{\eps i}}\phi_{\pm,\eps}(y)\left(\nu_\eps(x, dy)-\bar\nu_\eps(x,dy)\right)\right|
    \\
    &\leq \max\left\{\left|\frac{i}{n_\eps}l^{p'}_\eps\nu_{\eps}(x,E_{\eps,i}^\pm) - \frac{i-1}{n_\eps}l^{p'}_\eps\bar\nu_\eps(x,E_{\eps,i}^\pm)\right|,\ \left|\frac{i}{n_\eps}l^{p'}_\eps\bar\nu_\eps(x,E_{\eps,i}^\pm) - \frac{i-1}{n_\eps}l^{p'}_\eps\nu_{\eps}(x,E_{\eps,i}^\pm)\right|\right\}
    \\
    &\leq l^{p'}_\eps\left|\nu_{\eps}(x,E_{\eps,i}^\pm)-\bar\nu_\eps(x,E_{\eps,i}^\pm)\right| + \frac{1}{n_\eps}l^{p'}_\eps\bar\nu_{\eps}(x,E_{\eps,i}^\pm).
\end{align*}
Summing up these estimates over all $i$, we obtain
\begin{align*}
   \mathtt{II}\le\left(\sum_{\bullet\in\{+,-\}}\sum_{i=-n_\eps+1}^{n_\eps}l^{p'}_\eps\left|\nu_{\eps}(x,E_{\eps,i}^\bullet)-\bar\nu_\eps(x,E_{\eps,i}^\bullet)\right|\right) + 2 \frac{l^{p'}_\eps}{n_\eps}\bar\nu_{\eps}(x,K_{\vk'}(\eps)).
\end{align*}
Due to~\ref{item:nu-bar_nu},~\ref{item:nu(K)} and the definition of $n_\eps$ in~\eqref{eq:n_eps}, this is bounded by
\begin{align*}
    4(\eps^{-\frac{\delta}{2}}+1)l^{p'}_\eps \smallo{\eps^\delta}+2l^{p'+p}_\eps \eps^\frac{\delta}{2} = o(\eps^{\delta_2})
\end{align*}
for some $\delta_2>0$ uniformly for $x\in K_\vk(\eps)$.
\end{proof}

\subsubsection{Typical transitions}
We begin with a result on  the one-step typical transition. Its proof is postponed to Section~\ref{sec:pf_one-step_typical_trans}.

\begin{lemma}\label{lem:typical_loc_lim}
Suppose that conditions~\ref{setting:general},~\ref{setting:geometry-domain},~\ref{setting:initial-cond},~\ref{setting:conjugacy} hold.
Let $\alpha\in(0,1]$ and $\alpha'=(\alpha\rho)\wedge 1$. Let
\[
m=\begin{cases}
3,&\rho<1,\\ 
4,& \rho\ge 1.
\end{cases}
\]
Then there are:
\begin{itemize} 
    \item an $m$-dimensional centered Gaussian vector $N$,
    \item deterministic continuous functions $\Phi_{1,\eps}$, $\Phi_{2,\eps}:\R\times \R^m\to\R$ indexed by $\eps\in(0,1)$,
\end{itemize}
such that 

\begin{enumerate}

\item \label{item:comp_in_lem_typical_loc_lim} for every $\vk,\vk'>0$, there is $\delta>0$ such that
\begin{align*}
    \sup_{\substack{x\in K_\vk(\eps), \\ [a,b]\subset K_{\vk'}(\eps)}}\left|\Pp^{x_0+\eps^\alpha xv}\Big\{X_\tau\in q_++\eps^{\alpha'}[a,b]v_+\right\}-\Pp\left\{\Phi_{1,\eps}(x,N)\in [a,b],\ \Phi_{2,\eps}(x,N)\geq 0\Big\} \right|=\smallo{\eps^\delta};
\end{align*}

    \item \label{item:cvg_Phi}
    there are constants $c_1,c_2>0$ and vectors $u_1\in\{0\}^2\times (0,\infty)^{m-2}$, $u_2\in(0,\infty)^2\times \{0\}^{m-2}$
    such that, 
     for $i=1,2$, $\Phi_{i,\eps}$ converges in LU, as $\eps\to0$, to a continuous function $\Phi_i$, defined for all $(x,y)\in\R\times\R^m$, by 
    \begin{align*}
        \Phi_1(x,y) & = (1-\ONE_{\rho>1,\,\alpha\rho>1}) c_1|\Phi_2(x,y)|^\rho + \ONE_{\alpha'=1}u_1\cdot y,
        \\
        \Phi_2(x,y) & = c_2x+\ONE_{\alpha=1}u_2\cdot y;
    \end{align*}
   these functions $\Phi_i$, $i=1,2$, satisfy the following:
    \begin{itemize}
    \item if $\alpha = 1$, then for all $x\in\R$,
        \begin{itemize}
            \item $\Leb\{y:\Phi_i(x,y)=0\}=0$, $i=1,2$, 
            \item $\Leb\{y:\Phi_1(x,y)\wedge\Phi_2(x,y)\geq 0\}>0$;
        \end{itemize}
    \item if $\alpha < 1$, then 
        \begin{itemize}
            \item $\Leb\{y:\Phi_i(x,y)=0\}=0$ for all $x\neq 0$, $i=1,2$,
            \item $\Leb\{y:\Phi_1(x,y)\wedge \Phi_2(x,y)\geq 0\}>0$  for all $x >0$,
            \item $\Leb\{y:\Phi_2(x,y)\geq 0\}=0$ for all $x<0$;
        \end{itemize}
    \end{itemize}
     
    \item \label{item:nondecreasing_Phi} for every $y\in\R^m$ and every $\eps\in(0,1)$, the function $\Phi_{1,\eps}(\cdot,y)$ is nondecreasing on $\{x:\Phi_{2,\eps}(x,y)\geq 0\}$ and the function $\Phi_{2,\eps}(\cdot,y)$ is nondecreasing on~$\R$;
    \item \label{item:phi_pm_decomp} for every $[a,b]\subset \R$ and every $\eps\in(0,1)$ sufficiently small, the function 
    \begin{align*}
        x\mapsto \Pp\left\{\Phi_{1,\eps}(x,N)\in [a,b],\ \Phi_{2,\eps}(x,N)\geq 0\right\}
    \end{align*}
    can be written as a sum of two monotone functions $\phi_{+,\eps}$ and $\phi_{-,\eps}$, both with values in $[-1,1]$;
    \item \label{item:Phi_1_range} for each $\vk>0$ and sufficiently large $\vk'>0$, 
\begin{align*}
    \sup_{x\in K_\vk(\eps)}\Pp\{\Phi_{1,\eps}(x,N)\not\in K_{\vk'}(\eps)\} = o_e(1);
\end{align*}
    \item \label{item:Phi>|x|^p} 
    if $\alpha\rho\leq 1$, then there are constants $C,R,p,q>0$ such that, 
    \begin{align*}
        |\Phi_{1,\eps}(x,y)|\geq C|x|^p, \qquad \text{for all }\ \eps\in(0,1),\ |x|\geq R,\ |y|_\infty\leq |x|^q,
    \end{align*}
    where $|y|_\infty = \max\left\{|y^i|:i\in\{1,2,\dots,m\}\right\}$.

\end{enumerate}

\end{lemma}

We will use the above lemma to prove the following result on the typical sequential transitions not involving slowdown saddle points.

\begin{lemma}\label{lem:typical_chain_exit}
Suppose that conditions~\ref{setting:general}, \ref{setting:het-chain}, \ref{setting:scaling-at-saddle-0}, \ref{setting:scaling-limit-at-saddle-0}, \ref{setting:conjugacy-all} hold.

Then for each $k\in\{1,2,\dots,n-1\}$, there are
\begin{itemize}
    \item $m_k\in\N$ and an $m_k$-dimensional centered Gaussian vector $N_k$,
    \item deterministic continuous functions $\Phi^k_{1,\eps},\Phi^k_{2,\eps}:\R\times \R^{m_k}\to \R$ indexed by $\eps\in(0,1)$,
\end{itemize}
such that the following holds:

\begin{enumerate}
\item \label{item:use_Phi_to_approx}
for every $\vk>0$ and sufficiently large $\vk'>0$, there is $\delta>0$ such that
\begin{multline}
\label{eq:use_Phi_to_approx}
    \sup_{\substack{x\in K_{\vk}(\eps), \\ [a,b]\subset K_{\vk'}(\eps)}}
    \Big|\Pp^{x_0+\eps^{\alpha_0} xv_0}\{X_{\eps, \tau^k_\eps}\in  x_k+\eps^{\alpha_k}[a,b]v_k\}\\
-\Pp\{\Phi^k_{1,\eps}(x,N_k)\in [a,b],\ \Phi^k_{2,\eps}(x,N_k)\geq 0\} \Big|=\smallo{\eps^\delta};
\end{multline}
    \item \label{item:cvg_Phi^k} for $i=1,2$, $\Phi^k_{i,\eps}$ converges in LU, as $\eps\to0$, to a continuous function~$\Phi^k_{i}$ satisfying the following: 
    \begin{itemize}
    \item if $\alpha_0 = 1$, then for all $x\in\R$;,
        \begin{itemize}
            \item $\Leb\{y:\Phi^k_i(x,y)=0\}=0$ for $i=1,2$,
            \item $\Leb\{y:\Phi^k_1(x,y)\wedge\Phi^k_2(x,y)\geq 0\}>0$,
        \end{itemize}
   \item if $\alpha_0 < 1$, then 
        \begin{itemize}
            \item $\Leb\{y:\Phi^k_i(x,y)=0\}=0$ for all $x\neq 0$, $i=1,2$,
            \item $\Leb\{y:\Phi^k_1(x,y)\wedge\Phi^k_2(x,y)\geq 0\}>0$ for all $x>0$,
            \item $\Leb\{y:\Phi^k_2(x,y)\geq 0\}=0$ for all $x<0$;
        \end{itemize} 
    \end{itemize} 
    \item \label{item:nondecreasing_Phi^k} for every $y\in\R^{m_k}$ and every $\eps\in(0,1)$, the functions $\Phi^k_{1,\eps}(\cdot,y)$ and $\Phi^k_{2,\eps}(\cdot,y)$ are nondecreasing on $\{x:\Phi^k_{2,\eps}(x,y)\geq 0\}$;
    \item \label{item:phi_pm_decomp_k} for every $[a,b]\subset \R$ and every $\eps\in(0,1)$ sufficiently small, the function 
    \begin{align*}
        x\mapsto \Pp\{\Phi^k_{1,\eps}(x,N_k)\in [a,b],\ \Phi^k_{2,\eps}(x,N_k)\geq 0\}
    \end{align*}
    can be written as a sum of two monotone functions $\phi_{+,\eps}$ and $\phi_{-,\eps}$, both with values in $[-1,1]$;
    
    \item \label{item:Phi_1_range_multi_saddle}  
    for each $\vk>0$ and sufficiently large $\vk'>0$,
    \begin{align*}
        \sup_{x\in K_\vk(\eps)}\Pp\{\Phi^k_{1,\eps}(x,N_k)\not\in K_{\vk'}(\eps)\}=o_e(1);
    \end{align*}

    \item \label{item:k_chain_Phi>|x|^p} if 
    \begin{align*}\alpha_i\rho_{i+1}\leq 1,\qquad \text{for all}\ i\in\{0,1,\dots,k-1\},
     \end{align*}
    then there are positive constants $C,R,p,q$ such that, uniformly in $\eps$,
    \begin{align*}
        |\Phi^k_{1,\eps}(x,y)|\geq C|x|^p,\qquad \text{for all}\  |x|\geq R,\ |y|_\infty\leq |x|^q,
    \end{align*}
    where $|y|_\infty = \max\{|y^i|:i\in\{1,2,\dots,m_k\}\}$. 
\end{enumerate}
\end{lemma}

\begin{remark} \rm In the proof of this lemma, we actually give an explicit recursive definition of functions $\Phi^k_{i,\eps}$ and $\Phi^k_i$ based on 
compositions of functions $\Phi_{i,\e}$, $\Phi_i$ introduced in Lemma[\ref{lem:typical_loc_lim}, see \eqref{eq:Phi_k+1_1}, \eqref{eq:Phi_k+1_2}, \eqref{eq:Phi_k+1_1=Phi'...}.
\end{remark}

\begin{proof}
The base case $k=1$ is covered by Lemma~\ref{lem:typical_loc_lim}. 
Now, let us assume that  the lemma holds for $k$ and prove it for $k+1$. Using the induction hypothesis~\eqref{item:use_Phi_to_approx} and defining
\begin{align*}
    \nu_{\eps}(x, d y) & = \Pp^{x_0+\eps^{\alpha_0} xv_0}\{X_{\eps,\tau^{k}_\eps}\in x_{k}+\eps^{\alpha_{k}}( d y)v_{k}\}, \\
    \bar \nu_{\eps}(x, d y) & = \Pp\{\Phi^k_{1,\eps}(x,N_{k})\in d y,\ \Phi^k_{2,\eps}(x,N_{k})\geq 0\},
\end{align*}
we have that the Lemma~\ref{lem:trans_ker}~\ref{item:nu-bar_nu} is satisfied for sufficiently large $\vk'$. Since $\bar\nu(x,\cdot)$ is always a sub-probability measure, Lemma~\ref{lem:trans_ker}~\ref{item:nu(K)} also holds.  
Let us fix any $\vk''>0$.
Applying Lemma~\ref{lem:typical_loc_lim} to the saddle $O_{k+1}$, we can find $m'\in\N$, an $m'$-dimensional centered Gaussian vector $N'$ and functions $\Phi'_{1,\eps},\,\Phi'_{2,\eps}:\R\times \R^{m'}\to \R$ satisfying properties described in that lemma such that the measures
\begin{align*}
    \mu_\eps(y,dz) & = \Pp \{X_{\eps,\tau^{k+1}_\eps}\in x_{k+1}+\eps^{\alpha_{k+1}}(dz)v_{k+1}\ |\  X_{\eps,\tau^{k}_\eps}=x_{k}+\eps^{\alpha_{k}}yv_{k}\},\\
    \bar \mu_\eps(y,dz) &= \Pp\{\Phi'_{1,\eps}(y,N')\in dz,\ \Phi'_{2,\eps}(y,N')\geq 0\},
\end{align*}
satisfy Lemma~\ref{lem:trans_ker}~\ref{item:mu-bar_mu} and~\ref{item:mu_decomp}. Hence, we can invoke Lemma~\ref{lem:trans_ker} to see that $\nu_\eps,\bar\nu_\eps,\mu_\eps,\bar\mu_\eps$ satisfy~\eqref{eq:nu_mu_K}. 

Now, we proceed to derive part~\eqref{item:use_Phi_to_approx}.
Using Lemma~\ref{lem:typical_scale} and adjusting $\vk'$ if necessary,  we can rewrite the first integral in~\eqref{eq:nu_mu_K} as
\begin{align*}
    &\int_{K_{\vk'}(\eps)} \nu_\eps(x, dy)\mu_\eps(y,[a,b])\\
    &=\Pp^{x_0+\eps^{\alpha_0}xv_0}\{X_{\eps,\tau^{k+1}_\eps} \in x_{k+1}+\eps^{\alpha_{k+1}}[a,b]v_{k+1},\ X_{\eps,\tau^{k}_\eps} \in x_{k}+\eps^{\alpha_{k}}K_{\vk'}(\eps)v_{k}\}\\
    & = \Pp^{x_0+\eps^{\alpha_0}xv_0}\{X_{\eps,\tau^{k+1}_\eps} \in x_{k+1}+\eps^{\alpha_{k+1}}[a,b]v_{k+1}\}+o_e(1)
\end{align*}
uniformly in $x\in K_\vk(\eps)$ and $[a,b]\subset K_{\vk''}(\eps)$. This is the first term in~\eqref{eq:use_Phi_to_approx} for~$k+1$.

\smallskip

Let us now treat the second integral in~\eqref{eq:nu_mu_K}.
Using the induction assumption~\eqref{item:Phi_1_range_multi_saddle}, we can rewrite it  as 
\begin{align*}
    &\int_{K_{\vk'}(\eps)} \bar \nu_\eps(x, dy)\bar\mu_\eps(y,[a,b])\\
    &=\Pp\Big\{\Phi'_{1,\eps}(\Phi^k_{1,\eps}(x,N_{k}),N')\in[a,b],\quad  \Phi'_{2,\eps}(\Phi^k_{1,\eps}(x,N_{k}),N')\geq 0, \\  
    &\qquad\qquad\qquad\qquad\qquad \qquad\qquad  \Phi^k_{1,\eps}(x,N_k)\in K_{\vk'}(\eps),\quad  \Phi^k_{2,\eps}(x,N_{k})\ge 0\Big\}\\
    &= \Pp\left \{\Phi'_{1,\eps}(\Phi^k_{1,\eps}(x,N_{k}),N')\in[a,b],\quad \Phi^k_{2,\eps}(x,N_{k})\wedge \Phi'_{2,\eps}(\Phi^k_{1,\eps}(x,N_{k}),N')\geq 0\right \}\\
    & \qquad\qquad\qquad\qquad\qquad\qquad\qquad\qquad\qquad\qquad + o_e(1),
\end{align*}
uniformly in $x\in K_\vk(\eps)$ and $[a,b]\subset K_{\vk''}(\eps)$. For $(x,(y,y')) \in \R\times (\R^{m_k}\times \R^{m'})$, we define
\begin{align}
    &\Phi^{k+1}_{1,\eps}(x,(y,y')) = \Phi'_{1,\eps}(\Phi^k_{1,\eps}(x,y),y'), \label{eq:Phi_k+1_1}
    \\
    & \Phi^{k+1}_{2,\eps}(x,(y,y')) = \Phi^k_{2,\eps}(x,y)\wedge \Phi'_{2,\eps}(\Phi^k_{1,\eps}(x,y),y').\label{eq:Phi_k+1_2}
\end{align}
We set $m_{k+1}= m_k+m'$ and define $N_{k+1} = (N_k,N')$. The second integral in~\eqref{eq:nu_mu_K} becomes
\begin{align*}
    \Pp\{\Phi^{k+1}_{1,\eps}(x,N_{k+1})\in [a,b],\ \Phi^{k+1}_{2,\eps}(x,N_{k+1})\geq 0\} + o_e(1),
\end{align*}
where the main term is exactly the expression appearing in~\eqref{eq:use_Phi_to_approx} for $k+1$.
This completes our verification of part~\eqref{item:use_Phi_to_approx} for $k+1$.

\bigskip

Now, we turn to~\eqref{item:cvg_Phi^k}. For brevity, we write $\bar y= (y,y')$. In view of \eqref{eq:Phi_k+1_1} and~\eqref{eq:Phi_k+1_2}, using the continuity and the LU convergence of $\Phi'_{i,\eps}$ and $\Phi^k_{i,\eps}$ for $i=1,2$ (due to Lemma~\ref{lem:typical_loc_lim}~\eqref{item:cvg_Phi} and the induction hypothesis~\eqref{item:cvg_Phi^k}), we can easily derive the LU convergence of $\Phi^{k+1}_{i,\eps}$, for $i=1,2$ and that the limits are given by 
\begin{align}
    \Phi^{k+1}_{1}(x,\bar y) & = \Phi'_{1}(\Phi^k_{1}(x,y),y') \label{eq:Phi_k+1_1=Phi'...}
    \\
    & = (1- \ONE_{\rho_{k+1}>1,\, \alpha_k\rho_{k+1}>1})c_1\left|c_2 \Phi^k_1(x,y) + \ONE_{\alpha_k=1} u_2\cdot y'\right|^{\rho_{k+1}}+\ONE_{\alpha_{k+1}=1}u_1\cdot y', \notag
    \\
    \Phi^{k+1}_{2}(x,\bar y) & = \Phi^k_2(x,y)\wedge \left(c_2 \Phi^k_1(x,y)+\ONE_{\alpha_k=1}u_2\cdot y'\right), \notag
\end{align}
for constants $c_1,c_2$ and vectors $u_1,u_2$ given in Lemma~\ref{lem:typical_loc_lim}~\eqref{item:cvg_Phi}, where $\Phi^k_i$ is the limit of $\Phi^k_{i,\eps}$. Moreover, due to Lemma~\ref{lem:typical_loc_lim}~\eqref{item:cvg_Phi}, for all possible values of $\alpha_k$ and $\rho_{k+1}$, we have
\begin{align*}
    \Leb\{y': \Phi'_1(x,y')\}=0,\quad \text{if } x\neq 0,
\end{align*}
which along with \eqref{eq:Phi_k+1_1=Phi'...} implies
\begin{align*}
    \Leb\{\bar y:\Phi^{k+1}_1(x,\bar y)=0,\, \Phi^k_1(x,y)\neq 0\}=0.
\end{align*}
Then, \eqref{item:cvg_Phi^k} follows from the induction assumptions, the orthogonality between $u_1$ and $u_2$, and
\begin{align*}
    \{\bar y:\Phi^{k+1}_1(x,\bar y)  = 0\}&\subset \{\bar y:\Phi^{k+1}_1(x,\bar y)=0,\, \Phi^k_1(x,y)\neq 0\}\cup \{\bar y:\Phi^k_1(x,y) = 0\},
    \\
    \{\bar y: \Phi^{k+1}_1(x,\bar y) \geq 0\}&\supset  \{\bar y: u_1\cdot y'\geq0\},
    \\
    \{\bar y: \Phi^{k+1}_2(x, \bar y)  = 0\}&\subset \{\bar y:\Phi^k_2(x,y) = 0\}\cup \{\bar y: u_2\cdot y' = -c_2 \Phi^k_1(x,y) \}\cup \{\bar y:\Phi^k_1(x,y)=0\}, 
    \\
    \{\bar y: \Phi^{k+1}_2(x,\bar y) \geq 0\}&\supset \{\bar y:\Phi^k_1(x,y)\wedge\Phi^k_2(x,y)\geq 0\}\cap \{\bar y: u_2\cdot y'\geq 0\},
    \\
    \{\bar y: \Phi^{k+1}_2(x,\bar y) \geq 0\} & \subset \{\bar y:\Phi^k_2(x,y)\geq 0\}. 
\end{align*}

Let us verify~\eqref{item:nondecreasing_Phi^k}. Fix $(y,y')$ and $\eps$. Due to~\eqref{eq:Phi_k+1_2}, on the set
\begin{align*}
    A= \{x:\Phi^{k+1}_{2,\eps}(x,(y,y'))\geq 0\},
\end{align*}
we have
\begin{align} \label{eq:Phi^k_2_geq_0}
    \Phi^{k}_{2,\eps}(x,y)\geq 0,\qquad\text{and}\qquad \Phi'_{2,\eps}(\Phi^k_{1,\eps}(x,y),y')\geq 0.
\end{align}
Due to the induction assumption~\eqref{item:nondecreasing_Phi^k}, the first inequality in \eqref{eq:Phi^k_2_geq_0} implies that
\begin{align}\label{eq:x_to_Phi^k_nond}
    x\mapsto \Phi^k_{i,\eps}(x,y)\text{ is nondecreasing on }A,\qquad i=1,2.
\end{align}
Lemma~\ref{lem:typical_loc_lim}~\eqref{item:nondecreasing_Phi} states that $ \Phi'_{2,\eps}(\cdot,y')$ is nondecreasing on $\R$. This along with~\eqref{eq:Phi_k+1_2} and~\eqref{eq:x_to_Phi^k_nond} yields that $\Phi^{k+1}_{2,\eps}(\cdot,(y,y'))$ is nondecreasing on $A$. 

Lemma~\ref{lem:typical_loc_lim}~\eqref{item:nondecreasing_Phi} also gives that $\Phi'_{1,\eps}(\cdot,y')$ is nondecreasing on $\{z:\Phi'_{2,\eps}(z,y')\geq 0\}$. From this, the definition of $\Phi^{k+1}_{1,\eps}$ in \eqref{eq:Phi_k+1_1}, the second inequality in \eqref{eq:Phi^k_2_geq_0}, and \eqref{eq:x_to_Phi^k_nond}, we can deduce that $\Phi^{k+1}_{1,\eps}(\cdot,(y,y'))$ is nondecreasing on $A$. This completes the verification of \eqref{item:nondecreasing_Phi^k} for $k+1$. 

\medskip

Setting
\begin{align*}
    \phi_{+,\eps}(x) & = \Pp\{\Phi_{1,\eps}^{k+1}(x,N_{k+1}) \geq a,\ \Phi^{k+1}_{2,\eps}(x,N_{k+1})\geq 0\},
    \\
    \phi_{-,\eps}(x) & = - \Pp\{\Phi^{k+1}_{1,\eps}(x,N_{k+1}) >b,\ \Phi^{k+1}_{2,\eps}(x,N_{k+1})\geq 0\},
\end{align*}
we can see that \eqref{item:phi_pm_decomp_k} for $k+1$ follows from \eqref{item:nondecreasing_Phi^k} for $k+1$ proved above.

\medskip

Let us verify~\eqref{item:Phi_1_range_multi_saddle}. For any $\vk''>0$, 
\begin{align*}
    &\Pp\{\Phi^{k+1}_{1,\eps}(x,N_{k+1})\not\in K_{\vk'}(\eps)\} = \Pp\{\Phi'_{1,\eps}(\Phi^k_{1,\eps}(x,N_{k}),N')\not\in K_{\vk'}(\eps)\}\\
    &\leq \Pp\{\Phi'_{1,\eps}(\Phi^k_{1,\eps}(x,N_{k}),N')\not\in K_{\vk'}(\eps),\  \Phi^k_{1,\eps}(x,N_{k})\in K_{\vk''}(\eps)\}\\
    &\qquad + \Pp\{\Phi^k_{1,\eps}(x,N_{k})\not\in K_{\vk''}(\eps)\}.
\end{align*}
Due to the induction assumption~\eqref{item:Phi_1_range_multi_saddle}, the last term is $o_e(1)$ uniformly in $x\in K_\vk(\eps)$ for large $\vk''$.
Choosing $\vk'$ sufficiently large and using~\eqref{item:Phi_1_range} in Lemma~\ref{lem:typical_loc_lim}, we can see that the first term after the inequality sign is $o_e(1)$.
Thus,~\eqref{item:Phi_1_range_multi_saddle} holds for $k+1$.

\medskip

Let us verify~\eqref{item:k_chain_Phi>|x|^p}. Applying Lemma~\ref{lem:typical_loc_lim}~\eqref{item:Phi>|x|^p} to $\Phi'_{1,\eps}$, we can find constants $C',R',p',q'$ such that
\begin{align*}
    |\Phi'_{1,\eps}(x',y')|\geq C'|x'|^{p'}
\end{align*}
for
\begin{align}\label{eq:x'_cond}
    |x'|\geq R',\qquad |y'|_\infty\leq |x'|^{q'}.
\end{align}
Using the induction assumption~\eqref{item:k_chain_Phi>|x|^p} for $k$, we get that, for $C,R,p,q>0$,
\begin{align}\label{eq:proof_Phi_1>}
    |\Phi^k_{1,\eps}(x,y)|\geq C|x|^{p}
\end{align}
for $|x|\geq R$ and $|y|_\infty\leq |x|^q$. This still holds if we make  $R>1$ larger and $q$ smaller. We can do so to ensure $C|x|^p\geq R'$ and $ |x|^q\leq (C|x|^p)^{q'}$ whenever $|x|\geq R$. This along with~\eqref{eq:proof_Phi_1>} allows us to see that whenever $|x|\geq R$ and $|y|_\infty,|y'|_\infty\leq |x|^q$,~\eqref{eq:x'_cond} is satisfied with $x'$ replaced by $\Phi^k_{1,\eps}(x,y)$. Therefore, we obtain that
\begin{align*}
    \left|\Phi'_{1,\eps}(\Phi^k_{1,\eps}(x,y),y')\right|\geq CC'|x|^{pp'}
\end{align*}
as desired, for $|x|>R$ and $|(y,y')|_\infty \leq |x|^q$. The left-hand side of the above display is exactly $\Phi^{k+1}_{1,\eps}(x,(y,y'))$ due to~\eqref{eq:Phi_k+1_1}. This completes the verification of~\eqref{item:k_chain_Phi>|x|^p} for $k+1$.
\end{proof}

\subsubsection{Transitions near binding saddle points}
For each $k\in\{0,1,\dots,n\}$, we define
\begin{equation}
\label{eq:theta_k_H}
\theta_k=\sum_{j:\ j\le k,\  j\in J}  \left(\frac{\bar\alpha_{j}}{\rho_{j}}-1\right).
\end{equation}
Due to the definition of $J$ in \eqref{eq:def_J},
\begin{align}\label{eq:theta_k=0,k_leq_kstar}
    \theta_k = 0,\qquad k \leq \kstar.
\end{align}
For $ k\in\{\kstar,\ldots,n-1\},\  \e>0,$ and a Borel set $E\subset\R$, we denote 
\begin{equation}
\label{eq:events-B}
B_{k,\e,E}=A_{k,\eps}\cap \{ X_{\e,\tau^k_\e}\in x_k+\e^{\bar\alpha_k}E v_k\}.
\end{equation}
For these $k$ and $\e$, 
and for every vector $(\vk)=(\vk_i)_{i\in \{\kstar\}\cup H\cup J}$
we define a transition kernel $\nu_{k,\e}(\cdot,\cdot)=\nu_{k,(\vk),\e}(\cdot,\cdot)$ by
\begin{equation}
\label{eq:nu_k}
\nu_{k,(\vk),\e}(x,E)=\eps^{-\theta_k}\Pp^{x_0+\eps^{\alpha_0} xv_0}\left( B_{k,\e,E}\cap\bigcap_{i\in (\{\kstar\}\cup H\cup J)\cap\{0,1,\dots , k\} } \bar A_{i,\vk_i,\e}  \right).
\end{equation}
Note that we can rewrite
\begin{align}\label{eq:inter_bar_A}
    \bigcap_{i\in (\{\kstar\}\cup H\cup J)\cap\{0,1,\dots , k\} } \bar A_{i,\vk_i,\e}=\bigcap_{i\in H'\cap\{0,\ldots,k-1\} }( \bar A_{i,\vk_i,\e}\cap \bar A_{i+1,\vk_{i+1},\e}).
\end{align}

For $k\in \{\kappa+1,\dots,n-1\}$, we set
\begin{align}\label{eq:underline_k}
    \underline{k} = \max\{j\in J: j\leq k\}.
\end{align}

We recall that the collection $\GoodMeasures$ of measures is defined  by \eqref{eq:GoodMeasures>0}--\eqref{eq:GoodMeasure_density}.

\begin{lemma}\label{lem:local-limit-for-H}

Suppose  conditions~\ref{setting:general}, \ref{setting:het-chain}, \ref{setting:scaling-at-saddle-0}, \ref{setting:scaling-limit-at-saddle-0}, \ref{setting:conjugacy-all} hold
and assume that~$\kstar$ is well-defined. Then, there is
a family of transition kernels $\bar \nu_{k,\eps}$ indexed by
$k\in \{\kappa+1,\dots,n-1\}$ and $\eps\in(0,1)$ with the following property:
for every $\vk_0>0$ and every vector $(\vk'_k)_{k\in\{\kstar\}\cup H\cap J}$ of positive numbers, there are $(\vk_k)_{k\in\{\kstar\}\cup H\cap J}$ satisfying $\vk_k>\vk'_k$ for each $k$, and 
a constant $\delta>0$, such that for each $k$,
\begin{align}\label{eq:local_limit_nu-bar_nu}
    \sup_{\substack{x\in K_{\vk_0}(\eps), \\ [a,b]\subset K_{\vk_{k}}(\eps)}}\left|\nu_{k,\eps}(x,[a,b])- \bar \nu_{k,\eps}(x,[a,b])\right|=\smallo{\eps^\delta},
\end{align}
and
\begin{align}\label{eq:bar_nu_k,eps}
    \bar\nu_{k,\eps}(x, dy) = h_{k,\eps}(x)\varsigma_{k,\eps}(dy),
\end{align}
where 
\begin{enumerate}[label=(\roman*)]
    \item \label{item:h_k_cvg} the measurable functions $h_{k,\eps}:\R\to[0,\infty)$ indexed by $\eps\in(0,1)$ are bounded uniformly in $\eps$, and converge as $\eps \to 0$ to a bounded continuous function
    \begin{itemize}
        \item $h_k:\R\to(0,\infty)$ in LU on $\R$, if $\alpha_0=1$,
        \item $h_k:\R\setminus\{0\}\to[0,\infty)$ in LU on $\R\setminus\{0\}$, satisfying $h_k=0$ on $(-\infty,0)$ and $h_k>0$ on $(0,\infty)$,  if $\alpha_0<1$;
    \end{itemize}
    
    \item \label{item:gamma_k,k=underbar_k} if $k=\underline{k}$, then $\varsigma_{k,\eps} = \varsigma_k$ is independent of $\eps$ for some $\varsigma_k \in \GoodMeasures$;
    \item \label{item:gamma_k,k>underbar_k} if $k>\underline{k}$, then $\varsigma_{k,\eps}$ is given by
    \begin{align}\label{eq:gamma_k,eps_formula}
        \varsigma_{k,\eps}(dy) = \int_{K_{\vk_{\underline{k}}}(\eps)} \varsigma_k(dz)\Pp\left\{\Phi^k_{1,\eps}(z,N_k)\in dy,\  \Phi^k_{2,\eps}(z,N_k)\geq 0\right\},
    \end{align}
    where
    \begin{itemize}
        \item the Borel measure $\varsigma_k$ does not  depend on $\eps$ and satisfies 
        \begin{align}\label{eq:gamma_good}
            \varsigma_{k}\in\GoodMeasures,
        \end{align}
        
        \item $N_k$ is an $m_k$-dimensional centered Gaussian vector for some $m_k$,
        \item functions $\Phi^k_{1,\eps}, \Phi^k_{2,\eps}:\R\times\R^{m_k}\to\R$ 
        are the functions given by Lemma~\ref{lem:typical_chain_exit} applied to the transition from the vicinity of $x_{\underline{k}}$ at scale $\bar\alpha_{\underline{k}}$ to the vicinity of $x_{k}$ at scale $\bar\alpha_k$, and thus satisfy properties~\eqref{item:cvg_Phi^k}, \eqref{item:nondecreasing_Phi^k}, \eqref{item:phi_pm_decomp_k},~\eqref{item:Phi_1_range_multi_saddle}, and~\eqref{item:k_chain_Phi>|x|^p} in Lemma~\ref{lem:typical_chain_exit}.
            \end{itemize}
\end{enumerate}

\end{lemma}

\begin{remark}\rm
In fact, for $\varsigma_k$ in \eqref{eq:gamma_k,eps_formula}, we always have $\varsigma_k = \varsigma_{\underline{k}}$, which will be clear from the proof. Since $\Phi^k_{i,\eps}$ and $N_k$ account for the transition from the vicinity of~$x_{\underline{k}}$ to that of $x_k$, a more accurate but heavier notation would be $\Phi^{\underline{k}\to k}_{i,\eps}$ and $N_{\underline{k}\to k}$. Hence, it would be more precise to rewrite \eqref{eq:gamma_k,eps_formula} as
\begin{align*}
        \varsigma_{k,\eps}(dy) = \int_{K_{\vk_{\underline{k}}}(\eps)} \varsigma_{\underline k}(dz)\Pp\left\{\Phi^{\underline{k}\to k}_{1,\eps}(z,N_{\underline{k}\to k})\in dy,\  \Phi^{\underline{k}\to k}_{2,\eps}(z,N_{\underline{k}\to k})\geq 0\right\}.
    \end{align*}
For brevity, however, we stick to the notation of the lemma.
\end{remark}

\begin{proof} In this proof,
we will use Lemmas~\ref{lem:local-limit-theorem}, \ref{lem:structure-of-reverse-exponents}, \ref{lem:trans_ker}, and \ref{lem:typical_chain_exit}. Among them, only Lemmas~\ref{lem:local-limit-theorem} and~\ref{lem:typical_chain_exit} impose restrictions on $\vk$'s, but both of them allow us to choose $\vk_k$ arbitrarily large. Hence, whenever these two lemmas are applied in this proof, we choose the relevant $\vk_k$ sufficiently large to ensure $\vk_k>\vk'_k$. With this clarified, we will omit mentioning this technicality for brevity.

We will use induction, sequentially showing that the result holds for all $k\in\{\kstar+1,\ldots,j\}$, where $j$ runs through elements of $H$.

\smallskip

\textbf{Basis of induction.}
We first verify that our claim holds for $j=\min H$. Due to the definition of $\underline k$ in \eqref{eq:underline_k}, we have
\begin{align}\label{eq:underline_k_base}
    \underline k = \kstar + 1.
\end{align}

We split the argument into four steps.
Step~1: we use Lemma~\ref{lem:typical_chain_exit} on typical transitions to approximate the distribution of $X_{\eps,\tau^\kappa_\eps}$. Step~2: to
approximate the distribution of~$X_{\eps,\tau^{\kappa+1}_\eps}$, we apply Lemma~\ref{lem:local-limit-theorem} to atypical 
transitions from $X_{\eps,\tau^\kappa_\eps}$ to~$X_{\eps,\tau^{\kappa+1}_\eps}$. Step~3: if $j> \kstar +1$, we 
approximate the distribution of $X_{\eps,\tau^{k}_\eps}$ applying Lemma~\ref{lem:typical_chain_exit} to typical transitions from $X_{\eps,\tau^{\kappa+1}_\eps}$ to $X_{\eps,\tau^{k}_\eps}$.
Step~4: for approximations obtained in Steps~2 and~3, we verify their  properties
claimed in the lemma.

\bigskip
Step~1. We study the distribution of $X_{\eps,\tau^\kstar_\eps}$. 
Recalling that $\bar \alpha_{\kstar}= \alpha_{\kstar} =1$ and rewriting \eqref{eq:nu_k} with $k$ replaced by $\kstar$: 
\[
    \nu_{\kappa, \eps}(x,dy) = \Pp^{x_0+\eps^{\alpha_0} xv_0}\{\tau^\kstar_\eps<\infty,\ X_{\eps,\tau^\kstar_\eps}\in x_\kstar+\eps^1(dy)v_n\},
\] 
setting
\[
 \bar\nu_{\kappa,\eps}(x,dy) = \Pp\{\Phi^\kstar_{1,\eps}(x,N_\kstar)\in dy,\ \Phi^\kstar_{2,\eps}(x,N_\kstar)\geq 0\},
\]
and applying Lemma~\ref{lem:typical_chain_exit} to these measures, 
we  see that they satisfy conditions~\ref{item:nu-bar_nu} and~\ref{item:nu(K)} of Lemma~\ref{lem:trans_ker}. 

\bigskip

Step~2. We study the distribution of $X_{\eps,\tau^{\kstar+1}_\eps}$. Note that $\alpha_{\kstar}=1$ due to the definition of $\kstar$ in~\eqref{eq:def_kstar}.
Applying Lemma~\ref{lem:local-limit-theorem} to the transition from $X_{\eps,\tau^\kappa_\eps}$ to $X_{\eps,\tau^{\kappa+1}_\eps}$, we have that 
for some $c_{\kstar+1}\geq 0$ and 
\begin{align}\label{eq:bar_mu_kstar+1_Good}
    \bar\mu_{\kstar+1}\in\GoodMeasures,
\end{align}
the kernels given by
\begin{align}
    \mu_{\kappa+1,\eps}(x,dy)& = \eps^{-(\frac{\bar\alpha_{\kappa+1}}{\rho_{\kappa+1}}-1)}\Pp^{x_{\kappa}+\eps x v_{\kappa}}\left\{\tau^{\kstar+1}_\e<\infty,\ X_{\e,\tau^{\kstar+1}_\e}\in x_{\kstar+1}+\e^{\bar\alpha_{\kappa+1}} (dy)v_{\kstar+1} \right\},   \label{eq:mu_kstar+1}
    \\
    \bar \mu_{\kappa+1,\eps}(x,dy)& = g_{c_{\kappa+1}}(x)\bar\mu_{\kappa+1}(dy)\label{eq:bar_mu_kstar+1}
\end{align}
satisfy condition~\ref{item:mu-bar_mu} of Lemma~\ref{lem:trans_ker}.
Note that in fact $\bar\mu_{\kstar+1,\eps}$ does not depend on~$\eps$.
Due to~\eqref{eq:bar_mu_kstar+1_Good} and property~\eqref{eq:GoodMeasures} enjoyed by measures in $\GoodMeasures$, for every $\vk_{\kstar+1}>0$, there is $p>0$ such that
\begin{align}\label{eq:mu_kstar(K)}
    \bar\mu_{\kstar+1}(K_{\vk_{\kstar+1}}(\eps))\leq l_\eps^p, 
\end{align}
Using this and the fact that $g_{c_{\kappa+1}}$ is a Gaussian density (see~\eqref{eq:g_c_Gaussian}), we derive that
condition~\ref{item:mu_decomp}  of Lemma~\ref{lem:trans_ker} also holds for $\bar\mu_{\kappa+1,\eps}$. In fact, we can explicitly decompose $g_{c_{\kappa+1}}(x)$ into a sum of two bounded monotone functions:
\begin{align}
    g_{c_{\kappa+1}}(x) & = g^+_{c_{\kappa+1}}(x) + g^-_{c_{\kappa+1}}(x) \notag
    \\
    &= \left(g_{c_{\kappa+1}}(x)\ONE_{x< 0}+g_{c_{\kappa+1}}(0)\ONE_{x\geq 0}\right) + \left(g_{c_{\kappa+1}}(x)\ONE_{x\geq 0}-g_{c_{\kappa+1}}(0)\ONE_{x\geq 0}\right). \label{eq:g_decomp}
\end{align}

Having checked all the conditions of Lemma~\ref{lem:trans_ker} for $\nu_{\kstar,\eps},\bar \nu_{\kstar,\eps}, \mu_{\kstar+1,\eps}, \bar \mu_{\kstar+1,\eps}$, we can now
apply it and obtain that, for any $\vk_{\kstar}>0$, the kernels given by
\begin{align}
    \nu_{\kstar+1,\eps}(x,dy)& = \int_{K_{\vk_{\kstar}}(\eps)}\nu_{\kstar,\eps}(x,dz)\mu_{\kstar+1,\eps}(z,dy), \notag
    \\
    \bar \nu_{\kstar+1,\eps}(x,dy) &= \int_{K_{\vk_{\kstar}}(\eps)}\bar\nu_{\kstar,\eps}(x,dz)\bar\mu_{\kstar+1,\eps}(z,dy)\label{eq:bar_nu_kstar+1}
\end{align}
also satisfy condition~\ref{item:nu-bar_nu} of Lemma~\ref{lem:trans_ker}. In particular, \eqref{eq:local_limit_nu-bar_nu} with $k=\kstar+1$ holds for $\nu_{\kstar+1,\eps}, \bar\nu_{\kstar+1,\eps}$.

For later use, we note that $\bar\nu_{\kstar+1,\e}$ satisfies condition~\ref{item:nu(K)} of Lemma~\ref{lem:trans_ker}, as a result of~\eqref{eq:mu_kstar(K)}, the boundedness of~$g_{c_{\kstar+1}}$, and the fact that $\bar\nu_{\kstar,\eps}(x,dy)$ is a sub-probability measure.

\bigskip

Step~3. If $k\in\{\kstar+2,\ldots, j\}$ for $j=\min H$, then, to study the distribution $X_{\eps,\tau^{k}_\eps}$, we need to study the transition from $X_{\eps, \tau^{\kstar+1}_\eps}$ to $X_{\eps,\tau^{k}_\eps}$. 
The scaling upon the exit from saddle   $O_{\kstar+1}$ described by \eqref{eq:mu_kstar+1} is $\eps^{\bar\alpha_{\kstar+1}}$. Since  there are no elements of $H$ between $\kstar$ and $j$, part~\ref{item:backward-recursion1} of Lemma~\ref{lem:structure-of-reverse-exponents} guarantees that $\bar \alpha_{k}=\bar \alpha_{k-1} \rho_{k}\le 1$ for all
$k\in\{\kstar+2,\ldots,j\}$ and, moreover, $\bar\alpha_{j}=1$. Therefore, the dynamics of exponents $\bar \alpha_k$ for these saddles is described by \eqref{eq:new-alpha},
i.e.\ the evolution is typical and described by
 Lemma~\ref{lem:typical_chain_exit}.
Applying parts~\eqref{item:use_Phi_to_approx} and~\eqref{item:phi_pm_decomp_k} of this lemma to the dynamics starting near $x_{\kstar+1}$, we  see that the kernels given by 
\begin{align}
    \mu_{k,\eps}(x,dy)& = \Pp^{x_{\kappa+1}+\eps^{\bar\alpha_{\kappa+1}} x v_{\kappa+1}}\left\{\tau^{k}_\e<\infty,\  X_{\e,\tau^{k}_\e}\in x_{k}+\e^{\bar\alpha_k} (dy)v_{k} \right\},\notag
    \\
    \bar \mu_{k,\eps}(x,dy)& = \Pp\{\Phi^k_{1,\eps}(x,N_k)\in dy ,\ \Phi^k_{2,\eps}(x,N_k)\geq 0\} \label{eq:bar-mu_k,eps}
\end{align}
for $k\in\{\kstar+2,\ldots,j\}$ 
satisfy conditions~\ref{item:mu-bar_mu} and~\ref{item:mu_decomp} of Lemma~\ref{lem:trans_ker}.

This, along with the conclusions from Step~2, allows to apply Lemma~\ref{lem:trans_ker} to $\nu_{\kappa+1,\eps}$, $\bar \nu_{\kappa+1,\eps}$, $ \mu_{k,\eps}$, $\bar \mu_{k,\eps}$ and obtain that, for some $\vk_{\kstar+1}>0$, the kernels given by
\begin{align}
\notag
    \nu_{k,\eps}(x,dy)& = \int_{K_{\vk_{\kstar+1}}(\eps)}\nu_{\kappa+1,\eps}(x,dz)\mu_{k,\eps}(z,dy),
    \\
    \label{eq:bar_nu_k_eps}
    \bar \nu_{k,\eps}(x,dy) &= \int_{K_{\vk_{\kstar+1}}(\eps)}\bar\nu_{\kappa+1,\eps}(x,dz)\bar\mu_{k,\eps}(z,dy)
\end{align}
satisfy condition~\ref{item:nu-bar_nu} of Lemma~\ref{lem:trans_ker}, which yields the desired result~\eqref{eq:local_limit_nu-bar_nu}.
Combining this with Step~2, we can conclude that~\eqref{eq:local_limit_nu-bar_nu} holds for $k\le j$.

\bigskip

Step~4. We verify that $\nu_{k,\eps}$ and $\bar\nu_{k,\eps}$ are of the desired form as in \eqref{eq:nu_k} and~\eqref{eq:bar_nu_k,eps}, respectively.

First, we verify this for $\nu_{k,\eps}$. We can check, using the definition~\eqref{eq:theta_k_H}, that $\theta_k = \frac{\bar\alpha_{\kappa+1}}{\rho_{\kappa+1}}-1$ for all $k\in \{\kstar+1,\ldots,\min H\}$.
Tracing the definitions of these kernels, we can see that, in agreement with~\eqref{eq:nu_k} (and \eqref{eq:inter_bar_A}),  
\begin{align*}
    \nu_{k,\eps}(x,[a,b]) = \eps^{-\theta_k}\Pp^{x_0+\eps^{\alpha_0}xv_0}(B_{k,\eps,[a,b]}\cap \bar A_{\kstar,\vk_{\kstar},\eps} \cap \bar A_{\kstar+1,\vk_{\kstar+1},\eps}  ).
\end{align*}

Next, we verify the decomposition~\eqref{eq:bar_nu_k,eps} for $\bar\nu_{k,\eps}$ along with \ref{item:h_k_cvg}, \ref{item:gamma_k,k=underbar_k}, \ref{item:gamma_k,k>underbar_k}.

For all $k\in\{\kstar+1,\ldots,j\}$, we set  
\begin{equation}
\label{eq:def_gamma_k}
\varsigma_k=\bar\mu_{\kstar+1}\in \GoodMeasures
\end{equation} and
\begin{align}
    h_{k,\eps}(x) & = \E\left[g_{c_{\kstar+1}}\left(\Phi^\kstar_{1,\eps}(x,N_\kstar)\right)\ONE_{\Phi^\kstar_{1,\eps}(x,N_\kstar)\in K_{\vk_\kstar}(\eps),\ \Phi^\kstar_{2,\eps}(x,N_\kstar)\geq 0}\right]  \label{eq:h_k,eps(x)}.
\end{align}
Recalling $\underline k = \kstar+1$ from~\eqref{eq:underline_k_base}, we thus have $\zeta_k = \zeta_{\underline k}$ and $h_{k,\eps} = h_{\underline k,\eps}$ for all $k\in \{\kappa+1,\ldots,j\}$.

For $k=\kstar+1$ (equivalently, $k=\underline{k}$), relation~\eqref{eq:bar_nu_k,eps} follows from \eqref{eq:bar_mu_kstar+1} and~\eqref{eq:bar_nu_kstar+1},
and property~\ref{item:gamma_k,k=underbar_k} follows from~\eqref{eq:def_gamma_k}.

For $k\in\{\kstar+2,\ldots,j\}$,  we define $\varsigma_{k,\eps}$ by  \eqref{eq:gamma_k,eps_formula} and~\eqref{eq:def_gamma_k}.
Due to $\underline k = \kstar +1$, relation~\eqref{eq:bar_nu_k,eps}  for these values of $k$ follows now from \eqref{eq:bar_mu_kstar+1}, \eqref{eq:bar_nu_kstar+1}, \eqref{eq:bar-mu_k,eps}, \eqref{eq:bar_nu_k_eps},
\eqref{eq:def_gamma_k}, \eqref{eq:h_k,eps(x)}. We also obtain \eqref{eq:gamma_good} from \eqref{eq:def_gamma_k}. Since $N_k,\Phi^k_{1,\eps},\Phi^k_{2,\eps}$ were introduced through the application of Lemma~\ref{lem:typical_chain_exit}, they also have the desired properties. Therefore, \ref{item:gamma_k,k>underbar_k} holds.

It remains to verify \ref{item:h_k_cvg}.
It is clear that $h_{k,\eps}$ is bounded uniformly in $\eps$. Let us show that they converge in LU and the limit is strictly positive everywhere. Setting
\begin{align}\label{eq:tilde_h_k,eps(x)}
    \tilde h_{k,\eps}(x) = \E\left[g_{c_{\kstar+1}}\left(\Phi^\kstar_{1,\eps}(x,N_\kstar)\right)\ONE_{\Phi^\kstar_{2,\eps}(x,N_\kstar)\geq 0}\right],
\end{align}
and using the fact that  $g_{c_{\kstar+1}}$ is a Gaussian density and Lemma~\ref{lem:typical_chain_exit}~(\ref{item:Phi_1_range_multi_saddle}), we obtain
\begin{align*}
    \|  h_{k,\eps} -\tilde h_{k,\eps}\|_\infty = o_e(1).
\end{align*}
Hence, it suffices to show that $\tilde h_{k,\eps}$ satisfies the desired properties. Recalling the decomposition in \eqref{eq:g_decomp}, we define
\begin{align*}
    \tilde h^\pm_{k,\eps}(x) = \E\left[g^\pm_{c_{\kstar+1}}\left(\Phi^\kstar_{1,\eps}(x,N_\kstar)\right)\ONE_{\Phi^\kstar_{2,\eps}(x,N_\kstar)\geq 0}\right]
\end{align*}
satisfying
\begin{align*}
\tilde h_{k,\eps}(x) = \tilde h^+_{k,\eps}(x) + \tilde h^-_{k,\eps}(x).
\end{align*}

Let us first show that $\tilde h^\pm_{k,\eps}$ converges pointwise, and then upgrade this to convergence in LU.
Fix any $x\in\R$ if $\alpha_0=1$, or $x\in \R\setminus\{0\}$ if $\alpha_0<1$.
Using the convergence of $\Phi^\kstar_{i,\eps}$ given in Lemma~\ref{lem:typical_chain_exit}~\eqref{item:cvg_Phi^k}, we have that $g^\pm_{c_{\kstar+1}}\left(\Phi^\kstar_{1,\eps}(x,N_\kstar)\right)$ converges a.s. Using the property of the limit $\Phi^\kstar_2$ of $\Phi^\kstar_{2,\eps}$ described in Lemma~\ref{lem:typical_chain_exit}~\eqref{item:cvg_Phi^k}, we get that $\Phi^\kstar_2(x,N_k)\neq 0$ a.s.\ and thus $\ONE_{ \Phi^\kstar_{2,\eps}(x,N_\kstar)\geq 0}$ converges a.s. Then, in view of \eqref{eq:tilde_h_k,eps(x)}, the bounded convergence theorem yields that $\tilde h^\pm_{k,\eps}$ converges pointwise to
\begin{align*}h^\pm_{k}: x \mapsto \E\left[g^\pm_{c_{\kstar+1}}\left(\Phi^\kstar_{1}(x,N_\kstar)\right)\ONE_{ \Phi^\kstar_{2}(x,N_\kstar)\geq 0}\right].
\end{align*}
It is clear that $h^\pm_{k}$ is bounded, and the continuity of $h^\pm_{k}$ follows from the properties of $\Phi^\kstar_i$ in Lemma~\ref{lem:typical_chain_exit}~\eqref{item:cvg_Phi^k}.
Since $\pm g^\pm_{c_{\kstar+1}}$ is nondecreasing, using Lemma~\ref{lem:typical_chain_exit}~\eqref{item:nondecreasing_Phi^k}, we can see that $\pm \tilde h^\pm_{k,\eps}(x)$ is nondecreasing, which upgrades the pointwise convergence to LU convergence on $\R$ if $\alpha_0=1$ or on $\R\setminus\{0\}$ if $\alpha_0<1$.

Combining this with the above displays, we obtain the LU convergence of  $h_{k,\eps}$~to
\begin{align*}
    h_k:x\mapsto \E\left[g_{c_{\kstar+1}}\left(\Phi^\kstar_{1}(x,N_\kstar)\right)\ONE_{ \Phi^\kstar_{2}(x,N_\kstar)\geq 0}\right]=h_k^+(x)+h_k^-(x).
\end{align*}
Since $g_{c_{\kstar+1}}$ is positive everywhere, Lemma~\ref{lem:typical_chain_exit}~\eqref{item:cvg_Phi^k} implies that this expectation is positive for all $x\in\R$ if $\alpha_0=1$, and all $x\in(0,\infty)$ if $\alpha_0<1$; it is identical zero for all $x\in(-\infty,0)$ if $\alpha_0<1$. The boundedness and continuity of $h_k$ follows from those properties for  $h^\pm_k$. This completes the verification of properties of $h_{k,\eps}$.

\medskip

This completes the proof of the basis case, i.e., for $k\in\{\kstar+1, \ldots, \min H\}$.

\bigskip

\textbf{Induction step.} Let us assume that the desired result holds for all $\kappa+1,\kappa+2,\dots,j$ for some $j\in H$. Let $\overline{j} = \min\{i\in H: i>j\}$. Our goal is to extend the result to~ values $k\in \{j+1,\dots, \overline{j}\}$.
Note that
\begin{align}\label{eq:underline_k=j+1}
    \underline{k}=j+1.
\end{align}

\medskip
The argument is very similar to that for the base case. We split it into three steps.
Step~1: we use Lemma~\ref{lem:local-limit-theorem} on atypical transitions to obtain an approximation for the distribution of $X_{\eps,\tau^{j+1}_\eps}$. Step~2: if $k> j +1$, we approximate the distribution of $X_{\eps,\tau^{k}_\eps}$ applying 
Lemma~\ref{lem:typical_chain_exit} on typical transitions to the transition from $X_{\eps,\tau^{j+1}_\eps}$ to $X_{\eps,\tau^{k}_\eps}$. Step~3:
for the approximations obtained in Steps 1 and 2, we verify the properties claimed in the lemma.
\bigskip

Step~1. We study the distribution of $X_{\eps,\tau^{j+1}_\eps}$ through the transition from $X_{\eps,\tau^{j}_\eps}$ to $X_{\eps,\tau^{j+1}_\eps}$. Using the induction assumption (in particular,~\eqref{eq:local_limit_nu-bar_nu}), we have that $\nu_{j,\eps}$ given in~\eqref{eq:nu_k} and some measure $\bar\nu_{j,\eps}$ of the form~\eqref{eq:bar_nu_k,eps} satisfy Lemma~\ref{lem:trans_ker}~\ref{item:nu-bar_nu}. 
In addition, $\bar \nu_{j,\eps}$ satisfies Lemma~\ref{lem:trans_ker}~\ref{item:nu(K)} due to the uniform boundedness of $h_{j,\eps}$,~\eqref{eq:bar_nu_k,eps} and~\eqref{eq:gamma_good} (see the property~\eqref{eq:GoodMeasures} for measures in $\GoodMeasures$).

Due to~\eqref{eq:backward-iteration1}, $\bar \alpha_{j}=1$. Now applying Lemma~\ref{lem:local-limit-theorem}, we obtain that the kernels   $\mu_{j+1,\eps},\bar \mu_{j+1,\eps}$ given by
\begin{align}
    \mu_{j+1,\eps}(x,dy)& = \eps^{-(\frac{\bar\alpha_{j+1}}{\rho_{j+1}}-1)}\Pp^{x_{j}+\eps x v_{j}}\left\{\tau^{j+1}_\e<\infty,\ X_{\e,\tau^{j+1}_\e}\in x_{j+1}+\e^{\bar\alpha_{j+1}} (dy)v_{j+1} \right\},\label{eq:mu_k+1,eps} 
    \\
    \bar \mu_{j+1,\eps}(x,dy)& = g_{c_{j+1}}(x)\bar\mu_{j+1}(dy) \label{eq:bar_mu_k+1,eps},
\end{align}
satisfy Lemma~\ref{lem:trans_ker}~\ref{item:mu-bar_mu} and that
\begin{align}\label{eq:bar_mu_k+1_Good}
    \bar \mu_{j+1} \in \GoodMeasures.
\end{align}
Similarly to the argument used to derive~\eqref{eq:mu_kstar(K)}, we have that for every $\vk>0$ there is $p>0$ such that
\begin{align}\label{eq:mu_k+1(K)}
    \bar\mu_{j+1}(K_\vk(\eps))\leq l_\eps^p. 
\end{align}
Using a decomposition similar to~\eqref{eq:g_decomp}, we can verify that $\bar\mu_{j+1,\eps}$ satisfies Lemma~\ref{lem:trans_ker}~\ref{item:mu_decomp}.

Hence, we are now allowed to apply Lemma~\ref{lem:trans_ker} to $\nu_{j,\eps},\bar \nu_{j,\eps}, \mu_{j+1,\eps}, \bar \mu_{j+1,\eps}$ to see that the kernels 
$\nu_{j+1,\eps}, \bar \nu_{j+1,\eps}$
given by
\begin{align}
    \nu_{j+1,\eps}(x,dy)& = \int_{K_{\vk_{j}}(\eps)}\nu_{j,\eps}(x,dz)\mu_{j+1,\eps}(z,dy),\label{eq:nu_k+1,eps}
    \\
    \bar \nu_{j+1,\eps}(x,dy) &= \int_{K_{\vk_{j}}(\eps)}\bar\nu_{j,\eps}(x,dz)\bar\mu_{j+1,\eps}(z,dy)\label{eq:bar_nu_k+1,eps},
\end{align}
satisfy Lemma~\ref{lem:trans_ker}~\ref{item:nu-bar_nu}.
One can easily check that this definition of $\nu_{j+1,\eps}$ coincides with~\eqref{eq:nu_k} for $j+1$.
Since we have shown that $\bar\nu_{j,\eps}$ satisfies Lemma~\ref{lem:trans_ker}~\ref{item:nu(K)}, displays~\eqref{eq:bar_mu_k+1,eps}, \eqref{eq:mu_k+1(K)} and the boundedness of the Gaussian density $g_{c_{j+1}}$ imply that $\bar\nu_{j+1,\eps}$ satisfies~\ref{item:nu(K)}.

\bigskip

Step~2. For $k\in\{j+2,\ldots, \overline{j} \}$, we study the distribution of $X_{\eps,\tau^{k}_\eps}$ through the transition from $X_{\eps,\tau^{j+1}_\eps}$ to $X_{\eps,\tau^{k}_\eps}$.
The scaling upon the exit from saddle   $O_{j+1}$ described by \eqref{eq:mu_k+1,eps} is $\eps^{\bar\alpha_{j+1}}$. Since  there are no elements of $H$ between $j$ and $k$, part~\ref{item:backward-recursion1} of Lemma~\ref{lem:structure-of-reverse-exponents} guarantees that $\bar \alpha_{k}=\bar \alpha_{k-1} \rho_{k}\le 1$ for all
$k\in\{j+2,\ldots,\overline{j}\}$ and, moreover, $\bar\alpha_{\overline{j}}=1$. Therefore, the dynamics of exponents $\bar \alpha_k$ for these saddles is described by \eqref{eq:new-alpha},
i.e. the evolution is typical and described by
 Lemma~\ref{lem:typical_chain_exit}.

Applying parts~\eqref{item:use_Phi_to_approx} and~\eqref{item:phi_pm_decomp_k} of this lemma to the dynamics starting near $x_{j+1}$ shows that kernels given by 
\begin{align}
    \mu_{k,\eps}(x,dy)& = \Pp^{x_{j+1}+\eps^{\bar\alpha_{j+1}} x v_{j+1}}\left\{\tau^{k}_\e<\infty,\ X_{\e,\tau^{k}_\e}\in x_{k}+\e^{\bar\alpha_k} (dy)v_{k} \right\}\label{eq:mu_k',eps}\\
    \bar \mu_{k,\eps}(x,dy)& = \Pp\{\Phi^{k}_{1,\eps}(x,N_{k})\in dy ,\ \Phi^{k}_{1,\eps}(x,N_{k})\geq 0\}, \label{eq:mu_bar_k'_eps}
\end{align}
for $k\in\{j+2,\ldots, \overline{j} \}$,
satisfy Lemma~\ref{lem:trans_ker}~\ref{item:mu-bar_mu} and~\ref{item:mu_decomp}.
This and the result in Step~1 allow us to apply Lemma~\ref{lem:trans_ker} to $\nu_{j+1,\eps}, \bar \nu_{j+1,\eps}, \mu_{k,\eps}, \bar \mu_{k,\eps}$ to get that, for any $\vk_{j+1}>0$, the kernels  $\bar \nu_{k,\eps},  \nu_{k,\eps}$ defined by
\begin{align}
    \nu_{k,\eps}(x,dy)& = \int_{K_{\vk_{j+1}}(\eps)}\nu_{j+1,\eps}(x,dz)\mu_{k,\eps}(z,dy),\label{eq:nu_k',eps} 
    \\
    \bar \nu_{k,\eps}(x,dy) &= \int_{K_{\vk_{j+1}}(\eps)}\bar\nu_{j+1,\eps}(x,dz)\bar\mu_{k,\eps}(z,dy) \label{eq:bar_nu_k',eps},
\end{align}
satisfy~\eqref{eq:local_limit_nu-bar_nu}.
This, along with Step~1, completes the verification of~\eqref{eq:local_limit_nu-bar_nu} for $k\in\{j+1,\ldots, \overline{j} \}$.

\medskip

Step~3. We show that $\nu_{k,\eps}$ and $\bar\nu_{k,\eps}$ coincide with those given by~\eqref{eq:nu_k} and~\eqref{eq:bar_nu_k,eps}.

First, we verify that $\nu_{k,\eps}$ of \eqref{eq:nu_k',eps} coincides with~\eqref{eq:nu_k}.
Using the expressions for $\nu_{k,\eps}$ in \eqref{eq:nu_k',eps}, for $\nu_{j+1,\eps}$ in \eqref{eq:nu_k+1,eps}, for $\mu_{k,\eps}$ in \eqref{eq:mu_k',eps}, for $\mu_{j+1,\eps}$ in \eqref{eq:mu_k+1,eps}, and~$\nu_{j,\eps}$~in~\eqref{eq:nu_k}, we can compute that 
\begin{align*}
    &\nu_{k,\eps}(x,[a,b]) = \int_{K_{\vk_{j}}(\eps)}\nu_{j,\eps}(x,dz')\int_{K_{\vk_{j+1}}(\eps)}\mu_{j+1,\eps}(z',dz)\mu_{k,\eps}(z,[a,b])
    \\
    &=\eps^{-\theta_k -(\frac{\bar\alpha_{j+1}}{\rho_{j+1}}-1)}\Pp^{x_0 + \eps^{\alpha_0}xv_0}\Bigg(\bigcap_{i\in H'\cap\{0,\ldots,j-1\} }( \bar A_{i,\vk_i,\e}\cap \bar A_{i+1,\vk_{i+1},\e}) \cap B_{j,\eps,K_{\vk_{j}}(\eps)}
    \\
    &\quad\cap\left\{\tau^{k}_\e<\infty,\  X_{\eps,\tau^{j+1}_\eps}\in x_{j+1}+\eps^{\bar\alpha_{j+1}}K_{\vk_{j+1}}(\eps)v_{j+1},\ X_{\eps,\tau^{k}_\eps}\in x_{k}+\eps^{\bar\alpha_{k}}[a,b]v_{k}\right\}  \Bigg)
\end{align*}
The right-hand side of this display coincides with the right-hand side of~\eqref{eq:nu_k} (for~$k$ in the range that we are considering). To see this, we need to note a few things. First, we use
the definition of $\theta_{j}$ in~\eqref{eq:theta_k_H} and the fact that there are no elements of~$H$ between $j$ and $k$ to see that $\theta_{k}=\theta_{j}+ \frac{\bar\alpha_{j+1}}{\rho_{j+1}}-1$. Next,
due to the~\eqref{eq:def-of-corridor} and~\eqref{eq:events-B}, we have $B_{j,\eps,K_{\vk_j}}(\eps) = \bar A_{j,\vk_j,\eps}$. Also, the event in the last line of the last display is exactly  $\bar A_{j+1,\vk_{j+1},\eps}\cap B_{k,\eps,[a,b]}$.
Finally, we have
\begin{align*}
    H'\cap\{0,\dots,k-1\} = \left(H'\cap\{0,\dots,j-1\}\right)\cup\{j\}.
\end{align*}
Applying these observations to the last display together with \eqref{eq:inter_bar_A}, we complete the proof of \eqref{eq:nu_k}.

\smallskip

Let us check the properties of $\bar\nu_{k,\eps}$ claimed in Lemma~\ref{lem:local-limit-for-H}, namely, decomposition~\eqref{eq:bar_nu_k,eps} along with \ref{item:h_k_cvg}, \ref{item:gamma_k,k=underbar_k}, \ref{item:gamma_k,k>underbar_k}. Recall $\underline k= j+1$ as in \eqref{eq:underline_k=j+1}.

If $k=j+1$ (equivalently, $k=\underline{k}$), then, using the expressions for $\bar\nu_{j+1,\eps}$ in \eqref{eq:bar_nu_k+1,eps}, $\bar\mu_{j+1,\eps}$ in~\eqref{eq:bar_mu_k+1,eps}, and $\bar\nu_{j,\eps}$ in~\eqref{eq:bar_nu_k,eps} (applied to $j$), we can see that \eqref{eq:bar_nu_k,eps} holds for
\begin{align}
    h_{k,\eps}(x) & = \left(\int_{K_{\vk_{j}}(\eps)}\varsigma_{j,\eps}(dz')g_{c_{j+1}}(z')\right)h_{j,\eps}(x),\label{eq:h_k,eps(x)_induction}
\end{align}
and $\varsigma_{k,\eps}= \bar\mu_{j+1}\in \GoodMeasures$ (due to \eqref{eq:bar_mu_k+1_Good}), verifying \ref{item:gamma_k,k=underbar_k}.

If $k\in\{j+2,\ldots, \overline{j} \}$, then, using the expressions for $\bar \nu_{k,\eps}$ in \eqref{eq:bar_nu_k',eps}, 
$\bar\nu_{j+1,\eps}$ in~\eqref{eq:bar_nu_k+1,eps}, $\bar\mu_{k,\eps}$ in \eqref{eq:mu_bar_k'_eps}, $\bar\mu_{j+1,\eps}$ in~\eqref{eq:bar_mu_k+1,eps}, and $\bar\nu_{j,\eps}$ in~\eqref{eq:bar_nu_k,eps} (applied to $j$), we can see that~\eqref{eq:bar_nu_k,eps} holds for $h_{k,\eps}$ defined in \eqref{eq:h_k,eps(x)_induction} and
\begin{align*}
    \varsigma_{k,\eps}(dy) & = \int_{K_{\vk_{j+1}}(\eps)}\bar\mu_{j+1}(dz)\Pp\{\Phi^{k}_{1,\eps}(z,N_{k})\in dy ,\ \Phi^{k}_{1,\eps}(z,N_{k})\geq 0\}. \end{align*}
Since $\underline{k} =j+1$ in this case, we can set  $\varsigma_{k}=\bar\mu_{j+1}$ to guarantee~\eqref{eq:gamma_k,eps_formula}. Now~\eqref{eq:gamma_good} follows from~\eqref{eq:bar_mu_k+1_Good}. The random vector $N_{k}$ and the map $\Phi^{k}_{i,\eps}$ were introduced in Step~2 through the application of Lemma~\ref{lem:typical_chain_exit}.
Thus they possess  the desired properties automatically. Hence, we have verified \ref{item:gamma_k,k>underbar_k}.

It remains to show \ref{item:h_k_cvg}, which will follow from the induction assumption on $h_{j,\eps}$ once we show that
\begin{align}\label{eq:int_gamma_k,eps_g}
    \int_{K_{\vk_{j}}(\eps)}\varsigma_{j,\eps}(dz')g_{c_{j+1}}(z')
\end{align}
is bounded uniformly in $\eps$ and converges as $\eps\to 0$ to a positive constant. 
To that end,
we expand \eqref{eq:int_gamma_k,eps_g} using the induction assumption on $\varsigma_{k,\eps}$:
\begin{align}\label{eq:int_gamma_kE[...]}
    \int_{K_{\vk_{\underline{j}}}(\eps)}\varsigma_{j}(dz)\E \left[g_{c_{j+1}}\left(\Phi^{j}_{1,\eps}(z,N_{j})\right)\ONE_{\Phi^{j}_{1,\eps}(z,N_{j})\in K_{\vk_{j}}(\eps),\ \Phi^{j}_{2,\eps}(z,N_{j})\geq 0}\right].
\end{align}
Since part~\eqref{item:k_chain_Phi>|x|^p} of Lemma~\ref{lem:typical_chain_exit} holds for $\Phi^{j}_{1,\eps}$, the fact that $g_{c_{j+1}}$ is a Gaussian density and the Gaussianity of $N_{j}$ imply 
\begin{align*}
    &\E\left[g_{c_{j+1}}\left(\Phi^{j}_{1,\eps}(z,N_{j})\right)\right]\\
    &= \E\left[g_{c_{j+1}}\left(\Phi^{j}_{1,\eps}(z,N_{j})\right)\left(\ONE_{|z|<R}+\ONE_{|z|\geq R,\, |N_{j}|_\infty\leq |z|^q}+\ONE_{|z|\geq R,\, |N_{j}|_\infty>|z|^q}\right)\right]\\
    &\leq C\left(\ONE_{|z|<R}+ e^{-c|z|^{2p}}+e^{-c|z|^{2q}} \right)
\end{align*}
for some $C,c>0$.
Using this and~\eqref{eq:GoodMeasures} enjoyed by $\varsigma_{j}$ (due to \eqref{eq:gamma_good}), the boundedness of the expression in~\eqref{eq:int_gamma_k,eps_g} is immediate.
Moreover, the integrand in \eqref{eq:int_gamma_kE[...]} is dominated by a function integrable with respect to $\varsigma_j$. Since in the limit,
as $\e\to 0$, $K_{\vk_j}(\eps)$ and $K_{\vk_{\underline {j}}}(\e)$ 
expand to cover the entire $\R$, we can use arguments similar to those in Step~4 of the basis case to conclude that the integrand converges pointwise everywhere. Therefore, the dominated convergence theorem gives the convergence of \eqref{eq:int_gamma_k,eps_g} to
\begin{align*}
    \int_\R\varsigma_{j}(dz)\E \left[g_{c_{j+1}}\left(\Phi^{j}_{1}(z,N_{j})\right)\ONE_{ \Phi^{j}_{2}(z,N_{j})\geq 0}\right].
\end{align*}
The induction assumption guarantees that $\varsigma_{j}\in\GoodMeasures$. In particular, \eqref{eq:GoodMeasures>0} holds for~$\varsigma_{j}$.  Thus, 
to show the positivity of the above integral, it suffices to show the integrand is positive for every $z\in(0,\infty)$. In turn, this follows since
the Gaussian density~$g_{c_{j+1}}$ is positive and  the condition on $\Phi_2^{j}$ in part \eqref{item:cvg_Phi^k} of Lemma~\ref{lem:typical_chain_exit} holds.
Hence, the expression in \eqref{eq:int_gamma_k,eps_g} converge pointwise everywhere to a function that is positive everywhere, and so does~$h_{j,\eps}$. 
Using monotonicity similarly to Step~4 of the base case, we upgrade pointwise convergence to LU convergence.

\smallskip
This completes the proof of the induction step and  of the entire Lemma~\ref{lem:local-limit-for-H}.
\end{proof}

\subsubsection{Proof of \eqref{eq:main_polynomial_asymptotics}}

Since $\bar \alpha_{n-1} =1$ (see Lemma~\ref{lem:structure-of-reverse-exponents}~\eqref{item:backward-recursion1}), we set
\begin{align*}
    \phi_\eps(x) = \Pp^{x_{n-1}+\eps xv_{n-1}}(A_{n,\eps}).
\end{align*}
We start by choosing $\vk_0$ and  $\vk_k$'s used in the definition for $\nu_{n-1,\eps}=\nu_{n-1,(\vk),\eps}$ given in \eqref{eq:nu_k}. First, we 
use the tameness of $\xi_{0,\eps}$ to
choose $\vk_0$ sufficiently large enough to ensure $\Pp\{|\xi_{0,\eps}|>l^{\vk_0}_\eps\}=o_e(1)$. Then, we choose $\vk_k$ in $\nu_{n-1,\eps}$ large to ensure that Lemmas~\ref{lem:restricting-to-a-narrow-corridor} and~\ref{lem:local-limit-for-H} are applicable. We note that
if $\kappa =0$, then  $\vk_0$ is used in the definition of $\nu_{n-1,\eps}$. In this case, we simply make the previously chosen~$\vk_0$ larger, and adjust the others accordingly.

Using Lemma~\ref{lem:restricting-to-a-narrow-corridor} with $\widetilde \vk_0$ replaced by $\vk_0$ therein, we have, uniformly in $x\in K_{\vk_0}(\eps)$,
\begin{align}\label{eq:P(A_n,eps)_approx}
    \Pp^{x_0+\eps^{\alpha_0} x v_0}(A_{n,\eps}) = \eps^{\theta_{n-1}}\int_{K_{\vk_{n-1}}(\eps)}\nu_{n-1,\eps}(x,dy)\phi_\eps(y) + o_e(1),
\end{align}
for $\nu_{n-1,\eps}$ given in~\eqref{eq:nu_k} and $\theta_{n-1}$ defined in~\eqref{eq:theta_k_H}. In fact,  $\theta_{n-1}=\theta$, where the latter is defined in~\eqref{eq:theta}.

The limiting behavior of the right-hand side of \eqref{eq:P(A_n,eps)_approx}, can be analyzed using Lemma~\ref{lem:trans_ker}. The latter is actually
targeted at transition kernel convolutions but we can make it work for this simpler case.

Applying Lemma~\ref{lem:local-limit-for-H}, we have that $\nu_{n-1,\eps}$ and $\bar \nu_{n-1,\eps}$ (given in \eqref{eq:bar_nu_k,eps}) satisfy Lemma~\ref{lem:trans_ker}~\ref{item:nu-bar_nu}. Due to~\eqref{eq:gamma_k,eps_formula} and~\eqref{eq:gamma_good}, Lemma~\ref{lem:trans_ker}~\ref{item:nu(K)} is satisfied by $\bar \nu_{n-1,\eps}$.  Lemma~\ref{lem:postive-limit-prob-if-close-to-manifold} implies that, for some constant $s>0$, kernels given by
\begin{align*}
    \mu_{n,\eps}(x,dy) &= \phi_\eps(x)\delta_0(dy),\\
    \bar \mu_{n,\eps}(x,dy) &= \psi_s(-x)\delta_0(dy),
\end{align*}
where $\delta_0$ is the Dirac mass at $0$ (any probability measure that does not depend on~$\eps$ would work equally well) and $\psi_s$ is given in that lemma, satisfy Lemma~\ref{lem:trans_ker}~\ref{item:mu-bar_mu}. Due to the definition of $\psi_s$ in~\eqref{eq:gaussian-cdf}, Lemma~\ref{lem:trans_ker}~\ref{item:mu_decomp} is satisfied by $\bar \mu_{n,\eps}$, as $\psi_s$ is monotone. Therefore, we can apply Lemma~\ref{lem:trans_ker} to $\nu_{n-1,\eps}, \bar\nu_{n-1,\eps}, \mu_{n,\eps},\bar\mu_{n,\eps}$ to see that
\begin{align}\label{eq:int_nu_phi_pf_main_thm_3}
    \sup_{x\in K_{\vk_0}(\eps)}\left|\int_{K_{\vk_{n-1}}(\eps)}\nu_{n-1,\eps}(x,dy)\phi_\eps(y) - \int_{K_{\vk_{n-1}}(\eps)}\bar\nu_{n-1,\eps}(x,dy)\psi_s(-y)\right| = \smallo{\eps^\delta}
\end{align}
for some $\delta>0$. In view of \eqref{eq:P(A_n,eps)_approx}, it remains to verify that the second integral in the above display converges to a positive constant as $\eps\to 0$.

The expression for $\bar \nu_{n-1,\eps}$ in~\eqref{eq:bar_nu_k,eps} (for $k=n-1$) allows us to compute that, for some $\vk'_{n-1}>0$,
\begin{align}
    &\int_{K_{\vk_{n-1}}(\eps)}\bar\nu_{n-1,\eps}(x,dy)\psi_s(-y) = h_{n-1,\eps}(x) \label{eq:approx_pf_main_thm_3}\\
    &\times \int_{K_{\vk'_{n-1}}(\eps)}\varsigma_{n-1}(dz) \E\left[\psi_s\left(-\Phi^{n-1}_{1,\eps}(z,N_{n-1})\right)\ONE_{\Phi^{n-1}_{1,\eps}(z,N_{n-1})\in K_{\vk_{n-1}}(\eps), \ \Phi^{n-1}_{2,\eps}(z,N_{n-1})\geq 0}\right]. \notag
\end{align}
Lemma~\ref{lem:local-limit-for-H} ensures that $h_{n-1,\eps}$ is bounded uniformly in $\eps$ and that $h_{n-1,\eps}$ converges in LU to some positive bounded continuous function on $\R$ if $\alpha_0 =1$; or a nonnegative bounded continuous function on $\R\setminus\{0\}$, which is positive on $(0,\infty)$, if $\alpha_0<1$. The argument we used to derive the convergence of \eqref{eq:int_gamma_kE[...]} yields the convergence of the integral on the right-hand side of \eqref{eq:approx_pf_main_thm_3} to a positive constant. Hence, the left-hand side of \eqref{eq:approx_pf_main_thm_3}, viewed as a function of $x$, is bounded uniformly in $\eps$ and converges in LU to some bounded continuous function $\bar h:\R\to (0,\infty)$ if $\alpha_0=1$, or $\bar h:\R\setminus\{0\}\to [0,\infty)$, satisfying $\bar h>0$ on $(0,\infty)$, if $\alpha_0<1$.

This along with \eqref{eq:P(A_n,eps)_approx} and \eqref{eq:int_nu_phi_pf_main_thm_3} implies that the function
\begin{align}\label{eq:def_bar_h_eps}
    \bar h_\eps : x\mapsto \eps^{-\theta_{n-1}}\P^{x_0+\eps^{\alpha_0} x v_0}(A_{n,\eps}) \ONE_{x\in K_{\vk_0}(\eps)}
\end{align}
is bounded uniformly in $\eps$, and converges in LU to $\bar h$ as $\eps \to 0$. 
We have
\begin{align}
\label{eq:P(A_n,e)_bar_h}
   \Pp(A_{n,\eps}) &= \E\left[ \P^{x_0+\eps^{\alpha_0} \xi_{0,\eps} v_0}(A_{n,\eps})\right] = \E\left[ \P^{x_0+\eps^{\alpha_0} \xi_{0,\eps} v_0}(A_{n,\eps})\ONE_{\xi_{0,\eps}\in K_{\vk}(\eps)}\right] + \Delta_\e
   \\ &= \eps^{\theta_{n-1}}\E\bar h_\eps (\xi_{0,\eps}) + \Delta_\e, \notag
\end{align}
where 
\begin{equation}
\label{eq:Delta_for_P_A_n}
0\le\Delta_\e\le \Pp\{\xi_{0,\eps}\notin K_{\vk}(\eps)\}.
\end{equation}
Due to the tameness of $\xi_{0,\eps}$, we have $\Delta_\e=o_e(1)$, so
\begin{align*}
    \Pp(A_{n,\eps})
     = \eps^{\theta_{n-1}}\E\bar h_\eps (\xi_{0,\eps}) + o_e(1).
\end{align*}
It remains to verify
\begin{align}\label{eq:E-bar-h_eps_lim}
    \lim_{\eps\to0} \E \bar h_\eps(\xi_{0,\eps}) = \E \bar h(\xi_0)>0.
\end{align}
First, we consider the case $\alpha_0=1$. We start with the upper bound
\begin{align}
    &\left|\E \bar h_\eps(\xi_{0,\eps}) - \E \bar h(\xi_0)\right|  \leq \left|\E \bar h_\eps(\xi_{0,\eps}) - \E \bar h(\xi_{0,\eps})\right| + \left|\E \bar h(\xi_{0,\eps}) - \E \bar h(\xi_0)\right|\notag
    \\
    & \leq \E \left[ \left|\bar h_\eps(\xi_{0,\eps}) - \bar h(\xi_{0,\eps})\right|\ONE_{|\xi_{0,\eps}|\leq R} \right] + C\Pp\{|\xi_{0,\eps}|> R\} + \left|\E \bar h(\xi_{0,\eps}) - \E \bar h(\xi_0)\right|\label{eq:split_pf_main_3}
\end{align}
which holds for some $C>0$ and all $R>0$.
The second term on the right-hand side can be made arbitrarily small, uniformly in small $\eps$, by choosing $R$ sufficiently large. The third term decays to zero 
as $\eps\to0$ due to condition  \ref{setting:scaling-limit-at-saddle-0}.
The first term in \eqref{eq:split_pf_main_3} converges to $0$ due to the LU convergence proved above. Hence, we conclude that \eqref{eq:E-bar-h_eps_lim} holds and the right-hand side is positive due to the positivity of $\bar h$. 

The argument is similar for $\alpha_0<1$. The estimate \eqref{eq:split_pf_main_3} is replaced by
\begin{align*}
    &\left|\E \bar h_\eps(\xi_{0,\eps}) - \E \bar h(\xi_0)\right|  
    \\
    & \leq \E \left[ \left|\bar h_\eps(\xi_{0,\eps}) -  \bar h(\xi_{0,\eps})\right|\ONE_{|\xi_{0,\eps}|\in[\delta, R]} \right] + C\Pp\{|\xi_{0,\eps}|\not \in [\delta,R]\} + \left|\E \bar h(\xi_{0,\eps}) - \E \bar h(\xi_0)\right|.
\end{align*}
Here the second term can be made arbitrarily small by choosing sufficiently small $\delta>0$ and sufficiently large $R>0$. The first term converges to 0 due to the LU convergence of $\bar h_\e$ to $\bar h$. 
To deduce the convergence of the last term to 0, besides the weak convergence of $\xi_{0,\eps}$ to $\xi_0$, we also use the fact that the only discontinuity point
$0$ of $\bar h$ is not an atom of the distribution of $\xi_0$. We also note that the right-hand side of \eqref{eq:E-bar-h_eps_lim} is positive
because of our assumption  $\Pp\{\xi_0>0\}>0$ and the fact that  $\bar h>0$ on $(0,\infty)$ and non-negative elsewhere.
 This completes the proof of \eqref{eq:main_polynomial_asymptotics}  of Theorem~\ref{th:multi-saddle-escape}~\eqref{item:cell-escape-probability}. 
 \epf

\subsection{Proof of (\ref{eq:tau_main_th}) in Theorem~\ref{th:multi-saddle-escape}~(\ref{item:cell-escape-probability})}

We need the following lemma describing the typical exit time near a saddle point where the initial condition is of order $\eps^\alpha$ for $\alpha\in(0,1)$. Here, we recall that Lemma~\ref{lem:exit_time_is_log} describes the typical exit time for $\alpha=1$.
\begin{lemma}\label{lem:exit_time_is_log_alpha<1} 
Under conditions~\ref{setting:general},~\ref{setting:geometry-domain}, and~\ref{setting:conjugacy}, for $\alpha\in(0,1)$ and every $\vk,\delta>0$, there is $\delta'>0$ such that
\begin{align*}
    \Pp^{x_0 + \eps^{\alpha}xv_0}\left\{\left|\frac{\tau_\e}{\frac{\alpha}{\lambda} l_\eps }-1\right|>\delta\right\}\leq \ONE_{|x|\leq \eps^{\delta'}} +  o_e(1).
\end{align*}
holds uniformly in $x\in K_\vk(\eps)$.
\end{lemma}
\bpfm  Due to \eqref{eq:weak_convergence_to_Gauss}, \eqref{eq:expr-for-tau}, \eqref{eq:Z_eps}, we have $\tau_\eps\approx \frac{1}{\lambda}\log\frac{R}{|\eps^\alpha x+\eps \frU|}$.
Thus, uniformly in $x\in K_\vk(\eps)$,
\begin{align*}
    \Pp^{x_0 + \eps^{\alpha}xv_0}\left\{\tau_\e<\frac{\alpha-\delta}{\lambda} l_\eps \right\} \approx \Pp^{x_0 + \eps^{\alpha}xv_0}\{ |x+\eps^{1-\alpha}\frU|>\eps^{-\delta}R\},
    \\
    \Pp^{x_0 + \eps^{\alpha}xv_0}\left\{\tau_\e>\frac{\alpha+\delta}{\lambda} l_\eps \right\} \approx \Pp^{x_0 + \eps^{\alpha}xv_0}\{ |x+\eps^{1-\alpha}\frU|<\eps^{\delta}R\}.
\end{align*}
The first display is $o_e(1)$ due to the Gaussianity of $\frU$ and $x\in K_\vk(\eps)$. The Gaussianity of $\frU$ yields that the second display is bounded above by $\ONE_{|x|\leq \eps^{\delta'}}+  o_e(1)$ for some $\delta'>0$.
\epf

Slightly extending the proof of \eqref{eq:main_polynomial_asymptotics} in Theorem~\ref{th:multi-saddle-escape}, we obtain the following
lemma, where the scaling limit assumption~\ref{setting:scaling-limit-at-saddle-0}
is replaced by the tameness of the initial condition:

\begin{lemma}\label{lem:rare_trans_upp_bdd}
Under conditions~\ref{setting:general}, \ref{setting:het-chain},~\ref{setting:scaling-at-saddle-0}, and \ref{setting:conjugacy-all}, if $\kstar<n-1$, then, for each $\vk_0>0$ and for $\theta$ defined in~\eqref{eq:theta},
\begin{align*}
    \sup_{x\in K_{\vk_0}(\e)} \Pp^{x_0+\e^{\alpha_0}x v_0}(A_{n,\eps})= O(\eps^\theta).
\end{align*}

\end{lemma}

\begin{proof}
In our proof of \eqref{eq:main_polynomial_asymptotics} in Theorem~\ref{th:multi-saddle-escape}, for an arbitrary initial condition $\xi_{0,\e},$ we obtained~\eqref{eq:P(A_n,e)_bar_h}, an expression for $\Pp^{x_0+\e^{\alpha_0}x v_0}(A_{n,\eps})$ in
terms of a function $\bar h_\eps$ defined in \eqref{eq:def_bar_h_eps} and a small correction $\Delta_\eps$.
To finish the proof, it now suffices to recall that we showed that $\bar h_\eps$  is bounded uniformly in $\eps$ and to note that 
\eqref{eq:Delta_for_P_A_n} implies that  for $\xi_{0,\e}=x\in K_\vk(\eps)$,  $\Delta_\e=0$. 
\end{proof}

Now, we are ready to prove \eqref{eq:tau_main_th} in Theorem~\ref{th:multi-saddle-escape}. For brevity, we write
\begin{align*}
    \chi_i = 
    \begin{cases}
    \frac{\bar \alpha_{i-1}}{\lambda_i}, & i\in \{1,\dots,n\}\setminus J,
    \\
    \frac{\bar \alpha_i}{\mu_i},& i\in J.
    \end{cases}
\end{align*}
Comparing this with \eqref{eq:bar_chi_def}, we have $\bar \chi = \sum_{i=1}^n \chi_i$. We also set $\tau^0_\eps =0$. Let $\delta>0$, and we have
\begin{align*}
    &\Pp^{x_0 + \eps^{\alpha_0}xv_0}\left\{\left|\tau^n_\eps - \bar\chi l_\eps \right|>n\delta l_\eps ,\ A_{n,\eps}\right\} \le  \sum_{i=1}^n P_i,
 \end{align*}
  where
  \begin{align*}
  P_i=P_i(x)=\Pp^{x_0 + \eps^{\alpha_0}xv_0}\left\{\left|\tau^i_\eps-\tau^{i-1}_\eps - \chi_i l_\eps \right|>\delta l_\eps ,\ A_{n,\eps}\right\}.
\end{align*}
Due to~\eqref{eq:main_polynomial_asymptotics}, it suffices to show that for all $i$, $P_i=o(\eps^\theta)$ uniformly
in $x\in K_{\vk_{0}}(\eps)$ for~$\theta$ from~\eqref{eq:theta}. 
Using Lemma~\ref{lem:restricting-to-a-narrow-corridor} and the strong Markov property, for $\vk_k$'s chosen as in the proof of Theorem~\ref{th:multi-saddle-escape}~\eqref{eq:main_polynomial_asymptotics}, we have, uniformly in  $x\in K_{\vk_0}(\eps)$,
\begin{align}
   &P_i = \Pp^{x_0+\eps^\alpha x v}\left\{ D_\e\cap A_{n,\e}\cap \bigcap_{k\in \{\kstar\} \cup H \cup J} \bar A_{k,\vk_k,\e} \right\}+o_e(1)\notag
   \\
   & = \E^{x_0 + \eps^{\alpha_0}xv_0}\left[\ONE_{\bar A_{\leq i-1,(\vk),\eps}}\E\left[\ONE_{D_\eps\cap\bar A_{i,\vk_i,\eps}}\Pp\left(A_{n,\eps}\cap\bar A_{\geq  i+1,(\vk),\eps}\big|X_{\eps,\tau^i_\eps}\right)\Big| X_{\eps,\tau^\eps_{i-1}}\right]\right]+o_e(1), \label{eq:prob_tau^i-tau^i-1}
\end{align}
where
\begin{gather*}
    D_\eps = \{|\tau^i_\eps - \tau^{i-1}_\eps - \chi_i l_\eps |>\delta  l_\eps \},
    \\
    \bar A_{\leq i-1,(\vk),\eps}= \bigcap_{\substack{k\leq i-1 \\ k\in \{\kstar\} \cup H\cup J}}\bar A_{k,\vk_k,\eps},\\
    \bar A_{\geq  i+1,(\vk),\eps} = \bigcap_{\substack{k\geq i+1\\ k \in \{\kstar\} \cup H\cup J}}\bar A_{k,\vk_k,\eps}.
\end{gather*}

To estimate \eqref{eq:prob_tau^i-tau^i-1}, we consider three transitions separately: from $X_{\eps,0}$ to $X_{\eps,\tau^{i-1}_\eps}$, from $X_{\eps,\tau^{i-1}_\eps}$ to $X_{\eps,\tau^{i}_\eps}$, and from $X_{\eps,\tau^{i}_\eps}$ to $X_{\eps,\tau^{n}_\eps}$. We will apply Lemma~\ref{lem:typical_chain_exit} or Lemma~\ref{lem:local-limit-for-H} to the first part, Lemma~\ref{lem:local-limit-theorem}~\eqref{item:large-negative-values-w.l.p.} to the second part, and Lemma~\ref{lem:rare_trans_upp_bdd} to the third part. 

First, we consider the third part, i.e., the transition from $X_{\eps,\tau^{i}_\eps}$ to $X_{\eps,\tau^{n}_\eps}$. Let us first assume $i\leq n-1$.
Our goal is to apply Lemma~\ref{lem:rare_trans_upp_bdd} to the diffusion along the heteroclinic chain $(O_i, \gamma_i,O_{i+1},\ldots,\gamma_{n-1},O_n,\gamma_{n+1}, O_{n+1})$ with initial condition belonging to $I_\e=x_i+\e^{\bar \alpha_i}K_{\vk_i}(\e) v_i$. 
This initial condition is, in fact, given by $X_{\eps,\tau^i_\eps}$; it belongs to $I_\e$ on $\bar A_{i,\vk_i,\eps}$, see the definition of the latter in~\eqref{eq:def-of-corridor}.

To apply Lemma~\ref{lem:rare_trans_upp_bdd}, we need to introduce a new sequence of exponents playing the role of $(\alpha_0,\alpha_1,\ldots,\alpha_{n-1})$ in Theorem~\ref{th:multi-saddle-escape} and Lemma~\ref{lem:rare_trans_upp_bdd}, where the role of~$\alpha_0$ is played by $\bar\alpha_i$, 
and compute all the other elements of the construction of the exponent $\theta$.

So we define a new sequence~$(\tilde\alpha_j)_{j\in\{i,\ldots,n\}}$ recursively by $\tilde\alpha_{i} = \bar\alpha_{i}$ and $\tilde\alpha_{j+1} = \tilde\alpha_{j}\rho_{j+1} \wedge 1$. We set
$\tilde\kstar = \max\{j:i\leq j\leq n-1,\ \tilde \alpha_j =1\}$, then we define the set~$\tilde H$ of binding indices for this stage of evolution.
Similarly to~\eqref{eq:H'},~\eqref{eq:def_J}, \eqref{eq:bar-alpha}, we
define~${\tilde H}',\tilde J$ and a new sequence $(\bar \alpha_j)_{j=i}^{n-1}$. Using Lemma~\ref{lem:structure-of-reverse-exponents}~\eqref{item:backward-recursion1}, we  see that 
\begin{align*}
    \tilde\kstar\leq n-1,\quad \tilde H = H\cap \{i,\dots,n-1\},\quad \tilde J = J\cap \{i+1,\dots ,n\},
\end{align*}
and the new $(\bar \alpha_j)_{j=i}^{n-1}$ is simply the restriction of $(\bar \alpha_j)_{j=0}^{n-1}$
to $j\in\{i,\ldots,n-1\}$.

Therefore, applying Lemma~\ref{lem:rare_trans_upp_bdd} to this stage of evolution we see that, uniformly on the event $\bar A_{i,\vk_i,\eps}$, we have $\Pp\left(A_{n,\eps}\cap\bar A_{\geq  i+1,(\vk),\eps}|X_{\eps,\tau^i_\eps}\right) = O(\eps^{\tilde\theta})$ where
\begin{align*}
    \tilde \theta = \sum_{j\in\tilde J} \left(\frac{\bar\alpha_j}{\rho_j}-1\right)= \sum_{j\in J\cap\{i+1,\dots,n\}} \left(\frac{\bar\alpha_j}{\rho_j}-1\right).
\end{align*}
Therefore, \eqref{eq:prob_tau^i-tau^i-1} can be continued as
\begin{align*}
    P_i= O(\eps^{\tilde\theta})\E^{x_0 + \eps^{\alpha_0}xv_0}\left[\ONE_{\bar A_{\leq i-1,(\vk),\eps}}\Pp\left(D_\eps\cap \bar A_{i,\xi,\eps}\Big| X_{\eps,\tau^{i-1}_\eps}\right)\right]+o_e(1).
\end{align*}

If $i=n$, then $\tilde J$ is empty and the above bound is still valid with $\tilde \theta =0$. To see this, we simply apply $\Pp(A_{n,\eps}\cap\bar A_{\geq  n+1,\vk,\eps}|X_{\eps,\tau^n_\eps})\le1$ in \eqref{eq:prob_tau^i-tau^i-1}.

Next, 
we study the transition from $X_{\eps,\tau^{i-1}_\eps}$ to $X_{\eps,\tau^{i}_\eps}$. If $i-1 \in \{\kstar\}\cup H$, we apply Lemma~\ref{lem:local-limit-theorem}~\eqref{item:large-negative-values-w.l.p.}. If $i-1 \not\in \{\kstar\}\cup H$, we apply Lemma~\ref{lem:exit_time_is_log} for $\bar\alpha_{i-1}=1$ or Lemma~\ref{lem:exit_time_is_log_alpha<1} for $\bar\alpha_{i-1}<1$. Then, the last display implies
that, uniformly in $x\in K_{\vk_0}(\eps)$,
\begin{align}
   P_i= O(\eps^{\theta'})\P^{x_0 + \eps^{\alpha_0}xv_0}\left(\eps^{-\bar\alpha_{i-1}}\left|X_{\eps,\tau^{i-1}_\eps}-x_{i-1}\right|\leq \eps^{\delta'}, \bar A_{\leq i-1,(\vk),\eps}\right)+o_e(1), \label{eq:tran_time_special_i}
    \\
    \text{if }i-1\not\in\{\kstar\}\cup H,\text{ and }\bar \alpha_{i-1}<1;  \notag
    \\
    P_i=o(\eps^{\theta'+\delta_i})\P^{x_0 + \eps^{\alpha_0}xv_0}\left( \bar A_{\leq i-1,(\vk),\eps}\right)+o_e(1),\quad \text{otherwise}, \label{eq:tran_time_otherwise}
\end{align}
for some $\delta_i>0$. 
Here
\begin{align*}
    \theta' = \sum_{j\in J\cap\{i,i+1,\dots,n\}} \left(\frac{\bar\alpha_j}{\rho_j}-1\right)
    =
    \begin{cases}
    \tilde \theta + \frac{\bar \alpha_i}{\rho_i}-1, &\quad \text{if }i-1 \in \{\kstar\}\cup H,
    \\
    \tilde \theta ,  &\quad \text{if }i-1 \not\in \{\kstar\}\cup H.
    \end{cases}
\end{align*}

Lastly, we study the transition from $X_{\eps,0}$ to $X_{\eps,\tau^{i-1}_\eps}$. Recalling the definition of~$\theta_{i-1}$ in~\eqref{eq:theta_k_H} and that of~$\theta$ in~\eqref{eq:theta}, we obtain $\theta_{i-1} + \theta'=\theta$. 
If $i-1\leq\kstar$ (implying $\theta_{i-1}=0$ by \eqref{eq:theta_k=0,k_leq_kstar} and thus $\theta'=\theta$), we apply Lemma~\ref{lem:typical_chain_exit} to $k=i-1$. If $i-1>\kstar$, we apply Lemma~\ref{lem:local-limit-for-H}. 
Then, we obtain the following results. 

First we estimate~\eqref{eq:tran_time_special_i}. 
Under the condition $i-1\not\in\{\kstar\}\cup H$ and $\bar \alpha_{i-1}<1$, the main term in \eqref{eq:tran_time_special_i}
can be bounded from above by
\begin{align*}
    & O(\eps^\theta) \Pp\{\Phi^{i-1}_{1,\eps}(x,N_{i-1})\in |v_{i-1}|^{-1}[-\eps^{\delta'},\eps^{\delta'}] \},\quad \text{if } i-1<\kstar , 
    \\ 
    & O(\eps^\theta) \bar\nu_{i-1,\eps}(x,|v_{i-1}|^{-1}[-\eps^{\delta'},\eps^{\delta'}]),\quad \text{if } i-1>\kstar.
\end{align*}

Next, we estimate \eqref{eq:tran_time_otherwise}.
When $i-1\leq \kstar$,  we bound the probability in~\eqref{eq:tran_time_otherwise} by~$1$ and thus the main term in \eqref{eq:tran_time_otherwise} is $o(\eps^\theta)$. When $i-1> \kstar$, using Lemma~\ref{lem:local-limit-for-H} and \eqref{eq:nu_k}, we can bound the probability on the r.h.s.\ of \eqref{eq:tran_time_otherwise} by $\eps^{\theta_{i-1}}(\bar\nu_{i-1,\eps}(x,K_{\vk_{i-1}}(\eps))+o(\eps^\delta))$ for some $\delta>0$, uniformly in $x\in K_{\vk_0}(\eps)$. 
Recalling the expression for~$\bar \nu_{i-1,\eps}$ in \eqref{eq:bar_nu_k_eps}, the boundedness of $h_{i-1,\eps}$ in Lemma~\ref{lem:local-limit-for-H}~\ref{item:h_k_cvg}, the expression for~$\varsigma_{i-1,\eps}$ in Lemma~\ref{lem:local-limit-for-H}~\ref{item:gamma_k,k=underbar_k} and \ref{item:gamma_k,k>underbar_k} where $\varsigma_{i-1}$ satisfies \eqref{eq:GoodMeasures} due to $\varsigma_{i-1}\in\GoodMeasures$, we can see that $\bar\nu_{i-1,\eps}(x,K_{\vk_{i-1}}(\eps))\leq l_\eps^p$ for some $p>0$. We can conclude that the main term 
in~\eqref{eq:tran_time_otherwise} is $o(\eps^{\theta'+\theta_{i-1}} )= o(\eps^\theta)$.

Hence, in view of the tameness of $\xi_{0,\eps}$, to prove \eqref{eq:tau_main_th}, it suffices to verify:
\begin{align}
    \lim_{\eps\to 0} \Pp\{\Phi^{i-1}_{1,\eps}(\xi_{0,\eps},N_{i-1})\in |v_{i-1}|^{-1}[-\eps^{\delta'},\eps^{\delta'}] \} = 0, \label{eq:P(xi,N)_to_0}
    \\
    \lim_{\eps\to 0}\E\bar\nu_{i-1,\eps}(\xi_{0,\eps},|v_{i-1}|^{-1}[-\eps^{\delta'},\eps^{\delta'}]) =0. \label{eq:E_nu(xi)_to_0}
\end{align}

To prove  \eqref{eq:P(xi,N)_to_0}, we will  show that $\lim_{\eps\to 0}\E w_\eps(\xi_{0,\eps}) = 0$ for
\begin{gather*}
    w_\eps(x) = \Pp\{\Phi^{i-1}_{1,\eps}(x,N_{i-1})\in |v_{i-1}|^{-1}[-\eps^{\delta'},\eps^{\delta'}] \}.
\end{gather*}
Denoting 
\[
v_\eta(x)=
\Pp\{\Phi^{i-1}_{1}(x,N_{i-1})\in [-2\eta,2\eta]\},
\]
we use Lemma~\ref{lem:typical_chain_exit}~\eqref{item:cvg_Phi^k} to obtain
\begin{align}\label{eq:lim_eta_0_Phi}
    \lim_{\eta\to0}v_\eta(x)=0,
\end{align}
for every $x\in\R$ if $\alpha_0=1$ or for every $x\in \R\setminus\{0\}$ if $\alpha_0<1$.
Due to our assumption on $\xi_0$, this implies
\[
    \lim_{\eta\to0}v_\eta(\xi_0)\stackrel{a.s.}{=}0.
\]
For each $\eta\in(0,1)$, let $\zeta_\eta:\R\to[0,1]$ be a smooth bump function that is constantly $1$ on $[-\eta,\eta]$ and supported on $[-2\eta,2\eta]$. Hence, setting
\begin{align*}
    u_\eta(x) = \E \zeta_\eta\circ \Phi^{i-1}_1(x,N_{i-1}),
\end{align*}
we obtain that $u_\eta(\xi_0)$ converges to $0$ a.s.\ as $\eta\to 0$, which implies $\lim_{\eta\to 0} \E u_\eta(\xi_0)=0$.
Now, fixing any $\delta>0$, we choose $\eta$ sufficiently small so that
\begin{align}\label{eq:Eu_eta<delta}
    \E u_\eta(\xi_0)\leq \delta.
\end{align}
Setting
\begin{align*}
    u_{\eta,\eps}(x) = \E \zeta_\eta\circ \Phi^{i-1}_{1,\eps}(x,N_{i-1}),
\end{align*}
we want to estimate
\begin{align*}
    |\E u_{\eta,\eps}(\xi_{0,\eps})-\E u_\eta(\xi_0)| \leq |\E u_{\eta,\eps}(\xi_{0,\eps})-\E u_\eta(\xi_{0,\eps})| + |\E u_{\eta}(\xi_{0,\eps})-\E u_\eta(\xi_0)|.
\end{align*}
Since $u_\eta$ is bounded and continuous (due to the continuity of $\Phi^{i-1}_1$ ensured by Lemma~\ref{lem:typical_chain_exit}), and since $\xi_{0,\eps}\indistr \xi_0$, the second term on the right can be made arbitrarily small for sufficiently small $\eps$. To treat the first term, we bound it by
\begin{align*}
    \E |u_{\eta,\eps}(\xi_{0,\eps})-\E u_\eta(\xi_{0,\eps})|\ONE_{|\xi_{0,\eps}|\leq R} + 2\Pp\{|\xi_{0,\eps}|> R\}.
\end{align*}
Due to the LU convergence of $\Phi^{i-1}_{1,\eps}$ given in Lemma~\ref{lem:typical_chain_exit}~\eqref{item:cvg_Phi^k}, and the smoothness of~$\zeta_\eta$, we see that $u_{\eta,\eps}$ converges in LU to $u_{\eta}$. Hence, choosing $R$ large and then~$\eps$ sufficiently small, the above can be made arbitrarily small. In view of \eqref{eq:Eu_eta<delta}, we can conclude that $\E u_{\eta,\eps}(\xi_{0,\eps})\leq 2\delta$ 
for sufficiently small $\eps$.

Since $\delta>0$ is arbitrary and  $\E w_\eps(\xi_{0,\eps})\leq \E u_{\eta,\eps}(\xi_{0,\eps})$ for sufficiently small $\eps$, we can thus conclude \eqref{eq:P(xi,N)_to_0}.

Now, we turn to \eqref{eq:E_nu(xi)_to_0}. Using \eqref{eq:bar_nu_k,eps} and \ref{item:h_k_cvg} in Lemma~\ref{lem:local-limit-for-H}, the expectation 
in~\eqref{eq:E_nu(xi)_to_0}
is bounded by a constant times
\begin{align*}
   E_\e= \E \varsigma_{i-1,\eps}(|v_{i-1}|^{-1}[-\eps^{\delta'},\eps^{\delta'}]).
\end{align*}
If $\varsigma_{i-1,\eps}$ is given by Lemma~\ref{lem:local-limit-for-H}~\ref{item:gamma_k,k=underbar_k}, i.e., it does not depend on $\eps$ and belongs to~$\GoodMeasures$ (thus being absolutely continuous), then $\lim_{\eps\to0}E_\eps=0$. If $\varsigma_{i-1,\eps}$ is given by Lemma~\ref{lem:local-limit-for-H}~\ref{item:gamma_k,k>underbar_k}, then
\begin{align*}
E_\eps\le    \int \varsigma_{i-1}(d z)\Pp\{\Phi^{i-1}_{1,\eps}(z,N_{i-1})\in [-\eta,\eta]\}
\end{align*}
for every $\eta\in(0,1)$ and sufficiently small $\eps$. 
Due to Lemma~\ref{lem:typical_chain_exit}~\eqref{item:k_chain_Phi>|x|^p}, there is $q>0$ such that the following holds for all sufficiently large $L$: if $|z|> L$ and $|N_{i-1}|_\infty \leq |z|^q$, then $|\Phi^{i-1}_{1,\eps}(z,N_{i-1})|\geq 1>\eta$. Hence,
\begin{align*}
E_\eps\le    \int_{|z|\leq L} \varsigma_{i-1}(d z)\Pp\{\Phi^{i-1}_{1,\eps}(z,N_{i-1})\in [-\eta,\eta]\} + \int_{|z|>L} \varsigma_{i-1}(d z)\Pp\{|N_{i-1}|_\infty >|z|^q\}.
\end{align*}
Due to the Gaussian tail of $N_{i-1}$ and property \eqref{eq:GoodMeasures} enjoyed by $\varsigma_{i-1}$, the second term on the right-hand side can be made arbitrarily small by choosing $L$ sufficiently large. 
Noting that 
\begin{gather*}
    \limsup_{\eps\to 0}\Pp\{\Phi^{i-1}_{1,\eps}(z,N_{i-1})\in [-\eta,\eta]\}  \le  \Pp\{\Phi^{i-1}_{1}(z,N_{i-1})\in [-2\eta,2\eta]\},
\end{gather*}
using Fatou's lemma, \eqref{eq:lim_eta_0_Phi}, and choosing $\eta$
to be small we obtain that the first 
term can be made arbitrarily small as $\eps\to0$. This completes our proof. \epf

\section{Long-term asymptotics of diffusions near heteroclinic networks}
\label{sec:hierarchy}

In this section, we use the main result of Section~\ref{sec:long-escape-chains} to discuss --- briefly and informally, without any attempt at rigor --- the behavior of diffusions near heteroclinic networks over long periods of time.

We will work with a specific example but the picture of hierarchy of clusters and timescales that we describe holds for arbitrary planar stable heteroclinic networks. The periodic structure of our example allows to approach the question of homogenization.

Combining the vector field shown in Figure~\ref{fig:two-cells} with its own reflection we obtain a vector field on the torus  $\TT^2$ shown on Figure~\ref{fig:cellular-flow}. Once can also view this vector
field as $\Z^2$-periodic with a square fundamental domain, and lift the diffusion from~$\TT^2$ to its universal cover,  $\R^2$. 

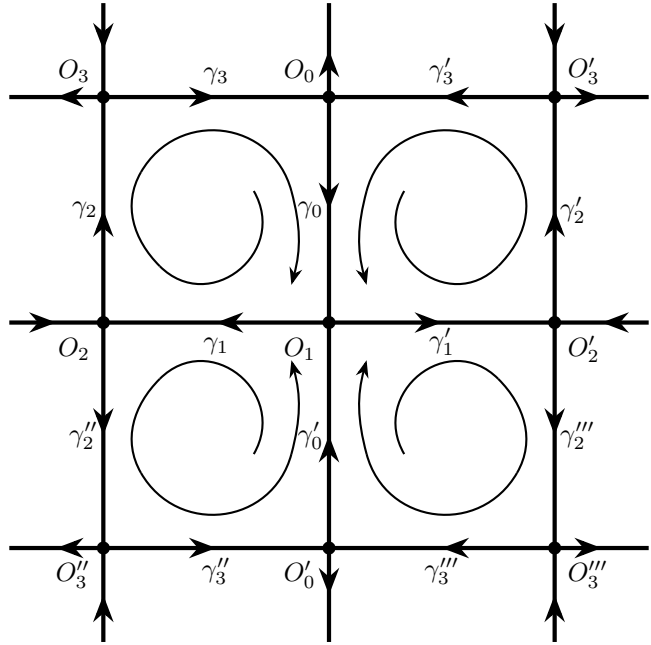
\begin{figure}[ht]
\centering
\begin{tikzpicture}

    \node (0) at (0, 0) {};
    \node (1) at (3, 0) {};
    \node (2) at (3, 3) {};
    \node (3) at (0, 3) {};
    \node (4) at (-3, 3) {};
    \node (5) at (-3, 0) {};
    \node (6) at (-3, -3) {};
    \node (7) at (0, -3) {};
    \node (8) at (3, -3) {};
    \node (9) at (4.25, 0) {};
    \node (10) at (4.25, 3) {};
    \node (11) at (3, 4.25) {};
    \node (12) at (0, 4.25) {};
    \node (13) at (-3, 4.25) {};
    \node (14) at (-4.25, 3) {};
    \node (15) at (-4.25, 0) {};
    \node (16) at (-1, 1.75) {};
    \node (18) at (-2.25, 0.75) {};
    \node (19) at (-2.5, 2) {};
    \node (21) at (-0.5, 1.75) {};
    \node (22) at (-0.5, 0.5) {};
    \node (23) at (1, 1.75) {};
    \node (24) at (2.25, 0.75) {};
    \node (25) at (2.5, 2) {};
    \node (26) at (0.5, 1.75) {};
    \node (27) at (0.5, 0.5) {};
    \node (27) at (0.5, 0.5) {};
    \node (28) at (-4.25, -3) {};
    \node (29) at (-3, -4.25) {};
    \node (30) at (0, -4.25) {};
    \node (31) at (3, -4.25) {};
    \node (32) at (4.25, -3) {};
    \node (33) at (-1, -1.75) {};
    \node (34) at (-2.25, -0.75) {};
    \node (35) at (-2.5, -2) {};
    \node (36) at (-0.5, -1.75) {};
    \node (37) at (-0.5, -0.5) {};
    \node (38) at (1, -1.75) {};
    \node (39) at (2.25, -0.75) {};
    \node (40) at (2.5, -2) {};
    \node (41) at (0.5, -1.75) {};
    \node (42) at (0.5, -0.5) {};

\foreach \n in {0, 1, 2, 3, 4, 5, 6, 7, 8}
        \node at (\n)[circle,fill,inner sep=1.75pt]{};

    \draw [ultra thick, postaction={on each segment={mid arrow=black}}] (3.center) to (0.center);
    \draw [ultra thick, postaction={on each segment={mid arrow=black}}] (0.center) to (1.center);
    \draw [ultra thick, postaction={on each segment={mid arrow=black}}] (0.center) to (5.center);
    \draw [ultra thick, postaction={on each segment={mid arrow=black}}] (5.center) to (4.center);
    \draw [ultra thick, postaction={on each segment={mid arrow=black}}] (1.center) to (2.center);
    \draw [ultra thick, postaction={on each segment={mid arrow=black}}] (4.center) to (3.center);
    \draw [ultra thick, postaction={on each segment={mid arrow=black}}] (2.center) to (3.center);
    \draw [ultra thick, postaction={on each segment={mid arrow=black}}] (2.center) to (10.center);
    \draw [ultra thick, postaction={on each segment={mid arrow=black}}] (4.center) to (14.center);
    \draw [ultra thick, postaction={on each segment={mid arrow=black}}] (9.center) to (1.center);
    \draw [ultra thick, postaction={on each segment={mid arrow=black}}] (15.center) to (5.center);
    \draw [ultra thick, postaction={on each segment={mid arrow=black}}] (5.center) to (6.center);
    \draw [ultra thick, postaction={on each segment={mid arrow=black}}] (13.center) to (4.center);
    \draw [ultra thick, postaction={on each segment={mid arrow=black}}] (7.center) to (0.center);
    \draw [ultra thick, postaction={on each segment={mid arrow=black}}] (3.center) to (12.center);
    \draw [ultra thick, postaction={on each segment={mid arrow=black}}] (11.center) to (2.center);
    \draw [ultra thick, postaction={on each segment={mid arrow=black}}] (1.center) to (8.center);
    \draw [thick] [in=-120, out=135] (18.center) to (19.center);
    \draw [thick] [in=-45, out=-60, looseness=1.50] (16.center) to (18.center);
    \draw [thick] [in=105, out=60, looseness=1.25] (19.center) to (21.center);
    \draw [thick, -Stealth] [in=75, out=-75] (21.center) to (22.center);
    \draw [thick] [in=-135, out=-120, looseness=1.50] (23.center) to (24.center);
    \draw [thick] [in=-60, out=45] (24.center) to (25.center);
    \draw [thick] [in=75, out=120, looseness=1.25] (25.center) to (26.center);
    \draw [thick, -Stealth] [in=105, out=-105] (26.center) to (27.center);
    \draw [ultra thick, postaction={on each segment={mid arrow=black}}](6.center) to (28.center);
    \draw [ultra thick, postaction={on each segment={mid arrow=black}}](6.center) to (7.center);
    \draw [ultra thick, postaction={on each segment={mid arrow=black}}](8.center) to (7.center);
    \draw [ultra thick, postaction={on each segment={mid arrow=black}}](8.center) to (32.center);
    \draw [ultra thick, postaction={on each segment={mid arrow=black}}](29.center) to (6.center);
    \draw [ultra thick, postaction={on each segment={mid arrow=black}}](7.center) to (30.center);
    \draw [ultra thick, postaction={on each segment={mid arrow=black}}](31.center) to (8.center);
    \draw [thick] [in=45, out=60, looseness=1.50] (33.center) to (34.center);
    \draw [thick] [in=120, out=-135] (34.center) to (35.center);
    \draw [thick] [in=-105, out=-60, looseness=1.25] (35.center) to (36.center);
    \draw [thick, -Stealth] [in=-75, out=75] (36.center) to (37.center);
    \draw [thick] [in=135, out=120, looseness=1.50] (38.center) to (39.center);
    \draw [thick] [in=60, out=-45] (39.center) to (40.center);
    \draw [thick] [in=-75, out=-120, looseness=1.25] (40.center) to (41.center);
    \draw [thick, -Stealth] [in=-105, out=105] (41.center) to (42.center);

    \node [left] at (-0.05, 3.35) {$O_0$};
    \node [left] at (-3.05, 3.35) {$O_3$};
    \node [right] at (3.05, 3.35) {$O'_3$};
    \node [left] at (-0.05, -0.35) {$O_1$};
    \node [left] at (-3.05, -0.35) {$O_2$};
    \node [right] at (3.05, -0.35) {$O'_2$};
    \node [left] at (-3.05, -3.35) {$O''_3$};
    \node [left] at (-0.05, -3.35) {$O'_0$};
    \node [right] at (3.05, -3.35) {$O'''_3$};

    \node [left] at (0.05, 1.5) {$\gamma_0$};
    \node [left] at (-2.95, 1.5) {$\gamma_2$};
    \node [right] at (2.95, 1.5) {$\gamma'_2$};
    \node [above] at (1.5, 3.05) {$\gamma'_3$};
    \node [above] at (-1.5, 3.05) {$\gamma_3$};
    \node [above] at (1.5, -0.55) {$\gamma'_1$};
    \node [above] at (-1.5, -0.55) {$\gamma_1$};
    \node [above] at (1.5, -3.6) {$\gamma'''_3$};
    \node [above] at (-1.5, -3.6) {$\gamma''_3$};
    \node [left] at (-2.95, -1.5) {$\gamma''_2$};
    \node [left] at (0.05, -1.5) {$\gamma'_0$};
    \node [right] at (2.95, -1.5) {$\gamma'''_2$};
\end{tikzpicture}

\caption{An example on $\TT^2$ or, by $\Z^2$-periodicity, on $\R^2$}
{\label{fig:cellular-flow}}
\end{figure}

On the torus, points $O_0,O'_0$ are identified, and so are  $O_2,O'_2$, and $O_3,O'_3, O^{\prime\prime}_3,O^{\prime\prime\prime}_3$.

We already know that under the assumptions on stability indices $\rho_0,\rho_1,\rho_2,\rho_3$ made in Section~\ref{sec:typical}, for small $\e$, the diffusion started near the connection $\gamma_0$ 
stays within the union of two cells on both sides of $\gamma_0$ at least for times comparable with $ l_\eps $, circulating near the boundaries of these two cells and making occasional transitions between them upon passing the neighborhood of $O_0$. The exit distribution upon passing $O_0$ is symmetric Gaussian, scaling as $\e^1$, and the next distributions from $O_1$, $O_2$ (or $O'_2$), $O_3$ (or $O'_3$) scale as $\e^{\rho_1}$,  
$\e^{\rho_1\rho_2}$, $\e^{\rho_1\rho_2\rho_3}$, respectively, and the scaling limit distributions are one-sided.

However, the elliptic diffusion on the torus must have an invariant distribution absolutely continuous with respect to the Lebesgue measure, so the process must eventually visit arbitrarily small neighborhoods of all points of the torus escaping from the pair of cells and realizing a rare transition or a sequence of those, on time scales longer than logarithmic.

Theorem~\ref{th:multi-saddle-escape} explains how cell escapes get realized. If we start at distance of order~$\e$ from $\gamma_0$, then it is easy to see that $\bar\alpha_1=\bar\alpha_2=\bar\alpha_3=1$, so the escape through $\gamma_1$ (or $\gamma'_1$), $\gamma_2$ (or $\gamma'_2$),   $\gamma_3$ (or $\gamma'_3$)  happens with probability of order $\e^{\frac{1}{\rho_1}-1}$, 
$\e^{\frac{1}{\rho_1}+\frac{1}{\rho_2}-2}$, $\e^{\frac{1}{\rho_1}+\frac{1}{\rho_2}+\frac{1}{\rho_2}-3}$, respectively. If the escape attempt is unsuccessful, the process typically returns to a neighborhood of 
the connection $\gamma_0$, passing it at a distance of order $\e$. To see a successful escape one needs to make about  $\e^{-(\frac{1}{\rho_1}-1)}$, 
$\e^{-(\frac{1}{\rho_1}+\frac{1}{\rho_2}-2)}$, $\e^{-(\frac{1}{\rho_1}+\frac{1}{\rho_2}+\frac{1}{\rho_2}-3)}$ attempts, respectively. Each of them takes time of order $ l_\eps $. Therefore, by a time $t(\e)$
satisfying
\begin{equation}
\label{eq:times-transitions-type-0}
 l_\eps \ll t(\e) \ll \e^{-(\frac{1}{\rho_1}-1)} l_\eps ,
\end{equation}
it is likely for the diffusion to visit small neighborhoods of all the saddle points multiple times but
it is unlikely to see any transitions between cells except crossing $\gamma_0$ and $\gamma'_0$ (let us call them transitions of type~$0$). Moreover,  one can easily compute the limit of  the empirical measure of the process
\[
\nu_{t(\e)}(A)=\frac{1}{t(\e)}\int_{0}^{t(\e)} \ONE_{X_{\e,s}\in A} ds.
\]
Since during one cycle, $X_\e$ spends time of order $\frac{\alpha_{i-1}}{\lambda_i} l_\eps $ near a saddle $O_i$ and the time it takes to travel between those saddles is of order of constant, we obtain that the limit is given by
\begin{equation}
\label{eq:limit-emp-measure}
    m_0\delta_{O_0}+ m_1\delta_{O_1}+ m_2\delta_{O_2}+ m_3\delta_{O_3}.
\end{equation}
Here
\begin{align*}
    m_0= \frac{\rho_1\rho_2\rho_3}{\lambda_0 Z},\quad m_1= \frac{1}{\lambda_1 Z},\quad  
    m_2= \frac{\rho_1}{\lambda_2 Z},\quad 
    m_3= \frac{\rho_1\rho_2}{\lambda_3 Z},
\end{align*}
with $Z$ being the normalizing constant
\[
    Z=  \frac{1}{\lambda_1}+\frac{\rho_1}{\lambda_2}+ \frac{\rho_1\rho_2}{\lambda_3}+\frac{\rho_1\rho_2\rho_3}{\lambda_0}.
\] 
By a time $t(\e)$ satisfying
\begin{equation}
\label{eq:times-transitions-type-1}
\e^{-(\frac{1}{\rho_1}-1)}   l_\eps \ll t(\e) \ll \e^{-(\frac{1}{\rho_1}+\frac{1}{\rho_2}-2)}   l_\eps ,
\end{equation}
it is likely to see a growing number of transitions through connections $\gamma_1,\gamma'_1$ (let us call them transitions of type~$1$) but no other new transitions. So the process circulates within the pair of cells for a long time making only transitions of type~$0$, then, at a random time, via a transition of type $1$, escapes to the neighboring pair of cells, where the same process begins anew, etc. For these times $t(\e)$,
the process is still confined, with high probability, to the 4-cell cluster composed of two 2-cell clusters. At longer time scales though, for $t(\e)$ satisfying
 \begin{equation}
\label{eq:times-transitions-type-2}
 \e^{-(\frac{1}{\rho_1}+\frac{1}{\rho_2}-2)}   l_\eps \ll t(\e) \ll  \e^{-(\frac{1}{\rho_1}+\frac{1}{\rho_2}+\frac{1}{\rho_2}-3)}   l_\eps ,
\end{equation}
we will see multiple transitions through $\gamma_2,\gamma'_2,\gamma^{\prime\prime}_2,\gamma^{\prime\prime\prime}_2$ (transitions of type~$2$) but still, typically, no transitions through  $\gamma_3,\gamma'_3,\gamma^{\prime\prime}_3,\gamma^{\prime\prime\prime}_3$ (transitions of type $3$). Between those transitions of type $2$, there will be multiple transitions of type~$1$, and between those there will be multiple transitions of type~$0$.   If one views the diffusion as a process on $\R^2$,  then the entire infinite strip composed of 4-cell complexes separated by heteroclinic connections of type~$2$, is accessible for the diffusion for these times.
 
For times $t(\e)$ satisfying
\begin{equation}
\label{eq:times-transitions-type-3}
  t(\e) \ge  \e^{-(\frac{1}{\rho_1}+\frac{1}{\rho_2}+\frac{1}{\rho_2}-3)}   l_\eps ,
\end{equation}
transitions of type $3$ finally become typical, making all cells in the entire $\R^2$ accessible for the diffusion.  

In effect, we have the following hierarchy of clusters: singular cells, 2-cell complexes, 4-cell complexes, strips of cells, the entire plane. Each cluster is equipped with a range of time scales
on which the diffusion is typically confined to it.
One can deduce from Theorem~\ref{th:multi-saddle-escape} that such a picture, in fact, emerges for a broad class of planar heteroclinic networks under a stability assumption.
In our example, viewed as a diffusion on the torus, due to the symmetry of the model, one can actually claim that for times belonging to any of the scales described by any of the relations~\eqref{eq:times-transitions-type-0}, \eqref{eq:times-transitions-type-1}, \eqref{eq:times-transitions-type-2}, \eqref{eq:times-transitions-type-3}, the limit of the empirical distribution as $\e\to0$ is the same and given by~\eqref{eq:limit-emp-measure}. In particular, it also gives the limit of the invariant measure for the Markov semigroup
associated with SDE~\eqref{eq:basic-sde}. In general, though,  the limiting empirical distribution for each cluster (or timescale) of the hierarchy  can be computed by averaging the limiting distributions associated with the subordinate clusters.

The hierarchical structure that we are describing is reminiscent of the hierarchy of  cycles in the Freidlin--Wentzell theory of metastability. One key difference though is that in the metastability theory, transitions happen at times exponential in $\e^{-2}$ whereas in our picture the transition times are polynomial.

Metastable cycling was studied in \cite{FK-metastable:MR3652517} in the more abstract setting of a Markov chain on a graph where probabilities of various transitions 
depend on a small parameter $\e$ and 
are of different order of magnitude. It was shown under mild regularity assumptions that there is a sequence of time scales
\[
  1\equiv T_0(\e)\ll T_1(\e)\ll\dots\ll T_N(\e)\ll T_{N+1}(\e) \equiv \infty
\] 
and a family of measures $(\mu_i^j)$ called  metastable distributions such that
if~$t(\e)$  satisfies $T_i(\e)\ll t(\e) \ll T_{i+1}(\e)$ for some $i$, then the process equilibrates to one of~$\mu_i^j$ over time~$t(\e)$. Here $i$ enumerates timescales and $j$ enumerates clusters, i.e., elements  of the partition of the state space associated with a particular timescale.

Although our setting is described by 
the construction of~\cite{FK-metastable:MR3652517}  only approximately, we still can draw a connection. The timescales are given by $T_i(\e) = \e^{-\theta_i} l_\eps $ for $i \ge 1$ and an increasing sequence of exponents~$\theta_i$ determined by the network geometry and contraction/expansion rates near all saddles. 
Each saddle point produces four graph vertices, one per incident cell. Edges of the graph correspond to heteroclinic connections.
The diffusion spends a logarithmic in $\e^{-1}$ time near each saddle point, so one can say that for times  $t(\e)$ satisfying  $T_0\equiv 1\ll t(\e)\ll T_1(\e)= l_\eps =\e^{0}  l_\eps $ (i.e., $\theta_1=0$), the empirical measure equilibrates to the delta measure at one of the saddles. The next level clusters are composed of  vertices/saddles on the boundary of cells that are  mutually accessible in logarithmic times. At time scales between
$T_1(\e)= l_\eps $ and $T_2(\e)=\e^{-\theta_2} l_\eps $, the diffusion equilibrates to a mixture of the delta measures at those saddle points. For longer time scales, 
similarly to our cellular flow example,
more and more transitions become 
available, so
more massive clusters emerge 
and the metastable distributions at each level are mixtures of metastable distributions a level below. 
Imposing additional recurrence conditions, one can use the top level of the hierarchy to claim  convergence of stationary distributions of the diffusion to a limiting measure and describe the mixing properties.

\medskip

In general, diffusions near planar noisy heteroclinic networks can exhibit a variety of behaviors. 
In our relatively simple cellular flow example, the vector field and the heteroclinic network are $\Z^2$-periodic, so at the time scales~\eqref{eq:times-transitions-type-2} and ~\eqref{eq:times-transitions-type-3}, one can approximate
the diffusion with a  symmetric random walk on $\Z^1$ and $\Z^2$ respectively (just recording the $\Z^2$ coordinates of the cell occupied by the process), obtaining Gaussian approximations.
One can conjecture a  Central Limit Theorem for the regime~\eqref{eq:times-transitions-type-2}: there is a constant $c^2>0$ (the effective diffusivity) such that
\[
\frac{X_{\e,t(\e)}^1}{\displaystyle\left(\frac{t(\e)}{\e^{-(\frac{1}{\rho_1}+\frac{1}{\rho_2}-2)}   l_\eps }\right)^{1/2}}\indistr \Nc(0,c^2),
\]
and, moreover, for each $T>0$, the process 
\[
Z_{\eps,s}=\frac{X_{\e,s t(\e)}^1}{\displaystyle\left(\frac{t(\e)}{\e^{-(\frac{1}{\rho_1}+\frac{1}{\rho_2}-2)}   l_\eps }\right)^{1/2}}, \quad s\in[0,T],
\]
converges in distribution to a Brownian motion on $[0,T]$. In the regime~\eqref{eq:times-transitions-type-3}, 
a Gaussian scaling limit also should hold, albeit two components must scale differently:
defining the diagonal scaling matrix $D_{\e}$ by
\[
D_\e = \frac{l_\e}{t(\e)}\textrm{diag}\left(\e^{\frac{1}{\rho_1}+\frac{1}{\rho_2}-2}, \e^{\frac{1}{\rho_1}+\frac{1}{\rho_2}+\frac{1}{\rho_3}-3}\right),
\]  
we conjecture that $\sqrt{D_\e}X_{\e, t(\e)}$ converges in distribution to a centered Gaussian vector with independent components.

These statements can also be reformulated in terms of homogenization for a Fokker--Planck PDE with small diffusion but it seems that it is harder to obtain such a result by PDE methods.

We expect similar but perhaps more sophisticated scaling limits to hold for more complex heteroclinic networks. 

 An important feature of the example considered in this section is the stability of the network. Due to
 the relation $\rho_0\rho_1\rho_2\rho_3 >1$, one typically has to wait for the first departure
from a small neighborhood of the network for a very long time. In general,
although the process travels over large scales only when close to the network, one also has to take into account the time spent away from the network. This leads to a subordinated Brownian limit  in the Hamiltonian dynamics case where the network is given by a level set of the Hamiltonian, see~\cite{HKPG:MR3573288} 
and~\cite{HIKNPG:MR3773377}.

For general systems with multiple attractors, departures from the heteroclinic network towards other attractors may also be an intrinsic part of the picture thus giving rise to longer (Kramers--Freidlin--Wentzel) characteristic time scales.  This is related to the concept of excitability, see, e.g.,~\cite{AP:MR3566895}.

\section{Proofs in rectified coordinates}
\label{sec:rectified}
In Sections~\ref{sec:rectified}--\ref{section:density_est}, we give rigorous proofs of all lemmas that were studied heuristically in preceding sections.

Using assumption~\ref{setting:conjugacy} 
in a neighborhood of a saddle point,  
changing coordinates by the conjugacy $f$ introduced in~\ref{setting:conjugacy}, we can begin our program with studying the process $Y_{\e,t}=f(X_{\e,t})$, in a setting  that is simpler than general, where the domain is a small rectangle containing the saddle point at the origin and the drift is linear up to a $O(\eps^2)$ correction. We collect useful preliminary results on processes related to $Y_{\eps,t}$ in Sections~\ref{sec:basic_estimates} and~\ref{sec:est_N}. 
We will describe the setting in more detail in Section~\ref{sec:rect-setting}. In this simpler setting, in Section~\ref{sec:proofs_of_lemmas_rectified}, we will 
use local limit theorems from Sections~\ref{sec:Gaussian-approx} and~\ref{sec:LLT} to 
give rigorous
proofs of the lemmas that were only proved heuristically in Sections~\ref{sec:2saddles} and~\ref{sec:long-escape-chains}. We will prove them in full generality in Section~\ref{sec:original-coords}.

\subsection{Basic estimates}\label{sec:basic_estimates}
Let $\lambda >0>-\mu$ and let $F^1_1,F^1_2,F^2_1,F^2_2,G^1,G^2:\R^2\to\R$ be continuous and bounded.
We assume that the matrix $F(x)=(F^k_l(x))_{k,l=1}^2$ is uniformly elliptic (see condition~\ref{setting:general}).
For each $\e>0$,  we consider the SDE
\begin{align}\label{eq:def_SDE_near_a_saddle_point}
\begin{split}
        dY^1_t& = \lam Y^1_tdt + \eps F^1_l(Y_t)dW^l_t+\eps^2G^1(Y_t)dt,\\
    dY^2_t &= -\mu Y^2_tdt + \eps F^2_l(Y_t)dW^l_t+\eps^2G^2(Y_t)dt,
\end{split}
\end{align}
where~$(W_t,\mathcal{F}_t)$ is a standard  $2$-dimensional Wiener process, and the Einstein convention of summation over repeated indices is used. In Section~\ref{sec:rect-setting} we show that $Y_{\eps,t}=f(X_{\e,t})$ solves an equation of this form with coefficients 
$F$ and~$G$ that we compute.

Starting with this section, we will often suppress the dependence of various processes on $\e$, e.g., $Y_t=Y_{\e,t}$ in~\eqref{eq:def_SDE_near_a_saddle_point}. 

The joint distribution of $((Y_t)_{t\ge0}, (W_t)_{t\ge 0})$ given that $Y_0=y\in\R^2$
will be denoted by $\Pp^y$.  We also follow the convention of Section~\ref{sec:notation1}
denoting various probability measures by $\Pp$ if the joint distribution of r.v.'s involved is unambiguously defined. The expectation w.r.t.\ $\Pp^y$ is denoted by $\E^y$.

Let us define
\begin{gather}\label{eq:def_V_M_U_N}
\begin{aligned}
    V^1_t& = \int_0^te^{-\lam s} G^1(Y_s)ds, & V^2_t &= \int_0^te^{\mu s} G^2(Y_s)ds,\\
    M^1_t&=\int_0^te^{-\lam s}F^1_l(Y_s)dW^l_s,& M^2_t&=\int_0^te^{\mu s}F^2_l(Y_s)dW^l_s,\\
    U^i_t&=M^i_t+\eps V^i_t, \quad i=1,2,\\
S_t&=e^{-\mu t} M_t^2=e^{-\mu t}\int_0^te^{\mu s}F^2_l(Y_s)dW^l_s,\\
N_t&=N^2_t =e^{-\mu t}U^2_t= S_t+\eps e^{-\mu t}V^2_t.
\end{aligned}
\end{gather}
This notation and Duhamel's formula allow to write the solutions of~\eqref{eq:def_SDE_near_a_saddle_point}:
\begin{align}\label{eq:SDE_near_a_saddle_point_after_duhamel1}
        Y^1_t&=e^{\lam t}(Y^1_0+\eps U^1_t)= e^{\lam t}(Y^1_0+\eps M^1_t+\eps^2 V^1_t ),\\
        \label{eq:SDE_near_a_saddle_point_after_duhamel2}
        Y^2_t&=e^{-\mu t}Y^2_0+\eps N_t=e^{-\mu t}(Y^2_0+\eps U^2_t)=e^{-\mu t}Y^2_0+\eps S_t +\eps^2 e^{-\mu t} V^2_t.
\end{align}

In this section we prove various useful estimates on processes introduced in~\eqref{eq:def_V_M_U_N}.

Let us first state 
the following well-known exponential martingale inequality (see, e.g., Problem~12.10 in~\cite{Bass}):

\begin{lemma}\label{lem:exp-marting-ineq}
Let $M_t$ be a continuous local martingale satisfying $M_0=0$, with quadratic variation process $\qd{M}_t$. Then, for any $a,b>0$,
\[
\Pp\left\{\sup_{t\geq 0}|M_t|\geq a;\; \qd{M}_{\infty}\leq b\right\}\leq 2e^{-a^2/(2b)}.
\]
\end{lemma}

\begin{lemma}\label{lem:estimates-for-terms}  Processes introduced in~\eqref{eq:def_V_M_U_N} satisfy the following: 
\leavevmode
\begin{enumerate}
\item \label{item:tail-of_M_1}
There is a constant $C>0$ such for all $\e>0$, $r>0$,  $y\in\R^2$, the process~$M^1$, defined in~\eqref{eq:def_V_M_U_N}, satisfies
\[
\Pp^y\left\{\sup_{t\ge 0}|M_t^1|\ge r\right\} \le 2e^{-r^2/C}.
\]
In particular,  $\sup_{t\ge 0}|M^1_t|$ are tame under $\Pp^y$, uniformly  over $y\in\R^2$.
\item   \label{item:V-bounded} There is a constant $C>0$ such that for all $\e>0$ and all $y\in\R^2$,
\[
\sup_{t\ge 0}|V_t^1|<C,\quad \Pp^y\text{\rm -a.s.}
\]
\[
\sup_{t\ge T}|V_t^1-V_T^1|<Ce^{-\lambda T},\quad T\geq 0,\quad \Pp^y\text{\rm -a.s.}
\]
and 
\[\sup_{t\ge 0} e^{-\mu t} |V_t^2|<C, \quad \Pp^y\text{\rm -a.s.}\]
\item \label{item:tail-of_U_1}
There are constants $C,c>0$ such for all $\e\in(0,1)$, $r>c$,  $y\in\R^2$, the process~$U^1$, defined in~\eqref{eq:def_V_M_U_N}, satisfies
\[
\Pp^y\left\{\sup_{t\ge 0}|U_t^1|\ge r\right\} \le 2e^{-r^2/C}. 
\]
In particular,  under $\Pp^y$, $\sup_{t\ge 0}|U_t^1|$ is  tame  uniformly in $y\in\R^2$,    and, uniformly in $y\in\R^2$, $\e\in(0,1)$, has bounded moments of all orders.

\item \label{item:sup_of-S-over-large-intervals}
There is $C>0$ such that for all $\Delta>0$, all $\e>0$, all $r>0$, all $y\in\R^2$.
\[
\sup_{T\ge 0}\Pp^y\left\{\sup_{t\in[T,T+\Delta]}|S_t|\ge r\right\}\le 4([\Delta]+1) e^{-r^2/C}.
\]
\item  \label{item:sup_of-N-over-large-intervals}

There are $C,c>0$ such that for all $\Delta>0$, all $\e>0$, all $r>c$, all $y\in\R^2$,
\begin{equation}
\label{eq:tail-of-N}
\sup_{T\ge 0}\Pp^y\left\{\sup_{t\in[T,T+\Delta]}|N_t|\ge r\right\}\le 4([\Delta]+1) e^{-r^2/C}.
\end{equation}
In particular, for every $p\geq 1$ and every $\Delta\geq0$, there is a constant $C>0$ such that the following holds for every $y\in\R^2$, every $T\geq 0$ and every $\eps\in(0,1)$:
\begin{align}\label{eq:moment-N}
    \E^y\sup_{t\in[T,T+\Delta]}|N_t|^p\leq C.
\end{align}

\item \label{item:lemma:N^2_sigma}
For each $\beta>0$, there are $C,c>0$ such that
\begin{align*}
   \Probx{y}{\sup_{t\in[0,\zeta]}|N_t|> r} \leq C\left(\beta l_\eps +1\right)e^{-r^2/C} + \Probx{y}{\zeta \geq \beta  l_\eps  }
\end{align*}
holds for every stopping time $\zeta$, every $\eps\in(0,1)$, every $r>c$, and every $y\in\R^2$.

\item
\label{item:localization-of-times} 
For any $\Delta>0$,
there is $C>0$ such that if 
deterministic times $(t_\e)_{\e>0}$, stopping times $(\tau_\e)_{\e>0}$, events $(B_\e)_{\e>0}$, and parameter $\e_0>0$ satisfy
\begin{equation}
\label{eq:range-of-stopping-time}
\Pp^y(B_\e\cap \{\tau_\e\notin [t_\e,t_\e+\Delta]\})= 0,\quad \e\in(0,\e_0),\ y\in I_\e,
\end{equation}
then
the following estimate holds:
\begin{equation*}
\Pp^y(B_\e\cap\{|S_{\tau_\e}|>r\})\le C e^{-r^2/C},\quad  \e\in(0,\e_0),\ r>0,\ y\in I_\e.
\end{equation*}
\end{enumerate}
\end{lemma}

\bpf Part~\ref{item:tail-of_M_1} is directly implied by the exponential martingale inequality of Lemma~\ref{lem:exp-marting-ineq} and the boundedness of $F$.
Part~\ref{item:V-bounded} follows from the boundedness of~$G$. 
Part~\ref{item:tail-of_U_1} follows from parts~\ref{item:tail-of_M_1} and~\ref{item:V-bounded}.
To prove part~\ref{item:sup_of-S-over-large-intervals},
we write
\begin{align*}
\sup_{T\ge 0}\Pp^y\left\{\sup_{t\in[T,T+\Delta]}|S_t|\ge r\right\}
&\le \sup_{T\ge 0}\sum_{k\in\N\cup\{0\}: k \le \Delta }\Pp^y\left\{\sup_{t\in[T+k,T+k+1]}|S_t|\ge r\right\}
\\
&\le ([\Delta]+1)\sup_{u\ge 0} \Pp^y\left\{\sup_{t\in[u,u+1]}|S_t|\ge r\right\}.
\end{align*} 
\begin{multline*}
\Pp^y\left\{\sup_{t\in[u,u+1]}|S_t|\ge r\right\}\le \Pp^y\left\{e^{-\mu u} |M_u^2|\ge r/2 \right\} \\
+\Pp^y\left\{e^{-\mu u} \sup_{t\in[u,u+1]}\left|\int_u^te^{\mu s} F^2(Y_s)dW_s\right|\ge r/2\right\},
\end{multline*}
and each term on the right-hand side may be estimated by $2e^{-r^2/C}$ for some $C$ and all $u,r,\e$ due to the exponential martingale inequality and
boundedness of $F$, so our claim follows.

Part~\ref{item:sup_of-N-over-large-intervals} follows from parts~\ref{item:sup_of-S-over-large-intervals}, \ref{item:V-bounded} (we integrate by parts with respect to $r$ to obtain~\eqref{eq:moment-N}.)

To prove part~\ref{item:lemma:N^2_sigma}, we apply \eqref{eq:tail-of-N} to each term in the sum on the r.-h.s of
\begin{align*}
    &\Probx{y}{\sup_{t\in[0,\zeta]}|N_t|> r}
    \leq \Probx{y}{\zeta \geq \beta l_\eps }+\sum_{n=0}^{\lfloor \beta l_\eps \rfloor}\Probx{y}{\sup_{t\in[n,n+1]}|N_t|> r}.
\end{align*}

To prove part~\ref{item:localization-of-times} , we 
use~\eqref{eq:range-of-stopping-time} and write
\begin{align*} 
\Pp^y(B_\e\cap\{|S_{\tau_\eps}|>r\})=&\Pp^y\left(B_\e\cap\{|S_{\tau_\e}|>r \}\cap  \{\tau_\e\in [t_\e,t_\e+\Delta]\}\right)
\\ \le& \Pp^y\left\{\sup_{t\in[t_\e,t_\e+\Delta]} |S_{t}|>r\right\},
\end{align*}
so our claim follows from part~\ref{item:sup_of-S-over-large-intervals}.
\epf

Let us give a useful identity for the (non-Markov) process $S_t$ defined in~\eqref{eq:def_V_M_U_N}. It can be viewed as a  generalization of the Ornstein--Uhlenbeck semigroup property.

Due to the strong uniqueness of solutions of SDEs,  for $\Fc_0$-measurable $Y_0$, we can write 
\begin{equation}
\label{eq:OU-solution}
S_t=S_t(Y_0,dW_\cdot), 
\end{equation}
where by $dW_\cdot$ we mean
the collection of increments $(W_t-W_0)_{t\ge 0}$.
Let $\Theta^t$ denote the time shift of the Wiener path:
\[
\Theta^t W_s=W_{t+s}-W_t,\quad s\ge 0,
\]
so that for any stopping time $\tau$, the random  shift $\Theta^{\tau}W$ is also a Wiener process.
\begin{lemma}
If stopping times $\tau,\tau'$ satisfy $\tau'\ge \tau\ge 0$, then, with probability $1$,
\begin{align}
\label{eq:OU-property-stopping}
S_{\tau'}(Y_0,dW_\cdot)=e^{-\mu (\tau'-\tau)}S_{\tau}(Y_0,dW_\cdot)+S_{\tau'-\tau}(Y_{\tau},d(\Theta^{\tau}W)_\cdot).
\end{align}
\end{lemma}
\bpf
For deterministic times $\tau$ and $\tau'$,~\eqref{eq:OU-property-stopping} is a result of a direct computation which is a simple version of the reasoning below. For arbitrary stopping times, we need to be more careful. Let us introduce two auxiliary SDE's,
\begin{align}
\label{eq:SDE-for-S}
dS_t&=-\mu S_t dt +F^2(Y_t)dW_t,\\
\label{eq:SDE-for-W}
d\widetilde W_t&=dW_t.
\end{align}

The system of autonomous SDEs  
\eqref{eq:def_SDE_near_a_saddle_point},\eqref{eq:SDE-for-S},\eqref{eq:SDE-for-W} 
generates unique strong solutions, a strong Markov semigroup, and  an adapted flow of solution maps
\[
(Y,S,\widetilde W)_t((Y,S,\widetilde W)_0, dW.).
\]
By Duhamel's principle, we have that for any random initial  conditions $(Y,S,W)_0$,  with probability~$1$,
\begin{equation}
\label{eq:duhamel-for-S}
S_t((Y,S, \widetilde W)_0, dW_\cdot)=e^{-\mu t}S_0+ e^{-\mu t}\int_0^t e^{\mu s} F^2\left(Y_s\left((Y,S,\widetilde W )_0,dW_\cdot\right)\right)dW_s, \quad t\ge0.
\end{equation}
Comparing this to \eqref{eq:OU-solution}, we see that, with probability~$1$,
\begin{equation}
\label{eq:two-def-of-S}
S_t((Y_0,0, 0), dW_\cdot) =S_t(Y_0,dW_\cdot),\quad t\ge 0.
\end{equation}

Combining~\eqref{eq:duhamel-for-S} with the strong Markov property, we obtain that if $\tau$ is a stopping time, then with probability 1, we have, for all $t\ge0$,
\begin{align*}
&S_{\tau+t}\left((Y,S,\widetilde W)_0,\, dW_\cdot\right) 
=S_t\left((Y,S,\widetilde W)_{\tau}\left((Y,S,\widetilde W)_0,\,dW_\cdot\right) ,\,d(\Theta^\tau W)_\cdot\right)
\\
&\qquad = e^{-\mu t}S_\tau\left((Y,S,\widetilde W)_0,\,dW_\cdot\right)
\\
&\qquad + e^{-\mu t}\int_0^t e^{\mu s} F\left(Y_s\left((Y,S,\widetilde W)_{\tau}\left((Y,S,\widetilde W)_{0},  dW_\cdot\right),\,d(\Theta^\tau W)_\cdot\right)\right)d\Theta^\tau W_s.
\end{align*}
Now, plugging in the values $t=\tau'-\tau$, $S_0=\widetilde W_0=0$ and using~\eqref{eq:two-def-of-S} to interpret both sides of this identity, we obtain~\eqref{eq:OU-property-stopping}. \epf

\subsection{Estimating the stopped process \texorpdfstring{$N$}{N} }\label{sec:est_N}

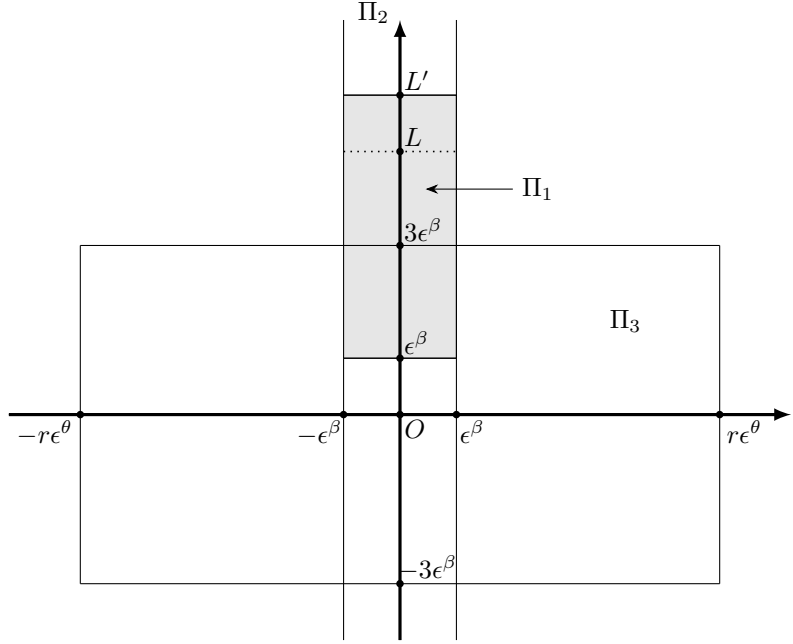
\begin{figure}[ht]
\centering
\begin{tikzpicture}

    \node (0) at (-5.2, 0) {};
    \node (1) at (5.2, 0) {};
    \node (2) at (0, 0) {};
    \node (3) at (-4.25, 0) {};
    \node (4) at (4.25, 0) {};
    \node (5) at (-4.25, 2.25) {};
    \node (6) at (4.25, 2.25) {};
    \node (7) at (-4.25, -2.25) {};
    \node (8) at (4.25, -2.25) {};
    \node (9) at (0, 5.25) {};
    \node (10) at (0, -3) {};
    \node (11) at (-0.75, 0.75) {};
    \node (12) at (0.75, 0.75) {};
    \node (13) at (-0.75, 5.25) {};
    \node (14) at (0.75, 5.25) {};
    \node (15) at (-0.75, -3) {};
    \node (16) at (0.75, -3) {};
    \node (17) at (-0.75, 4.25) {};
    \node (18) at (0.75, 4.25) {};
    \node (19) at (-0.75, 3.5) {};
    \node (20) at (0.75, 3.5) {};
    \node (21) at (0, 0.75) {};
    \node (22) at (0, 3.5) {};
    \node (23) at (0, 4.25) {};
    \node (24) at (0, 2.25) {};
    \node (25) at (0.75, 0) {};
    \node (26) at (-0.75, 0) {};
    \node (27) at (0, -2.25) {};

    \draw[fill=gray!20]    (-0.75, 0.75) -- (0.75, 0.75) -- (0.75, 4.25) -- (-0.75, 4.25)  ;

    \node [below] at (4.58, 0.05) {$r\epsilon^\theta$};
    \node [below] at (-4.72, 0.05) {$-r\epsilon^\theta$};
    \node [below] at (0.96, 0.05) {$\epsilon^\beta$};
    \node [below] at (-1.07, 0.05) {$-\epsilon^\beta$};
    \node [right] at (-0.07, 0.96) {$\epsilon^\beta$};
    \node [right] at (-0.07, 2.46) {$3\epsilon^\beta$};
    \node [right] at (-0.15, -2.04) {$-3\epsilon^\beta$};
    \node [right] at (-0.07, 3.7) {$L$};
    \node [right] at (-0.07, 4.45) {$L'$};
    \node [below] at (0.2,0.05) {$O$};
    \node at (3, 1.25) {$\Pi_3$};
    \node at (-0.35, 5.35) {$\Pi_2$};
    \node [right] at (1.5, 3) {$\Pi_1$};

    \draw [-Stealth] (1.5, 3) to (0.34, 3);

\foreach \n in {3, 2, 4, 21, 22, 23, 24, 25, 26, 27}
        \node at (\n)[circle,fill,inner sep=1pt]{};

    \draw (0.center) to (3.center);
    \draw (3.center) to (2.center);
    \draw (2.center) to (4.center);
    \draw [very thick, -latex] (0.center) to (1.center);
    \draw (10.center) to (2.center);
    \draw [very thick, -latex] (10.center) to (9.center);
    \draw (3.center) to (5.center);
    \draw (5.center) to (6.center);
    \draw (6.center) to (4.center);
    \draw (4.center) to (8.center);
    \draw (8.center) to (7.center);
    \draw (7.center) to (3.center);
    \draw (11.center) to (13.center);
    \draw (12.center) to (14.center);
    \draw (11.center) to (12.center);
    \draw (15.center) to (11.center);
    \draw (16.center) to (12.center);
    \draw (17.center) to (18.center);
    \draw [semithick, dotted] (19.center) to (20.center);

\end{tikzpicture}

\caption{Narrow channels around the invariant manifolds.}
\label{fig:cross-channels}
\end{figure}

Let $L,L'>0$ satisfy $L'>L$.
Recall the definition of $\rho$ in \eqref{eq:rho}.
Throughout this section we assume that constants $\alpha$, $\rho$, $\theta$ satisfy
\begin{align}
\label{eq:condition_on_alpha_rho_gamma}
\alpha\in(0,1],\qquad 0\leq \theta& <  \frac{1}{2}\wedge \frac{\alpha}  {1+\rho^{-1}}.
\end{align}
We also fix $r>0$ and define
\begin{align}\label{eq:tau_gamma}
    \tau &= \tau_{r,\theta,\eps} = \inf\left\{t>0:Y_t\not\in [-r\eps^\theta, r\eps^\theta]\times [-L', L']\right\},\\
\label{eq:exit-from-strip}
    \zeta&= \zeta_{r,\theta,\eps} = \inf\left\{t>0:Y_t\not\in [-r\eps^\theta, r\eps^\theta]\times\R \right\}.
\end{align}

 We will be later interested in a specific case  of the exit time for $Y$ from
\begin{equation}
\label{eq:main_rectangle}
\fR=[-R,R]\times[-L',L']
\end{equation}
for some $R>0$.   This time is denoted by $\tau_\fR = \tau_{\fR,\eps}$ and satisfies
\begin{align}\label{eq:tau_fR}
    \tau_\fR = \tau_{R,0,\eps}.
\end{align}

For $\theta=0$, $r=R$, we have $\tau=\tau_\fR$.

Let us fix an arbitrary $\beta$  satisfying
\begin{equation}
\label{eq:condition_on_beta1}
\theta<\beta < \frac{1}{2}\wedge \frac{\alpha}  {1+\rho^{-1}}.
\end{equation}

For small $\e$, we introduce domains 
\begin{align*}
\Pi_{1}&=\Pi_{1,\e}=[-\e^\beta,\e^\beta]\times[\e^\beta,L'], \\
\Pi_{2}&=\Pi_{2,\e}=[-\e^\beta,\e^\beta]\times\R, \\
\Pi_{3}&=\Pi_{3,\e}=[-r\eps^\theta,r\eps^\theta]\times[-3\e^\beta,3\e^\beta], \\
\end{align*}
shown in Figure~\ref{fig:cross-channels}.
Let us also define $D_\e= \Pi_{1,\e}\cup \Pi_{3,\e}$ and 
\begin{align}\label{eq:bar_tau_r,theta,eps}
    \bar \tau=\bar \tau_{r,\theta,\e}=\inf\{t\ge0:\  Y_t\notin D_\e\}.
\end{align}
and observe that $\bar\tau \leq\tau\leq\zeta$. Defining (note that $\alpha>\beta$)
\begin{equation*}
I_\e=[-\e^\alpha,\e^\alpha]\times\{L\},
\end{equation*} 
we can state the main result of this section:

\begin{lemma}\label{lem:second-coord-at-exit}
Let $\tau,\zeta,\bar\tau$ be given in \eqref{eq:tau_gamma}, \eqref{eq:exit-from-strip}, \eqref{eq:bar_tau_r,theta,eps}, respectively.
If  \eqref{eq:condition_on_alpha_rho_gamma} and \eqref{eq:condition_on_beta1} hold,  then there are constants $C,\e_0>0$ and a family of events $(B_\e)_{\e>0}$ such that
the following holds:

\begin{enumerate}
\item \label{item:1_lem:second-coord-at-exit}
\begin{equation*}
\sup_{y\in I_\e}\Pp^y(B_\e^c)=o_e(1);
\end{equation*}
on $B_\e$, we have
\[
Y_{\bar\tau}\in \{-r\eps^\theta,r\eps^\theta\}\times[-3\e^\beta,3\e^\beta]
\]
(i.e., the exit happens through the lateral sides of $\Pi_3$),
and 
\begin{equation}
\label{eq:lem:second-coord-at-exit-for-N}
\sup_{y\in I_\e}\Pp^y\left(B_\e\cap\{|N_{\bar\tau}|>z\}\right)\le C e^{-z^2/C}, \quad  \e\in(0,\e_0),\ z>0.
\end{equation}
In particular, $N_{\bar \tau}$ is  tame under $\Pp^y$, uniformly in $y\in I_\e$.
\item \label{item:bartau-tame}$\bar\tau$ is tame under $\Pp^y$, uniformly in $y\in I_\e$.
\item  \label{item:bartau=tau}
On $B_\e$, $\tau =\bar\tau=\zeta$, so $\tau$, $N_\tau$, $\zeta$, $N_\zeta$ are also tame under $\Pp^y$, uniformly in $y\in I_\e$. 
\end{enumerate} 
\end{lemma}

Most of the conclusions of this lemma do not depend on a particular choice of $\beta$ satisfying \eqref{eq:condition_on_beta1}. Also, if  \eqref{eq:condition_on_alpha_rho_gamma} holds, then one can make 
$\alpha$ smaller still retaining this condition. Thus, 
recalling the definition of $K_\vk(\eps)$ in \eqref{eq:def_K(eps)},
we obtain the following immediate consequence of Lemma~\ref{lem:second-coord-at-exit}:

\begin{lemma}\label{lem:N_tau_tail}
Let \eqref{eq:condition_on_alpha_rho_gamma} hold, $\vk>0$, and $\tau,\zeta$ be given in \eqref{eq:tau_gamma}, \eqref{eq:exit-from-strip}, respectively. Then, under $\Pp^y$,
$\tau=\zeta$  w.h.p.\ uniformly in $y\in (\eps^\alpha K_\vk(\eps))\times \{L\}$  and
\begin{align*}
    \sup_{y\in (\eps^\alpha K_\vk(\eps))\times \{L\}}\Pp^y\{|N_{\tau}|>l_\eps\}=o_e(1),
\end{align*}
\end{lemma}

\medskip

To prove Lemma~\ref{lem:second-coord-at-exit},
we need an auxiliary result. 
Our goal is to split the evolution until $\bar\tau$ into three parts.
Let us define
\begin{align*}
\notag
\tau^1=&\inf\{t\ge 0:\ Y\notin\Pi_{1}\},&
\\
\tilde \tau^2=&\inf\{t\ge \tau^1 :\ Y\notin \Pi_2\},& \tau^2=&\inf\{t\ge 0:\ Y\notin \Pi_2\},
\\
\notag
\tilde\tau^3=&\inf\{t\ge \tau^2:\ Y\notin \Pi_3\},& \tau^3=&\inf\{t\ge 0:\ Y\notin \Pi_3\},
\end{align*}
(all these times are a.s.-finite due to the ellipticity of the noise) and
\begin{align*}
I^1_\eps&= [-\e^\beta,\e^\beta]\times\{\e^\beta\},\\
I^2_\eps&=I^2_{\eps,+}\cup I^2_{\eps,-}=\left(\{\e^\beta\}\times [-2\e^\beta,2\e^\beta]\right) \cup \left(\{-\e^\beta\}\times [-2\e^\beta,2\e^\beta]\right).
\end{align*}

\smallskip

\begin{lemma} \label{lem:controlling-OU}
Under the setting in Lemma~\ref{lem:N_tau_tail}, the following hold.
\begin{enumerate}

\item \label{part:along-stable}
There are constants $C,\e_0>0$ 
and a family of events $(B_\e)_{\e>0}$  such that 
\[
\sup_{y\in I_\e}  \Pp^y(B_\e^c)=o_e(1),
\]
\[
\sup_{y\in I_\e} \Pp(B_\e\cap\{|S_{\tau^1}|>z\})\le C e^{-z^2/C},\quad \e<\e_0,\ z>0,
\]
\[
\Pp^y(Y_{\tau^1}\in I^1_\e |\ B_\e )=1,\quad y\in I_\e.
\]
Also, the stopping time $\tau^1$ is tame under $\Pp^y$, uniformly in $y\in I_\e$.

\item \label{part:small-square-neighb} There are constants $C,\e_0>0$  
and a family of events $(B_\e)_{\e>0}$
  such that 
 \[
\sup_{y\in I^1_\e}  \Pp(B_\e^c)=o_e(1),
 \]
\[
\sup_{y\in I^1_\e}\Pp^y\left(B_\e\cap\{e^{-\mu\tau_3}|S_{\tau^2}|>z\}\right)\le C e^{-z^2/C},\quad  \e<\e_0,\ z>0,
\]
\[
\Pp^y(Y_{\tau^2}\in I^2_\e|\ B_\e)=1,\quad y\in I^1_\e.
\]
Also, the stopping time $\tau^2$ is tame under $\Pp^y$, uniformly in $y\in I^1_\e$.
\item  \label{part:along-unstable}There are constants $C,\e_0>0$  
and a family of events $(B_\e)_{\e>0}$
such that 
\[
\sup_{y\in I^2_\e}  \Pp^y(B_\e^c)=o_e(1),
 \]
\[
\sup_{y\in I^2_\e}\Pp^y\left(B_\e\cap\{|S_{\tau^3}|>z\}\right)\le C e^{-z^2/C},\quad  \e<\e_0,\ z>0,
\]
\[
\Pp^y\left(Y_{\tau^3}\in \{r\e^\theta,-r\e^\theta\}\times [-3\e^\beta,3\e^\beta] \ |\ B_\e \right)=1,\quad y\in I^2_\e.
\]
Also, the stopping time $\tau^3$ is tame under $\Pp^y$, uniformly in $y\in I^2_\e$.
\end{enumerate}
\end{lemma}

Let us derive Lemma~\ref{lem:second-coord-at-exit} from Lemma~\ref{lem:controlling-OU} first  and then prove the latter.

\bpf[Proof of Lemma~\ref{lem:second-coord-at-exit}]
Decomposing the evolution into three stages corresponding to times $\tau^1, \tilde\tau^2, \tilde\tau^3$ and described in Lemma~\ref{lem:controlling-OU},
and using the strong Markov property, we obtain the existence of a set $B_\e$ with properties described in part~\ref{item:1_lem:second-coord-at-exit}, except \eqref{eq:lem:second-coord-at-exit-for-N}, which we still need to check. Also, decomposing $\bar\tau$ into a sum of three exit times and combining three parts of Lemma~\ref{lem:controlling-OU}, we  immediately obtain part~\ref{item:bartau-tame}

To prove \eqref{eq:lem:second-coord-at-exit-for-N}, it suffices (due to Lemma~\ref{lem:estimates-for-terms}~\eqref{item:V-bounded}) 
to check
\begin{equation}
\label{eq:lem:second-coord-at-exit-for-S}
\sup_{y\in I_\e}\Pp(B_\e\cap\{|S_{\bar \tau}|>z\})\le C e^{-z^2/C}, \quad  \e\in(0,\e_0),\ z>0,
\end{equation}
for some $C,\e_0>0$.

Applying~\eqref{eq:OU-property-stopping} twice, we obtain that, with probability 1,
\begin{multline*}
S_{\tilde\tau_3}(Y_0,W)=e^{-\mu (\tilde\tau^3-\tau^1)}S_{\tau^1}(Y_0,W)\\+e^{-\mu (\tilde\tau^3-\tilde\tau^2)}S_{\tilde\tau^2-\tau^1}(Y_{\tau^1},\Theta^{\tau^1}W)+S_{\tilde\tau^3-\tilde\tau^2}(Y_{\tilde \tau_2},\Theta^{\tilde\tau^2}W).
\notag
\end{multline*}
The estimate~\eqref{eq:lem:second-coord-at-exit-for-S} follows directly  from this representation  and Lemma~\ref{lem:controlling-OU}.
This completes the proof of part~\ref{item:1_lem:second-coord-at-exit}.

To prove Part~\ref{item:bartau=tau}, we recall that  $\bar \tau,\tau,\zeta$ are defined as the times of exit
 from sets
$D_\eps$, $[-r\eps^\theta, r\eps^\theta]\times [-L', L']$, and $[-r\eps^\theta, r\eps^\theta]\times \R$, respectively.  
On $B_\e$, the exit from~$D_\eps$ happens through the lateral sides of $\Pi_3$. Since they
belong to the boundaries of all these sets, we conclude that
$\tau=\bar\tau=\zeta$ holds on~$B_\eps$. Combining this with 
part~\ref{item:bartau-tame},we obtain
the tameness claim of part~\ref{item:bartau=tau}.
\epf

\medskip
Let us now prove  Lemma~\ref{lem:controlling-OU}.
We first prove its part~\ref{part:along-stable}, then  part~\ref{part:along-unstable}, and then~part~\ref{part:small-square-neighb}.

\medskip
\bpf[Proof of part~\ref{part:along-stable}]
We will assume 
\begin{equation}
\label{eq:ini-cond-assumption}
y\in I_\eps,
\end{equation}
throughout the proof.
In addition to $Y$, let us consider the deterministic process $(\overline Y_t)$ given by
\begin{align}
\label{eq:model-process}
\begin{split}
    \overline Y^1_t&=e^{\lambda t} Y^1_0,
\\
\overline Y^2_t&=e^{-\mu t} Y^2_0=e^{-\mu t} L,
\end{split}
\end{align}

We see that $\overline Y^2_t$ decreases in $t$.  For
\begin{equation}
\label{eq:def_of_exit_from_Pi_1}
t_\e=-\frac{1}{\mu}\log\frac{\e^\beta}{2L},
\end{equation}
we have 
\begin{equation*}
\overline Y^2_{t_\e}=\frac{1}{2}\e^{\beta}.
\end{equation*}

Due to \eqref{eq:condition_on_beta1}, $-\frac{\beta}{\rho}+\alpha > \beta$. We can use this
and~\eqref{eq:ini-cond-assumption} to see that  for some $\e_0$ (which does not depend on $Y_0=y$), all $\e<\e_0$, and all $t\in[0,t_\e]$,
\begin{equation*}\left|\overline Y^1_{t}\right|\le\ e^{\lambda t_\e}\e^\alpha<  \frac{1}{2}\e^{\beta}.
\end{equation*}
So 
$t_\e$ is the exit time from
$\widetilde \Pi_1=[-\frac{\e^\beta}{2}, \frac{\e^\beta}{2}]\times [\frac{\eps^\beta}{2}, L']$:
\begin{equation*}
t_\e=\inf\left\{t\ge 0:\ \overline Y_t\notin \widetilde \Pi_1\right\}.
\end{equation*}

Let us use $\overline Y$ to prove that $Y$ exits $\Pi_1$ through the bottom side w.h.p.

Parts~\ref{item:tail-of_U_1} and \ref{item:sup_of-N-over-large-intervals} of Lemma~\ref{lem:estimates-for-terms} and  assumption~\eqref{eq:condition_on_beta1}  imply that there are constants $h_1,h_2,h_3>0$ such that
for $\e<\e_0$,
and for all $y$ satisfying~\eqref{eq:ini-cond-assumption},
we have 
\begin{align*}
\Pp^{y}\left\{\sup_{t\in[0,t_\e]}(e^{\lam t}\eps \left|U^1_t\right|) >\frac{\e^\beta}{2}\right\}< h_1 e^{-h_3\e^{-h_2}}
\end{align*}
and
\begin{align*}
\Pp^{y}\left\{\sup_{t\in[0,t_\e]}\eps |N_t| >\frac{\e^\beta}{2} \right\}< h_1 e^{-h_3\e^{-h_2}}.
\end{align*}
This allows to define an event $B_\e$ with $\Pp^y(B_\e)>1-2 h_1 e^{-h_3\e^{-h_2}}$ such that on~$B_\e$,
\begin{equation*}
\sup_{t\in[0,t_\e]} |Y_t-\overline Y_t|_\infty < \frac{\e^\beta}{2},
\end{equation*}
where we used \eqref{eq:SDE_near_a_saddle_point_after_duhamel1} and \eqref{eq:SDE_near_a_saddle_point_after_duhamel2}.
In particular, on~$B_\e$, the exit from $\Pi_1$ happens through its bottom, before time $t_\e$.

Similarly to~\eqref{eq:def_of_exit_from_Pi_1}, we can define
\begin{equation*}
t'_\e=-\frac{1}{\mu}\log\frac{3\e^\beta}{2L}= t_\e - \frac{1}{\mu}\log 3,
\end{equation*}
interpret it as the exit time from a smaller rectangle $[-\frac{\e^\beta}{2},\frac{\e^\beta}{2}]\times [\frac{3\e^\beta}{2},L']$, through its bottom $[-\frac{\e^\beta}{2},\frac{\e^\beta}{2}]\times\{ \frac{3\e^\beta}{2}\}$ and derive that  $\tau^1\ge t'_\e$ on $B_\e$.
Therefore, on $B_\e$, we have
\begin{equation}
\label{eq:tau1-sandwiched}
t'_\e \le \tau^1\le  t_\e.
\end{equation}
so we can apply  Lemma~\ref{lem:estimates-for-terms}~\eqref{item:localization-of-times} with $\Delta=\frac{1}{\mu}\log 3$ to derive the first claim of part~\ref{part:along-stable}. The tameness of $\tau^1$ follows from the upper bound in~\eqref{eq:tau1-sandwiched}.\epf

\bigskip
\bpf[Proof of part~\ref{part:along-unstable}]
 We only consider initial conditions given by
 \begin{equation}
\label{eq:ini-condition-positive-side-square}
y\in  I^2_{\eps,+}.
\end{equation}
 The case of $y\in I^2_{\eps,-}$ is similar.
We recall the  process $\overline Y_t$ defined in~\eqref{eq:model-process}.
We see that $\overline Y^1_t$ increases in $t$ and for the time
\begin{equation*}
t_\e=\frac{1}{\lambda}\log\frac{(r+1)\eps^\theta}{\e^\beta},
\end{equation*}
we have
\[
\overline Y^1_{t_\e}= (r+1)\eps^\theta,
\]
and for some $\eps_0$, all $\e<\eps_0$, and all $t\in[0,t_\e]$,
\[
|\overline Y^2_{t_\e}|\le 2\e^\beta,
\]
so
$t_\e$ is the exit time from
$\widetilde \Pi_3=[-(r+1)\eps^\theta,(r+1)\eps^\theta]\times [-3\eps^\beta, 3\eps^\beta]$:
\begin{equation*}
t_\e=\inf\{t\ge 0:\ \overline Y\notin \widetilde \Pi_3\}.
\end{equation*}
We can use $\beta<1/2$ (guaranteed by~\eqref{eq:condition_on_beta1}) and parts~\ref{item:tail-of_U_1} and \ref{item:sup_of-N-over-large-intervals} of Lemma~\ref{lem:estimates-for-terms} to find constants $h_1,h_2,h_3>0$ such that 
for all $y$ satisfying~\eqref{eq:ini-condition-positive-side-square},
\begin{align*}
\Pp^{y}\left\{\sup_{t\in[0,t_\e]}e^{\lam t}\eps \left|U^1_t\right| >\e^\beta\ \right\}< h_1 e^{-h_3\e^{-h_2}}
\end{align*}
and 
\begin{align*}
\Pp^{y}\left\{\sup_{t\in[0,t_\e]}\eps |N_t| >\e^\beta\ \right\}< h_1 e^{-h_3\e^{-h_2}}.
\end{align*}
This allows to define an event $B_\e$ with $\Pp^y(B_\e)>1-2 h_1 e^{-h_3\e^{-h_2}}$ such that on~$B_\e$,
\begin{equation}
\label{eq:compare-to-deterministic-orbit}
\sup_{t\in[0,t_\e]} |Y_t-\overline Y_t|_\infty < \e^\beta.
\end{equation}
In particular, due to $\theta<\beta$ (see \eqref{eq:condition_on_beta1}), on~$B_\e$, the exit  from $\Pi_3$ happens through the right lateral side, before time~$t_\e$. One can also define 
\begin{equation*}
t'_\e=\frac{1}{\lambda}\log\frac{r\eps^\theta/2}{\e^\beta}
\end{equation*}
and see that, due to ~\eqref{eq:compare-to-deterministic-orbit}, for sufficiently small $\e$, $t'_\e \le \tau^3\le  t_\e$
and $t_\e-t'_\e=\frac{1}{\lambda}\log\frac{r+1}{r/2}$. Thus  $\tau^3$ is tame, and we can apply  Lemma~\ref{lem:estimates-for-terms}~\eqref{item:localization-of-times} to derive the remaining claim of  part~\ref{part:along-unstable}.\epf

\bigskip

To  prove part~\ref{part:small-square-neighb}, we need  several auxiliary results (Lemmas~\ref{lem:Markov-chain-exit1}, \ref{lem:exit-from_Pi-cannot-take-too-long}, \ref{lem:exit-from_Pi_3_cannot-be-short} below).
We define
\begin{align*}
\tau'&=\inf\{t\ge0: \left|Y^2_t\right|\ge 2\e^\beta\},\\
\hat\tau&=\tau^2\wedge\tau',\\
t_\e&=\frac{1-\beta}{\lambda} l_\eps,
\end{align*}
so that $\hat\tau$ is the exit time from the rectangle $[-\e^{\beta},\e^{\beta}]\times[-2\e^{\beta}, 2\e^\beta]$. We also define  $\widetilde \Pi_2=\widetilde \Pi_{2,\e}=[-\e^\beta,\e^\beta]^2$.

\begin{lemma}
\label{lem:Markov-chain-exit1}
There is $c\in(0,1)$ such that for sufficiently small $\e,$
\begin{equation}
\label{eq:prob-of-exit-lower}
\sup_{y\in\widetilde \Pi_2}\Pp^y\{\tau^2>t_\e\}<c.
 \end{equation}
We also have
\begin{equation}
\label{eq:prob-of-staying-in-Pi}
\sup_{y\in\widetilde \Pi_2}\Pp^y\left\{\tau^2>t_\e,\ |Y_{t_\e}|\notin\widetilde \Pi_2 \right\}=o_e(1),
 \end{equation}
\begin{equation}
\label{eq:prob-of-exiting-too-far}
\sup_{y\in\widetilde \Pi_2}\Pp^y\left\{\tau^2\le t_\e\ |Y^2_{\tau^2}|>2\e^\beta\right\}=o_e(1).
 \end{equation}
\end{lemma}
\bpf Throughout this proof, $\sup_y$ means $\sup_{y\in\widetilde \Pi_2}$. Part \ref{item:sup_of-N-over-large-intervals} of Lemma~\ref{lem:estimates-for-terms} implies
\begin{equation*}
\sup_{y}\Pp^y\left\{\sup_{t\in[0,t_\e]} |N_t|> l_\eps \right\}=o_e(1).
\end{equation*}
Therefore, due to \eqref{eq:SDE_near_a_saddle_point_after_duhamel2},
\begin{equation}
\label{eq:tau_second}
\sup_{y}\Pp^y\left\{\tau'\le t_\e\right\} =\sup_{y}\Pp^y\left\{\sup_{t\in[0,t_\e]}\left|Y^2_t\right|\ge 2\e^\beta\right\}=o_e(1)
\end{equation}
and 
\[
\sup_{y}\Pp^y\left\{|Y^2_{t_\e}|\ge \e^\beta\right\}=o_e(1).
\]
Estimates~\eqref{eq:prob-of-staying-in-Pi} and~\eqref{eq:prob-of-exiting-too-far} follow from these bounds.

Let us define 
\begin{align*}
 \overline M^1_t&=\int_0^te^{-\lam s}F^1_l(0)dW^l_s,\\
 \widetilde M^1_t&= \int_0^t e^{-\lambda s}(F^1_l(Y_s)-F^1_l(0))d W_s^l.
\end{align*}
Due to~\eqref{eq:tau_second},
\begin{align*}
\sup_{y}\Pp^y\{t_\e\le \tau^2 \}=& \sup_{y}\Pp^y\left\{t_\e\le \tau^2 ,\ e^{\lam t_\e}|y^1+\eps U^1_{t_\e}|\leq\e^\beta\right\}\\
= & \sup_{y} \Pp^y\left\{t_\e\le \hat\tau ,\ e^{\lam t_\e}|y^1+\eps U^1_{t_\e}|\leq \e^\beta\right\}+o_e(1)\\
\le & \sup_{y}\Pp^y\left\{t_\e\le \hat\tau ,\ |y^1+\eps \overline M^1_{t_\e} +\eps \widetilde M^1_{t_\e} +\eps^2 V^1_{t_\e}|\leq\e\right\}+o_e(1)
\\
\le &\sup_{y} \Pp^y\{|y^1+\eps \overline M^1_{t_\e}|\leq 2\e\} \\ &\ \qquad\qquad+ \sup_{y} \Pp^y\left\{t_\e\le \hat\tau ,\ |\eps \widetilde M^1_{t_\e} +\eps^2 V^1_{t_\e}|>\e\right\}+o_e(1).
\end{align*}
The first term on the r.h.s.\  is bounded away from $1$ because $\overline M_{t_\e}$ is a Gaussian r.v.\ with variance bounded away from 0.  Due to the exponential martingale inequality, 
the second term is $o_e(1)$ since  $V$ is bounded and the estimate $|F(Y_s)-F(0)|\le C\e^\beta$ holds for some $C>0$, all $\e>0$ and $s\le \hat\tau$. This completes the proof of~\eqref{eq:prob-of-exit-lower} and the entire lemma.\epf

\begin{lemma} \label{lem:exit-from_Pi-cannot-take-too-long}
For every $\vk>2$,
\begin{equation}
\label{eq:exit-from-Pi-tail}
\sup_{y\in\widetilde \Pi_2} \Pp^y\{\tau^2>( l_\eps )^{\vk}\}=o_e(1),
\end{equation}
and
\begin{equation}
\label{eq:prob-of-exiting-too-far-eventually}
\sup_{y\in\widetilde \Pi_2}\Pp^y\left\{\tau^2\le ( l_\eps )^{\vk},\ |Y^2_{\tau^2}|>2\e^\beta\right\}=o_e(1).
\end{equation}
\end{lemma}

\bpf
Using~\eqref{eq:prob-of-exit-lower},~\eqref{eq:prob-of-staying-in-Pi},
and the Markov property iteratively, we obtain uniformly in $y\in\widetilde \Pi_2$ and $k=1,2,\ldots$:
\[
\Pp^y\{\tau^2 > kt_\eps \}\le (c+o_e(1))^k.
\]
Setting 
\[
k = n_\e=\left\lfloor\frac{( l_\eps )^{\vk}}{t_\e}\right\rfloor+1,
\]
gives~\eqref{eq:exit-from-Pi-tail}. To prove~\eqref{eq:prob-of-exiting-too-far-eventually}, we start by defining
\begin{align*}
    \eta = \min\{k\ge 1: Y_{kt_\e}\not\in\widetilde \Pi_2\}
\end{align*}
and estimating
\begin{multline*}
    \Pp^y\left\{\tau^2\le ( l_\eps )^{\vk},\ |Y^2_{\tau^2}|>2\e^\beta\right\}
    \leq \sum_{k=0}^{n_\e} \Pp^y\left\{\tau^2\in(kt_\e,(k+1)t_\e],\  |Y^2_{\tau^2}|>2\e^\beta\right\}\\
    \leq \sum_{k=0}^{n_\e}\left( \Pp^y\left\{\tau^2\in(kt_\e,(k+1)t_\e],\  |Y^2_{\tau^2}|>2\e^\beta,\ \eta >k\right\} + \Pp^y\left\{\tau^2> kt_\e,\ \eta\leq k\right\}\right).
\end{multline*}
To see that the first term  in the $k$-th summand  is uniformly $o_e(1)$, we condition  on~$Y_{kt_\eps}$ and apply the Markov property and~\eqref{eq:prob-of-exiting-too-far}. For the second term, we write
\begin{align*}
    \Pp^y\{\tau^2> kt_\e,\ \eta\leq k\} &= \sum_{i=1}^k\Pp^y\{\tau^2> kt_\e,\ \eta=i\}\\
    &\leq \sum_{i=1}^k \Pp^y\left\{\tau^2> kt_\e,\ Y_{it_\e}\not\in\widetilde \Pi_2,\  Y_{(i-1)t_\e}\in\widetilde \Pi_2\right\}\leq k o_e(1),
\end{align*}
uniformly in $y$, where the last inequality follows from~\eqref{eq:prob-of-staying-in-Pi} and 
conditioning on~$Y_{(i-1)t_\e}$.
Combining these estimates, we obtain
\begin{align*}
    \sup_{y\in \widetilde \Pi_2}\Pp^y\left\{\tau^2\le ( l_\eps )^{\vk},\ |Y^2_{\tau^2}|>2\e^\beta\right\} \leq (n_\e +1)^2o_e(1) = o_e(1),
\end{align*}
thus proving~\eqref{eq:prob-of-exiting-too-far-eventually}.
\epf

\begin{lemma}\label{lem:exit-from_Pi_3_cannot-be-short}
 Uniformly in $y\in\{\e^\beta, -\e^\beta\}\times[-2\e^\beta,2\e^\beta]$,
\[
\Pp^y\left\{\tau^3< \frac{\beta-\theta}{2\lambda} l_\eps \right\}=o_e(1).
\]
\end{lemma}
\bpf Let us denote $t_\e=\frac{\beta-\theta}{2\lambda} l_\eps$
, and write
\begin{align*}
\Pp^y\{\tau^3< t_\e\}\le
\Pp^y\left\{\sup_{t\in[0, \tau^3\wedge t_\e]}\left|Y^1_t\right|\ge r\eps^\theta\right\}+\Pp^y\left\{\sup_{t\in[0, \tau^3\wedge t_\e]}\left|Y^2_t\right|\ge3\e^\beta\right\}=I_1+I_2.
\end{align*}
Since $e^{\lambda t_\e}=\e^{(\theta-\beta)/2}$, parts~\ref{item:tail-of_U_1} and \ref{item:sup_of-N-over-large-intervals} of Lemma~\ref{lem:estimates-for-terms} imply
\begin{align*}
I_1\le \Pp^y\left\{e^{\lambda t_\e}\left(\e^\beta+\sup_{t\in[0, \tau^3\wedge t_\e]}\left|\e U^1_t\right|\right)\ge r\eps^\theta \right\}=o_e(1)
\end{align*}
and
\begin{align*}
I_2\le \Pp^y\left\{2\e^\beta + \sup_{t\in[0, \tau^3\wedge t_\e]} \left|\e N_t\right|\ge 3\eps^\beta\right\}=o_e(1),
\end{align*}
uniformly in $y$, 
and our lemma follows.
\epf

\bigskip

\bpf[Proof of part~\ref{part:small-square-neighb}  of Lemma~\ref{lem:controlling-OU}] 
The tameness of the exit time has already been proven in Lemma~\ref{lem:exit-from_Pi-cannot-take-too-long}. To prove the remaining main claim of part~\ref{part:small-square-neighb}, we take an arbitrary $\vk>2$ and use Lemmas~\ref{lem:exit-from_Pi-cannot-take-too-long} and~\ref{lem:exit-from_Pi_3_cannot-be-short}, and  Lemma~\ref{lem:estimates-for-terms}~\eqref{item:sup_of-S-over-large-intervals}  to find uniformly high probability events $B_\e$ such that
\begin{align*}
\Pp^y\left(B_\e\cap\{e^{-\mu\tau_3}|S_{\tau^2}|>z\}\right)&\le \Pp^y\left(B_\e\cap\left\{\sup_{t\le ( l_\eps )^{\vk}}|S_{t}|>z \e^{-\mu(\beta-\theta)/(2\lambda)}\right\}\right)
\\&\le C (( l_\eps )^{\vk}+1) \exp\{-z^2\e^{-\mu(\beta-\theta)/\lambda}/C\}
\\&= Ce^{-p(z,\e)}, 
\end{align*}
where
\begin{align*}
p(z,\e)=z^2\e^{-\mu(\beta-\theta)/\lambda}/C-\log (( l_\eps )^\vk+1).
\end{align*}
There is $\eps_0>0$ such that for $z\ge1$ and $\eps\in(0,\eps_0)$, 
\begin{align*}
p(z,\e)\ge z^2\e^{-\mu(\beta-\theta)/\lambda}/(2C)  + \e^{-\mu(\beta-\theta)/\lambda}/(2C)-\log (( l_\eps )^\vk+1)\ge  z^2/(2C).
\end{align*}
For $z<1$, we estimate the probability by $1$. Combining these estimates and adjusting the value of the constant, we complete the proof of Lemma~\ref{lem:controlling-OU}~\eqref{part:small-square-neighb} and hence, Lemma~\ref{lem:second-coord-at-exit}.
\epf

\subsection{The setting in rectified coordinates.}\label{sec:rect-setting}
We recall that 
Condition~\ref{setting:conjugacy} introduces a family of linearizing conjugacies and implies that
for any  $R, L'>0$, we may assume that $f(U)$, the domain where the pushforward of $b$ under $f$ is linear, contains the rectangle $\fR$ defined in~\eqref{eq:main_rectangle}.

We are going to study the process $Y=f(X)$ until the time $\tau_\fR$, the exit time from $\fR$. 
The It\^o formula implies that 
until that time the evolution of $Y$ is governed by SDE~\eqref{eq:def_SDE_near_a_saddle_point}  with coefficients $F$ and $G$ 
given by
\begin{align*}
F_j^i(y)&=\partial_k f^i(f^{-1}(y))\sigma^k_j(f^{-1}(y)),\quad y\in f(U),
\\
G^i(y)&=\frac{1}{2}\partial_{jk}^2f^i(f^{-1}(y))\left(\, \sigma^j(f^{-1}(y))\cdot\sigma^k(f^{-1}(y)) \,\right),\quad y\in f(U),
\end{align*}
where the Einstein convention of summation over repeated indices is used.
Since~$\sigma$ is assumed to be $C^3_\mathrm{b}$ (see~\ref{setting:conjugacy}) and $f$ is assumed to be $C^5_\mathrm{b}$, we see that 
$F, G\in C^3_\mathrm{b}$, and we can extend them to $\R^2$ preserving smoothness and boundedness (but not the linearizing property) and study solutions of~\eqref{eq:def_SDE_near_a_saddle_point} with thus extended coefficients.  
Estimates from sections  \ref{sec:basic_estimates}, \ref{sec:est_N} hold for these solutions, hence, they apply to the process $f(X)$ stopped at $\tau_\fR$.

Let us describe the setting and show that it is compatible with \ref{setting:general}, \ref{setting:geometry-domain}, \ref{setting:initial-cond}, and \ref{setting:entrance-scaling-limit}, up to a small correction.

The role of vector field $b$ in \ref{setting:general} is played by $\bar b:x\mapsto (\lambda x^1, -\mu x^2)$.
The role of the diffusion $X_t$ is played by $Y_t$, so \ref{setting:general} holds only  up to a small correction given by $\eps^2G$ in the drift term.

The interior of $\fR$ plays the role of $D$ in \ref{setting:geometry-domain}, namely
\begin{align}\label{eq:D=(-R,R)...}
    D = (-R,R)\times (-L',L'),
\end{align}
and the origin $(0,0)$ is the saddle point $O$ associated with $\bar b$. We also assume that
\begin{align}\label{eq:R,L'>1}
    R,\,L'\geq 1,
\end{align}
which we can always arrange by scaling $f$. We set
\begin{align}\label{eq:x_0,v,q,v}
    x_0 = (0,L),\quad v=(1,0),\quad q_\pm =(\pm R,0), \quad v_\pm =(0,1)
\end{align}
where we choose $L>0$ sufficiently small so that \ref{setting:geometry-domain} is satisfied. One viable choice is $R,L'=1$ and $L=1/2$.

The process $Y$ starting near $x_0$ exits~$\fR$, at time $\tau=\tau_{\fR}$ given in \eqref{eq:tau_fR}, typically near $q_\pm$. See Figure~\ref{fig:rectified_coord} for this setting.

\begin{figure}[ht]
\centering
\begin{tikzpicture}

    \node (0) at (-3.75, 0) {};
    \node (1) at (3.75, 0) {};
    \node (2) at (0, 0) {};
    \node (3) at (-2.75, 0) {};
    \node (4) at (2.75, 0) {};
    \node (5) at (-2.75, 1.375) {};
    \node (6) at (2.75, 1.375) {};
    \node (7) at (-2.75, -2.75) {};
    \node (8) at (2.75, -2.75) {};
    \node (9) at (-2.75, 2.75) {};
    \node (10) at (2.75, 2.75) {};
    \node (11) at (0, -3.65) {};
    \node (12) at (0, 3.75) {};
    \node (13) at (0, 1.375) {};
    \node (14) at (-0.5, 2.25) {};
    \node (15) at (-2.25, 0.5) {};
    \node (16) at (0.5, 2.25) {};
    \node (17) at (2.25, 0.5) {};
    \node (18) at (-0.5, -2.25) {};
    \node (19) at (-2.25, -0.5) {};
    \node (20) at (0.5, -2.25) {};
    \node (21) at (2.25, -0.5) {};

    \node [below] at (0.2, 0.05) {$O$};
    \node [below] at (3, 0.05) {$q_+$};
    \node [below] at (-2.96, 0.05) {$q_-$};
    \node at (0.2, 1.55) {$x_0$};
    \node at (2, 1.55) {$v$};
    \node at (3.05, 2) {$v_+$};
    \node at (-3, 2) {$v_-$};
    \node at (2, -1.5) {$\Pi$};

\foreach \n in {3, 4, 2, 13}
        \node at (\n)[circle,fill,inner sep=1.25pt]{};

    \draw [semithick, dotted, -latex] (0.center) to (1.center);
    \draw [semithick, dotted, -latex] (11.center) to (12.center);
    \draw [semithick] (7.center) to (8.center);
    \draw [thick, dotted] (5.center) to (6.center);
    \draw [very thick, -latex] (13.center) to (2,1.375); \draw [semithick] (5.center) to (9.center);
    \draw [semithick] (9.center) to (10.center);
    \draw [semithick] (10.center) to (6.center);
    \draw [semithick] (7.center) to (3.center);
    \draw [semithick] (8.center) to (4.center);
    \draw [very thick, -latex] (3.center) to (-2.75,2); \draw [very thick, -latex] (4.center) to (2.75,2); \draw [semithick, -Stealth] [in=0, out=-90] (14.center) to (15.center);
    \draw [semithick, -Stealth] [in=180, out=-90] (16.center) to (17.center);
    \draw [semithick, -Stealth] [in=0, out=90] (18.center) to (19.center);
    \draw [semithick, -Stealth] [in=-180, out=90] (20.center) to (21.center);

\end{tikzpicture}

\caption{Dynamics in rectified coordinates}
{\label{fig:rectified_coord}}

\end{figure}
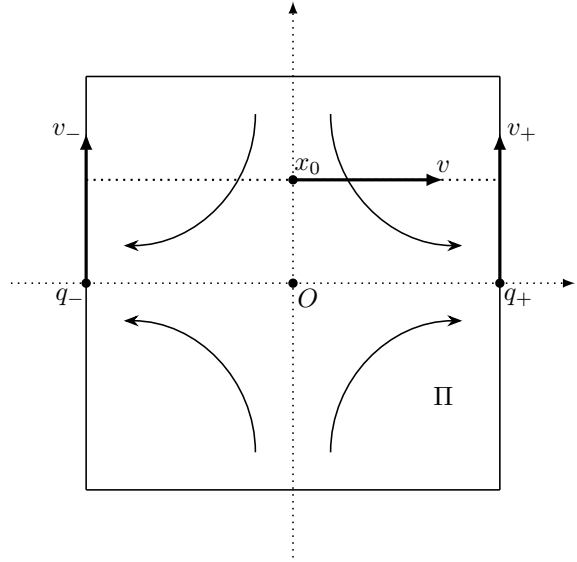

We are mostly interested in initial conditions described by Condition~\ref{setting:initial-cond} which can be rewritten as follows:  $\alpha\in(0,1]$;
the initial condition $Y_0=Y_{\e,0}$ is measurable with respect to $\Fc_0$ and satisfies
\begin{align}
\label{eq:ini-cond-rect}
Y_0=x_0 + \eps^\alpha\xi_\eps v = (\eps^\alpha\xi_\eps, L),
\end{align}
for some  real-valued r.v.'s~$\xi_\e$  such that $\e^\alpha \xi_\e\in[-1,1]$, $\e>0$.

We also assume that Condition~\ref{setting:entrance-scaling-limit} holds for some r.v.\ $\xi$.

In agreement with~\eqref{eq:entrance-scaling-conditioned} and  the definition of $\Psc^x$ above that display, 
in this section, 
$\Psc^x=\Pp^{x_0+\eps x v}=\Pp^{(\eps x, L)}$ denotes the distribution of the diffusion~\eqref{eq:def_SDE_near_a_saddle_point} with initial condition $Y_0=x_0+\eps x v= (\eps x, L)$.

As our main goal, in the next subsection, we prove lemmas stated in Sections~\ref{sec:2saddles} and~\ref{sec:long-escape-chains} for the ``rectified coordinates'' setting described above.
Let us summarize the setting for the convenience of reference:
\begin{remark}\label{rem:irc}
\rm
A lemma is said to hold \irc\ if it holds for $Y$ given in \eqref{eq:def_SDE_near_a_saddle_point} in place of $X$, $D$ given in \eqref{eq:D=(-R,R)...}, and $x_0,v,q_\pm,v_\pm$ given in \eqref{eq:x_0,v,q,v}, where $R,L'$ satisfy \eqref{eq:R,L'>1} and $L>0$ is sufficiently small so that \ref{setting:geometry-domain} is satisfied.
\end{remark}

\subsection{Proofs of lemmas \irc}\label{sec:proofs_of_lemmas_rectified}

Here, we collect proofs of some lemmas in Section~\ref{sec:2saddles} and~\ref{sec:long-escape-chains} \irc (see Remark~\ref{rem:irc}).
Some of our proofs use nontrivial local limit theorems that we postpone to Sections~\ref{sec:Gaussian-approx} and~\ref{sec:LLT}. These two sections assume the setting in Sections~\ref{sec:basic_estimates} and~\ref{sec:est_N} and, additionally, that $F$ and $G$ in \eqref{eq:def_SDE_near_a_saddle_point} are $C^3_\mathrm{b}$ (see the beginning of Section~\ref{sec:Gaussian-approx}). Hence, the results from those sections are applicable here.

We recall that  we are considering the initial conditions described in \eqref{eq:ini-cond-rect}, i.e., belonging to  $I=[-R,R]\times \{L\}$.
If $\xi_\e=x$ in~\eqref{eq:ini-cond-rect} is~deterministic, then the initial condition
is  \begin{align}
\label{eq:deterministic-ini-cond-rect}
y=(0,L)+\eps^\alpha (x,0).
\end{align}

Throughout this subsection, $\tau$ stands for $\tau_\fR$.

\begin{lemma} 
\label{lem:hp-events-in-rectangle}
Under $\Pp^y,$ events $A_{+,\e}\cup A_{-,\e}$ happen w.h.p., uniformly in $y\in I$. 
On that event,
\begin{gather}
\label{eq:expresssion_for-tau}
\tau=\frac{1}{\lambda}\log\frac{R}{\left|\eps^\alpha x+\e U^1_\tau\right|},
\\
\label{eq:dir-exit-rect}
Y^1_\tau =e^{\lam \tau}\left(\eps^\alpha x+\eps U^1_\tau\right),
\\
\label{eq:expr-for-second-coord-at-exit}
Y^2_\tau=e^{-\mu \tau}L+\e N_\tau =\frac{L}{R^\rho}\left|\e^\alpha x+\e U^1_\tau\right|^\rho +\e N_\tau, 
\end{gather}
and (recalling \eqref{eq:scaling-at-exit}) 
\begin{equation}
\label{eq:expr-for-xi-prime}
\xi'_\e=\begin{cases}
\displaystyle
\frac{L}{R^\rho}\left|x+\e^{1-\alpha} U^1_\tau\right|^\rho+\e^{1-\rho\alpha} N_\tau,& \alpha\rho \le 1, \vspace{2mm} \\ 
\displaystyle
\frac{L}{R^\rho}\left|\e^{\alpha-\frac{1}{\rho}}x+\e^{1-\frac{1}{\rho}} U^1_\tau\right|^\rho+N_\tau,&\alpha \rho>1.
\end{cases}
\end{equation}
\end{lemma}
\bpf  Lemma~\ref{lem:second-coord-at-exit} directly implies that $A_{+,\e}\cup A_{-,\e}$ happens w.h.p., uniformly in $y\in I_\e =[-\e^{\alpha'},\e^{\alpha'}]\times\{L\}$ for any $\alpha'\in(0,1)$.
 It also happens w.h.p., uniformly in $y\in I\setminus I_\e$ due to a simple large deviation estimate.
Identities~\eqref{eq:expresssion_for-tau}, \eqref{eq:dir-exit-rect}, \eqref{eq:expr-for-second-coord-at-exit} follow from \eqref{eq:SDE_near_a_saddle_point_after_duhamel1}, \eqref{eq:SDE_near_a_saddle_point_after_duhamel2} and \eqref{eq:deterministic-ini-cond-rect};  \eqref{eq:SDE_near_a_saddle_point_after_duhamel2} and
 \eqref{eq:scaling-at-exit} imply~\eqref{eq:expr-for-xi-prime}.
\epf

In the proof of Lemma~\ref{lem:postive-limit-prob-if-close-to-manifold} and multiple other instances throughout the paper, we will need 
 the following obvious lemma.

\begin{lemma}\label{lem:gaussian_sym_diff}
Suppose that $\Nc$ is a r.v.\ with density bounded by a constant $C$. Then, for any Borel sets $A_1,A_2\subset\R$,
\[
|\Pp\{\Nc\in A_1\} - \Pp\{\Nc\in A_2\}|\le C\, \Leb(A_1\triangle A_2).
\]
\end{lemma}

\subsubsection{Proof of Lemma~\ref{lem:exit-straight} in rectified coordinates} 
\label{sec:rec-lem:exit-straight}
The representation for $\xi'$ in~\eqref{eq:expr-for-xi-prime} holds w.h.p., uniformly in $x\in K_{\vk}(\e)$, due to 
Lemma~\ref{lem:hp-events-in-rectangle}. The lemma follows, since~$N_\tau$ and $U_\tau^1$ are uniformly tame due to Lemmas~\ref{lem:estimates-for-terms}~\eqref{item:tail-of_U_1} and \ref{lem:second-coord-at-exit}~\eqref{item:bartau=tau}.
\epf

\subsubsection{Proof of Lemma~\ref{lem:exit-on-the-same-whp} in rectified coordinates.}
\label{sec:proofof-lem:exit-on-the-same-whp}
Let $I_\e=[\e l_\e^\vk, R]\times\{L\}$.  Using~\eqref{eq:dir-exit-rect}  and  Lemma~\ref{lem:estimates-for-terms}~\eqref{item:tail-of_U_1}, we obtain
\[
\sup_{y\in I_\e}\Pp^y\left(A_{-,\e}\right)\le \sup_{y \in I_\e}\Pp^y\left\{Y^1_\tau < 0\right\}\le \sup_{y\in I_\e}\Pp^y\left\{\sup_{t\ge 0}\left|U^1_t\right|>l_\e^\vk\right\}=o_e(1),
\]
for sufficiently large $\vk$. \epf

\subsubsection{Proof of Lemma~\ref{lem:postive-limit-prob-if-close-to-manifold} in rectified coordinates}
\label{sec:rec-lem:postive-limit-prob-if-close-to-manifold}
Using~\eqref{eq:dir-exit-rect}, we obtain
\[
\Psc^{x}\left(A_{-,\e}\right)=\Psc^x\left\{x+ U^1_\tau<0\right\}+o_e(1) = 1 - \Psc^x\left\{x+ U^1_\tau\geq 0\right\}+ o_e(1),
\]
uniformly over $x\in K_\vk(\eps)$. 
Choosing $\vk'>\vk$ and using Lemma~\ref{lem:estimates-for-terms}, we have
\begin{align*}
    \Psc^x\left\{x+ U^1_\tau\geq 0\right\}& = \Psc^x\left\{x+ U^1_\tau\in\left[0,\,l^{\vk'}_\eps\right]\right\} + \Psc^x\left\{ U^1_\tau >l^{\vk'}_\eps -x\right\} \\
    &= \Psc^x\left\{x+ U^1_\tau\in\left[0,\,l^{\vk'}_\eps\right]\right\} + o_e(1),
\end{align*}
uniformly in $x\in K_\vk(\eps)$. 
Applying Lemma~\ref{lem:gaussian_approx}
 with $1,0,0,\tau$ substituted for $\alpha,\xi,\theta,\zeta$, we have
\begin{align*}
    \sup_{x\in K_\vk(\eps)}\left|\Psc^x\left\{x+ U^1_\tau\in\left[0,\,l^{\vk'}_\eps\right]\right\}- \Pp\left\{x+ \frU\in \left[\mp \eps^\eta,\, l^{\vk'}_\eps\pm \eps^\eta\right]\right\}\right|=\smallo{\eps^\delta},
\end{align*}
for some $\delta,\eta>0$ and a centered Gaussian r.v.\ $\frU$ with variance $\cc_1$ defined in~\eqref{eq:def_cc_1}. 
Using the choice $\vk'>\vk$, the Gaussian tail of $\frU$, and Lemma~\ref{lem:gaussian_sym_diff}, we can verify that
\begin{align*}
    \sup_{x\in K_\vk(\eps)}\left|\Pp\left\{x+ \frU\in \left[\mp \eps^\eta,\, l^{\vk'}_\eps\pm \eps^\eta\right]\right\}-\Pp\{x+ \frU\geq 0\}\right|=\smallo{\eps^{\delta'}}
\end{align*}
for some $\delta'>0$. Setting $s=\cc_1$, we have $\Pp\{y+ \frU\geq 0\} = 1 - \psi_s(-y)$. Combining these estimates, we complete the proof.
\epf

\subsubsection{Proof of Lemma~\ref{lem:alpha<1-and-rho-alpha<1-- concentration at exit} in rectified coordinates}
\label{sec:rec-lem:alpha<1-and-rho-alpha<1-- concentration at exit}
Using our assumption $\alpha<1$, $\alpha\rho<1$, the fact that  $\xi_\eps\in [l_\eps^{-\vk},l_\eps^{\vk}]$ 
w.h.p., and Lemma~\ref{lem:estimates-for-terms} in ~\eqref{eq:expr-for-xi-prime}, we obtain that $\xi'_\eps\in [l_\eps^{-\vk'},l_\eps^{\vk'}]$ w.h.p.\ for sufficiently large $\vk'$.

\subsubsection{Proof of Lemma~\ref{lem:local-limit-theorem} in rectified coordinates}
\label{sec:rec-lem:local-limit-theorem}
The proof relies on results from Section~\ref{sec:LLT}.
The first part of Lemma~\ref{lem:local-limit-theorem} is a combination of Propositions~\ref{prop:exit_loc_beta<1} and~\ref{prop:exit_loc}.

For part~\eqref{item:exit-below-unlikely}, it suffices to rewrite~\eqref{eq:expr-for-second-coord-at-exit} (holding w.h.p.~uniformly in $x\in K_\vk(\e)$):
\begin{align}
\label{eq:Y_if_exits_normally}
   Y_\tau^2= \eps^\rho\left|x+U^1_\tau\right|^\rho R^{-\rho}L + \eps N^2_\tau,
\end{align}
and use Lemma~\ref{lem:second-coord-at-exit}~\eqref{item:1_lem:second-coord-at-exit} to write
\begin{align*}
    \sup_{x\in K_\vk(\e)}\Psc^x\left\{\eps^\rho\left|x+U^1_\tau\right|^\rho R^{-\rho}L + \eps N^2_\tau \leq -\eps l^{\vk'}_\eps\right\}\leq \sup_{x\in K_\vk(\e)}\Psc^x\{ N^2_\tau \leq - l^{\vk'}_\eps\} = o_e(1)
\end{align*}
for $\vk'>1/2$.

To prove part~\eqref{item:large-negative-values-w.l.p.}, we first note that, for any $C\in\R$, due to Proposition~\ref{prop:long_exit_time_asym_linear_box_case} (with $1, 0, \frac{\beta(1+\delta)}{\rho}, R,-C$ substituted for $\alpha, \theta, \beta, r,c$, respectively),
\begin{align}\label{eq:Q(tau+C>...)}
    \sup_{x\in K_\vk(\e)}\Psc^x\left\{\tau + C\geq \frac{\beta(1+\delta)}{\mu} l_\eps \right\} = O\left(\eps^{\frac{\beta(1+\delta)}{\rho}-1}\right),
\end{align}
uniformly in $x\in K_{\vk}(\e)$. Note that the upper bound in part~\eqref{item:large-negative-values-w.l.p.} is a specific case of this estimate, with $C=0$. For the proof in original coordinates, we will need \eqref{eq:Q(tau+C>...)} with nonzero $C$.
To prove a matching lower bound for $\tau$, we note that, due to~\eqref{eq:Y_if_exits_normally},
the symmetric difference between events 
$D_{1,\e}=\left\{Y_\tau \in \{R\}\times (-\infty,\eps^\beta l^{\vk'}_\eps]\right\}$ and \\ $D_{2,\e}=\left\{\eps^\rho(x+U^1_\tau)^\rho R^{-\rho}L + \eps N^2_\tau \leq \eps^\beta l^{\vk'}_\eps,\, x+U^1_\tau>0\right\}$
is a small probability event under $\Psc^x$, uniformly in $x\in K_\vk(\e)$.

Due to~\eqref{eq:expresssion_for-tau}, on $D_{2,\eps}$ we have
\begin{gather*}
    \tau \geq \frac{\beta}{\mu} l_\eps +\frac{1}{\mu}\log \frac{L}{l^{\vk'}_\eps - \eps^{1-\beta }N^2_\tau },\\
    l^{\vk'}_\eps - \eps^{1-\beta }N^2_\tau>0.
\end{gather*}
Therefore, on $D_{2,\eps}$,  $\tau - \frac{\beta}{\mu} l_\eps <- \frac{\beta\delta}{\mu} l_\eps $ implies $ |N^2_\tau |> L\eps^{-\beta\delta} - l^{\vk'}_\eps$, 
but the latter occurs (uniformly) w.l.p.\ due to Lemma~\ref{lem:second-coord-at-exit}~\eqref{item:1_lem:second-coord-at-exit}. 
\epf

\subsubsection{Proof of Lemma~\ref{lem:exit_time_is_log} in rectified coordinates}
\label{sec:rec-lem:exit_time_is_log}
Using~\eqref{eq:expresssion_for-tau}, we obtain
\begin{align*}
    \Psc^x\left\{\tau<\frac{1-\delta}{\lambda} l_\eps \right\} = \Psc^x\left\{ \left|x+U^1_\tau\right|>\eps^{-\delta}R\right\} +o_e(1),
\end{align*}
uniformly in $x\in K_\vk(\eps)$.
Due to Lemma~\ref{lem:estimates-for-terms}~\eqref{item:tail-of_U_1}, the r.h.s.\  is $o_e(1)$, uniformly in $x\in K_\vk(\eps)$.
A matching lower bound is implied by Proposition~\ref{prop:long_exit_time_asym_linear_box_case} with $1,0,1+\delta$ substituted for $\alpha, \theta, \beta$:
\begin{align*}
    \Psc^x\left\{\tau>\frac{1+\delta}{\lambda} l_\eps \right\}=O\left(\eps^{\delta}\right),
\end{align*}
uniformly in $x$. 
\epf

\subsubsection{Proof of Lemma~\ref{lem:if-entrance-far-exit-far} in rectified coordinates}
\label{sec:rec-lem:if-entrance-far-exit-far}
Using~\eqref{eq:expr-for-xi-prime}  from Lemma~\ref{lem:hp-events-in-rectangle}
along with Lemmas~\ref{lem:estimates-for-terms}~\eqref{item:tail-of_U_1} and~\ref{lem:second-coord-at-exit}~\eqref{item:bartau=tau}, we obtain that 
if $\vk>\vk'/\rho$, then
$\xi'_\e>l_\e^{\vk'}$ w.h.p., uniformly in $x\in(l_\e^{\vk},\e^{-\alpha}]$,  and our claim follows. \epf

\subsubsection{Proof of Lemma~\ref{lem:one-step_lower_bound} in rectified coordinates}
\label{sec:rec-lem:one-step_lower_bound}
Using~\eqref{eq:expr-for-second-coord-at-exit} from Lemma~\ref{lem:hp-events-in-rectangle}, we can bound the probability of interest, up to a $o_e(1)$ term, by
\begin{align*}
    \Pp^{y }\left\{\eps^{\alpha\rho}R^{-\rho}L\left|x+\eps^{1-\alpha}U^1_\tau\right|^\rho+\eps N_\tau <-\eps^\beta l_\eps^{\vk'}\right\}\leq \Pp^{y }\left\{|N_\tau|>\eps^{\beta-1}l_\eps^{\vk'}\right\} =o_e(1),
\end{align*}
uniformly in $x\in  K_\vk(\eps)$,
where we used Lemma~\ref{lem:second-coord-at-exit}~\eqref{item:1_lem:second-coord-at-exit}. \epf

\subsubsection{Proof of Lemma~\ref{lem:exit_time_is_log_alpha<1} in rectified coordinates}
\label{sec:rec-lem:exit_time_is_log_alpha<1}
Using~\eqref{eq:expresssion_for-tau}, we obtain
\begin{align*}
    \Pp^{y}\left\{\tau<\frac{\alpha-\delta}{\lambda} l_\eps \right\} &\leq \Pp^{y}\left\{ \left|x+\eps^{1-\alpha}U^1_\tau\right|>\eps^{-\delta}R\right\}+o_e(1)
    \\ & \leq \Pp^{y}\left\{ \eps^{1-\alpha}\left|U^1_\tau\right|>\eps^{-\delta}R-|x|\right\} +o_e(1),  
    \\
    \Pp^{y}\left\{\tau>\frac{\alpha+\delta}{\lambda} l_\eps \right\} &\leq \Pp^{y}\left\{ \left|x+\eps^{1-\alpha}U^1_\tau\right|<\eps^{\delta}R\right\}+o_e(1)\\ &\leq \Pp^{y}\left\{ |x|<\eps^{\delta}R+\eps^{1-\alpha}\left|U^1_\tau\right|\right\}+o_e(1),
\end{align*}
uniformly in $x\in K_\vk(\eps)$.
Applying Lemma~\ref{lem:estimates-for-terms}~\eqref{item:tail-of_U_1} to $U^1_\tau$, we see that the first display is $o_e(1)$ uniformly in $x \in K_\vk(\eps)$, and the second display is bounded from above by $\ONE_{|x|\leq \eps^{\delta'}}+ o_e(1)$ for some $\delta'>0$.\epf

\section{Proofs in the original coordinates}
\label{sec:original-coords}

The goal of this section is to prove the results from Sections~\ref{sec:2saddles} and~\ref{sec:long-escape-chains} in full generality. The plan is to 
use the results obtained in Section~\ref{sec:rectified} \irc\ to study the diffusion inside the domain of the linearizing conjugacy, and 
combine these results with the analysis of motion along heteroclinic orbits outside of that domain.  We begin with the latter.

\subsection{Diffusion along a heteroclinic orbit}
\label{sec:finite-time-horizon-lemmas}
The results in this section concern finite time horizon and are close to those in~\cite{FW} and~\cite{Almada-Bakhtin:MR2739004}.

Given a vector field $b$, we call a $C^1$ curve $\trcurve:[0,1]\to \R^2$ $b$-transversal if, for every $t\in[0,1]$,

\begin{align*}
    b(\trcurve(t))\cdot \frac{d}{dt}\trcurve(t) \neq 0.
\end{align*}
For brevity, we often use  $\trcurve$ to denote $\trcurve([0,1])\subset \R^2$, the image of $\trcurve$. In addition, we denote by $\mathring \trcurve $ the set $\trcurve((0,1))$. 
We recall the definition of the flow $(\flow^t)_{t\in\R}$ from~\eqref{eq:flow}.

\begin{lemma}\label{lem:hit_loc}
 Let $E\subset \R^2$ be compact and let $\trcurve:[0,1]\to \R^2\setminus E$ be $C^2$, 
$b$-transversal. Suppose further that for every $z\in E$, there is a minimal time $t_z>0$ such that $\flow^{t_z}z \in \mathring \trcurve$. 
Let $\zeta=\inf\{t\ge 0:\ X_t\in\trcurve \}$, where $X_t$ is a solution of~\eqref{eq:basic-sde}. Then there is
a constant $C$ such that 
\begin{equation}
\label{eq:whp-bounded-time}
    \sup_{z\in E} \Pp^z\{\zeta >C\} = o_e(1),
\end{equation}
 and there are constants $c_1,c_2$ such that
  for all $\eta>0$,
 \[
    \sup_{z\in E} \Pp^z\left\{|X_{\zeta} - \flow^{t_z}z|>\eta;\ \zeta <\infty\right\} \leq c_1 e^{-c_2 \eta^2\eps^{-2}}.
\]
In particular,
for any fixed $(\beta,\vk)\in([0,1)\times \R)\cup (\{1\}\times (1/2,+\infty))$, 
\begin{align*}
    \sup_{z\in E} \Pp^z\left\{|X_{\zeta} - \flow^{t_z}z|>\eps^\beta l^\vk_\eps;\ \zeta <\infty\right\} =o_e(1).
\end{align*}
\end{lemma}

\bpf The lemma follows from the classical Freidlin--Wentzell Large Deviation Principle, which holds  uniformly with respect  to the initial condition $z$, 
see~\cite[Chapter 5, Theorem 3.2]{FW}.
\epf

\begin{lemma}\label{lem:bi-lip}
 Let $\trcurve_1, \trcurve_2:[0,1]\to \R^2$ be $C^1$ and $b$-transversal. Suppose for every $z \in \trcurve_1$, there is a minimal time $t_z>0$ such that $\flow^{t_z}z \in \trcurve_2$. Then the map~$\phi$ defined by
 \begin{equation}
 \label{eq:def-phi}
 \phi(z)=\flow^{t_z}z. 
 \end{equation}
is a diffeomorphism on $\trcurve_1$. 
\end{lemma}
\bpf Due to the transversality assumption, this is a consequence of the implicit function theorem. \epf

The following result is an extension of Theorem~\ref{th:along-hetero}.

\begin{lemma}\label{lem:expansion}
Let $\alpha\in(0,1]$ and  $\vk>0$. Assume that $\trcurve$ is $b$-transversal. 
Let $x_0\in\R^2$ and let
$T>0$ be the minimal time such that $\flow^{T}x_0\in \mathring \trcurve$.
 Let ~$(X_t)$ be a solution of~\eqref{eq:basic-sde} and $\zeta=\inf\{t>0: X_t \in\trcurve\}$.
 Then there is $\eta>0$, a deterministic rank-one matrix~$A$, a~centered Gaussian vector $\NProj$ (once $x_0$ is fixed, $\NProj$ is a function of the
 noise realization $W$, so it does not depend on the initial condition $z\in\R_2$),
 it is concentrated on
 the tangent line to $\chi$ at $\flow^Tx_0$, and a family of random vectors~$(r_{z,\eps})_{ |z-x_0|\le \eps^\alpha l^\vk_\e,\ \e>0}$ 
  such that under $\Pp^z$, w.h.p., uniformly in $z$, $\zeta<\infty$ and 
\begin{align*}
    X_\zeta = \flow^Tx_0 + \eps^\alpha(A\bar z + \eps^{1-\alpha} \NProj + r_{z,\eps}),
\end{align*}
where 
\begin{equation}
\label{eq:bar_x}
\bar z=\eps^{-\alpha}(z-x_0)
\end{equation}
and $|r_{z,\eps}|\leq \eps^\eta$ w.h.p.\ uniformly in $z$ satisfying $|z-x_0|\le \eps^\alpha l^\vk_\e$.
More precisely, there is $\eta>0$ such that
\begin{gather*}
    \sup_{|z-x_0|\le \e^\alpha l^\vk_\e}\Pp^z\left\{\left|\eps^{-\alpha}(X_\zeta - \flow^Tx_0)-A\bar z-\eps^{1-\alpha} \NProj\right|>\eps^\eta,\quad \zeta<\infty\right\} = o_e(1),
    \\
    \sup_{|z-x_0|\le \e^\alpha l^\vk_\e} \Pp^z\{\zeta =\infty\} = o_e(1).
\end{gather*}
\end{lemma}

\bpf  
By Taylor's theorem, functions $Q_1(\cdot,\cdot)$ and $Q_2(\cdot,\cdot)$ defined by 
\begin{align}
\label{eq:def_Q1}b(z)&=b(y)+Db(y)(z-y)+Q_1(y,z-y),\quad z,y\in\R^2,\\
\label{eq:def_Q2}\sigma(z)&=\sigma(y)+Q_2(y,z-y),\quad z,y\in\R^2,
\end{align}
are continuous and  satisfy, for some $K>0$, 
\begin{align}
\label{eq:quadratic_term}
|Q_1(y,v)|\le K |v|^2,\quad y\in\R^2,\ |v|\le1,
\\
\label{eq:linear_term}
|Q_2(y,v)|\le K(1\wedge |v|),\quad y,\,v\in\R^2.
\end{align}
On the left-hand side of the last inequality,  we use $|\cdot|$ to denotes the operator norm of a matrix.  
We define the linearization (fundamental solution) near the orbit $(\flow^tx_0)$ by
\begin{align*}
\frac{d}{dt}\Lf (t)=Db(\flow^tx_0) \Lf(t),\quad \Lf(0)=I,
\end{align*}
where $I$ is the $2\times 2$ identity matrix. The standard theory of differential equations combined with the properties of $Db(\cdot)$ under our assumptions on $b$ then imply that $(\Lf(t))_{t\geq 0}$ has the semigroup property $\Lf(t+s) = \Lf(t)\Lf(s)$ and there is a constant $c>0$ such that $|\Lf(t)|\leq e^{c t}$. We set 
\begin{equation}
\label{eq:N_t}
N_t=\int_0^t\Lf(t-s)\sigma(\flow^sx_0)dW_s,\quad t\ge 0,
\end{equation}
then, recalling that $\bar z$ and $z$ are related by \eqref{eq:bar_x}, we set
\begin{equation}
\label{eq:two-terms-for-Theta}
\Theta_\e(t,z)=\Lf(t)\bar z+\e^{1-\alpha}N_t,\quad z\in\R^2,\ t\ge0, \ \e>0,
\end{equation}
and define  $r_{\e}(t,z)$ by
\begin{equation}
\label{eq:def-of-r_e}
X_t=\flow^tx_0+\e^\alpha(\Theta_\e(t,z)+r_{\e}(t,z)),\quad  z\in\R^2,\ t\ge 0,\ \e>0.
\end{equation}

\begin{lemma}\label{lem:linearization-error}
For any $T'>0$ and  $\beta\in(0,\alpha)$,
\[
\sup_{|z-x_0|\le \eps^\alpha l^\vk_\e}\Pp^z\left\{\sup_{t\in[0,T']}|r_{\e}(t,z)|> \e^\beta\right\}=o_e(1).
\]
\end{lemma}
\bpf
Let $\Delta_t=X_t-\flow^t x_0$, so  $\Pp^z\{\Delta_0=\eps^\alpha \bar z\}=1$. Using~\eqref{eq:def-of-r_e}, we obtain 
\begin{equation}
\label{eq:Delta_via_Theta}
\Delta_t=\e^\alpha (\Theta_\e(t,z) + r_\e(t,z)).
\end{equation}
Then, since $X_0=z$, we have
\begin{align*}
d\Delta_t
&=(b(X_t)-b(\flow^t x_0))dt+\e \sigma(X_t)dW_t\\
&=Db(\flow^t x_0)\Delta_t dt+ \e \sigma(\flow^tx_0)dW_t
+Q_1(\flow^tx_0,\Delta_t) dt + \e Q_2(\flow^tx_0,\Delta_t) dW_t.
\end{align*}
Applying Duhamel's principle to this identity, using~\eqref{eq:Delta_via_Theta} and~\eqref{eq:two-terms-for-Theta},  we obtain
\begin{equation}
\label{eq:r_eps}
r_{\e}(t,z)=\Theta'_\e(t,z)+\Theta''_\e(t,z),
\end{equation}
where
\begin{align*}
\Theta'_\e(t,z)&=\e^{-\alpha}\int_0^t \Lf(t-s)Q_1(\flow^sx_0,\Delta_s)ds,
\\ \Theta''_\e(t,z)&=\e^{1-\alpha}\int_0^t \Lf(t-s)Q_2(\flow^sx_0,\Delta_s)dW_s.
\end{align*}
 Let us take an arbitrary $\beta'\in(0,\alpha)$ and define $\ell=\inf\{t\ge 0: |\Delta_t|>\e^{\beta'}\}\wedge T'$.
Then, using \eqref{eq:quadratic_term}, \eqref{eq:linear_term}, and the exponential martingale inequality, we obtain that for some constant $C>0$ and for small~$\e$,
\begin{equation}
\label{eq:est_Theta-prime}
\sup_{|z-x_0|\le \eps^\alpha l^\vk_\e}\sup_{t\le \ell} |\Theta'_\e(t,z)|\le C \e^{2\beta'-\alpha}
\end{equation}
and
\begin{equation}
\label{eq:est_Theta-pprime}
\sup_{|z-x_0|\le \eps^\alpha l^\vk_\e}\Pp^z\left\{\sup_{t\le \ell} |\Theta''_\e(t,z)|\ge \e^{2\beta'-\alpha}\right\}=o_e(1). 
\end{equation}
In addition, $\sup_{t\le T'} |N_t|$ is tame, so $\sup_{t\le \ell} |\Theta_\e(t,z)|$ is  tame, uniformly in $z$ satisfying $|z-x_0|\le \eps^\alpha l^\vk_\e$. Using this, \eqref{eq:est_Theta-prime},
and~\eqref{eq:est_Theta-pprime} in~\eqref{eq:Delta_via_Theta}, we obtain that  for any $\beta''<\alpha\wedge (2\beta')$,  w.h.p, uniformly in $z$,
$\sup_{t\le \ell}|\Delta_t|\le \e^{\beta''}.$
Choosing $\beta''\in(\beta',\alpha \wedge (2\beta'))$, we thus obtain that $\ell=T'$ w.h.p., uniformly in $z$.
Combining this with~\eqref{eq:r_eps},~\eqref{eq:est_Theta-prime},~\eqref{eq:est_Theta-pprime},
we complete the proof.
\epf

Going back to the proof of Lemma~\ref{lem:expansion}, 
we first note that its last claim follows from Lemma~\ref{lem:hit_loc}. To prove the main claim,
we choose some $\alpha'\in(0,\alpha)$ (we will impose a tighter requirement later) and note that \eqref{eq:def-of-r_e}, \eqref{eq:two-terms-for-Theta},
Lemma~\ref{lem:linearization-error}, and the $b$-transversality of $\trcurve$ imply that under~$\Pp^z$,
\begin{align}\label{eq:zeta_in_(T-,T+)}
    \zeta\in \big(T-\e^{\alpha'}, T+\e^{\alpha'}\big)
\end{align}
w.h.p., uniformly in~$z$ satisfying $|z-x_0|\le \eps^\alpha l^\vk_\e$. Let us study the path $X_t$ on this time interval.

First, let us introduce projection operators $\pi_{b}$ and $\pi_\trcurve$ via a unique decomposition
\begin{equation}
\label{eq:projecting1}
v=\pi_b v + \pi_\trcurve v,\quad  v\in\R^2,
\end{equation}
where $\pi_b v$ is collinear with $b(\flow^Tx_0)$ and $\pi_\trcurve v$ is tangent to $\trcurve$ at $\flow^Tx_0$.
We will prove that the lemma holds with 
\begin{align}
A\bar z&=\pi_\trcurve A (T)\bar z,\notag
\\
\label{eq:projecting3}
\NProj&= \pi_\trcurve N_T.
\end{align}
so that
\[
A \bar z+\e^{1-\alpha} \NProj=\pi_\trcurve(\Theta_\e(T,z)).
\]
Let us impose an additional requirement that $\alpha'>\alpha/2$ and  
 prove that for any $\beta$ satisfying
 \[
 0<\beta< \alpha' \wedge (1-\alpha+\alpha'/3)\wedge (2\alpha'-\alpha),
 \]
 we have that  w.h.p.\ under~$\Pp^z$, uniformly in $z$ satisfying $|z-x_0|\le \eps^\alpha l^\vk_\e$,
\begin{equation}
\label{eq:intersect-at-transversal-level}
\sup_{t\in (T-\e^{\alpha'}, T+\e^{\alpha'})}|\e^{-\alpha}\pi_\trcurve (X_t-\flow^Tx_0) - \pi_\trcurve\Theta_\e(T,z) |\le \e^{\beta}.
\end{equation}
To that end, let us  use \eqref{eq:def-of-r_e} to write
\begin{multline*}
|\e^{-\alpha} \pi_\trcurve (X_t-\flow^Tx_0) -\pi_\trcurve \Theta_\e(T,z)|\le I_1+I_2+I_3\\
=\e^{-\alpha} |\pi_\trcurve (\flow^tx_0-\flow^Tx_0) |+ |\pi_\trcurve(\Theta_\e(t,z) -\Theta_\e(T,z))|
+|\pi_\trcurve(r_{\e}(t,z))|,
\end{multline*}
and estimate each term on the right-hand side.
Since $\flow^tx_0$ is a $C^2$ function of $t$ and $\frac{d}{dt}\flow^tx_0\big\rvert_{t=T}=b(\flow^Tx_0)$, we have 
\[
I_1\le  \e^{-\alpha} C(t-T)^2\le \e^{2\alpha'-\alpha} ,\quad t \in (T-\e^{\alpha'}, T+\e^{\alpha'}).
\]
To estimate $N_t-N_T$, we assume $t\leq T$ the opposite case following by interchanging the role of $t$ and $T$. Since $A(t)$ is smooth in $t$, we obtain 
\[
|A(t-s)-A(T-s)|\le C\e^{\alpha'},\quad t \in (T-\e^{\alpha'}, T+\e^{\alpha'}).
\]
Using this, $|A(T-s)|<e^{cT}$, and the exponential martingale inequality, we obtain that w.h.p.\ under~$\Pp^z$, uniformly in~$z$,
\begin{multline*}
|N_t-N_T|\le\\ \left|\int_t^TA(T-s)\sigma(\flow^sx_0)dW_s\right| + \left|\int_0^t(A(T-s)-A(t-s))\sigma(\flow^sx_0)dW_s\right| < C\e^{\alpha'/3},
\end{multline*}
for all $t \in (T-\e^{\alpha'}, T+\e^{\alpha'})$. So, w.h.p.\ uniformly in $z$,
\[
 \sup_{t \in (T-\e^{\alpha'}, T+\e^{\alpha'})} I_2\le C (\e^{\alpha'} +\e^{1-\alpha+\alpha'/3}).
\]
Finally, due to Lemma~\ref{lem:linearization-error}, we know that w.h.p.\ under $\Pp^z$, uniformly in $z$,
\[
 \sup_{t \in (T-\e^{\alpha'}, T+\e^{\alpha'})} I_3\le \e^{\alpha'}.
\]
Combining these estimates for $I_1,I_2,I_3$, we obtain that \eqref{eq:intersect-at-transversal-level} holds w.h.p.\ under~$\Pp^z$, uniformly in 
$z$ satisfying $|z-x_0|\le \eps^\alpha l^\vk_\e$.
Therefore, due to \eqref{eq:zeta_in_(T-,T+)}, w.h.p., uniformly in~$z$,
\begin{equation}
\label{eq:intersect-at-transversal-level2}
|\e^{-\alpha} \pi_\trcurve (X_\zeta-\flow^Tx_0)-\pi_\trcurve \Theta_\e(T,z) |\le \e^{\beta}.
\end{equation}
For small $\e$, this estimate implies
$|\pi_\trcurve (X_\zeta-\flow^Tx_0)| \le \e^{3\alpha/4 }$ and, 
since $\trcurve\in C^2$, 
$|\pi_b (X_\zeta-\flow^Tx_0)|\le C\e^{3\alpha/2}$. Combining the latter with \eqref{eq:intersect-at-transversal-level2} and choosing any  $\eta\in(0,\beta\wedge (\alpha/2))$,
we complete the proof of the lemma.
\epf

\bigskip

We will need another extension of Theorem~\ref{th:along-hetero}. Let us adopt the setting of Lemma~\ref{lem:expansion}. Then for all $z$ in a small
neighborhood of $x_0$, the minimal time $t_z$
such that $\flow^{t_z}z\in \trcurve$ is well-defined and finite.
In that entire neighborhood, we can define the map~$\phi$ by \eqref{eq:def-phi}.

We recall the definition of $N_t$ in \eqref{eq:N_t} and define a random vector $\NProj$ by~\eqref{eq:projecting1} and \eqref{eq:projecting3}. 
Note that $\NProj$
is a function of the noise realization~$W$. We can now state one more extension of Theorem~\ref{th:along-hetero} that we need.

\begin{lemma}\label{lem:scal_lim}
In the setting of Lemma~\ref{lem:expansion}, there is $\eta>0$ and a family of random vectors $(h_{z,\eps})_{|z-x_0|<c,\,\e>0 }$
such that for each $\alpha\in(0,1]$ and $\vk>0$, the following holds w.h.p.\ under $\Pp^z$, uniformly in $z$ satisfying  $|z-x_0|<\eps^\alpha l^\vk_\eps$:
\begin{align*}
\zeta&<\infty,\\
X_{\zeta}&= \phi(z)+\eps \NProj  +\eps h_{z,\eps},\\
|h_{z,\eps}|&\le \eps^{\eta}.
\end{align*}

\end{lemma}

\begin{remark}\label{rem:initial_line}\rm
Let us restrict $\phi$ to a small segment $\trcurve$ such that $x_0\in\trcurve \subset x_0+\R v$ for some $v$ transversal to $b(x_0)$.
 Then we can write $\NProj = \NProj'D\phi(x_0)v$ for some centered Gaussian r.v.~$\NProj'$, where $D\phi$ is the differential of the restriction of $\phi$. Extending $\phi$ smoothly to the entire $x_0 +\R v$, we also have 
\begin{align*}
 \left|\phi(x_0+(\eps^\alpha x+\eps \NProj')v) - \phi(x_0+\eps^\alpha xv) - \eps \NProj\right|\leq C(\eps^2|\NProj'|^2+\eps^{1+\alpha}|\NProj'||x|).
\end{align*}
The error can be absorbed into $h'_{x,\eps}=h_{x_0+\eps^\alpha xv,\,\eps}$, and hence,  w.h.p.\  under $\Pp^{x_0+\eps^\alpha xv}$, uniformly in $x\in K_\vk(\eps)$,
\begin{align*}
    X_\zeta=\phi(x_0+(\eps^\alpha x+\eps M')v)+\eps h'_{x,\eps}
\end{align*}
and $|h'_{x,\eps}|\leq \eps^{\eta}$.
\end{remark}

The proof of Lemma~\ref{lem:scal_lim} is similar to that of Lemma~\ref{lem:expansion}. 
First, we prove the following auxiliary result:

\begin{lemma}\label{lem:linearization-error2}
  Under the conditions stated above, for $h_{\e}(t,z)$ defined via
  \[
  X_t=\flow^tz+\e N_{t}+\e h_{\e}(t,z),\quad z\in\R^2,\ t\ge 0,\ \e>0,
  \]
  the following holds:  if $T'>0$ and  $\beta\in(0,\alpha)$, then
  \[
  \sup_{|z-x_0|\le \eps^\alpha l^\vk_\eps}\Pp^z\left\{\sup_{t\in[0,T']}|h_{\e}(t,z)|> \e^\beta\right\}=o_e(1).
  \]
\end{lemma}
\bpf
Let 
\begin{equation}
\label{eq:Delta-via-r}
\Delta_t=X_t-\flow^t z=\e N_t+\e h_\e(t,z).
\end{equation}
In addition to the definitions of $Q_1,Q_2$ in \eqref{eq:def_Q1}, \eqref{eq:def_Q2},
 we define 
\[
Q_3(z,y)=Db(z+y)- Db(z),\quad z,y\in\R^2,
\]
and adjust the  constant $K$ in \eqref{eq:quadratic_term}, \eqref{eq:linear_term}, to ensure that 
\begin{equation}
\label{eq:Q_3_lin_bound}
|Q_3(z,y)|\le K(1\wedge |y|),\quad z,\,y\in\R^2.
\end{equation}
Then
\begin{align*}
d\Delta_t
=&(b(X_t)-b(\flow^t z))dt+\e \sigma(X_t)dW_t\\
=&Db(\flow^t x_0)\Delta_t dt+Q_3(\flow^t x_0, \flow^t z-\flow^t x_0)\Delta_t dt
+Q_1(\flow^tz,\Delta_t) dt \\ &+ \e \sigma(\flow^tx_0)dW_t + \e Q_2(\flow^tx_0,    \flow^tz-\flow^tx_0 )  dW_t
+ \e Q_2(\flow^tz,\Delta_t) dW_t.
\end{align*}
Applying the Duhamel principle to this identity, we obtain that
\begin{equation}
\label{eq:Delta_via_Thetaprime}
h_\e(t,z)=\e^{-1}\Delta_t-N_t=\Theta'_\e(t,z)+\Theta''_\e(t,z)+\Xi'_\e(t,z)+\Xi''_\e(t,z), 
\end{equation}
where
\begin{align*}
\Theta'_\e(t,z)&=\e^{-1}\int_0^t \Lf(t-s)Q_1(\flow^sz,\Delta_s)ds,
\\ \Theta''_\e(t,z)&=\int_0^t \Lf(t-s)Q_2(\flow^sz,\Delta_s)dW_s, 
\\
\Xi'_\e(t,z)&=\e^{-1}\int_0^t \Lf(t-s) Q_3(\flow^s x_0, \flow^s z-\flow^s x_0)\Delta_s   ds,
\\ \Xi''_\e(t,z)&=\int_0^t \Lf(t-s)Q_2(\flow^sx_0,    \flow^sz-\flow^sx_0 )dW_s.
\end{align*}

Let us take an arbitrary $\beta'\in(0,1)$ and define $\ell=\inf\{t\ge 0: |\Delta_t|\ge\e^{\beta'}\}\wedge T'$.
Then, using \eqref{eq:quadratic_term}, \eqref{eq:linear_term}, \eqref{eq:Q_3_lin_bound} and
the Lipschitzness of $\flow^t$,  we obtain that for all $\beta''\in(0,\beta')$, $\alpha'\in(0,\alpha)$,
\begin{equation}
\label{eq:est_Theta-prime1}
\sup_{|z-x_0|<\eps^\alpha l^\vk_\eps}\, \sup_{t\le \ell} |\Theta'_\e(t,z)|=o(\e^{2\beta''-1}),
\end{equation}
\begin{equation*}
\sup_{|z-x_0|<\eps^\alpha l^\vk_\eps}\Pp^z\left\{\sup_{t\le \ell} |\Theta''_\e(t,z)|\ge \e^{\beta''}\right\}=o_e(1), 
\end{equation*}
\begin{equation}
\label{eq:est_Xi-prime1}
\sup_{|z- x_0|<\eps^\alpha l^\vk_\eps}\, \sup_{t\le \ell} |\Xi'_\e(t,z)|=o(\e^{\alpha+\beta''-1}),
\end{equation}
\begin{equation}
\label{eq:est_Xi-pprime1}
\sup_{|z-x_0|<\eps^\alpha l^\vk_\eps}\Pp^z\left\{\sup_{t\le \ell} |\Xi''_\e(t,z)|\ge \e^{\alpha'}\right\}=o_e(1). 
\end{equation}
Choosing $\beta'$ and $\beta''$  sufficiently close to $1$ and $\alpha'$ sufficiently close to $\alpha$,
using these relations along with \eqref{eq:Delta-via-r} and the tameness of $\sup_{t\le T'} |N_t|$,
we obtain that 
\[
\sup_{|z- x_0|<\eps^\alpha l^\vk_\eps }\Pp^z\left\{\sup_{t\le \ell} 
|\Delta_t|\ge \e^{\beta'}
\right\}=o_e(1), 
\]
which implies that w.h.p.\ under $\Pp^z$, uniformly in $|z- x_0|<\eps^\alpha l^\vk_\eps$, we have $\ell=T'$.
Therefore,
\eqref{eq:est_Theta-prime1}--\eqref{eq:est_Xi-pprime1} hold with $\ell$ replaced by $T'$ (w.h.p.\ for \eqref{eq:est_Theta-prime1} and \eqref{eq:est_Xi-prime1}). 
Once $\alpha',\beta',\beta''$ are chosen to ensure relations $2\beta''-1>\beta$,\quad $\beta'>\beta$,\quad $\alpha+\beta''-1>\beta$, \quad $\alpha'>\beta$, we can use 
these estimates in 
\eqref{eq:Delta_via_Thetaprime} to complete the proof.
\epf

\bpf[Proof of Lemma~\ref{lem:scal_lim}]
Let $\beta_1\in(0,1)$ and note that 
Lemma~\ref{lem:linearization-error2} implies that
\begin{equation}
\label{eq:sandwich-for-hitting-time}
\zeta\in (t_z-\e^{\beta_1}, t_z+\e^{\beta_1})
\end{equation}
w.h.p.\ under $\Pp^z$, uniformly in~$z$. Let us study the path $X_t$ on this time interval.

First, we define projection operators $\pi_{b,z}$ and $\pi_{\trcurve,z}$ via a unique decomposition 
\begin{equation*}
v=\pi_{b,z} v + \pi_{\trcurve,z} v,\quad  v\in\R^2,
\end{equation*}
where $\pi_{b,z}$ is collinear with $b(\phi(z))$ and $\pi_{\trcurve,z} v$ is tangent to $\trcurve$ at $\phi(z)$. We define  $\NProj(z)=\pi_{\trcurve,z}N_{t_z}$. In particular, $\NProj=\NProj(x_0)=\pi_{\trcurve,x_0}N_{t_{x_0}}$.

We claim that there is $\beta_2>0$ such that w.h.p.\ under $\Pp^z$, uniformly in $z$,
\begin{equation}
\label{eq:intersect-at-transversal-level1}
\sup_{t\in (t_z-\e^{\beta_1}, t_z+\e^{\beta_1})}|\e^{-1}\pi_{\trcurve,z} (X_t-\phi(z)) - M|\le \e^{\beta_2}.
\end{equation}
To prove this, let us use the representation for $X_t$ from Lemma~\ref{lem:linearization-error2} and write
\begin{multline*}
|\e^{-1} \pi_{\trcurve,z} (X_t-\phi(z)) -\NProj|\le I_1+I_2+I_3+I_4\\
=\e^{-1} |\pi_{\trcurve,z} (\flow^tz-\phi(z)) |+ |\pi_{\trcurve,z}N_t -\NProj(z)| + |\NProj(z)-\NProj|
+|\pi_{\trcurve,z}(h_{\e}(t,z))|,
\end{multline*}
and estimate each term on the right-hand side. Since $|t-t_z|<\e^{\beta_1}$, and the tangent vector to the $C^2$ trajectory $(\flow^tz)_{t\in  (t_z-\e^{\beta_1}, t_z+\e^{\beta_1})}$ 
 at $t=t_z$ is $b(\phi(z))$, we see that
\[
\sup_{t\in  (t_z-\e^{\beta_1}, t_z+\e^{\beta_1})} I_1\le \e^{2\beta_1-1}.
\] 
Using the exponential martingale inequality to control $N$, we obtain that, w.h.p., uniformly in $z$,
\[
\sup_{t\in  (t_z-\e^{\beta_1}, t_z+\e^{\beta_1})} I_2\le \e^{\beta_1/3}.
\]
Let us estimate $I_3$. The definitions of $\NProj(z)$ and $\NProj$ imply that
\[
I_3\le |\pi_{\trcurve,z}(N_{t_z}-N_{t_{x_0}})|+|(\pi_{\trcurve,z}-\pi_{\trcurve,x_0})N_{t_{x_0}}|
=I_{3,1}+I_{3,2}.
\]
The operator norm of $\pi_{\trcurve,z}$ is bounded,  so for a constant $C>0$ and an arbitrary $\beta_3\in(0,\alpha/2)$, we have w.h.p.\ under $\Pp^z$, uniformly in $|z-x_0|\le \e^{\alpha}l_\e^\vk$,
\[
I_{3,1}\le C |N_{t_z}-N_{t_{x_0}}|\le \e^{\beta_3},
\]
where in the second inequality we used the Lipschitzness of $t_z$ in $z$ and the fact that~$N_{t}$ is a diffusion process. Since the projection operator $\pi_{\trcurve,z}$
is Lipschitz in $z$, we also conclude that for $\beta'_3\in(\beta_3,\alpha)$, w.h.p.\ under $\Pp^z$, uniformly in $|z-x_0|\le \e^{\alpha}l_\e^\vk$, 
\[
I_{3,2}\le \e^{\beta'_3} |N_{t_{x_0}}|\le \e^{\beta_3},
\]
where the last estimate follows from the fast decay of the Gaussian tail.
We also use  Lemma~\ref{lem:linearization-error2} to find $\beta_4>0$ such that
\[
\sup_{t\in  (t_z-\e^{\beta_1}, t_z+\e^{\beta_1})} I_4\le e^{\beta_4}.
\]
Combining these estimates and choosing $\beta_1$ sufficiently close to $1$, we obtain our 
claim~\eqref{eq:intersect-at-transversal-level1}. Using \eqref{eq:sandwich-for-hitting-time}, we obtain that w.h.p., uniformly in $z$,
\begin{equation*}
|\pi_{\trcurve,z} (X_\zeta-\phi(z)) - \e \NProj|\le \e^{1+\beta_2}.
\end{equation*}
Since $\trcurve\in C^2$, this estimate implies that for some $K>0$ and any  $\beta_5\in(0,1)$, w.h.p.\ under $\Pp^z$, uniformly in $z$,
\begin{equation*}
|\pi_{b,z} (X_\zeta-\phi(z))|\le K(\e|\NProj|+\e^{1+\beta_2})^2\le \e^{1+\beta_5}.
\end{equation*}
Combining the last two estimates, we complete the proof of the lemma.\epf

\subsection{Proofs of lemmas from Sections~\ref{sec:2saddles} and~\ref{sec:long-escape-chains} in the original coordinates}
\label{sec:proofs-lemmas-original-coords}

We recall that the initial conditions for all the results we need to prove are described in 
assumption~\ref{setting:initial-cond} where $\alpha\in(0,1]$, $x_0\in \Ws$, $v$ is transversal to $\Ws$ at $x_0$, and in addition 
$\xi_\e$ is assumed to be tame. In other words, w.h.p.,
initial conditions belong to $x_0+\eps^{\alpha}K_{\vk}(\eps)v$ and we will restrict ourselves to these initial values only.

We are going to split the evolution into three stages (and rely on the strong Markov property for solutions of It\^o SDE's), see Figure~\ref{fig:dyn_3_stages}: (i)~along the stable manifold~$\Ws$, (ii)~in a small neighborhood of the saddle point $O$, (iii)~along the unstable manifold~$\Wu$. 

To that end, we recall that our choice of  parameters $R,L,L'>0$ and the drift-linearizing conjugacy~$f$ defined on a neighborhood $U$ of the saddle point ensures that the rectangle $\fR$ defined by~\eqref{eq:main_rectangle} satisfies $\fR\subset f(U)$, i.e., $f^{-1}(\fR)\subset U$ (see Section~\ref{sec:rect-setting}).

In the first stage, the process $X$ evolves mostly outside $\fR$. This stage ends at time  
$\zeta=\inf\{t\ge 0:  X_t\in \chi \},$ 
when the process $X$ hits $\chi=f^{-1}([-R,R]\times \{L\})$.
 The outcome of this first stage can be studied using results of Section~\ref{sec:finite-time-horizon-lemmas}.
 In particular,  $\zeta<\infty$ and $X_\zeta$ belongs to a small neighborhood of $f^{-1}(0,L)$ w.h.p.

 This means that, w.h.p., the evolution of $X$ after $\zeta$ is well-defined and, while~$X$ stays within $U$, can be described in terms of the process $Y$ given by $Y_t=f(X_{\zeta+t})$. This
 process solves the rectified SDE~\eqref{eq:def_SDE_near_a_saddle_point} with initial condition $Y_0 = f(X_\zeta)$ 
 (belonging to $\chi$ and close to $f^{-1}(0,L)$  w.h.p.), and $W$ replaced by $W(\cdot + \zeta) -W(\zeta)$. 
 The second stage lasts while the process $Y$ stays within $\fR$
 (i.e., the process $X_{\zeta+t}$ stays within $f^{-1}(\fR)$), i.e., until time $\tau_\fR=\inf\{t\ge 0:\ Y_t\in \partial\fR\}$ (in terms of $Y$), or until time $\zeta+\tau_\fR$ (in terms of $X$).  
 The exit time $\tau_\fR$ and exit location $Y_{\tau_\Pi}$ are studied in detail in Section~\ref{sec:rectified}. In particular,
 w.h.p.,  $\tau_\fR<\infty$, 
 events 
 \begin{equation}
 \label{eq:exit-lateral}
 A_{\fR,\pm,\eps}=\left\{Y_{\tau_\fR}\in \{\pm R\}\times [-L',L']\right\}.
 \end{equation}
 get realized (i.e., the exit happens through one of the lateral sides of $\fR$), and
  $Y_{\tau_\Pi}$ 
  is close to  $(-R,0)$ or $(R,0)$, i.e., $X_{\zeta+\tau_\Pi}$ is close to $f(-R,0)$ or $f(R,0)$. 
 
 This, in turn, means that, w.h.p., the evolution of $X$ after $\zeta+\tau_\Pi$ is well-defined.
 The process $\widetilde X$ given by $\widetilde X_t=X_{\zeta+\tau +t}$ 
 solves SDE~\eqref{eq:basic-sde} with $W$ replaced by $W(\cdot+\zeta+\tau_\Pi)-W(\zeta+\tau_\Pi)$ and 
 satisfies $\widetilde X_0=X_{\zeta+\tau_\Pi}$. The third stage lasts for time $\tilde \tau=\inf\{t\ge 0: \widetilde X_t\in \partial D\}$.
 For this stage, we can study the exit time $\tilde\tau$ and exit location $\widetilde X_{\tilde \tau}$ using
 the results of Section~\ref{sec:finite-time-horizon-lemmas}. 
 In particular, we can conclude that  w.h.p.\ $\tilde \tau<\infty$ and $\widetilde X_{\tilde \tau}$ belongs to a small
 neighborhood of $q_\pm$. 
 
 \smallskip

 There are nonrigorous elements in this description of the three-stage evolution. 
 Let us convert them into rigorous statements. To that end, let us define
 the following curves:
\begin{gather*}
        \chi_0 = x_0 +[-c_0,c_0]v,\qquad \chi_1 = f^{-1}\left([-R,R]\times\{L\}\right),\\
    \chi_{2,\pm} = f^{-1}\left(\{\pm R\}\times [-L',L']\right),\qquad \chi_2= \chi_{2,+}\cup \ \chi_{2,-},
    \\ \chi_{3,\pm} = q_\pm+[-1,1]v_+,\qquad \chi_3= \chi_{3,+}\cup \ \chi_{3,-},
\end{gather*}
where the constant $c_0\in(0,1)$ is chosen  to ensure that the deterministic flow $(\flow^t)_{t\ge0}$ transports $\chi_0$ into $\mathring\chi_1$. Note that $\chi_{2}$ is transported by  $(\flow^t)_{t\ge0}$ into $\mathring\chi_{3}$ due to  the part of condition~\ref{setting:conjugacy} on transport from $U$. We also define $t_x=\min\{t:\ \flow^tx\in \chi_1\}$, 
$\phi(x)=\flow^{t_x}x$ 
for $x\in \chi_0$, and $\tilde t_x=\min\{t:\ \flow^tx\in \chi_{3}\}$, $\tilde \phi(x)=\flow^{\tilde t_x}x$  for $x\in \chi_{2}$. It is easy to see that \begin{gather}
    \phi(x_0 )= f^{-1}(0,L), \label{eq:flow_x_0to...}
    \\
\tilde \phi\left(q_{\fR,\pm}\right) = q_\pm, \label{eq:flow...to_q_+}
\end{gather}
where 
\begin{align}\label{eq:q_fR,pm}
    q_{\fR,\pm} = f^{-1}(\pm R,0).
\end{align}

We will prove the following lemma in Section~\ref{subsec:3-stages}:

\begin{lemma}  \label{lem:3-stages}
The following holds  w.h.p.\ under~$\Pp^z$, uniformly in $z\in\chi_0$:
\begin{align}
    X_\zeta\in \chi_1, \qquad Y_0\in [-R,R]\times\{L\}, \label{eq:X_zeta_in_chi}\\
    Y_{\tau_\fR}\in \{\pm R\}\times [-L',L'],\qquad  X_{\zeta+\tau_\Pi}=\widetilde{X}_0 \in\chi_{2}, \label{eq:X_zeta_tau}\\
    X_{\tau} =
    \widetilde X_{\tilde \tau} \in \chi_{3},  \label{eq:tilde_X=X}\\
     \tau = \zeta + \tau_\fR + \tilde \tau,\label{eq:tau=sigma+tau_R+t_tau}
\end{align}
and for every $\vk>\frac{1}{2}$
\begin{gather}
    |X_\zeta - \phi(X_0)|\leq \eps l_\eps^{\vk}, \label{eq:|X_zeta-phi(X_0)|}
    \\
    |\widetilde X_{\tilde\tau} - \tilde \phi(\widetilde X_0)|\leq \eps l_\eps^{\vk}. \label{eq:|t_X_t_tau-phi(t_X_0)|}
\end{gather}
In addition,
\begin{equation}
\label{eq:A_Pi-vs-A}
\sup_{z\in \chi_0} \Pp^z(A_{\pm, \e}\triangle A_{\fR,\pm, \e})=o_e(1).
\end{equation}
\end{lemma}

In the proofs below we will combine 
the finite time horizon results obtained in Section~\ref{sec:finite-time-horizon-lemmas}
with the {\it rectified coordinates} versions of the lemmas proved in Section~\ref{sec:rectified}. 
 In our three-stage analysis, we will obviously rely on the strong Markov property for diffusions without mentioning it explicitly.

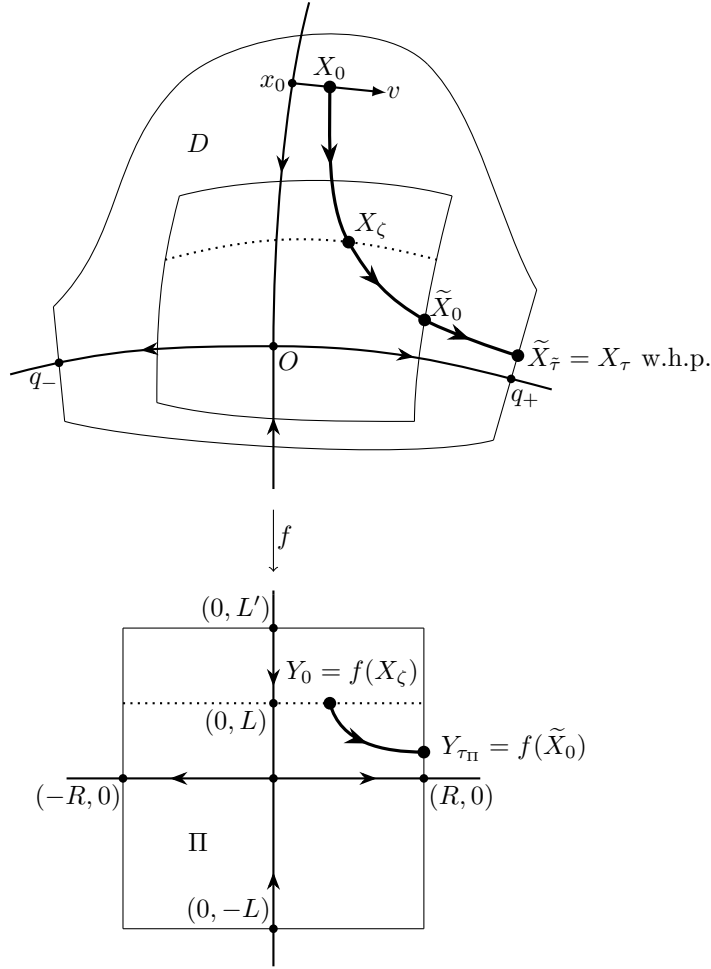
\begin{figure}[ht]
\centering
\begin{tikzpicture}

    \node (0) at (0, 0) {};
    \node (1) at (-3.5, -0.375) {};
    \node (2) at (3.7, -0.575) {};
    \node (3) at (0.475, 4.575) {};
    \node (4) at (0, -1.9) {};
    \node (5) at (-2.925, 0.525) {};
    \node (6) at (-2.775, -1) {};
    \node (7) at (3.5, 0.75) {};
    \node (8) at (2.925, -1.25) {};
    \node (9) at (-1.25, 3.5) {};
    \node (10) at (2.05, 3.75) {};
    \node (11) at (-1.25, 2) {};
    \node (12) at (-1.55, -0.75) {};
    \node (13) at (1.875, -1) {};
    \node (14) at (2.375, 2) {};
    \node (15) at (0.25, 3.5) {};
    \node (16) at (1.5, 3.375) {};
    \node (17) at (0.75, 3.449) {};
    \node (18) at (1, 1.385) {};
    \node (19) at (2.01, 0.35) {};
    \node (20) at (3.25, -0.125) {};
    \node (21) at (2, -2.25) {};
    \node (22) at (0, -2.175) {};
    \node (23) at (0, -3) {};
    \node (24) at (0, -5.75) {};
    \node (25) at (0, -3.25) {};
    \node (26) at (0, -8.25) {};
    \node (27) at (-2.75, -5.75) {};
    \node (28) at (2.75, -5.75) {};
    \node (29) at (-2, -4.75) {};
    \node (30) at (-2, -7.75) {};
    \node (31) at (2, -7.75) {};
    \node (32) at (2, -4.75) {};
    \node (33) at (-2, -3.75) {};
    \node (34) at (2, -3.75) {};
    \node (35) at (-1.44, 1.15) {};
    \node (36) at (2.175, 1.15) {};
    \node (37) at (0.75, -4.75) {};
    \node (38) at (2, -5.4) {};
    \node (39) at (0, -3.3) {};
    \node (40) at (0, -3.75) {};
    \node (41) at (0, -4.75) {};
    \node (42) at (2, -5.75) {};
    \node (43) at (-2, -5.75) {};
    \node (44) at (0, -7.75) {};
    \node (45) at (-2.85, -0.22) {};
    \node (46) at (3.16, -0.435) {};
 
    \node at (0, 3.5) {$x_0$};
    \node at (1.6, 3.375) {$v$};
    \node at (0.75, 3.7) {$X_0$};
    \node at (1.3, 1.6) {$X_\zeta$};
    \node at (2.3, 0.55) {$\widetilde X_{0}$};
    \node at (4.6, -0.125) {$\widetilde X_{\tilde\tau} = X_\tau \text{ w.h.p.}$};
    \node at (0.15, -2.55) {$f$};
    \node at (1.05, -4.35) {$Y_0 = f(X_\zeta)$};
    \node at (3.2, -5.3) {$Y_{\tau_\Pi} = f(\widetilde X_{0})$};
    \node at (0.2, -0.2) {$O$};
    \node at (-1, 2.7) {$D$};
    \node at (-1, -6.6) {$\Pi$};
    \node at (-0.5, -3.5) {$(0,L')$};
    \node at (-0.5, -5) {$(0,L)$};
    \node at (2.5, -6) {$(R,0)$};
    \node at (-2.6, -6) {$(-R,0)$};
    \node at (-0.6, -7.5) {$(0,-L)$};
    \node at (-3.05, -0.47) {$q_-$};
    \node at (3.35, -0.7) {$q_+$};

\foreach \n in {0, 15, 24, 40, 41, 42, 43, 44, 45, 46}
        \node at (\n)[circle,fill,inner sep=1.15pt]{};
    \foreach \n in {17, 18, 19, 20, 37, 38}
        \node at (\n)[circle,fill,inner sep=1.7pt]{};

    \draw [thick, postaction={on each segment={mid arrow=black}}] [in=90, out=-105, looseness=0.75] (3.center) to (0.center);
    \draw [thick, postaction={on each segment={mid arrow=black}}] (4.center) to (0.center);
    \draw [thick, postaction={on each segment={mid arrow=black}}] [in=15, out=180] (0.center) to (1.center);
    \draw [thick, postaction={on each segment={mid arrow=black}}] [in=165, out=360] (0.center) to (2.center);
    \draw (5.center) to (6.center);
    \draw (7.center) to (8.center);
    \draw [in=-135, out=45] (5.center) to (9.center);
    \draw [in=135, out=45, looseness=0.75] (9.center) to (10.center);
    \draw [in=120, out=-45, looseness=0.75] (10.center) to (7.center);
    \draw [in=-165, out=-15, looseness=0.50] (6.center) to (8.center);
    \draw [in=90, out=-105] (11.center) to (12.center);
    \draw [in=-180, out=-15, looseness=0.75] (12.center) to (13.center);
    \draw [in=165, out=15, looseness=0.75] (11.center) to (14.center);
    \draw [in=90, out=-105, looseness=0.75] (14.center) to (13.center);
    \draw [thick, -latex](15.center) to (16.center);
    \draw [very thick, postaction={on each segment={mid arrow=black}}] [in=120, out=-90, looseness=0.75] (17.center) to (18.center);
    \draw [very thick, postaction={on each segment={mid arrow=black}}] [in=150, out=-60] (18.center) to (19.center);
    \draw [very thick, postaction={on each segment={mid arrow=black}}] [in=165, out=-30, looseness=0.50] (19.center) to (20.center);
    \draw [-to] (22.center) to (23.center);
    \draw [thick] (25.center) to (39.center);
    \draw [thick, postaction={on each segment={mid arrow=black}}] (39.center) to (24.center);
    \draw [thick, postaction={on each segment={mid arrow=black}}] (26.center) to (24.center);
    \draw [thick, postaction={on each segment={mid arrow=black}}] (24.center) to (27.center);
    \draw [thick, postaction={on each segment={mid arrow=black}}] (24.center) to (28.center);
    \draw (33.center) to (30.center);
    \draw (30.center) to (31.center);
    \draw (31.center) to (34.center);
    \draw (34.center) to (33.center);
    \draw [thick, dotted] (29.center) to (32.center);
    \draw [thick, dotted, in=165, out=15] (35.center) to (36.center);
    \draw [very thick, postaction={on each segment={mid arrow=black}}] [in=180, out=-75] (37.center) to (38.center);

\end{tikzpicture}

\caption{Dynamics in three stages.}
\label{fig:dyn_3_stages}

\end{figure}

 \subsubsection{Proof of Lemma~\ref{lem:3-stages}}\label{subsec:3-stages}

In this proof we shorten ``w.h.p.\ under~$\Pp^z$ uniformly in $z\in \chi_0$'' to ``w.h.p.''

Lemma~\ref{lem:hit_loc} applied to the process $X$ traveling from $\chi_0$ to $\chi_1$ implies that \eqref{eq:X_zeta_in_chi} and \eqref{eq:|X_zeta-phi(X_0)|} hold w.h.p. Therefore, we can apply Lemma~\ref{lem:hp-events-in-rectangle} and conclude that
$A_{\fR,+,\eps}\cup A_{\fR,-,\eps}$ happens w.h.p.  Therefore,  \eqref{eq:X_zeta_tau} holds w.h.p.

Applying Lemma~\ref{lem:hit_loc} on each of the disjoint events $A_{\fR,+,\eps}$, and $A_{\fR,-,\eps}$ to the process $\widetilde X$ traveling
between $\chi_2$ and $\chi_3$, 
we obtain that \eqref{eq:tilde_X=X}, \eqref{eq:|t_X_t_tau-phi(t_X_0)|} and \eqref{eq:A_Pi-vs-A} hold w.h.p. Identity \eqref{eq:tau=sigma+tau_R+t_tau} simply computes
the total time spent by the process $X$ in all three stages.
 \epf

 \subsubsection{Proof of Lemma~\ref{lem:exit-straight}}For initial conditions in $x_0+\eps^{\alpha}K_{\vk}(\eps)v$ (which is a subset of $\chi_0$ for small $\e$), a strengthening of~\eqref{eq:X_zeta_in_chi}
follows from Lemmas~\ref{lem:hit_loc},~\ref{lem:bi-lip} and the smoothness of $f$:
\begin{lemma}\label{lem:Y_0}
For every $\vk >0$ and every $\alpha\in(0,1]$, there is $\vk'>0$ such that under~$\Pp^{x_0+\eps^\alpha x v}$
  w.h.p., uniformly in $x\in K_\vk(\eps)$, the following holds:
\begin{gather}
    |X_\zeta - f^{-1}(0,L)|\leq \eps^\alpha l_\eps^{\vk'}, \notag
\\
 \label{eq:Y_at_entrance}   Y_0\in (\eps^\alpha K_{\vk'}(\eps))\times \{L\}.
\end{gather}
\end{lemma}

Lemma~\ref{lem:Y_0} allows us to apply Lemma~\ref{lem:exit-straight} \irc{} (proved in Section \ref{sec:rec-lem:exit-straight}), so recalling~\eqref{eq:exit-lateral} and using
the smoothness of $f^{-1}$ and the identity $\widetilde X_{0} = f^{-1}(Y_{\tau_\fR})$, 
we obtain the
following:

\begin{lemma} \label{lem:outcome-of-stage-2}
Let $\vk >0$ and $\alpha\in(0,1]$.  Then under~$\Pp^{x_0+\eps^\alpha x v}$,
  w.h.p.\  uniformly in $x\in K_\vk(\eps)$ the event
  $A_{\fR,-,\eps}\cup A_{\fR,+,\eps}$
happens  and, moreover (for all sufficiently large  $\vk'>0$),
\begin{gather}
  Y_{\tau_\fR} \in \{-R,R\} \times  (\eps^{\alpha'}K_{\vk'}(\eps)), \notag
  \\
  \label{eq:exit-fR}
  |\widetilde  X_0 - f^{-1}(R,0)| \wedge  |\widetilde  X_0 - f^{-1}(-R,0)|  \leq \eps^{\alpha'} l_\eps^{\vk'}.
\end{gather}
\end{lemma}
Relation~\eqref{eq:exit-fR} allows to
 apply Lemmas~\ref{lem:hit_loc},~\ref{lem:bi-lip} to $\widetilde X$ strengthening 
 relation~\eqref{eq:tilde_X=X}  of Lemma~\ref{lem:3-stages} and obtaining relations~\eqref{eq:stronger_exit-straight1},~\eqref{eq:stronger_exit-straight} of Lemma~\ref{lem:exit-straight}.
 \epf

\subsubsection{Proof of Lemma~\ref{lem:exit-on-the-same-whp}}

For initial conditions in $x_0+ [c_0,1] v$, we can use the last part of assumption~\ref{setting:geometry-domain} and Lemma~\ref{lem:hit_loc} to show that $A_{-,\eps}$ happens w.l.p.\ under $\Psc^x$ uniformly over those initial conditions. So it suffices to consider only 
initial conditions in $ x_0+  \eps[l_\eps^\vk, c_0\eps^{-1}]v=x_0+  [\eps l_\eps^\vk, c_0]v\subset \chi_0$.

Using \eqref{eq:|X_zeta-phi(X_0)|} of Lemma~\ref{lem:3-stages}, the smoothness of $f$,  for an arbitrary $\vk''$,
we can find
$\vk$ large enough to guarantee that 
 $Y^1_0>\eps l^{\vk''}_\eps$ w.h.p.\ under~$\Psc^x$, uniformly in $x\in [l_\eps^\vk, c_0\eps^{-1}]$. 
 Lemma~\ref{lem:exit-on-the-same-whp} \irc\   (proved in Section~\ref{sec:proofof-lem:exit-on-the-same-whp}) implies that $A_{\fR,-,\eps}$ happens w.l.p. Now, 
 applying \eqref{eq:A_Pi-vs-A} of Lemma~\ref{lem:3-stages},
 we conclude that  and  $A_{-,\eps}$ happens w.l.p.\ (uniformly in $x \in [l_\eps^\vk, \eps^{-1}]$). \epf

\subsubsection{Proof of Lemma~\ref{lem:postive-limit-prob-if-close-to-manifold}}

Uniformly in $x\in K_\vk(\eps)$, we have, for some $\vk',\vk''>0$,
\begin{align}\notag
\Psc^x(A_{-,\e})&=\Psc^x\left(A_{-,\e}\cap A_{\fR,-,\e}\cap\{ Y_{\tau_\fR}^2 \in \eps^{\alpha'}K_{\vk''}(\eps)\}\right)
\\&\qquad\qquad+
\Psc^x\left(A_{-,\e}\cap A_{\fR,+,\e}\cap\{ Y_{\tau_\fR}^2 \in \eps^{\alpha'}K_{\vk''}(\eps)\}\right)+o_e(1)
\notag
\\
\notag
&=\Psc^x\left(A_{\fR,-,\e}\cap\{ Y_{\tau_\fR}^2 \in \eps^{\alpha'}K_{\vk''}(\eps)\}\right)+o_e(1)
\\
\notag
&=\Psc^x(A_{\fR,-,\e})+o_e(1)
\\
\notag
&=\Psc^x\left(A_{\fR,-,\e}\cap \{Y_0 \in (\eps K_{\vk'}(\eps))\times \{L\}\}\right)+o_e(1)
\\
\label{eq:P(A)_4.4}
&=
 \E^{x_0 + \eps x v}\left[ \Pp(A_{\fR,-,\eps}|Y_0) \ONE_{Y_0 \in (\eps K_{\vk'}(\eps))\times \{L\}}\right]+o_e(1).
\end{align}
Here the first identity follows from Lemma~\ref{lem:outcome-of-stage-2}, the second one from 
 \eqref{eq:A_Pi-vs-A} of Lemma~\ref{lem:3-stages}, the third one from 
Lemma~\ref{lem:outcome-of-stage-2}, the fourth one from Lemma~\ref{lem:Y_0}, and the last one is simply a disintegration with respect to $Y_0$.

To compute the expectation in \eqref{eq:P(A)_4.4}, we use Lemma~\ref{lem:postive-limit-prob-if-close-to-manifold} \irc{}
(proved in Section~\ref{sec:rec-lem:postive-limit-prob-if-close-to-manifold})
and obtain for some $s>0$:
\begin{align}\label{eq:sup_y_P(A)}
     \sup_{y\in K_{\vk'}(\eps)}\left|\Pp\left(A_{\fR,-,\eps}|Y_0 = (\eps y,L)\right) - \psi_s(-y)\right|=\smallo{\eps^\delta}.
\end{align}

To study the asymptotics of $\eps^{-1}Y^1_0 =\eps^{-1}f^1(X_\zeta)$ as $\eps \to 0$, where $f^1$ is the first coordinate of $f$, we will  apply Lemma~\ref{lem:expansion} with $\chi=\chi_1$,  $T=t_{x_0}$. Using $A, \NProj$ introduced in that lemma to define $\bar c = \nabla f^1(\phi(x_0))\cdot(Av)$, $\overline \NProj = \nabla f^1(\phi (x_0))\cdot \NProj$, and using \eqref{eq:flow_x_0to...} to see that $f^1(\phi (x_0))=0$,
 we obtain, due to  the smoothness of~$f^1$, that there is $\eta>0$ such that w.h.p.
\begin{align}\label{eq:eps^-1Y_0}
    |\eps^{-1}Y^1_0 - (\bar c x +\overline \NProj )|\leq  \eps^\eta.
\end{align}

Combining \eqref{eq:P(A)_4.4},~\eqref{eq:sup_y_P(A)}, and \eqref{eq:eps^-1Y_0}, choosing sufficiently large $\vk'>0$, using
 the Gaussianity of $\overline \NProj$, and the fact that $\psi_s$ is bounded and Lipschitz, we obtain that, for some $\delta'>0$,
\begin{align*}
    \sup_{x\in K_\vk(\eps)}\left|\Psc^{x}(A_{-,\eps}) - \E \psi_s(-\bar c x-\overline \NProj)\right| = \smallo{\eps^{\delta'}}.
\end{align*}
Since $\overline \NProj$ is centered and Gaussian, the function $x\mapsto \E\psi_s(-\bar c x-\overline \NProj)$ is  given by $x\mapsto\psi_{s'}(-x)$ for some $s'>0$ and thus the proof is complete. \epf

\subsubsection{Proof of Lemma~\ref{lem:alpha<1-and-rho-alpha<1-- concentration at exit}}

In this proof, we shorten ``w.h.p.\ under $\P^{x_0+\eps^\alpha xv}$, uniformly in $x\in K_\vk(\eps)$'' into ``w.h.p.''
Let us study three stages sequentially. First, using \eqref{eq:|X_zeta-phi(X_0)|} of Lemma~\ref{lem:3-stages} and the tameness of $\xi_\eps$, we have that \eqref{eq:|X_zeta-phi(X_0)|} holds w.h.p.\ for $X_0 = x_0 + \eps^\alpha \xi_\eps v$ and some $\vk>0$.
 Due to Lemma~\ref{lem:bi-lip}, the function $\phi$  is Lipschitz on its natural domain. Thus, \eqref{eq:flow_x_0to...} and the assumption that $\alpha<1 $ imply
\begin{align*}
    f\left(\phi(X_0)\right)\in [\eps^\alpha l_\eps^{-\vk_1}, \eps^\alpha l_\eps^{\vk_1}]\times \{L\},
\end{align*}
for some $\vk_1>0$. Since $f$ is Lipschitz and $Y_0= f(X_\zeta)$, the above two displays imply that, for some $\vk_2>0$,   w.h.p.\ the outcome of the first stage satisfies 
\begin{align*}
    Y_0^1 \in [\eps^\alpha l_\eps^{-\vk_2}, \eps^\alpha l_\eps^{\vk_2}].
\end{align*}

Combining this  with Lemma~\ref{lem:alpha<1-and-rho-alpha<1-- concentration at exit} \irc{} (proved in Section~\ref{sec:rec-lem:alpha<1-and-rho-alpha<1-- concentration at exit}), we obtain that
for some $\vk_3>0$, w.h.p.\ the outcome of the second stage satisfies
\begin{align*}
    Y_{\tau_\fR}\in \{R\}\times \eps^{\alpha\rho}[l_\eps^{-\vk_3},l_\eps^{\vk_3}].
\end{align*}
Using this, \eqref{eq:|t_X_t_tau-phi(t_X_0)|}, the fact that $\widetilde X_0=f^{-1}(Y_{\tau_\fR})$, that $f$ is smooth and orientation-preserving (see \ref{setting:conjugacy}),
  property \eqref{eq:flow...to_q_+},
  the Lipschitzness of the function $z\mapsto  \tilde\phi(f^{-1}(z))$ (due to Lemma~\ref{lem:bi-lip}), and the assumption $\alpha\rho<1$, we obtain 
 that  for some $\vk_4>0$,
 $\widetilde X_{\tilde \tau}\in q_+ + \eps^{\alpha\rho}[l_\eps^{-\vk_4},l_\eps^{\vk_4}]v_+$ w.h.p., which completes the proof.  \epf

\subsubsection{Proof of Lemma~\ref{lem:local-limit-theorem}}
      The lemma was proved \irc{} in Section~\ref{sec:rec-lem:local-limit-theorem}.
      We prove the lemma in the following order: part~\eqref{item:exit-below-unlikely},  part~\eqref{item:large-negative-values-w.l.p.}, part~\eqref{item:power-asymp}. In this proof,  ``w.h.p.'' is understood as w.h.p.\ under $\Psc^x$ uniformly in $x\in K_\vk(\eps)$ for a fixed $\vk>0$, and all $o_\e(1)$ are understood to be uniform in $x\in K_\vk(\eps)$.

Part~\eqref{item:exit-below-unlikely}.
Let 
\begin{gather*}
    H = A_{+,\eps}\cap\left\{\widetilde X_{\tilde \tau} \notin q_++[-\eps l^{\vk'}_\eps,+\infty)v_+\right\},
    \\
    E = \left\{Y_{ \tau_\fR} \not\in \{R\}\times [-\eps l^{\vk''}_\eps,+\infty)\right\},
\end{gather*}
where  $\vk''$ is to be chosen later.
In view of \eqref{eq:tilde_X=X}, it suffices to show that $\Psc^x(H) = o_e(1)$. 
Identity~\eqref{eq:A_Pi-vs-A} of Lemma~\ref{lem:3-stages} implies
$\Psc^x(A_{+,\eps}\cap A_{\fR,-,\eps})=o_e(1)$.  Also, $f^{-1}(Y_{\tau_\fR}) = \widetilde X_0 \in \chi_{2,+}$ on $A_{\fR,+,\eps}$. Hence
\begin{align}
    \Psc^x(H)&\leq \Psc^x(H\cap A_{\fR,+,\eps})+o_e(1) \notag
    \\
    & \leq \Psc^x(H\cap E^c\cap \{\widetilde X_0 \in \chi_{2,+}\}) +  \Psc^x(A_{\fR,+,\eps}\cap E) +o_e(1) \notag
    \\
    & \leq \Psc^x(H\cap E^c\cap \{\widetilde X_0 \in \chi_{2,+}\}) + o_e(1), \label{eq:Q^x(H)<}
\end{align}
where the last inequality follows from Lemma~\ref{lem:local-limit-theorem}~\eqref{item:exit-below-unlikely} \irc\ for sufficiently large $\vk''$. On $E^c\cap \{\widetilde X_0 \in \chi_{2,+}\}$, we have w.h.p.\
\begin{align*}
    f(\widetilde X_0) \in f(q_{\fR,+}) + \{0\}\times (-\infty, -\eps l^{\vk''}_\eps).
\end{align*}
Using \eqref{eq:|t_X_t_tau-phi(t_X_0)|}, the above display, \eqref{eq:flow...to_q_+}, the fact that $f$ and  $\tilde\phi$ are orientation-preserving diffeomorphisms
(see Lemma~\ref{lem:bi-lip}), we obtain that w.h.p.\
\begin{align*}
    \widetilde X_{\tilde \tau} \in q_+ + (-\infty, -C_1\eps l^{\vk''}_\eps + C_2\eps l^{\tilde\vk} )v_+
\end{align*}
on $E^c\cap \{\widetilde X_0 \in \chi_{2_+}\}$
for some constants $C_1,C_2>0$. Choosing $\vk''$ sufficiently large, we can use this to ensure w.h.p.
\begin{align}\label{eq:tilde_X_in_q_+...}
    \widetilde X_{\tilde \tau} \in q_+ + (-\infty, -\eps l^{\vk'}_\eps )v_+
\end{align}
on $E^c\cap \{\widetilde X_0 \in \chi_{2_+}\}$.
Then, the definition of $H$ implies that
\begin{align*}
    \Psc^x(H\cap E^c\cap \{\widetilde X_0 \in \chi_{2_+}\}) = o_e(1).
\end{align*}
Using this in \eqref{eq:Q^x(H)<}, we obtain $\Psc^x(H) = o_e(1)$ thus
completing the proof of part~\eqref{item:exit-below-unlikely}.

Part~\eqref{item:large-negative-values-w.l.p.}. 
Due to \eqref{eq:tilde_X=X}, up to an $o_e(1)$ error uniformly in $x\in K_\vk(\eps)$, the left-hand side of \eqref{eq:tau_cvg_in_prob_quantified} can be rewritten as
\begin{align}
\label{eq:proving_time_concentration}
    \Psc^x\left\{\left|\frac{\tau}{\frac{\beta}{\mu} l_\eps }-1\right|>\delta,\ A_{+,\eps},\ H^c \right\},
\end{align}
where
\[
 H=\left\{\widetilde X_{\tilde \tau} \in q_+ + \left(\eps^\beta l^{\vk'}_\eps, \infty\right)v_+\right\}.
\]
Let us estimate this expression.
Using arguments similar to those for~\eqref{eq:tilde_X_in_q_+...}, we can choose $\vk''>0$ sufficiently large to ensure
\begin{align}
\label{eq:Q(H_2-H_3)}
    \Psc^x\left(E\cap H^c\right)=o_e(1),
\end{align}
where
\begin{align*}
    E = \left\{Y_{\tau_\fR}\in \{R\}\times \left(\eps^\beta l^{\vk''}_\eps, \infty\right)\right\}.
\end{align*} 
Now, using \eqref{eq:A_Pi-vs-A} of Lemma~\ref{lem:3-stages} and
\eqref{eq:Q(H_2-H_3)}, representing $\tau$ via~\eqref{eq:tau=sigma+tau_R+t_tau}, 
and applying estimate~\eqref{eq:whp-bounded-time} of Lemma~\ref{lem:hit_loc} to times $\zeta$ and $\tilde\tau$, 
we can bound the expression in~\eqref{eq:proving_time_concentration} by
\begin{align*}
    \Psc^x\left\{\frac{\tau_\fR+C}{\frac{\beta}{\mu} l_\eps }-1>\delta,\ A_{\fR,+,\eps},\ E^c \right\}+\Psc^x\left\{\frac{\tau_\fR}{\frac{\beta}{\mu} l_\eps }-1<-\delta,\ A_{\fR,+,\eps},\ E^c \right\}+o_e(1),
\end{align*}
for some constant $C>0$.
Using \eqref{eq:Q(tau+C>...)} and Lemma~\ref{lem:local-limit-theorem}~\eqref{item:large-negative-values-w.l.p.} \irc, we conclude that the quantity above is $O(\eps^{\frac{\beta(1+\delta)}{\rho}-1})$, completing the proof of part~\ref{item:large-negative-values-w.l.p.}.

Part~\eqref{item:power-asymp}. In this part, we abbreviate ``w.h.p.\ in $\Psc^x$ uniformly in $x\in K_\vk(\eps)$ and $[a,b]\in K_{\vk'}(\eps)$'' to ``w.h.p.''; also all estimates involving $o(\cdot)$ are understood to hold  uniformly in $x\in K_\vk(\eps)$ and $[a,b]\in K_{\vk'}(\eps)$.
We start by analyzing the third stage. 
Setting
\begin{align*}
    E = \left\{Y_{\tau_\fR}\in \{R\}\times \eps^\beta K_{\vk''}(\eps)\right\},
\end{align*}
and using arguments similar to those for \eqref{eq:tilde_X_in_q_+...}, for sufficiently large $\vk''>0$, we have that
\begin{align}\label{eq:Q^x(...cap_E^c)}
    \Psc^{x}\left( \left\{\widetilde X_{\tilde \tau}\in q_++\eps^\beta[a,b]v_+\right\}\cap E^c\right) = o_e(1).
\end{align}
On the event $E$, we rewrite $\widetilde X_0 = f^{-1}(Y_{\tau_\fR})$ as
\begin{align*}
    \widetilde X_0 = q_{\fR,+} + \eps^\beta \frac{f^{-1}(Y_{\tau_\fR})-q_{\fR,+}}{\eps^\beta}
\end{align*}
where $q_{\fR,+}$ is given in \eqref{eq:q_fR,pm}.

We apply Lemma~\ref{lem:expansion} to the transition from $\chi_2$ to $\chi_3$ with $\chi_3,q_{\fR,+},t^{3}_{q_{\fR,+}},\beta$ substituted for $\chi,x_0,T,\alpha$ therein.
Let $A,M,r$ be given by that lemma. We set $\widetilde \NProj =\frac{ v_+}{|v_+|^2}\cdot \NProj$, $\tilde r_\eps = 
\frac{ v_+}{|v_+|^2}\cdot r_{\widetilde X_0,\eps}$, and define $\tilde f$ on a suitable subset of real numbers via its inverse
\begin{align*}
    \tilde f^{-1}(y) =\frac{ v_+}{|v_+|^2}\cdot A(f^{-1}(R,y) - q_{\fR,+}).
\end{align*}
Assumption~\ref{setting:conjugacy} implies that $\tilde f$ is an increasing $C^5_\mathrm{b}$-diffeomorphism,

Using \eqref{eq:flow...to_q_+} and Lemma \ref{lem:expansion}, we have on $E$,
\begin{align*}
   \frac{ v_+}{|v_+|^2}\cdot(\widetilde X_{\tilde \tau}- q_+)= \tilde f^{-1}(Y^2_{\tau_\fR}) + \eps\widetilde \NProj + \eps^\beta\tilde r_\eps,   
\end{align*}
where $\widetilde \NProj$ is a centered Gaussian variable independent of $Y_{\tau_\fR}$, and the r.v.\ $\tilde r_\eps$ satisfies $|\tilde r_\eps|\leq \eps^\eta$ w.h.p. for some $\eta>0$.
Then, we can write 
\begin{align*}
    \Psc^x&\left\{\widetilde X_{\tilde \tau}\in q_++\eps^\beta [a, b] v_+\right\}  
    \\ 
    &=\Psc^x\left(\left\{\tilde f^{-1}(Y^2_{\tau_\fR})+\eps\widetilde M+\eps^\beta\tilde r_\eps \in \eps^\beta[a,b],\, Y^1_{\tau_\fR}=R\right\}\cap E\right) + o_e(1)
    \\
     &= \Psc^x\left(\left\{Y_{\tau_\fR}\in\{R\}\times [\tilde f(\eps^\beta a-\eps\widetilde \NProj-\eps^\beta\tilde r_\eps),\  \tilde f(\eps^\beta b-\eps\widetilde \NProj-\eps^\beta\tilde r_\eps)]\right\}\cap E\right) + o_e(1), 
\end{align*}
where we used \eqref{eq:Q^x(...cap_E^c)} in the first equality, and the monotonicity of  $\tilde f$  in the second identity.
Let $\tilde c= \tilde f'(0)$. Since $\tilde f(0)=0$, there are deterministic constants $C,\eta_1>0$ such that w.h.p.
\begin{align*}
    \left|\frac{\tilde f(\eps^\beta a-\eps\widetilde \NProj-\eps^\beta\tilde r_\eps)}{\eps^\beta}-\tilde c\left(a-\eps^{1-\beta}\widetilde \NProj\right)\right|\leq \tilde c |\tilde r_\eps|+ C\eps^\beta\left(a-\eps^{1-\beta}\widetilde \NProj - \tilde r_\eps\right)^2\leq \eps^{\eta_1}.
\end{align*}
 Here, in the last inequality, we used the Gaussian tail of $\widetilde \NProj$ and the bound on $|\tilde r_\eps|$.  A similar estimate also holds for $a$ replaced by $b$. Set $\widehat \NProj = \widetilde \NProj \ONE_{\beta=1}$. For brevity, let us use the notation $\eqpm$ introduced in \eqref{eq:eqpm_notation}.
The above two displays yield that 
\begin{multline} \label{eq:3rd_dyn}
    \Psc^x\left\{\widetilde X_{\tilde \tau}\in q_++\eps^\beta [a, b] v_+\right\}
    \\ \eqpm \Psc^x
    \left\{Y_{\tau_\fR}\in\{R\}\times \eps^\beta\tilde c \left[a-\widehat \NProj \mp\eps^{\eta_1}, b-\widehat \NProj\pm \eps^{\eta_1}\right]\right\}
    + o_e(1),
\end{multline}
for some $\eta_2>0$, where we chose $\vk''$ sufficiently large and used the Gaussian tail of~$\widehat \NProj$ to drop the conditioning on $E$.

Next, we study the second stage of the dynamics and apply Lemma~\ref{lem:local-limit-theorem}~\eqref{item:power-asymp} \irc\ to see that  for some $\vk'',\delta,c>0,$ and $\nu\in \GoodMeasures$, uniformly in $y\in K_{\vk''}(\eps)$,
\begin{align}
    \eps^{-(\frac{\beta}{\rho}-1)}\Pp\left\{Y_{\tau_\fR}\in\{R\}\times \eps^\beta\tilde c \left[a-\widehat \NProj \mp\eps^{\eta_1}, b-\widehat \NProj\pm \eps^{\eta_1}\right]\,\Big|\,Y_0=(\eps y, L)\right\} \label{eq:2nd_dyn}
    \\
    = g_c(y)\E\nu\left(B_{\pm}-\widehat \NProj\right) + \smallo{\eps^\delta},\notag
\end{align}
where $B_\pm = [a\mp\eps^{\eta_1},b\mp\eps^{\eta_1}]$.

We want to evaluate the above with $y$ replaced by $\eps^{-1} Y_0^1$. To do so, we need to consider the dynamics in the first stage. 
Recall that~\eqref{eq:eps^-1Y_0} holds
w.h.p.\ for some $\eta>0$.
Using that $g_c$ is bounded and Lipschitz, properties \eqref{eq:GoodMeasures} and \eqref{eq:GoodMeasure_density} of $\nu$, Gaussian tails of $\overline M$ and $\widehat \NProj$, and the decay of $\bar r_\eps$, we can verify that  
\begin{align*}
    \left|\E^{x_0+\eps x v} g_c\left(\eps^{-1}Y^1_0\right)\E\nu\left(B_{\pm}-\widehat \NProj\right)- \E g_c\left(\bar c x +\overline M\right)\E\nu\left([a,b]-\widehat \NProj\right)\right|=\smallo{\eps^{\delta'}}
\end{align*}
for some $\delta'>0$. 
This together with~\eqref{eq:tilde_X=X}, \eqref{eq:3rd_dyn}, \eqref{eq:2nd_dyn}, and Lemma~\ref{lem:Y_0} completes the proof. \epf

\subsubsection{Proof of Lemma~\ref{lem:exit_time_is_log}}
The lemma follows from its version \irc{}
(proved in Section~\ref{sec:rec-lem:exit_time_is_log}) and
exactly the same argument based on~\eqref{eq:tau=sigma+tau_R+t_tau} as in the proof of Lemma~\ref{lem:local-limit-theorem}~\eqref{item:large-negative-values-w.l.p.}. \epf

\subsubsection{Proof of Lemma~\ref{lem:if-entrance-far-exit-far}}

For $x\in (c_0\eps^{-\alpha},\eps^{-\alpha}]$, we have $X_0 \in x_0 + (c_0,1]v$. In view of \ref{setting:geometry-domain}, applying Lemma~\ref{lem:hit_loc} to the transition from $x_0+(c_0,1]v$ to $q_++[-1,1]v_+$, we have $X_\tau \in q_++(c',1]v_+$ and thus $\xi'\ge\e^{-\alpha\rho}c'>l^{\vk'}_\eps$ for some $c'>0$ w.h.p.\ uniformly in $x\in(c_0\eps^{-\alpha},\eps^{-\alpha}]$.

For $x \in (l^\vk_\eps, c_0\eps^{-\alpha}]$, we use \eqref{eq:|X_zeta-phi(X_0)|} and \eqref{eq:|t_X_t_tau-phi(t_X_0)|} in Lemma~\ref{lem:3-stages} to obtain that
\begin{align*}
    \eps^{-\alpha}Y^1_0 \in \left(l^\vk_\eps\left(c_1 - c_3 \eps^{1-\alpha}l^{\vk''-\vk}_\eps\right),\ \eps^{-\alpha}\left(\left(c_2 + c_3 \eps l^{\vk''}_\eps\right)\wedge R\right)\right]
\end{align*}
for some $\vk''>\frac{1}{2}$ to be chosen and constants $c_1,c_2,c_3>0$, w.h.p.\ uniformly  in $x\in(l^\vk_\eps,c_0\eps^{-\alpha}]$, and that
\begin{align*}
    &\Pp^{x_0+\eps^\alpha x v}\left\{ \xi_\eps'\leq l_\eps^{\vk'},\ A_{+,\eps}\right\}
    \\ \leq\ & \Pp^{x_0+\eps^\alpha x v}\left\{Y_{\fR}\in \{R\}\times  c_4\eps^{\alpha'}(-\infty,\ l^{\vk'}_\eps(1+\eps^{1-\alpha'}l^{\vk''-\vk'}_\eps)]\right\}+o_e(1)
\end{align*}
uniformly in $x \in (l^\vk_\eps,c_0\eps^{-\alpha}]$ for some $c_4>0$. Choosing $\vk''$  sufficiently large, and then~$\vk$ sufficiently large, we can now deduce the desired result from these displays and Lemma~\ref{lem:if-entrance-far-exit-far} \irc{} proved in~Section~\ref{sec:rec-lem:if-entrance-far-exit-far}. \epf

\subsubsection{Proof of Lemma~\ref{lem:one-step_lower_bound}}
Using Lemma~\ref{lem:3-stages} and applying
Lemma~\ref{lem:hit_loc}  to the third stage, we obtain that for some $\vk''>0$,
\begin{align*}
    &\Pp^{x_0+\eps^{\alpha} xv}\left\{X_\tau \in q_++\eps^\beta\left(-\infty, -l^{\vk'}_\eps\right)v_+,\, A_{+,\eps}\right\}
    \\
    \leq\ & \Pp^{x_0+\eps^{\alpha} xv}\left\{Y_{\tau_\fR} \in {R}\times\left(-\infty, -\eps^\beta l^{\vk''}_\eps\right)\right\}+o_e(1)
\end{align*}
 uniformly in $x\in K_{\vk}(\eps)$.
 The desired result follows from this display, \eqref{eq:Y_at_entrance} of Lemma~\ref{lem:Y_0}, and Lemma~\ref{lem:one-step_lower_bound} \irc{}
 (proved in Section~\ref{sec:rec-lem:one-step_lower_bound}).

\subsubsection{Proof of Lemma~\ref{lem:exit_time_is_log_alpha<1}}

Using \eqref{eq:tau=sigma+tau_R+t_tau}, for any $\delta>0$, we have
\begin{align}
   & \Pp^{x_0 + \eps^{\alpha}xv_0}\left\{\left|\frac{\tau}{\frac{\alpha}{\lambda} l_\eps }-1\right|>3\delta\right\} \notag
    \\ \label{eq:sigma+tau_R+tilde_sigma}
    \leq\ & \Pp^{x_0 + \eps^{\alpha}xv_0}\left\{\zeta \geq \frac{\alpha\delta}{\lambda} l_\eps \right\} + \Pp^{x_0 + \eps^{\alpha}xv_0}\left\{\left|\frac{\tau_\fR}{\frac{\alpha}{\lambda} l_\eps }-1\right|>\delta\right\} \\ 
   & \qquad\qquad\qquad\qquad\qquad\qquad\qquad+ \Pp^{x_0 + \eps^{\alpha}xv_0}\left\{\tilde\tau \geq \frac{\alpha\delta}{\lambda} l_\eps \right\}.
     \notag
\end{align}
Lemma~\ref{lem:hit_loc} implies that $\zeta$ and $\tilde\tau$ are bounded by a positive constant w.h.p.\ uniformly in~$x\in K_\vk(\eps)$.
Hence, the first and third terms in \eqref{eq:sigma+tau_R+tilde_sigma} are $o_e(1)$.  Rewriting the second term in~\eqref{eq:sigma+tau_R+tilde_sigma} as
\begin{align*}
    \E^{x_0 + \eps^{\alpha}xv_0}\left[\Pp\left\{\left|\frac{\tau_\fR}{\frac{\alpha}{\lambda} l_\eps }-1\right|>\delta\ \bigg|\ Y_0 \right\}
    \right],
\end{align*}
and using Lemma~\ref{lem:Y_0} and Lemma~\ref{lem:exit_time_is_log_alpha<1} \irc{} 
(proved in Section~\ref{sec:rec-lem:exit_time_is_log_alpha<1}), we obtain that the second term in \eqref{eq:sigma+tau_R+tilde_sigma} is bounded from above by
\begin{align*}
    \Pp^{x_0 + \eps^{\alpha}xv_0}\left\{\left|\eps^{-\alpha}f^1(X_\zeta)\right|\leq\eps^{\delta'}\right\}+o_e(1)
\end{align*}
for some $\delta'>0$ uniformly in $x\in K_\vk(\eps)$. Using \eqref{eq:|X_zeta-phi(X_0)|} of Lemma~\ref{lem:3-stages}, \eqref{eq:flow_x_0to...}, the smoothness of $f\circ \phi$, and the fact that  $\frac{d}{dr}f^1\circ\phi(x_0+rv)\big \vert_{r=0}> 0$, we can bound the main term in the above display by $\ONE_{|x|\leq \eps^{\delta''}}+o_e(1)$ for some $\delta''>0$ uniformly in $x\in K_\vk(\eps)$, completing the proof. \epf

\subsection{Proof of Lemma~\ref{lem:typical_loc_lim}}\label{sec:pf_one-step_typical_trans}

We consider the dynamics in three stages as described in Section~\ref{sec:proofs-lemmas-original-coords}. We will use the notation $\phi, r$, etc.\ in the analysis of the first stage
and we will use~$\tilde \phi, \tilde r$,~etc.\ in the third stage. We shorten ``w.h.p.\ under $\P^{x_0+\eps^\alpha x v}$ uniformly in $x\in K_\vk(\eps)$ and $[a,b]\in K_{\vk'}(\eps)$'' to ``w.h.p.''

Applying Lemma~\ref{lem:scal_lim} and Remark~\ref{rem:initial_line} to the first stage, we have that, w.h.p.,
\begin{align}\label{eq:X_sigma=phi'...}
    X_\zeta = \phi(x_0 + (\eps^\alpha x + \eps \NProj)v)+\eps r_{\eps},
\end{align}
where $\NProj=\NProj'$ and $r_\eps = h'_{x,\eps}$ for $\NProj',h'$ given in Remark~\ref{rem:initial_line} (we suppress the dependence on~$x$ in the notation). Moreover,
\begin{align}\label{eq:r}
    |r_\eps|\leq \eps^{\eta'}, \quad\text{w.h.p.,}
\end{align}
for some $\eta'\in(0,1)$. 
Since $Y_0 = f(X_\zeta)$, using~\eqref{eq:X_sigma=phi'...}, we can write
\begin{align}\label{eq:Y_0=(y_eps,L)}
    Y_0= (\eps^\alpha y_\eps, L),
\end{align}
where
\begin{align}\label{eq:def_y_eps}
    y_\eps = \eps^{-\alpha} f^1(\phi(x_0+(\eps^\alpha x+\eps \NProj)v)+\eps r_\eps).
\end{align}

Since $f$ is orientation-preserving (see \ref{setting:conjugacy}), we can see from \eqref{eq:flow_x_0to...} that $s\mapsto f^1(\phi(x_0+sv))$ is nondecreasing in a neighborhood of $0$. For later use, we extend $f$ and $\phi$ as diffeomorphisms so that the function $s\mapsto f^1(\phi(x_0+sv))$ is nondecreasing on $\R$ and, moreover, its derivative is bounded above and below by positive constants.

Applying Lemma~\ref{lem:scal_lim} to the third stage, we get
\begin{align}\label{eq:tilde_X=phi(tilde_X_0)+...}
    \widetilde X_{\tilde \tau} = \tilde\phi(\widetilde X_0) + \eps (\widehat \NProj+\hat r_\eps)
\end{align}
w.h.p., where $\widehat \NProj = \NProj $ and $\hat r_\eps = h_{\widetilde X_0,\eps}$ for $\NProj,h$ given in that lemma. Moreover,
\begin{align}\label{eq:hat_r<}
    |\hat r_\eps|\leq \eps^{\hat \eta},\quad \text{w.h.p.,}
\end{align}
for some $\hat \eta>0$.

Recall $q_{\fR,+}$ in \eqref{eq:q_fR,pm}, and we set 
\begin{align}\label{eq:def_g(s)}
    g(s) =  \frac{v_+}{|v_+|^2}\cdot(\tilde\phi(f^{-1}(R,s))-\tilde\phi(q_{\fR,+})).
\end{align}
Since both $\tilde\phi$ and $f$ are diffeomorphisms we have that $g$ is invertible on $[-L',L']$, which contains the range of~$Y^2_{\tau_\fR}$. Due to \eqref{eq:flow...to_q_+} and the assumption that $f$ is orientation-preserving (see \ref{setting:conjugacy}), we can see that $g$ is nondecreasing and, moreover,
 its derivative is bounded below by a positive constant. For later use, we extend~$g$ smoothly to $\R$ preserving these properties.
 
Let  $\widetilde \NProj=\frac{v_+}{|v_+|^2}\cdot \widehat \NProj$ and $\tilde r_\eps=\frac{v_+}{|v_+|^2}\cdot \hat r_\eps$. Note that $\widetilde \NProj$ is a Gaussian r.v.\ and $\tilde r_\eps$ satisfies 
\begin{align}\label{eq:r<}
    |\tilde r_\eps|\leq \eps^{\eta''},\quad \text{w.h.p.,}
\end{align}
for some $\eta''>0$ (due to~\eqref{eq:hat_r<}). Using~\eqref{eq:tilde_X=phi(tilde_X_0)+...},~\eqref{eq:flow...to_q_+}, and \eqref{eq:def_g(s)}, we obtain that, uniformly in $x\in K_\vk(\eps)$,
\begin{align}
    &\Pp\left\{X_\tau \in q_++\eps^{\alpha'}[a,b]v_+\right\} = \Pp\left\{\widetilde X_{\tilde\tau} \in q_++\eps^{\alpha'}[a,b]v_+\right\}+o_e(1)\notag\\
    &= \Pp\left\{g(Y^2_{\tau_\fR})\in \left[\eps^{\alpha'}a-\eps(\widetilde \NProj+\tilde r_\eps),\eps^{\alpha'}b-\eps(\widetilde \NProj+\tilde r_\eps)\right],\ Y^1_{\tau_\fR}=L\right\}+o_e(1)\label{eq:Y_in_[a_eps,b_eps]}\\
    & = \Pp\left\{Y_{\tau_\fR}\in \{L\}\times \eps^{\alpha'}[a_\eps,b_\eps]\right\}+o_e(1),\notag
\end{align}
where
\begin{align}\label{eq:def_a_eps}
    a_\eps = \eps^{-\alpha'}g^{-1}\left(\eps^{\alpha'}a-\eps\left(\widetilde \NProj+\tilde r_\eps\right)\right),\\
    \label{eq:def_b_eps}
    b_\eps = \eps^{-\alpha'}g^{-1}\left(\eps^{\alpha'}b-\eps\left(\widetilde \NProj+\tilde r_\eps\right)\right).
\end{align}

Due to~\eqref{eq:Y_in_[a_eps,b_eps]} and~\eqref{eq:def_y_eps}, we can apply Proposition~\ref{prop:typical} to the dynamics in the second stage where $Y\in \fR$ evolves between times $0$ and $\tau_\fR$, with 
 $y_\eps$ and $[a_\eps,b_\eps]$ substituted for $y$ and $[a,b]$ in that proposition. Since there are four cases in Proposition~\ref{prop:typical}, we treat them separately here. We recall that $\frU$, $\NN$, and  $c$ are introduced just before the statement of Proposition~\ref{prop:typical}. 

\bigskip

\textbf{Case 1.} Let us consider the first case $\rho<1$. In this case, $\alpha'=\alpha\rho$.
Proposition~\ref{prop:typical}~\eqref{item:loc_lim_1} along with \eqref{eq:Y_0=(y_eps,L)}, \eqref{eq:Y_at_entrance} and \eqref{eq:Y_in_[a_eps,b_eps]} yields
\begin{align}\label{eq:P(X_tau)_to_be_comp}
    &\Pp\left\{X_\tau \in q_++\eps^{\alpha'}[a,b]v_+\right\} = P_\e
     + \smallo{\eps^\delta}
\end{align}
for some $\delta>0$, uniformly in $x\in K_\vk(\eps)$ and $[a,b]\in K_{\vk'}(\eps)$,
where 
\[
P_\eps=\Pp\left\{c|y_\eps+\eps^{1-\alpha}\frU|^\rho\in[a_\eps,b_\eps],\ y_\eps+\eps^{1-\alpha}\frU\geq 0\right\},
\]
with  $c = R^{-\rho}L$.

The next step is to get rid of $r_\eps$ and $\tilde r_\eps$ in our approximations, so that the only remaining randomness in the resulting approximations is Gaussian. The key properties to use are~\eqref{eq:r} and~\eqref{eq:r<}. We want to compare the right-hand side of~\eqref{eq:P(X_tau)_to_be_comp} to
\begin{align*}\widetilde P_\eps= \Pp\left\{c|\tilde y_\eps+\eps^{1-\alpha}\frU|^\rho\in\left[\tilde a_\eps,\tilde b_\eps\right],\ \tilde y_\eps+\eps^{1-\alpha}\frU\geq 0\right\},
\end{align*}
where
\begin{gather}
    \tilde y_\eps = \eps^{-\alpha} f^1(\phi(x_0+(\eps^\alpha x+\eps \NProj)v)),\label{eq:tilde_y_eps}
    \\
    \tilde a_\eps = \eps^{-\alpha'}g^{-1}\left(\eps^{\alpha'}a-\eps \widetilde \NProj\right)\label{eq:tilde_a_eps},
    \\
    \tilde b_\eps = \eps^{-\alpha'}g^{-1}\left(\eps^{\alpha'}b-\eps \widetilde \NProj\right)\label{eq:tilde_b_eps}.
\end{gather}
We can write
\begin{align*}P_\eps=\Pp\left\{\eps^{1-\alpha}\frU \in A_\eps\right\},\qquad \widetilde P_\eps=\Pp\left\{\eps^{1-\alpha}\frU \in \widetilde A_\eps\right\},
\end{align*} 
where
\begin{align*}
     A_\eps &= \left[\left((c^{-1} a_\eps)\vee0\right)^\frac{1}{\rho} -  y_\eps, \ \left((c^{-1}  b_\eps)\vee0\right)^\frac{1}{\rho} -  y_\eps\right]
    ,\\
    \widetilde A_\eps &= \left[\left((c^{-1}\tilde a_\eps)\vee0\right)^\frac{1}{\rho} - \tilde y_\eps, \ \left((c^{-1} \tilde b_\eps)\vee0\right)^\frac{1}{\rho} - \tilde y_\eps\right]
    .
\end{align*}
Comparing~\eqref{eq:def_y_eps},~\eqref{eq:def_a_eps},~\eqref{eq:def_b_eps} with~\eqref{eq:tilde_y_eps},~\eqref{eq:tilde_a_eps},~\eqref{eq:tilde_b_eps}, using the Lipschitzness of various functions involved, along with~\eqref{eq:r} and~\eqref{eq:r<}, we can verify that
\begin{gather}
    |y_\eps - \tilde y_\eps|\leq C\eps^{1-\alpha}|r_\eps|\leq C\eps^{1-\alpha+\eta'},\quad \text{w.h.p.,}\label{eq:y-tilde_y}\\
    |a_\eps - \tilde a_\eps|,\ |b_\eps - \tilde b_\eps|\leq C \eps^{1-\alpha'}|\tilde r_\eps|\leq C\eps^{1-\alpha'+\eta''},\quad\text{w.h.p.} \label{eq:a-tilde_a}
\end{gather}
Using the Gaussianity of $\widetilde \NProj,\NProj$, the Lipschitzness of $g^{-1}$ and the assumption that $[a,b]\subset K_{\vk'}(\eps)$, we can see that $a_\eps,b_\eps,\tilde a_\eps,\tilde b_\eps \in K_{\tilde \vk}(\eps)$ w.h.p.\ for some $\tilde \vk>0$. 
Using these together with $\rho<1$, we can see that the Lebesgue measure of the symmetric difference between $A_\eps$ and $\widetilde A_\eps$
is bounded w.h.p.\ by 
\begin{align*}
    C\left(\left(\left(l_\eps^{\tilde\vk}\right)^{\frac{1}{\rho}-1}\eps^{1-\alpha'+\eta''}\right)\vee \eps^{1-\alpha+\eta'}\right)<\eps^{1-\alpha+\eta'''},
\end{align*}
for $\eps$ sufficiently small and some small $\eta'''>0$, where the last equality is due to $\rho<1$ and thus $\alpha'=\alpha \rho<\alpha$.
Then, Lemma~\ref{lem:gaussian_sym_diff} implies that,
for some $\delta'>0$, 
\[
P_\eps=\widetilde P_\eps+\smallo{\eps^{\delta'}},
\]
uniformly in $x\in K_\vk(\eps)$ and $[a,b]\in K_{\vk'}(\eps)$. Using~\eqref{eq:tilde_y_eps},~\eqref{eq:tilde_a_eps},~\eqref{eq:tilde_b_eps}, and the obvious monotonicity of the function $g$,
 we can write
\begin{align*}
   \widetilde P_\eps= \Pp\left\{\Phi_{1,\eps}\left(x,\NProj,\frU,\widetilde \NProj\right)\in [a,b],\ \Phi_{2,\eps}(x,\NProj,\frU)\geq 0\right\},
\end{align*}
where
\begin{align}
    \Phi_{1,\eps}(x,y^1,y^2,y^3)&=\eps^{-\alpha'}\left(g\left(c\eps^{\alpha'}|\Phi_{2,\eps}(x,y^1,y^2)|^\rho\right)+\eps y^3\right), \notag\\
    \Phi_{2,\eps}(x,y^1,y^2)&= \eps^{-\alpha} f^1(\phi(x_0+(\eps^\alpha x+\eps y^1)v))+\eps^{1-\alpha}y^2. \label{eq:def-phi_2,eps}
\end{align}
We can also write  $\Phi_{2,\eps}(x,y^1,y^2)=\Phi_{2,\eps}(x, y^1,y^2,y^3)$ although it does not depend on~$y^3$ at all.
This completes the main part of the proof of~\eqref{item:comp_in_lem_typical_loc_lim}, with $m=3$.

Then, we verify the properties of $\Phi_{i,\eps}$, $i=1,2$, claimed in \eqref{item:cvg_Phi}.
Using the smoothness of functions involved and the identities
\begin{align}\label{eq:g(0)=0}
    f^1(\phi(x_0))=0, \quad\text{and}\quad g(0)=0
\end{align} 
(which are due to~\eqref{eq:flow_x_0to...} and~\eqref{eq:def_g(s)}), we can see that $\Phi_{i,\eps}$, $i=1,2$, converges in LU as $\eps\to 0$, and the limits are of the form described in \eqref{item:cvg_Phi}. The remaining properties follow from these expressions.

Let us verify \eqref{item:nondecreasing_Phi}.
We recall the extensions described below \eqref{eq:def_y_eps}.
Since $z\mapsto f^1(\phi(x_0+zv))$ is nondecreasing, we know that for fixed realizations of $\NProj$ and $\frU$, the function $x\mapsto\Phi_{2,\eps}(x,\NProj,\frU)$ is nondecreasing. Since the function~$g$ is also nondecreasing, we can see that on $\{x:\Phi_{2,\eps}(\cdot,\NProj,\frU)\geq 0\}$, the function $x\mapsto \Phi_{1,\eps}(x,\NProj,\frU,\widetilde \NProj)$ is nondecreasing for every fixed realization of randomness. Hence~\eqref{item:nondecreasing_Phi} holds. To prove~\eqref{item:phi_pm_decomp}, it suffices now to define monotone functions
\begin{align*}
    \phi_{+,\eps}(x) & = \Pp\left\{\Phi_{1,\eps}\left(x,\NProj,\frU,\widetilde \NProj\right) \geq a,\ \Phi_{2,\eps}(x,\NProj,\frU)\geq 0\right\},\\
    \phi_{-,\eps}(x) & = - \Pp\left\{\Phi_{1,\eps}\left(x,\NProj,\frU,\widetilde \NProj\right) >b,\ \Phi_{2,\eps}(x,\NProj,\frU)\geq 0\right\}.
\end{align*}
Then, we turn to~\eqref{item:Phi_1_range}. Using the fact that $\alpha' =\alpha\rho$, that  $f^1$, $\phi$, and $g$ are Lipschitz and~\eqref{eq:g(0)=0}, we derive 
\begin{align*}
    &\Pp\left\{\left|\Phi_{1,\eps}\left(x,\NProj,\frU,\widetilde \NProj\right)\right|\geq K_{\vk'}(\eps)\right\}\\
    &\leq \Pp\left\{c_1\eps^\rho|\frU|^\rho +c_2\eps^\rho|\NProj|^\rho +\eps\left|\widetilde \NProj\right| \geq \eps^{\alpha'}l^{\vk'}_\eps - c_3 \eps^{\alpha\rho}|x|^\rho\right\},
\end{align*}
for some positive constants $c_1,c_2,c_3$.
Since $\alpha'=\alpha\rho\le \rho<1$, the Gaussianity of $\frU,\NProj,\widetilde \NProj$, implies~\eqref{item:Phi_1_range}.

Lastly, we verify~\eqref{item:Phi>|x|^p}. Using \eqref{eq:g(0)=0}, and that $g$ and $s\mapsto f^1(\phi(x_0+sv))$ have derivatives bounded below by positive constants, we have that, for some constants $C,C'>0$,
\begin{align*}
    \left|\Phi_{2,\eps}\left(x,y^1,y^2\right)\right|&\geq C|x| - C\left|y^1\right|-\left|y^2\right|,
    \\
    \left|\Phi_{1,\eps}\left(x,y^1,y^2,y^3\right)\right|&\geq C'\left|\Phi_{2,\eps}\left(x,y^1,y^2\right)\right|^\rho-\left|y^3\right|.
\end{align*}
Choosing $q>0$ sufficiently small, and $R>0$ sufficiently large, we can see that for $|x|>R$ and $|y|_\infty<|x|^q$,
\begin{align*}
    \left|\Phi_{1,\eps}\left(x,y^1,y^2,y^3\right)\right|\geq C''|x|^\rho-\left|y^3\right|\geq C'''|x|^\rho.
\end{align*}

\medskip

\textbf{Case 2.} Let us treat the second case: $\rho=1$. Here, $\alpha'=\alpha\rho =\alpha$. Proposition~\ref{prop:typical}~\eqref{item:loc_lim_2} along with \eqref{eq:Y_0=(y_eps,L)}, \eqref{eq:Y_at_entrance} and \eqref{eq:Y_in_[a_eps,b_eps]}  gives that
\begin{align}
    &\Pp\{X_\tau \in q_++\eps^{\alpha'}[a,b]v_+\} \notag\\
    &= \Pp\left\{c| y_\eps+\eps^{1-\alpha}\frU|+\eps^{1-\alpha}\NN\in[ a_\eps, b_\eps],\  y_\eps+\eps^{1-\alpha}\frU\geq 0\right\}+\smallo{\eps^\delta}\notag\\
    & = \Pp\{\eps^{1-\alpha}\frU\in A_\eps\}+\smallo{\eps^\delta}, \label{eq:eps^(1-alpha)U_in_A}
\end{align}
where we redefine, for $y_\eps,  a_\eps, b_\eps$ given previously in \eqref{eq:def_y_eps} and \eqref{eq:def_a_eps},
\begin{align*}
    A_\eps = -y_\eps + \left[c^{-1}(a_\eps - \eps^{1-\alpha}\NN)\vee0,\ c^{-1}(b_\eps - \eps^{1-\alpha}\NN)\vee0\right].
\end{align*}
We want to compare \eqref{eq:eps^(1-alpha)U_in_A} with
\begin{align}
     \Pp\{c|\tilde y_\eps+\eps^{1-\alpha}\frU|+\eps^{1-\alpha}\NN\in[\tilde a_\eps,\tilde b_\eps],\ \tilde y_\eps+\eps^{1-\alpha}\frU\geq 0\}\label{eq:tilde_int_case_2}
= \Pp\{\eps^{1-\alpha}\frU\in \widetilde A_\eps\}, \end{align}
where we define, for $\tilde y_\eps, \tilde   a_\eps, \tilde b_\eps$ given previously in \eqref{eq:tilde_y_eps} and \eqref{eq:tilde_a_eps},
\begin{align*}
    \widetilde A_\eps = -\tilde y_\eps + \left[c^{-1}(\tilde a_\eps - \eps^{1-\alpha}\NN)\vee0,\ c^{-1}(\tilde b_\eps - \eps^{1-\alpha}\NN)\vee0\right].
\end{align*}
Using~\eqref{eq:y-tilde_y} and~\eqref{eq:a-tilde_a}, we can see that the symmetric difference between $A_\eps$ and~$\widetilde A_\eps$ has Lebesgue measure bounded w.h.p.\ by
\begin{align*}
    C\left(\eps^{1-\alpha+\eta'}\vee \eps^{1-\alpha'+\eta''}\right) < \eps^{1-\alpha'+\eta'''},
\end{align*}
for some $\eta'''>0$. Therefore, Lemma~\ref{lem:gaussian_sym_diff} implies that the difference between the Gaussian probabilities in~\eqref{eq:eps^(1-alpha)U_in_A} and~\eqref{eq:tilde_int_case_2} is $o(\eps^{\delta'})$ for some $\delta'>0$. 
Inserting the expressions for $\tilde a_\eps,\tilde b_\eps,\tilde y_\eps$ into~\eqref{eq:tilde_int_case_2}, we obtain \eqref{item:comp_in_lem_typical_loc_lim},
with $m=4$,
if~$\Phi_{2,\eps}$ is defined by~\eqref{eq:def-phi_2,eps} and
\begin{gather*}
    \Phi_{1,\eps}\left(x,y^1,y^2,y^3,y^4\right)=\eps^{-\alpha}\left(g\left(c\eps^{\alpha}\left|\Phi_{2,\eps}\left(x,y^1,y^2\right)\right|+\eps y^3\right)+\eps y^4\right).
\end{gather*}
The properties of $\Phi_{i,\eps}$ in \eqref{item:cvg_Phi}--\eqref{item:Phi>|x|^p} can be verified similarly to  Case~1.

\medskip

\textbf{Case 3.} We turn to the third case: $\rho>1$ and $\alpha\rho\leq 1$. In this case, we have
\begin{align}\label{eq:alpha'>alpha}
    \alpha' =\alpha \rho >\alpha.
\end{align}
Applying Proposition~\ref{prop:typical}~\eqref{item:loc_lim_3} and its modification in Remark~\ref{rem:loc_lim_3} to~\eqref{eq:Y_in_[a_eps,b_eps]}, we get
\begin{align}
    &\Pp\left\{X_\tau \in q_++\eps^{\alpha'}[a,b]v_+\right\} \notag\\
    &= \Pp\left\{c| y_\eps+\eps^{1-\alpha}\frU|^\rho+\eps^{1-\alpha'}\NN\in[ a_\eps, b_\eps],\quad  y_\eps+\eps^{1-\alpha}\frU\geq 0\right\}+\smallo{\eps^\delta}\notag\\
    & = \Pp\left\{\eps^{1-\alpha'}\NN\in A_\eps,\quad \eps^{1-\alpha}\frU\in B_\eps\right\}+\smallo{\eps^\delta}, \label{eq:eps^(1-alpha)U_in_A_B}
\end{align}
where
\begin{align*}
    A_\eps  = -c\left|y_\eps+\eps^{1-\alpha}\frU\right|^\rho + \big[a_\eps, b_\eps  \big],\qquad
    B_\eps  = [-y_\eps,\infty).
\end{align*}
We want to compare \eqref{eq:eps^(1-alpha)U_in_A_B} with
\begin{multline}
     \Pp\left\{c|\tilde y_\eps+\eps^{1-\alpha}\frU|^\rho+\eps^{1-\alpha'}\NN\in\left[\tilde a_\eps,\tilde b_\eps\right],\quad   \tilde y_\eps+\eps^{1-\alpha}\frU\geq 0\right\}\label{eq:tilde_int_case_3}
    \\
    = \Pp\left\{\eps^{1-\alpha'}\NN\in \widetilde A_\eps,\quad \eps^{1-\alpha}\frU\in \widetilde B_\eps\right\}, \end{multline}
where
\begin{align*}
    \widetilde A_\eps  = -c\left|\tilde y_\eps+\eps^{1-\alpha}\frU\right|^\rho + \left[\tilde a_\eps, \tilde b_\eps  \right],\qquad
    \widetilde B_\eps  = \left[-\tilde y_\eps,\infty\right).
\end{align*}
Using $\rho>1$,~\eqref{eq:alpha'>alpha},~\eqref{eq:y-tilde_y} and~\eqref{eq:a-tilde_a}, we can see that the symmetric difference between $A_\eps$ and $\widetilde A_\eps$ has Lebesgue measure bounded w.h.p.\ by
\begin{align*}
    C\left(\left(\left(l_\eps^{\tilde \vk}\right)^{\rho-1}\eps^{1-\alpha+\eta'}\right)\vee \eps^{1-\alpha'+\eta''}\right) < \eps^{1-\alpha'+\eta'''},
\end{align*}
for some $\eta'''>0$.
The symmetric difference between $B_\eps$ and $\widetilde B_\eps$ is bounded by~$C\eps^{1-\alpha+\eta'}$ w.h.p. 
Thus, due to Lemma~\ref{lem:gaussian_sym_diff}, the Gaussian probabilities
 in~\eqref{eq:eps^(1-alpha)U_in_A_B} and~\eqref{eq:tilde_int_case_3} differ by an error $o(\eps^{\delta'})$ for some $\delta'>0$. Inserting the expressions for $\tilde a_\eps,\tilde b_\eps,\tilde y_\eps$ in~\eqref{eq:tilde_a_eps} and~\eqref{eq:tilde_y_eps} into~\eqref{eq:tilde_int_case_3}, we obtain \eqref{item:comp_in_lem_typical_loc_lim}, with $m=4$, if we define~$\Phi_{2,\eps}$ as in \eqref{eq:def-phi_2,eps} and
\begin{gather*}
    \Phi_{1,\eps}\left(x,y^1,y^2,y^3,y^4\right)=\eps^{-\alpha'}\left(g\left(c\eps^{\alpha'}\left|\Phi_{2,\eps}\left(x,y^1,y^2\right)\right|^\rho+\eps y^3\right)+\eps y^4\right).
\end{gather*}
The properties of $\Phi_{i,\eps}$ in \eqref{item:cvg_Phi}--\eqref{item:Phi>|x|^p} can be verified similarly to Case~1.

\medskip

\textbf{Case 4.} Lastly, we consider the case: $\rho>1$ and $\alpha\rho>1$, implying $\alpha' =1$. Proposition~\ref{prop:typical}~\eqref{item:loc_lim_4} applied to~\eqref{eq:Y_in_[a_eps,b_eps]} yields
\begin{align}
    \Pp\left\{X_\tau \in q_++\eps^{\alpha'}[a,b]v_+\right\} 
    &= \Pp\left\{\NN\in[ a_\eps, b_\eps], \quad y_\eps+\eps^{1-\alpha}\frU\geq 0\right\}+\smallo{\eps^\delta}\notag
    \\
    & = \Pp\left\{\eps^{1-\alpha'}\NN\in A_\eps, \quad \eps^{1-\alpha}\frU\in B_\eps\right\}+\smallo{\eps^\delta}, \label{eq:N_in_A_B}
\end{align}
where $A_\eps  = \left[a_\eps, b_\eps  \right]$, $B_\eps  = [-y_\eps,\infty)$.
We want to compare \eqref{eq:N_in_A_B} with
\begin{align}
    \Pp\left\{\NN\in\left[\tilde a_\eps,\tilde b_\eps\right], \quad  \tilde y_\eps+\eps^{1-\alpha}\frU\geq 0\right\}
    \label{eq:tilde_int_case_4}
    = \Pp\left\{\NN\in \widetilde A_\eps, \quad \eps^{1-\alpha}\frU\in \widetilde B_\eps\right\}, \end{align}
where $\widetilde A_\eps  =   [\tilde a_\eps, \tilde b_\eps  ]$, $\widetilde B_\eps  = [-\tilde y_\eps,\infty)$.
Using~\eqref{eq:y-tilde_y} and~\eqref{eq:a-tilde_a}, we can see that the Lebesgue measure of the symmetric difference between $A_\eps$ and $\widetilde A_\eps$ is bounded by~$C \eps^{1-\alpha'+\eta''}$,
and the Lebesgue measure of the symmetric difference between $B_\eps$ and~$\widetilde B_\eps$ is bounded by $C\eps^{1-\alpha+\eta'}$.
Hence, due to Lemma~\ref{lem:gaussian_sym_diff}, the difference between the Gaussian probabilities in~\eqref{eq:N_in_A_B} and~\eqref{eq:tilde_int_case_4} is $o(\eps^{\delta'})$ for some $\delta'>0$. 
Inserting the expressions for $\tilde a_\eps,\tilde b_\eps,\tilde y_\eps$ in~\eqref{eq:tilde_a_eps} and~\eqref{eq:tilde_y_eps} into~\eqref{eq:tilde_int_case_4}, we obtain \eqref{item:comp_in_lem_typical_loc_lim}, with $m=4$, if $\Phi_{2,\eps}$ is defined by~\eqref{eq:def-phi_2,eps} and
\begin{gather*}
    \Phi_{1,\eps}\left(y^1,y^2,y^3,y^4\right)=\eps^{-1}\left(g\left(\eps y^3\right)+\eps y^4\right).
\end{gather*}
Since $\alpha\rho>1$ in this case, we do not need to verify~\eqref{item:Phi>|x|^p}. All the other properties of~$\Phi_{i,\eps}$ can be verified similarly as in Case~1.
\epf

\section{Gaussian approximation for the stopped process \texorpdfstring{$U^1$}{U}}
\label{sec:Gaussian-approx}

In this section  we assume the setting \irc{} and the notation described in
Sections~\ref{sec:basic_estimates} and~\ref{sec:est_N} with the additional assumption that $F$ and $G$ in \eqref{eq:def_SDE_near_a_saddle_point} are $C^3_{\mathrm{b}}$. 
 Throughout this section, we fix $L>0$ and  $\alpha\in (0,1]$, and study solutions of~\eqref{eq:def_SDE_near_a_saddle_point} 
 with initial conditions of the form  $(\eps^\alpha \scx ,L)$.
 For brevity, we write
\begin{align}\label{eq:def_Probxy}
    \Pp^{\eps^\alpha \scx }=\Pp^{(\eps^\alpha \scx ,L)}.
\end{align}

Our main goal here is to prove, for a family of stopping times,  a local Gaussian approximation (Lemma~\ref{lem:gaussian_approx} and its corollary) 
for the process $U^1$ (defined in \eqref{eq:def_V_M_U_N}) stopped at those times. It will be used then in Section~\ref{sec:LLT} to prove local limit theorems 
for the exit location and a precise estimate on the exit time \irc,
crucial for the proof of Lemma~\ref{lem:local-limit-theorem} in Section~\ref{sec:rec-lem:local-limit-theorem}. 
Our results here are based on the density estimates of Section~\ref{section:density_est}  which are collected in Lemma~\ref{lem:density_est_R^2}. The smoothness assumptions on $F$ and $G$ allow us to apply these results.

The exit times we consider are $\tau=\tau_{r,\theta,\eps}$ and $\zeta=\zeta_{r,\theta,\eps}$
defined in  \eqref{eq:tau_gamma} and~\eqref{eq:exit-from-strip}. 
The main results of this section are stated for $\zeta$.
Lemma~\ref{lem:N_tau_tail} implies though that if $\theta>0$ is small enough to 
satisfy~\eqref{eq:condition_on_alpha_rho_gamma}, then we can ignore the distinction between these exit times: for every $\vk>0$, we have 
 $\tau=\zeta$  w.h.p.\ under~$\Pp^{\eps^\alpha \scx}$ uniformly over $\scx\in K_\vk(\eps)$ (for any $\alpha\in(0,1]$). Thus, under ~\eqref{eq:condition_on_alpha_rho_gamma}, the results of this section with~$\tau$ replaced by $\zeta$ also hold.

Let us generalize the model case definition of \eqref{eq:var-of-N} and set
\begin{align}\label{eq:def_cc_1}
    \cc_1 = \int_0^\infty e^{-2\lam s}|F^1(0,e^{-\mu s}L)|^2ds.
\end{align}

\begin{lemma}\label{lem:gaussian_approx}
Let $\vk, \vk', r>0$, $\alpha\in(0,1]$, $\xi>-\alpha$, $\theta\in[0,\alpha)$ and $\zeta = \zeta_{r,\theta,\eps}$ be given in~\eqref{eq:exit-from-strip}. 
Then, for each $c>0$ and $\eta\in(0,1-\theta)$, there is $\delta>0$ such that
\begin{align*}
    \sup_{\substack{\scx \in K_\vk(\e), \\ [a,b]\subset K_{\vk'}(\e)}} \left|\P^{\eps^\alpha \scx }\left\{\scx +\eps^{1-\alpha}U^1_\zeta \in \eps^\xi [a,b] \right\}-\Pp\left\{\scx +\eps^{1-\alpha}\frU \in \eps^\xi [a\mp c\eps^\eta,b\pm c\eps^\eta] \right\}\right|
    \\= \smallo{\eps^{((\xi+\alpha -1)\vee0)+\delta}},
\end{align*}
where $\frU$ is a centered Gaussian r.v.\ with variance~$\cc_1$.
\end{lemma}

Using Lemma~\ref{lem:gaussian_sym_diff}~and adjusting $\delta$, we have the following consequence.

\begin{corollary}\label{cor:gaussian_approx}
In the setting of Lemma~\ref{lem:gaussian_approx}, if $\alpha=1$ and $\xi\geq 0$, then there is $\delta>0$ such that
\begin{align*}
    \sup_{\substack{x\in K_\vk(\e), \\ [a,b]\subset K_{\vk'}(\e)}} \left|\P^{\eps \scx }\left\{\scx +U^1_\zeta \in \eps^\xi [a,b] \right\}-\Pp\left\{\scx +\frU \in \eps^\xi [a,b] \right\}\right|= \smallo{\eps^{\xi+\delta}}.
\end{align*}
\end{corollary}

The rest of the section is devoted to the proof of Lemma~\ref{lem:gaussian_approx} which is divided into two steps:
an approximation by the process stopped at a  deterministic time
and a Gaussian approximation of the latter based on an iteration scheme. They are implemented separately in Section~\ref{subsection:rough_approx} and Section~\ref{subsection:iteration}.

Since notation is simpler at the scale $\eps^1$, we will primarily work under $\Pp^{\eps x}$ for $x$ in a  set larger than $K_\vk(\eps)$, which allows us to recover the desired result under $\Pp^{\eps^\alpha x}$ by substituting $\eps^{\alpha-1} x$ for $x$.

\subsection{Approximation by the process stopped at a deterministic time}\label{subsection:rough_approx}

\begin{lemma}\label{lem:rough_approx_new}
Let $\vk',r,c>0$, $\xi>-1$, $\theta\in[0,1)$, $\eta\in(0,1-\theta)$, and $\zeta = \zeta_{r,\theta,\eps}$ be given in \eqref{eq:exit-from-strip}.
Then, for every $\delta>0$,
\begin{align*}
    \sup_{\substack{\scx \in r\eps^{\theta-1}(-1,1)\\ [a,b]\subset K_{\vk'}(\e)}} \left|\P^{\eps \scx }\left\{\scx +U^1_\zeta \in \eps^\xi [a,b] \right\}-\P^{\eps \scx }\left\{\scx +U^1_\dtime \in \eps^\xi [a\mp c\eps^\eta,b\pm c\eps^\eta] \right\} \right| = \superpoly,
\end{align*}
where
\begin{align}\label{eq:def_dtime(eps)}
    \dtime=\dtime(\eps) = \frac{1}{\lambda}\log \frac{r\eps^{\theta - \xi -1}}{ l_\eps^{\vk'+\delta}}.
\end{align}
\end{lemma}

\begin{proof}
All estimates in this proof are understood to hold uniformly in $\scx \in\R$ and $[a,b]\subset K_{\vk'}(\e)$. For convenience, we set
\begin{align}\label{eq:def_A_{pm,eps}}
    A_{\pm,\eps} = \eps^\xi [a\mp c\eps^\eta,b\pm c\eps^\eta].
\end{align}

First, we establish an upper bound. The definition of $\zeta$ in~\eqref{eq:exit-from-strip} along with \eqref{eq:SDE_near_a_saddle_point_after_duhamel1} implies that
\begin{align}\label{eq:sigma_formula}
    r\eps^\theta=\eps e^{\lam \zeta}\left|\scx +U^1_\zeta\right|, \quad\quad \zeta=\frac{1}{\lam}\log\frac{r\eps^{\theta -1}}{\left|\scx +U^1_\zeta\right|}.
\end{align}
Let us start by showing $\zeta\geq\dtime$ on the relevant event. Indeed, using~\eqref{eq:def_dtime(eps)},~\eqref{eq:sigma_formula} and the definition of $K_{\vk'}(\e)$ in~\eqref{eq:def_K(eps)}, we have
\begin{align*}
    &\Probxy{\scx +U^1_\zeta \in \eps^\xi [a,b],\ \zeta<\dtime}\\
    &\leq \Probxy{\left|\scx +U^1_\zeta\right|\leq \eps^{\xi} l_\eps^{\vk'},\ \left|\scx +U^1_\zeta\right|> \eps^{\xi} l_\eps^{\vk'+\delta}} = 0.
\end{align*}
This implies that
\begin{align}\label{eq:upper_bound_approx_exit_loc_RA}
    \Probxy{\scx +U^1_\zeta \in \eps^\xi [a,b]}\leq \Probxy{\scx +U^1_{\zeta\vee\dtime} \in \eps^\xi [a,b]}.
\end{align}
Then, we compare $U^1_{\zeta\vee\dtime}$ with $U^1_\dtime$. We recall $U^1_t=M^1_t +\eps V^1_t$ (see~\eqref{eq:def_V_M_U_N}). 
Let us take any $\delta'\in(0,1-\theta-\eta)$.
The boundedness of $F^1$ implies that 
\begin{align*}
    \qd{M^1}_{\zeta\vee\dtime}-\qd{M^1}_\dtime\leq Ce^{-2\lam \dtime}\leq C \eps^{2(1+\xi-\theta-\delta')}.
\end{align*}
Applying the exponential martingale inequality (Lemma~\ref{lem:exp-marting-ineq}), we see that
\begin{align*}
    \Probxy{\left|M^1_{\zeta\vee\dtime}-M^1_\dtime\right|>\tfrac{1}{2}c\eps^{\xi+\eta}}\leq 2\exp\left(-C\eps^{2(\eta+\theta-1+\delta')}\right)=\superpoly.
\end{align*}
Using Lemma~\ref{lem:estimates-for-terms}~\eqref{item:V-bounded}, we also have
\begin{align*}
    \Probxy{\left|\eps V^1_{\zeta\vee\dtime}-\eps V^1_\dtime\right|>\tfrac{1}{2}c\eps^{\xi+\eta}}=0
\end{align*}
for small $\eps$.
From the above two displays, we obtain
\begin{align}\label{eq:|U^1_htau-U^1_T|_est_RA}
 \Probxy{\left|U^1_{\zeta\vee\dtime}-U^1_\dtime\right|>c\eps^{\xi+\eta}}=\superpoly,   
\end{align}
which together with~\eqref{eq:upper_bound_approx_exit_loc_RA} gives an upper bound.

\medskip

To find a lower bound,  we start with
\begin{align*}
\begin{split}
    \Probxy{\scx +U^1_\zeta \in \eps^\xi[a,b]}&\geq \Probxy{\scx +U^1_\zeta \in \eps^\xi[a,b];\ \left|U^1_\zeta - U^1_\dtime\right|\leq c\eps^{\xi+\eta}}\\
    &\geq \Probxy{\scx +U^1_\dtime\in A_{-,\eps};\ \left|U^1_\zeta - U^1_\dtime\right|\leq c\eps^{\xi+\eta}}\\
    &\geq \Probxy{\scx +U^1_\dtime\in A_{-,\eps}}-\Probxy{\scx +U^1_\dtime\in A_{-,\eps};\ \left|U^1_\zeta - U^1_\dtime\right|> c\eps^{\xi+\eta}}.
\end{split}
\end{align*}
To show that the second term on the right-hand side is $\superpoly$, we bound it by
\begin{align}\label{eq:lower_bound_split_exit_loc_RA}
    \Probxy{\zeta\geq\dtime;\ \left|U^1_\zeta - U^1_\dtime\right|>c\eps^{\xi+\eta}}+\Probxy{\zeta<\dtime;\ \scx +U^1_\dtime\in A_{-,\eps}}.
\end{align}
The first term is $\superpoly$ due to~\eqref{eq:|U^1_htau-U^1_T|_est_RA}. For the second term, we apply~\eqref{eq:SDE_near_a_saddle_point_after_duhamel1}, the definition of $\dtime$ in~\eqref{eq:def_dtime(eps)}, the definition of $A_{-,\eps}$ in~\eqref{eq:def_A_{pm,eps}} and the strong Markov property to see that
\begin{align*}
\begin{split}
    \Probxy{\zeta <\dtime;\ \scx +U^1_\dtime\in A_{-,\eps} } & = \Probxy{\zeta <\dtime;\ Y^1_\dtime\in \eps e^{\lam \dtime}A_{-,\eps}}\\
    &\leq \Probxy{\zeta <\dtime;\ |Y^1_\dtime|\leq r\eps^{\theta}  l_\eps^{-\delta} }\\
    &\leq \Exy{\left[\Probx{Y_\htau}{\inf_{t\in[0,\dtime]}\left|Y^1_t\right|\leq r\eps^{\theta}  l_\eps^{-\delta} }\right]}.
\end{split}
\end{align*}
 We have $\Pp^{Y_\htau}\{|Y^1_0|=r\eps^\theta\}=1$.
Hence~\eqref{eq:SDE_near_a_saddle_point_after_duhamel1} implies $\left|Y^1_t\right|= |e^{\lam t}(Y^1_0+\eps U^1_t)|\geq r\eps^\theta - \eps|U^1_t|$. From this, we can obtain
\begin{align*}
\begin{split}
\Probx{Y_\htau}{\inf_{t\in[0,\dtime]}\left|Y^1_t\right|\leq r\eps^{\theta}  l_\eps^{-\delta} }&\leq \Probx{Y_\htau}{\inf_{t\in[0,\dtime]}(r\eps^\theta-\eps\left|U^1_t\right|)\leq r\eps^{\theta}  l_\eps^{-\delta} }
\\
&\leq \Probx{Y_\htau}{\sup_{t\in[0,\dtime]}\left|U^1_t\right|\geq r \eps^{\theta-1}\left(1-  l_\eps^{-\delta}\right)}=\superpoly,
\end{split}
\end{align*}
where we used $\theta<1$ and Lemma~\ref{lem:estimates-for-terms}~\eqref{item:tail-of_U_1} in the third equality. This shows that~\eqref{eq:lower_bound_split_exit_loc_RA} is $\superpoly$ and completes the proof.
\end{proof}

\subsection{Gaussian approximation for the deterministically stopped process}\label{subsection:iteration}

\begin{lemma}\label{lem:iterative_result}  In the setting of Lemma~\ref{lem:rough_approx_new}, let  $\varrho>\varrho'>0$.
Let $\dtime(\eps)$ be a deterministic function of $\eps$ satisfying
\begin{align}\label{eq:condition_dtime(eps)}
    \dtime(\eps)\in [\varrho' \lam^{-1} l_\eps, \varrho \lam^{-1} l_\eps],\quad \eps \in (0,1/2).
\end{align}
For
\begin{align}\label{eq:zeta_xi_condition_iteration}
   \xi\ge -1+\varrho,
\end{align}
$ \kappa' >0$, and 
$\upsilon\in(0,1)$, there is $\delta>0$ such that
\begin{align*}
    \sup_{|\scx |\leq \eps^{\upsilon-1},\ A\subset \eps^\xi K_{\vk'}(\eps)}\left|\Probxy{\scx +U^1_{\dtime(\eps)}\in A} -  \Prob{\scx +\frU \in A}\right|=\smallo{\eps^{(\xi\vee 0)+\delta}},
\end{align*}
where $\frU$ is a centered Gaussian r.v.\ with variance $\cc_1$ defined in~\eqref{eq:def_cc_1}.
\end{lemma}

To prove Lemma~\ref{lem:iterative_result}, we need the following iterative scheme.

\begin{lemma}\label{lem:iterative_est}
Under the assumptions of Lemma~\ref{lem:iterative_result}, 
there is $N\in\mathbb{N}$ such that for each $\upsilon\in(0,1)$, there are positive constants  $\eps_k,C_k,\delta_k$, $k=1,2,...,N$ and~$\upsilon'$ such that
\begin{align}
    \sup_{\substack{|\scx |\leq \eps^{\upsilon-1},\  |w|\leq \eps^{\upsilon'-1}\\ A\subset \eps^\xi K_{\vk'}(\eps)}}\left|\Probxy{\scx +U^1_{t_k}+e^{-\lam t_k}w\in A} -  \Prob{\scx +\cU_k+e^{-\lam t_k}w\in A}\right|\notag\\
    \leq C_k \eps^{(\xi\vee0)+\delta_k}, \label{eq:induction_step_conclusion}
\end{align}
holds for all $k=1,2,..., N$ and $\eps\in(0,\eps_k]$. Here, for each $k=1,2,\dots, N$, 
\begin{align}\label{eq:t_k}
    t_k=t_k(\eps)=\frac{k}{N}\dtime(\eps)
\end{align}
and $\cU_k$ is a centered Gaussian random variable with variance
\begin{align}\label{eq:variance_cU_k}
    \E|\cU_k|^2 = \int_0^{t_1}e^{-2\lam s}\left|F^1(0,e^{-\mu s}L)\right|^2ds + \frac{1}{2\lam}\left|F^1(0,0)\right|^2\left(e^{-2\lam t_1}-e^{-2\lam t_k}\right).
\end{align}
\end{lemma}

Let us first use this lemma to prove Lemma~\ref{lem:iterative_result}.

\begin{proof}[Proof of Lemma~\ref{lem:iterative_result}]
Setting $k=N$ and $w = 0$, we obtain, for some $\delta>0$,
\begin{align*}
    \sup_{|\scx |\leq \eps^{\upsilon-1},\ A\subset \eps^\xi K_{\vk'}(\eps)}\left|\Probxy{\scx +U^1_{\dtime(\eps)}\in A} -  \Prob{\scx +\cU_N \in A}\right|=\smallo{\eps^{(\xi\vee0)+\delta}}.
\end{align*}

It remains to compare $\cU_N$ with $\frU$.    Using 
the definition of $\cc_1$, identities $t_1=\frac{1}{N}\dtime(\eps)$, $t_N = \dtime(\eps)$,~\eqref{eq:condition_dtime(eps)},~\eqref{eq:variance_cU_k}, 
and the boundedness of $F^1$, we obtain that there is $a>0$ such that
\begin{align*}
    \left|\E|\cU_N|^2 -\E|\frU|^2\right|\leq C\int_{\frac{\dtime(\eps)}{N}}^\infty e^{-2\lam s} ds + C e^{-2\lam \frac{\dtime(\eps)}{N}} + C e^{-2\lam \dtime(\eps)} \leq C\eps^a.
\end{align*}
Since $\cU_N$ with $\frU$ are Gaussian and centered, it can be checked that there is $a'>0$ such that the difference of densities $|\varphi_{\cU_N}(x)-\varphi_{\frU}(x)|\leq C\eps^{a'}e^{-c|x|^2}$, for all $x\in \R$. Therefore, for some $\delta>0$,

\begin{align*}
    \sup_{|\scx |\leq \eps^{\upsilon-1},\, A\subset \eps^\xi K_{\vk'}(\eps)}\left|  \Prob{\scx +\cU_N \in A}- \Prob{\scx +\frU \in A}\right|&\leq C  \eps^{a'}\left(\left(\eps^\xi l_\eps^{\vk'}\right)\vee 1\right),
    \\
    &=\smallo{\eps^{(\xi\vee0)+\delta}}
\end{align*}
which completes the proof.
\end{proof}

\begin{proof}[Proof of Lemma~\ref{lem:iterative_est}]
Recalling the range of $T(\eps)$ in~\eqref{eq:condition_dtime(eps)}, we fix $N\in\N$  sufficiently large to satisfy
\begin{align*}
    \frac{\dtime(\eps)}{N}\leq \bar\theta l_\eps,\quad \eps\in(0,1/2)
\end{align*}
for $\bar\theta$ given in Lemma~\ref{lem:density_est_R^2}. Then we use ~\eqref{eq:zeta_xi_condition_iteration} to fix  $\upsilon'$ satisfying
\begin{align}\label{eq:condition_upsilon'}
    0<\upsilon' < \left(\frac{1}{N}\varrho+\xi-\varrho+1\right)\wedge 1.
\end{align}

For $k=1$, the choice of $N$ allows us to apply Lemma~\ref{lem:density_est_R^2}~\eqref{item:den_est_1_R^2} to the deterministic time $t_1$ (given in \eqref{eq:t_k}) to obtain that, for some $\delta,\delta_1>0$,
\begin{align*}
     &\sup_{\substack{|\scx |\leq \eps^{\upsilon-1},\ |w|\leq \eps^{\upsilon'-1}\\ A\subset\eps^\xi K_{\vk'}(\eps)}}\left|\Probxy{\scx +U^1_{t_1}+e^{-\lam t_1}w\in A} -  \Prob{\scx +\cU_1+e^{-\lam t_1}w\in A}\right|
     \\
     &\leq \sup_{\substack{|\scx |\leq \eps^{\upsilon-1},\ |w|\leq \eps^{\upsilon'-1}\\ A\subset\eps^\xi K_{\vk'}(\eps)}} \int_{\{s\in\R:\scx +s+e^{-\lam t_1}w\in A\}}C\eps^\delta\left(1+\eps^{1-\upsilon}|\scx |\right)e^{-c|s|^2}ds
     \\
     &\leq C \eps^\delta \left(\left(\eps^\xi l_\eps^{\vk'}\right)\wedge 1\right)\leq C\eps^{(\xi\vee0)+\delta_1},
\end{align*}
as desired.

Then, we proceed by induction. 
Let $k\leq N$ and let us assume that~\eqref{eq:induction_step_conclusion} holds for $k-1$.
For $u\in \R^2$, we set
\begin{align}\label{eq:def_z(u)}
    z(u)= \left(z^1(u),\ z^2(u)\right)=\left(e^{\lam t_{k-1}}(\scx +u^1), e^{-\mu t_{k-1}}(\eps^{-1}L+u^2)\right),
\end{align}
where we suppressed the dependence on $\eps$ in the notation.
Using~\eqref{eq:SDE_near_a_saddle_point_after_duhamel1} and~\eqref{eq:SDE_near_a_saddle_point_after_duhamel2}, we have 
$\Pp^{\eps \scx }\{\eps z(U_{t_{k-1}}) = Y_{t_{k-1}}\}=1$. 
The Markov property of $Y$ together with~\eqref{eq:SDE_near_a_saddle_point_after_duhamel1} implies 
\begin{align}
    \begin{split}
        &\Probxy{\scx +U^1_{t_k}+e^{-\lam t_k}w\in A}  =  \Probxy{Y^1_{t_k}+\eps w\in \eps e^{\lam t_k}A} \\ 
\label{eq:strong_markov_Y}        
        &= \Exy{\left[\Probx{ Y_{t_{k-1}}}{Y^1_{t_1}+\eps w\in \eps e^{\lam t_k}A}\right]}  = \E^{\eps \scx }A_\eps(U_{t_{k-1}},w),
    \end{split}
\end{align}
where
\begin{align*}
    A_\eps(u,w)&=\Probx{ \eps z(u)}{z^1(u)+U^1_{t_1}+e^{-\lam t_1}w\in e^{\lam t_{k-1}}A}.
\end{align*}

Let $\cZZ$ be a centered Gaussian r.v.\ with variance 
\begin{align}\label{eq:var_Z_iteration}
    \E|\cZZ|^2 
    = \frac{1}{2\lambda}\left|F^1(0,0)\right|^2 \left(1-e^{-2\lambda t_1}\right)
\end{align}
and independent of all the other randomness. To check~\eqref{eq:induction_step_conclusion} for $k$ and complete the induction
step, we must show that
the error caused by replacing $U^1_{t_1}$ and $U^1_{t_{k-1}}$ by $\cZZ$ and $\cU_{k-1}$, respectively, in~\eqref{eq:strong_markov_Y} is small.
More precisely, ~\eqref{eq:induction_step_conclusion} for $k$ will follow immediately
once we prove that there are $\eps_k,\delta',\delta''>0$ such that
the following relations hold uniformly in $|\scx |\leq\eps^{\upsilon-1}$, $|w|\leq \eps^{\upsilon'-1}$, $A\subset\eps^\xi K_{\vk'}(\eps)$  and $\eps \in(0,\eps_k]$:
\begin{gather}
\label{eq:replacement-1}
|\E^{\eps \scx }A_\eps(U_{t_{k-1}},w)- \E^{\eps \scx } B_\eps(U_{t_{k-1}},w)|=\smallo{\eps^{(\xi\vee0)+\delta'}},
\\
\label{eq:replacement-2}
\left|\E^{\eps \scx }{ B_\eps(U_{t_{k-1}},w)}-C_\eps(\scx,w)\right|
=\smallo{\eps^{(\xi\vee0)+\delta''}},
\end{gather}
where
\begin{align*}
B_\eps(u,w)&=\Prob{z^1(u)+\cZZ+e^{-\lam t_1}w\in e^{\lam t_{k-1}}A},
\\
C_\eps(\scx,w)&=\Prob{\scx +\cU_k+e^{-\lam t_k}w\in A}.
\end{align*}

Let us derive~\eqref{eq:replacement-1}.
The choice of $N$ and definition of $t_1$ allow us to apply Lemma~\ref{lem:density_est_R^2}~\eqref{item:den_est_2_R^2}, by which there are $\hat\delta,\hat c>0$ such that
\begin{align}\label{eq:|A_eps(u,w)-B_eps(u,w)|_est}
\begin{split}
&|A_\eps(u,w)-B_\eps(u,w)|\\
&\leq \int_{\{s\in\R:z^1(u)+s+e^{-\lam t_1}w\in e^{\lam t_{k-1}}A\}}C\left(\eps |z^2(u)|+\eps^{\hat\delta}\left(1+\eps^{1-\upsilon'}|z^1(u)|\right)\right)e^{-\hat c|s|^2}ds.
\end{split}
\end{align}
Let us estimate the right-hand side. Using~\eqref{eq:zeta_xi_condition_iteration},~\eqref{eq:condition_dtime(eps)},~\eqref{eq:t_k} and ~\eqref{eq:condition_upsilon'}, we have, for $\eps$ sufficiently small,
\begin{align*}
    e^{\lam t_{k-1}}\eps^{\xi} l_\eps^{\vk'}\leq e^{ \frac{N-1}{N}\varrho l_\eps}\eps^{\xi} l_\eps^{\vk'} \leq  \eps^{\frac{1}{N}\varrho +\xi-\varrho } l_\eps^{\vk'} <  \eps^{\upsilon'-1},
\end{align*}
which along with $A\subset \eps^\xi K_{\vk'}(\eps)$ implies that 
if $z^1(u)+s+e^{-\lam t_1}w\in e^{\lam t_{k-1}}A$  and $|w|\leq \eps^{\upsilon'-1}$, then
\begin{align}\label{eq:|z(u)|_est}
    \eps^{1-\upsilon'}|z^1(u)| \leq C+\eps^{1-\upsilon'}|s|\leq C+|s|.
\end{align}
On the other hand, from~\eqref{eq:def_z(u)},~\eqref{eq:condition_dtime(eps)} and~\eqref{eq:t_k}, one can see that, for some $a>0$,
\begin{align}\label{eq:est_z^2(u)}
    \eps|z^2(u)|\leq e^{-\mu t_{k-1}}(L+\eps|u^2|)\leq \eps^{a}L+\eps e^{-\mu t_{k-1}}|u^2|.
\end{align}

Using $e^{-\hat c|s|^2}$ to absorb polynomials of $|s|$, from~\eqref{eq:|A_eps(u,w)-B_eps(u,w)|_est}, ~\eqref{eq:|z(u)|_est} and~\eqref{eq:est_z^2(u)} we obtain that, for some $\tilde \delta,  \tilde c>0$,
\begin{align*}
     &|A_\eps(u,w)-B_\eps(u,w)|\\
     &\leq  \eps^{\tilde \delta}\int_{\{s\in\R:z^1(u)+s+e^{-\lam t_1}w\in e^{\lam t_{k-1}}A\}}C(1+e^{-\mu t_{k-1}}|u^2|)e^{-\tilde c|s|^2}ds, \qquad |w|\leq \eps^{\upsilon'-1}.
\end{align*}
Let $\NN$ be a centered Gaussian r.v.\ with density proportional to $e^{-\tilde c|x|^2}$ and independent of other randomness. The last display implies that, if $|w|\leq \eps^{\upsilon'-1}$, then
\begin{align*}
&\left|\E^{\eps \scx }A_\eps(U_{t_{k-1}},w)- \E^{\eps \scx } B_\eps(U_{t_{k-1}},w)\right|\\
&\leq  C\eps^{\tilde \delta}\E^{\eps \scx }\left[(1+e^{-\mu t_{k-1}} |U^2_{t_{k-1}}|)\Ind{\scx +U^1_{t_{k-1}}+e^{-\lam t_k}w+e^{-\lam t_{k-1}}\NN \in A}\right].
\end{align*}
Let $p,p'> 1$ satisfy $\frac{1}{p}+\frac{1}{p'}=1$. We will choose $p$ very close to $1$ later. Using~\eqref{eq:def_V_M_U_N} and Lemma~\ref{lem:estimates-for-terms}~\eqref{item:sup_of-N-over-large-intervals}, we have that $\E^{\eps \scx }(e^{-\mu t_{k-1}} |U^2_{t_{k-1}}|)^{p'}<C$. Hence, applying H\"older's inequality to the above display, we have
\begin{align*}
    &\left|\E^{\eps \scx }A_\eps(U_{t_{k-1}},w)- \E^{\eps \scx } B_\eps(U_{t_{k-1}},w)\right|\\
&\leq  C\eps^{\tilde \delta}\left(\Probxy{\scx +U^1_{t_{k-1}}+e^{-\lam t_k}w+e^{-\lam t_{k-1}}\NN \in A}\right)^{\frac{1}{p}}.
\end{align*}
Since $e^{-\lambda t_1}$ decays like a small positive power of $\eps$, we have that, for small $\eps$,
\begin{align}\label{eq:induction_assumption_satisfied}
    |w| \leq \eps^{\upsilon'-1} \quad\text{ implies }\quad |e^{-\lam t_1}w|+  l_\eps\leq \eps^{\upsilon'-1}.
\end{align}
Therefore,
\begin{align*}\begin{split}
&|\E^{\eps \scx }A_\eps(U_{t_{k-1}},w)- \E^{\eps \scx } B_\eps(U_{t_{k-1}},w)|
\\
        &\leq C\eps^{\tilde \delta}\left(\Probx{\eps \scx }{\scx +U^1_{t_{k-1}}+e^{-\lam t_{k-1}}(e^{-\lam t_1}w+\NN) \in A;\ |\NN|\leq   l_\eps} + \superpoly\right)^{\frac{1}{p}}
        \\
    &\leq C\eps^{\tilde \delta}\left(\Prob{\scx +\cU_{k-1}+e^{-\lam t_{k-1}}(e^{-\lam t_1}w+\NN)\in A}+\smallo{\eps^{(\xi\vee 0)+\delta_{k-1}}}\right)^{\frac{1}{p}}
    \\
    & \leq C\eps^{\tilde \delta}\left((\eps^\xi l_\eps^{\vk'})\wedge 1+\smallo{\eps^{(\xi\vee 0)+\delta_{k-1}}}\right)^{\frac{1}{p}} = \mathcal{O}\left(\eps^{\frac{\xi\vee0}{p}+\tilde \delta}\, l_\eps^\frac{\vk'}{p}\right), 
\end{split}
\end{align*}
uniformly in $|\scx |\leq \eps^{\upsilon-1}$,  $|w|\leq \eps^{\upsilon'-1}$ and $A\subset\eps^\xi K_{\vk'}(\eps)$.
Here, in the second inequality we used the induction assumption \eqref{eq:induction_step_conclusion} for $k-1$ allowed by~\eqref{eq:induction_assumption_satisfied}, the Gaussian tail of $\NN$, and Fubini's theorem along with the independence of $\NN$. In the last line we used~$A\subset\eps^\xi K_{\vk'}(\eps)$, the uniform boundedness of the density of $\cU_{k-1}$ (see~\eqref{eq:variance_cU_k}), independence of $\NN$ and Fubini's theorem. Choosing $p$ sufficiently close to $1$ completes the proof of~\eqref{eq:replacement-1}.

Let us now prove~\eqref{eq:replacement-2}. 
The following holds uniformly in $|\scx |\leq \eps^{\upsilon-1}$,  $|w|\leq \eps^{\upsilon'-1}$ and $A\subset\eps^\xi K_{\vk'}(\eps)$:
\begin{align*} 
&\E^{\eps \scx } B_\eps(U_{t_{k-1}},w)
\\
    &= \Probx{\eps \scx }{\scx +U^1_{t_{k-1}}+e^{-\lam t_{k-1}}(e^{-\lam t_1}w+\cZZ)\in A}\\
    & = \Probx{\eps \scx }{\scx +U^1_{t_{k-1}}+e^{-\lam t_{k-1}}(e^{-\lam t_1}w+\cZZ)\in A;\ |\cZZ|\leq   l_\eps}+\superpoly\\
    &= \Prob{\scx +\cU_{k-1}+e^{-\lam t_{k-1}}(e^{-\lam t_1}w+\cZZ)\in A;\ |\cZZ|\leq   l_\eps}+\smallo{\eps^{(\xi\vee0)+\delta_{k-1}}}\\
    &= \Prob{\scx +\cU_{k-1}+e^{-\lam t_{k-1}}\cZZ+e^{-\lam t_k}w\in A}+\smallo{\eps^{(\xi\vee0)+\delta_{k-1}}},
    \\
    &=C_\eps(y,w) +\smallo{\eps^{(\xi\vee0)+\delta_{k-1}}}.
\end{align*}
 In the third identity, we used the induction assumption allowed by~\eqref{eq:induction_assumption_satisfied}, independence of $\cZZ$, and Fubini's theorem. In the last line, we used the identity in distribution between $\cU_{k-1}+e^{-\lam t_{k-1}}\cZZ$ and $\cU_k$ (see \eqref{eq:variance_cU_k} and \eqref{eq:var_Z_iteration}).
This proves~\eqref{eq:replacement-2} with $\delta''=\delta_{k-1}$
completing the induction step and the entire proof.
\end{proof}

Now, we are ready to prove Lemma~\ref{lem:gaussian_approx}.

\begin{proof}[Proof of Lemma~\ref{lem:gaussian_approx}]
We substitute
$1, \eps^{\alpha-1}\scx , \xi+\alpha-1$ for  $c,\scx ,\xi$  in both Lemmas~\ref{lem:rough_approx_new} and~\ref{lem:iterative_result}. 
We choose an arbitrary $\upsilon\in (0,\alpha)$ in Lemma~\ref{lem:iterative_result}. We  set $\varrho = \xi+\alpha - \theta$ and 
choose an arbitrary
$\varrho'\in (0,\varrho)$ for Lemma~\ref{lem:iterative_result}. Then, with $\xi$ replaced by $\xi+\alpha-1$,~\eqref{eq:zeta_xi_condition_iteration} holds and~\eqref{eq:condition_dtime(eps)} is satisfied for $T(\eps)$ given in \eqref{eq:def_dtime(eps)}, for sufficiently small $\eps$. Combining Lemma~\ref{lem:rough_approx_new} and Lemma~\ref{lem:iterative_result} (with $\vk'$ therein replaced by any $\vk''>\vk'$), we obtain the desired result.
\end{proof}

\section{Local limit theorems}
\label{sec:LLT}

In this section, we adopt the setting of Section~\ref{sec:Gaussian-approx}. The goal is to compute the tail asymptotics 
for exit times and obtain  
local limit theorems for exit locations, \irc.

We recall the notation $\Pp^{\eps^\alpha x}$ in \eqref{eq:def_Probxy} and the notation for Gaussian densities in~\eqref{eq:g_c_Gaussian}.

\subsection{Exit times}
\begin{proposition}\label{prop:long_exit_time_asym_linear_box_case}
Let $\vk,r>0$, $\alpha\in (0,1]$, $\theta\in[0,\alpha)$, $\beta\geq 1-\theta$, $c\in\R$, $\zeta = \zeta_{r,\theta,\eps}$ be given in \eqref{eq:exit-from-strip} and $\cc_1$ be given in \eqref{eq:def_cc_1}. 
There is $\delta>0$ such that the following hold:
If $\theta+\beta-\alpha>0$, then 
\begin{align*}
    \sup_{\scx \in K_\vk(\e)}\left|\eps^{-(\theta+\beta-\alpha)}\Pp^{\eps^\alpha \scx }\{\zeta\geq \beta\lam^{-1} l_\eps+c\}-2re^{-\lambda c} g_{\cc_1}\left(\eps^{\alpha-1}\scx \right) \right|=\smallo{\eps^\delta};
\end{align*}

If $\theta+\beta-\alpha=0$, then 
\begin{align*}
    \sup_{\scx \in K_\vk(\e)}\left|\Pp^{\eps^\alpha \scx }\{\zeta\geq \beta\lam^{-1} l_\eps+c\}-\int_{[-re^{-\lambda c}, re^{-\lambda c}]}g_{\cc_1}\left(\eps^{\alpha-1}\scx -s\right)ds \right|=\smallo{\eps^\delta}.
\end{align*}

\end{proposition}

\begin{proof}[Proof of Proposition~\ref{prop:long_exit_time_asym_linear_box_case}]
Due to $\theta<\alpha$,  for sufficiently small $\eps$, the initial condition we are interested in satisfies 
\begin{align*}
    |Y^1_0|=\eps^\alpha|\scx |\leq \eps^\alpha l^\vk_\eps\leq r\eps^\theta
\end{align*}
for all $x\in K_\vk(\eps)$. The definition of $\zeta$ thus ensures that $|Y^1_\zeta|= r\eps^\theta$, which along with the formula~\eqref{eq:SDE_near_a_saddle_point_after_duhamel1} implies that  $\Pp^{\eps^\alpha \scx }$-a.s.\
\begin{align*}e^{\lam \zeta}|\eps^\alpha \scx +\eps U^1_\zeta|= r\eps^{\theta},\quad\text{or equivalently,}\quad \zeta = \frac{1}{\lambda}\log\frac{r\eps^{\theta}}{|\eps^\alpha \scx +\eps U^1_\zeta|}.
\end{align*}
From this, we have
\begin{align*}\{\zeta\geq \beta\lambda^{-1} l_\eps+c\} \stackrel{\Pp^{\eps^\alpha \scx }}{=} \{|\scx +\eps^{1-\alpha}U^1_\zeta|\leq re^{-\lambda c}\eps^{\theta+\beta-\alpha}\}.
\end{align*}
Applying Lemma~\ref{lem:gaussian_approx} with $\xi=\theta+\beta-\alpha$ and arbitrary $\eta\in(0,1-\theta)$, we obtain
\begin{align*}
    \sup_{\scx \in K_\vk(\e)}\left|\Pp^{\eps^\alpha \scx }\{\zeta\geq \beta\lambda^{-1} l_\eps +c\}- \Prob{|\scx +\eps^{1-\alpha}\frU|\leq re^{-\lambda c}\eps^{\theta+\beta-\alpha}\pm\eps^\eta}\right|
    \\
    = \smallo{\eps^{(\theta+\beta-1)\vee 0 +\delta'}}
\end{align*}
for some $\delta'>0$. Rewriting the probability involving $\frU$, we have 
\begin{align*}
     &\eps^{-((\theta+\beta-\alpha)\vee 0)}\Prob{|\scx +\eps^{1-\alpha}\frU|\leq re^{-\lambda c}\eps^{\theta+\beta-\alpha}\pm\eps^\eta} \\
     &= \eps^{-((\theta+\beta-\alpha)\vee 0)}\int_{[-re^{-\lambda c}\eps^{\theta+\beta-\alpha}\mp\eps^\eta,\ re^{-\lambda c}\eps^{\theta+\beta-\alpha}\pm\eps^\eta]}g_{\cc_1}\left(\eps^{\alpha-1}\scx -s\right)ds.
\end{align*}
Estimating the right-hand side with the help of Lemma~\ref{lem:gaussian_sym_diff}, we obtain
the desired result.
\end{proof}

\subsection{Atypical exit locations}
Recalling stopping times given in~\eqref{eq:tau_gamma}, for $R>0$, we set \begin{align}\label{eq:def_tau_exit_loc}
    \tau = \tau_{R,0,\eps}=\tau_\fR.
\end{align}
We also recall the definition of stability index $\rho$ in \eqref{eq:rho}.

\begin{proposition}\label{prop:exit_loc_beta<1}
Suppose $\rho<1$. Let $\vk,\vk'>0$. Let $\tau$ be defined by~\eqref{eq:def_tau_exit_loc}, and~$\cc_1$ by~\eqref{eq:def_cc_1}. Then for each $\beta\in(\rho,1)$, there is $\delta >0$ such that
\begin{align*}
    \sup_{\substack{\scx \in K_\vk(\e) \\ [a,b]\subset K_{\vk'}(\e)}}\left|\eps^{-(\frac{\beta}{\rho}-1)}\Probxy{Y_\tau \in\{R\}\times \eps^\beta [a,b]}-RL^{-\frac{1}{\rho}}g_{\cc_1}\left(\scx \right)\left(|b\vee 0|^\frac{1}{\rho}-|a\vee 0|^\frac{1}{\rho}\right)\right|
    \\
    = \smallo{\eps^\delta}.
\end{align*}
\end{proposition}

\begin{proof} 
In this proof and further on, we often use the notation $\eqpm$ introduced in~\eqref{eq:eqpm_notation}.
Using Lemma~\ref{lem:hp-events-in-rectangle}, we have that, under $\Pp^{\eps x}$ uniformly in $x\in K_\vk(\eps)$ and $[a,b]\in K_{\vk'}(\eps)$,
\begin{align*}
    \left\{Y_\tau \in \{R\}\times \eps^\beta[a,b] \right\} 
    \eqmodlp \left\{\scx +U^1_\tau \geq 0;\ R^{-\rho}L\eps^\rho\left(\scx +U^1_\tau\right)^\rho+ \eps N^2_\tau \in \eps^\beta[a,b]\right\}.
\end{align*}
Then, Lemma~\ref{lem:N_tau_tail} implies that, for any $\eta\in(0,1-\beta)$,
\begin{align*}
    &\P^{\eps \scx }\left\{Y_\tau \in \{R\}\times \eps^\beta[a,b] \right\} 
    \\
    &\eqpm \P^{\eps \scx }\left\{\scx +U^1_\tau \geq 0;\ R^{-\rho}L\eps^\rho(\scx +U^1_\tau)^\rho\in \eps^\beta[a\mp\eps^\eta,b\pm \eps^\eta]\right\} \pm \superpoly
    \\
    & =  \P^{\eps \scx }\left\{\scx +U^1_\tau \in \Aepm\right \} \pm \superpoly,
\end{align*}
uniformly in $x\in K_\vk(\eps)$ and $[a,b]\subset K_{\vk'}(\eps)$, where
\begin{align*}
    \Aepm = RL^{-\frac{1}{\rho}}\eps^{\frac{\beta}{\rho}-1}\left[\left((a\mp\eps^\eta)\vee 0 \right)^\frac{1}{\rho},\left((b\mp\eps^\eta)\vee 0 \right)^\frac{1}{\rho}\right].
\end{align*}
Lemma~\ref{lem:N_tau_tail} ensures that
\begin{align}\label{eq:tau=zeta_R,0,eps_whp}
   \sup_{x\in K_\vk(\eps)}\Pp^{\eps^\alpha x}\{\tau \neq  \zeta_{R,0,\eps}\} = o_e(1),\quad \alpha \in(0,1].
\end{align}
Using \eqref{eq:tau=zeta_R,0,eps_whp} with $\alpha=1$, and Corollary~\ref{cor:gaussian_approx} with $R,\frac{\beta}{\rho}-1,0$ substituted for $r,\xi,\theta$, we obtain 
\begin{align*}
     \Probxy{Y_\tau \in\{R\}\times \eps^\beta [a,b]}\eqpm\Probxy{\scx +\frU \in \Aepm }\pm \smallo{\eps^{\frac{\beta}{\rho}-1+\delta}}
\end{align*}
for some $\delta>0$,
uniformly in $\scx \in K_\vk(\e)$ and $ [a,b]\subset K_{\vk'}(\e)$. Since the variance of~$\frU$ is~$\cc_1$, an elementary Gaussian integral estimate (see the proof of \eqref{eq:integral_est} below for a similar argument) yields that
\begin{align*}
    \sup_{\substack{x\in K_\vk(\e) \\ [a,b]\subset K_{\vk'}(\e)}}\left|\eps^{-(\frac{\beta}{\rho}-1)}\Probxy{\scx +\frU \in \Aepm } - RL^{-\frac{1}{\rho}}g_{\cc_1}\left(\scx \right)\left(|b\vee 0|^\frac{1}{\rho}-|a\vee 0|^\frac{1}{\rho}\right)\right|
    \\
    = \smallo{\eps^{\delta'}},
\end{align*}
for some $\delta'>0$. Combining
the last two displays we complete the proof.
\end{proof}

Let us now consider the case $\beta=1$.
In addition to $\cc_1$, we define
\begin{align}
    \label{eq:def_cc_2}
\cc_2&=\int_{-\infty}^0 e^{2\mu s}\left|F^2(Re^{-\lambda s},0)\right|^2 ds. 
\end{align}

\begin{proposition}\label{prop:exit_loc} Suppose  $\rho<1$.
Let $\vk,\vk'>0$. Let $\tau$ be given in~\eqref{eq:def_tau_exit_loc}, $\cc_1$~in~\eqref{eq:def_cc_1}, $\cc_2$ in~\eqref{eq:def_cc_2}. Then there is $\delta >0$ such that
\begin{align*}
    \sup_{\substack{\scx \in K_\vk(\e) \\ [a,b]\subset K_{\vk'}(\e)}}\left|\eps^{-(\frac{1}{\rho}-1)}\Probxy{Y_\tau \in\{R\}\times \eps [a,b]}-RL^{-\frac{1}{\rho}}g_{\cc_1}\left(\scx \right)\E h(a,b;\NN)\right|= \smallo{\eps^\delta}
\end{align*}
where
\begin{align}\label{eq:h(a,b;z)}
    h(a,b;z)=|(b-z)\vee 0|^\frac{1}{\rho}-|(a-z)\vee 0|^\frac{1}{\rho}
\end{align}
and $\NN$ is a centered Gaussian r.v.\ with variance $\cc_2$.
\end{proposition}

Recall the family of stopping times given in \eqref{eq:exit-from-strip}.
We need the next lemma, which is slightly more general than the setting of Proposition \ref{prop:exit_loc}.
In particular, we are not requiring $\rho<1$ here.

\begin{lemma}\label{lem:asymp_decouple}
Suppose $\alpha\in(0,1]$, $\rho>0$.
Let $\beta,\vk,\vk'>0$. For $\theta\in (0,1)$, we set
\begin{align}\label{eq:def_htau}
    \htau = \zeta_{1,\theta,\eps}.
\end{align}
Then for any sufficiently small $\theta>0$
and sufficiently small $\eta>0$, there is $\delta>0$ such that the following holds uniformly in $\scx \in K_\vk(\e)$ and $[a,b]\subset K_{\vk'}(\e)$:
\begin{align*}
    \Pp^{\eps^\alpha \scx }\left\{Y_\tau\in\{R\}\times\eps^\beta[a,b]\right\}\eqpm &\Pp^{\eps^\alpha \scx }\left\{\scx +\eps^{1-\alpha}U^1_\htau\in \Bepm(\NN)\right\}
    \\
    &\pm \eps^{\delta}\Pp^{\eps^\alpha \scx }\left\{\scx +\eps^{1-\alpha}U^1_\htau\in \Bepm(\cZZ)\right\}\pm \superpoly,
\end{align*}
where $\tau$ is given in \eqref{eq:def_tau_exit_loc}, and, for $s\in\R$,
\begin{align}\label{eq:def_Bepm}
    \Bepm(s)= \eps^{-\alpha}RL^{-\frac{1}{\rho}}\left[\left|(\eps^\beta a-\eps s\mp\eps^{1+\eta})\vee 0\right|^\frac{1}{\rho},\ \left|(\eps^\beta b-\eps s\pm\eps^{1+\eta})\vee 0\right|^\frac{1}{\rho}\right]\subset \R,
\end{align}
and $\NN, \cZZ$ are centered Gaussian r.v.'s (defined on an extended probability space) such that  the random vector $(U_\zeta^1,\NN,\cZZ)$ has independent components.
The variance of $\NN$ equals $\cc_2$ given in~\eqref{eq:def_cc_2}, and the variance of $\cZZ$ does not depend on $x$.

\end{lemma}

\begin{remark}\rm
In principle, the nonlinear dynamical system we are considering entangles the noisy perturbations in various directions in a sophisticated way. 
However, this key lemma describes the asymptotic disentanglement of noisy contributions in two coordinate directions and 
gives the asymptotics of the exit distribution in terms of
independent r.v.'s $U^1_\htau$ and $\NN$, 
 These two r.v.'s can be viewed as contributions from the white noise accumulated along two coordinate axes, $\NN$ being the distributional  limit of $N^2_\tau$. The asymptotic independence emerges since the determining noisy contributions along the first axis and the second axis 
 are mostly accumulated during two non-overlapping time intervals: 
 (i) during the motion along the stable manifold (until $\htau$), and (ii) during the motion along the unstable manifold (after $\htau$). 
\end{remark}

\begin{proof}[Proof of Proposition~\ref{prop:exit_loc}]
Let $\NN$ and $\cZZ$ be given in Lemma~\ref{lem:asymp_decouple} for $\alpha =1$ and $\beta =1$. 

The treatment for terms involving $\NN$ and $\cZZ$ is exactly the same since they are both independent centered Gaussian r.v.'s. Hence, we only present the argument for $\NN$ and estimate
\begin{align*}
    \Probxy{\scx +U^1_\htau \in \Bepm(\NN)}.
\end{align*}
We apply Corollary~\ref{cor:gaussian_approx} with 
$r=1$, $\xi=\frac{1}{\rho}-1$ to see that for some $\delta'>0$ 
\begin{align*}
    \sup_{\substack{\scx \in K_\vk(\e)\\ [a,b]\subset K_{\vk'}(\e)\\ |z|\leq  l_\eps}}\left|\Probxy{\scx +U^1_\htau\in \Bepm(z)}-\Prob{\scx +\frU\in \Bepm(z)} \right|=\smallo{\eps^{\frac{1}{\rho}-1+\delta'}}.
\end{align*}
Using the above display and the Gaussian tail of $\NN$, and integrating in $z$ with respect to the law of $\NN$, we obtain that
\begin{align*}\begin{split}
    \Probxy{\scx +U^1_\htau\in\Bepm(\NN)}&=  \Prob{\scx +\frU\in\Bepm(\NN)}+\smallo{\eps^{\frac{1}{\rho}-1+\delta'}}
\end{split}
\end{align*}
uniformly in $\scx \in K_\vk(\e)$ and $[a,b]\subset K_{\vk'}(\e)$.

Let us write
\begin{align*}
    \Prob{\scx +\frU\in \Bepm(\NN)}=\E \int_{\Bepm(\NN)}g_{\cc_1}\left(s-\scx \right)ds.
\end{align*}
We need the following estimate, the proof of which is postponed:
\begin{align}
    \sup_{\substack{\scx \in K_\vk(\e) \\ [a,b]\subset K_{\vk'}(\e)\\ z\in\R}}\left|\eps^{-(\frac{1}{\rho}-1)}\int_{\Bepm(z)}g_{\cc_1}\left(s-\scx \right)ds - RL^{-\frac{1}{\rho}}g_{\cc_1}\left(\scx \right) h(a,b;z)\right| \notag
    \\
    \leq C\eps^{\delta''}(|z|^p+1), \label{eq:integral_est}
\end{align}
for some $\delta'', p >0$, where $h$ is defined in~\eqref{eq:h(a,b;z)}.

Hence, the three displays above yield
\begin{align*}
    \sup_{\substack{\scx \in K_\vk(\e) \\ [a,b]\subset K_{\vk'}(\e)}}\left|\eps^{-(\frac{1}{\rho}-1)}\Probxy{\scx +U^1_\htau\in\Bepm(\NN)}-RL^{-\frac{1}{\rho}}g_{\cc_1}\left(\scx \right)\E h(a,b;\NN)\right|
    \\
    =\smallo{\eps^{\delta'\wedge\delta''}}.
\end{align*}
A similar result holds with $\NN$ replaced by $\cZZ$, which gives, due to $|h(a,b;z)|\leq C(|a|^\frac{1}{\rho}+|b|^\frac{1}{\rho}+|z|^\frac{1}{\rho})$ (see the definition of $h$ in \eqref{eq:h(a,b;z)}), that
\begin{align*}
    \sup_{\substack{x\in K_\vk(\e) \\ [a,b]\subset K_{\vk'}(\e)}} \Probxy{\scx +U^1_\htau\in \Bepm(\cZZ)}=\mathcal{O}\left(\eps^{\frac{1}{\rho}-1}l_\eps^\frac{\vk'}{\rho}\right).
\end{align*}
The above two displays together with Lemma~\ref{lem:asymp_decouple} imply the desired result.
\end{proof}

\begin{proof}[Proof of~\eqref{eq:integral_est}]
All statements below are understood to hold uniformly in $\scx \in K_\vk(\e)$ and $[a,b]\subset K_{\vk'}(\e)$.
Let us set
\begin{align*}
    B^\eps(z) = \eps^{\frac{1}{\rho}-1}RL^{-\frac{1}{\rho}}\left[\left|(a-z)\vee 0\right|^\frac{1}{\rho},\ \left|(b-z)\vee 0\right|^\frac{1}{\rho}\right]\subset \R.
\end{align*}
We shall compare the terms in \eqref{eq:integral_est} with
\begin{align*}
    \mathtt{I}= \eps^{-(\frac{1}{\rho}-1)}\int_{B^\eps(z)} g_{\cc_1}\left(s-\scx \right)ds.
\end{align*}
Using the definitions of $B^\eps_\pm(z)$ in~\eqref{eq:def_Bepm} (with $\alpha=\beta=1$), $B^\eps(z)$ above, $K_\vk(\eps)$ and $K_{\vk'}(\eps)$ in~\eqref{eq:def_K(eps)}, we have
\begin{align*}
    \left|\eps^{-(\frac{1}{\rho}-1)}\int_{\Bepm(z)}g_{\cc_1}\left(s-\scx \right)ds -\mathtt{I}\right|\leq C\eps^{-(\frac{1}{\rho}-1)}\left|\Bepm(z) \triangle B^\eps(z)\right|\leq C\eps^\delta(|z|^p+1),
\end{align*}
for some $\delta, p >0$.
The definition of $h(a,b;z)$ in~\eqref{eq:h(a,b;z)} implies that
\begin{align*}
     RL^{-\frac{1}{\rho}}g_{\cc_1}\left(\scx \right) h(a,b;z) = \eps^{-(\frac{1}{\rho}-1)}\int_{B^\eps(z)}g_{\cc_1}\left(\scx \right) ds.
\end{align*}
Due to the definitions of $B^\eps(z)$, $K_\vk(\eps)$ and $K_{\vk'}(\eps)$ and the fact that $|g_{\cc_1}\left(\scx \right) -g_{\cc_1}\left(s-\scx \right) |\leq C|s|$, we obtain, for some $\delta',p'>0$,
\begin{align*}
    \left|\mathtt{I}- \eps^{-(\frac{1}{\rho}-1)}\int_{B^\eps(z)}g_{\cc_1}\left(\scx \right)ds\right|\leq C\int_{B^\eps(z)}\eps^{-(\frac{1}{\rho}-1)}|s|ds\leq C\eps^{\delta'}(|z|^{p'}+1).
\end{align*}
Combining the above three displays, we arrive at~\eqref{eq:integral_est}.
\end{proof}

\subsubsection{Proof of Lemma~\ref{lem:asymp_decouple}}

Let us outline the plan. We will stop the process $Y$ at $\htau$ (given in~\eqref{eq:def_htau}) using the strong Markov property and show that from $\htau$ onward, the exit event can be approximated by a simpler event involving only $Y^2_\Tone$ (equivalently, $N^2_\Tone$ due to~\eqref{eq:SDE_near_a_saddle_point_after_duhamel2}) at a deterministic time $\Tone$ (Lemma~\ref{lem:approximation_starting_at_ty}); then we apply a density estimate result to show that this simpler event can be approximated by replacing $N^2_\Tone$ by $\NN$ (Lemma~\ref{lem:swap_by_gaussian}); finally, we undo the stopping at $\htau$ and complete the proof of Lemma~\ref{lem:asymp_decouple}.

In this proof, if not otherwise specified, all statements are understood to hold uniformly in $x\in K_\vk(\eps)$ and $[a,b]\subset K_{\vk'}(\eps)$.

Before proceeding, let us make a further notational simplification. It is easier to work with stopping times for exiting a vertical strip
as in \eqref{eq:exit-from-strip}. So, let us redefine 
\begin{align}\label{eq:tau=zeta}
    \tau = \zeta_{R,0,\eps}.
\end{align}
Due to \eqref{eq:tau=zeta_R,0,eps_whp}, working with this definition of $\tau$ instead of  the original one, we introduce a uniform probability error of order $o_e(1)$. Therefore, although we prove all the results in this section using the definition in~\eqref{eq:tau=zeta}, they 
also automatically
hold true for the original definition in~\eqref{eq:def_tau_exit_loc}.
 
In view of \eqref{eq:def_htau} and~\eqref{eq:tau=zeta}, we have 
\begin{align*}
    \htau\leq \tau.
\end{align*}
Using~\eqref{eq:SDE_near_a_saddle_point_after_duhamel1}, we have that, whenever $|Y^1_0|<\eps^\theta$,
\begin{align}\label{eq:htau_formula}
    \eps^\theta= e^{\lam \htau}|\eps^\alpha \scx +\eps U^1_\htau|, \quad\text{or equivalently,}\quad \htau=\frac{1}{\lam}\log\frac{\eps^\theta}{|\eps^\alpha \scx +\eps U^1_\htau|},
\end{align}
and, whenever $|Y^1_0|<R$, 
\begin{align}\label{eq:tau(=zeta)_formula}
    R= e^{\lam \tau}|Y_0^1 +\eps U^1_\tau|, \quad\text{or equivalently,}\quad \tau=\frac{1}{\lam}\log\frac{R}{|Y^1_0 +\eps U^1_\tau|}.
\end{align}

Let us disintegrate the distribution of $Y_\tau$ with respect to $Y_\htau$ using only the  typical values of the latter:

\begin{lemma}\label{lem:unlikely_position_Y_htau}
If
\begin{align}\label{eq:theta_0_condition}
    0<\vartheta_0<\beta\wedge 1,
\end{align}
then, for sufficiently small $\theta>0$, the following holds uniformly in $\scx \in K_\vk(\e)$
and in $[a,b]\in K_{\vk'}(\e)$,
\begin{multline}\label{eq:Prob(Y^2_tau_in_eps[a,b])_approx_1}
\Pp^{\eps^\alpha \scx }\{Y_\tau\in\{R\}\times\eps^\beta [a,b]\}\\
=\E^{\eps^\alpha \scx }\left[\Probx{Y_\htau}{Y_\tau\in\{R\}\times\eps^\beta [a,b]}\ONE_{\{Y^1_\htau = \eps^\theta,\  |Y^2_\htau|<\eps^{\vartheta_0}\}}\right]+\superpoly.
\end{multline}
\end{lemma}

\begin{proof}
Let us first exclude unlikely values of $Y_\htau$ and prove the following:
\begin{gather}
\label{eq:unlikely_position_Y_htau--1}
        \Pp^{\eps^\alpha \scx }\{Y_\tau \in \{R\}\times\eps^\beta K_{\vk'}(\e),\ |Y^2_\htau|\geq \eps^{\vartheta_0}\}=\superpoly,\\\label{eq:unlikely_position_Y_htau--2} 
        \Pp^{\eps^\alpha \scx }\{Y_\tau \in \{R\}\times\eps^\beta K_{\vk'}(\e), \ Y^1_\htau = -\eps^\theta\}=\superpoly.
\end{gather}

Since $|Y^1_0|\leq\eps^\alpha l_\eps^\vk\leq \eps^\theta$ for sufficiently small $\eps$, we know that
\begin{align}\label{eq:Y^1_htau}
    |Y^1_\htau|= \eps^\theta.
\end{align}
To estimate $Y^2_\htau$, we use the strong Markov property and~\eqref{eq:SDE_near_a_saddle_point_after_duhamel2} to see
\begin{align}\label{eq:strong_markov_Y_htau}
\begin{split}
    &\Pp^{\eps^\alpha \scx}\{Y^2_\tau\in\eps^\beta K_{\vk'}(\e),\ |Y^2_\htau|\geq \eps^{\vartheta_0}\}=\E^{\eps^\alpha \scx}\left[\Probx{Y_\htau}{Y^2_\tau\in\eps^\beta K_{\vk'}(\e)}\ONE_{|Y^2_\htau|\geq \eps^{\vartheta_0}}\right]\\
    &=\E^{\eps^\alpha \scx}\left[\Probx{Y_\htau}{e^{-\mu\tau}Y^2_0+\eps N^2_\tau\in\eps^\beta K_{\vk'}(\e)}\ONE_{|Y^2_\htau|\geq \eps^{\vartheta_0}}\right].
\end{split}
\end{align}
Due to \eqref{eq:Y^1_htau}, we can use~\eqref{eq:tau(=zeta)_formula} to see that 
\begin{align}\label{eq:e^(-mutau)Y}
    \Probx{Y_\htau}{e^{-\mu\tau}Y^2_0+\eps N^2_\tau\in\eps^\beta   K_{\vk'}(\e) }= \Probx{Y_\htau}{R^{-\rho }|Y^1_0+\eps U^1_\tau|^\rho Y^2_0+\eps N^2_\tau\in\eps^\beta K_{\vk'}(\e)}.
\end{align}
We want to control $N^2_\tau$ in the above display using Lemma~\ref{lem:estimates-for-terms}~\eqref{item:lemma:N^2_sigma}. Using~\eqref{eq:tau(=zeta)_formula} with $Y_0 = z$ satisfying $|z^1|=\eps^\theta$, we have for $q>\theta\lam^{-1}$ and $\eps$ sufficiently small,
\begin{align*}
    \Probx{z}{\tau \geq q l_\eps}\leq \Probx{z}{|U^1_\tau|>\eps^{\theta-1}-R\eps^{q\lam-1}}=\superpoly,
\end{align*}
where we used Lemma~\ref{lem:estimates-for-terms}~\eqref{item:tail-of_U_1}. This along with Lemma~\ref{lem:estimates-for-terms}~\eqref{item:lemma:N^2_sigma} implies that
\begin{align*}
    \Probx{Y_\htau}{|N^2_\tau|>\eps^{-p}} =\superpoly
\end{align*}
for every $p>0$. Using this, $[a,b]\subset K_{\vk'}(\eps)$ and the fact that $|Y^2_0|\geq\eps^{\vartheta_0}$ holds a.s.\ under $\Pp^{Y_\htau}$ with $|Y^2_\htau|\geq\eps^{\vartheta_0}$, we bound the left-hand side of~\eqref{eq:e^(-mutau)Y} from above by
\begin{align*}
    \Probx{Y_\htau}{C|Y^1_0+\eps U^1_\tau|^\rho \eps^{\vartheta_0}\leq \eps^{(\beta\wedge1)-\delta}} + \superpoly,\quad\text{if }|Y^2_\htau|\geq\eps^{\vartheta_0},
\end{align*}
for arbitrary $\delta>0$.
Since $|Y_0^1|=\eps^\theta$ holds a.s.\ under $\Pp^{Y_\htau}$ due to ~\eqref{eq:Y^1_htau}, we use Lemma~\ref{lem:estimates-for-terms}~\eqref{item:tail-of_U_1} 
and~\eqref{eq:theta_0_condition} to see that, for sufficiently small $\theta$, we can choose $\delta$ so that
 the main term in this display can be bounded by
\begin{align*}
     \Probx{Y_\htau}{\eps^{-1}(\eps^\theta- C\eps^{\frac{1}{\rho}((\beta\wedge1)-\delta-\vartheta_0)})\leq |U^1_\tau|}=\superpoly,\quad\text{if }|Y^2_\htau|\geq\eps^{\vartheta_0}.
\end{align*}
Hence, the left-hand side of~\eqref{eq:e^(-mutau)Y} is $\superpoly$ when $|Y^2_\htau|\geq\eps^{\vartheta_0}$. Inserting this into~\eqref{eq:strong_markov_Y_htau}, we obtain
\begin{align*}\Pp^{\eps^\alpha \scx }\{Y^2_\tau\in\eps^\beta K_{\vk'}(\e),\ |Y^2_\htau|\geq \eps^{\vartheta_0}\}=\superpoly.
\end{align*}
and \eqref{eq:unlikely_position_Y_htau--1} follows. To prove \eqref{eq:unlikely_position_Y_htau--2}, we apply the strong Markov property:
\begin{align*}
    &\Pp^{\eps^\alpha \scx } \{Y^1_\tau=R,\ Y_\htau = -\eps^\theta\}\leq \E^{\eps^\alpha \scx }\left[\Probx{Y_\htau}{Y^1_\tau>0}\ONE_{Y^1_\htau = -\eps^\theta}\right]\\
    &\leq \E^{\eps^\alpha \scx }\Probx{Y_\htau}{-\eps^\theta+\eps U^1_\tau>0}\leq  \E^{\eps^\alpha \scx }\Probx{Y_\htau}{ U^1_\tau>\eps^{\theta-1}}=\superpoly,
\end{align*}
where we used~\eqref{eq:SDE_near_a_saddle_point_after_duhamel1} in the second estimate and Lemma~\ref{lem:estimates-for-terms}~\eqref{item:tail-of_U_1} in the last one.
Finally, applying the strong Markov property and 
relations~\eqref{eq:unlikely_position_Y_htau--1},~\eqref{eq:unlikely_position_Y_htau--2}, 
we see that uniformly in $x\in K_{\vk}(\e)$ and $[a,b]\in K_{\vk'}(\e)$,
\begin{align*}
\begin{split}
&\Pp^{\eps^\alpha \scx }\{Y_\tau\in\{R\}\times\eps^\beta [a,b]\}\\
&= \Pp^{\eps^\alpha \scx }\{Y_\tau\in\{R\}\times\eps^\beta [a,b],\ Y^1_\htau = \eps^\theta, \  |Y^2_\htau|<\eps^{\vartheta_0}\}+\superpoly\\
&=\E^{\eps^\alpha \scx }\left[\Probx{Y_\htau}{Y_\tau\in\{R\}\times\eps^\beta [a,b]}\ONE_{\{Y^1_\htau = \eps^\theta,\  |Y^2_\htau|<\eps^{\vartheta_0}\}}\right]+\superpoly,
\end{split}
\end{align*}
so \eqref{eq:Prob(Y^2_tau_in_eps[a,b])_approx_1} holds, and the proof is completed. \end{proof}

Now, we investigate the dynamics after $\htau$. Taking into account the indicator function in the above display, we study $\Probx{\tscx}{Y^2_\tau\in\eps^\beta[a,b]}$ for $[a,b]\subset K_{\vk'}(\e)$ and $\tscx \in \R^2$ satisfying
\begin{align}\label{eq:range_of_ty_tilde_y}
    \tscx^1=\eps^\theta,  \quad\quad |\tscx^2|<\eps^{\vartheta_0}.
\end{align}
For these values of $y$, due to~\eqref{eq:tau(=zeta)_formula}, we have $\Pp^\tscx$-a.s.
\begin{align}\label{eq:tau_formula_under_ty}
    \tau =\frac{1}{\lam}\log\frac{R}{|\tscx^1+\eps U^1_\tau|},
\end{align}
which is to be compared with the following deterministic time
\begin{align}\label{eq:def_Tone_T_1_range}
    \Tone = \Tone(\eps)=\frac{1}{\lam}\log\frac{R}{|\tscx^1|}=\frac{1}{\lam}\log\frac{R}{\eps^\theta}.
\end{align}
We emphasize that $\Tone$ is in fact independent of $\tscx$ under assumption~\eqref{eq:range_of_ty_tilde_y}. The next result shows that $\Tone$ is a good approximation of $\tau$ under $\Pp^{\tscx}$. 

\begin{lemma}\label{lem:approximation_starting_at_ty}
If $ \vartheta_0,\vartheta_1>0$ satisfy
\begin{align}\label{eq:parameter_condition_2}
    \vartheta_1<1,\qquad \vartheta_1+\vartheta_0>1,
\end{align}
then, for sufficiently small $\theta,\eta>0$, the following holds uniformly in~$\tscx$ satisfying~\eqref{eq:range_of_ty_tilde_y} and~$[a,b]\subset K_{\vk'}(\eps)$,
\begin{align*}
    \Probx{\tscx}{Y_\tau\in\{R\}\times\eps^\beta[a,b]}\eqpm \Probx{\tscx}{Y^2_\Tone\in\left[\eps^\beta a\mp\tfrac{\eps^{1+\eta}}{2},\, \eps^\beta b\pm\tfrac{\eps^{1+\eta}}{2}\right]}\pm\superpoly.
\end{align*}
\end{lemma}

\begin{proof}
In this proof, if not otherwise specified, all statements are understood to hold uniformly in $\tscx$ satisfying \eqref{eq:range_of_ty_tilde_y} and $[a,b]\subset K_{\vk'}(\eps)$.
Since $Y^1_0=\tscx^1=\eps^\theta$ holds a.s.\ under~$\Pp^{\tscx}$, using \eqref{eq:SDE_near_a_saddle_point_after_duhamel1} and Lemma~\ref{lem:estimates-for-terms}~\eqref{item:tail-of_U_1}, we obtain
\begin{align*}
    \Probx{\tscx}{Y^1_\tau \neq R}\leq \Probx{\tscx}{\eps^\theta+\eps U^1_\tau <0}\leq\Probx{\tscx}{|U^1_\tau| >\eps^{\theta-1}}=\superpoly.
\end{align*}
The desired result will follow once we show that, for all sufficiently small $\eta>0$, 
\begin{align}\label{eq:Y_tau_Y_Tone_comparison}
    \Probx{\tscx}{|Y^2_\tau-Y^2_\Tone|> \tfrac{1}{2}\eps^{1+\eta}}=\superpoly.
\end{align}

To show~\eqref{eq:Y_tau_Y_Tone_comparison}, we start by controlling $\tau - \Tone$. Using~\eqref{eq:tau_formula_under_ty},~\eqref{eq:def_Tone_T_1_range},~\eqref{eq:range_of_ty_tilde_y}, Lemma~\ref{lem:estimates-for-terms}~\eqref{item:tail-of_U_1}, the fact the $e^s-1\geq s $,  and~\eqref{eq:parameter_condition_2}, we have
\begin{align*}
\begin{split}
        \Probx{\tscx}{\tau - T_1 \geq \eps^{\vartheta_1}}&\leq \Probx{\tscx}{\log\frac{|\tscx^1|}{|\tscx^1+\eps U^1_\tau|}\geq \lam\eps^{\vartheta_1}}\leq \Probx{\tscx}{e^{\lam\eps^{\vartheta_1}}\eps|U^1_\tau|\geq (e^{\lam\eps^{\vartheta_1}}-1)|\tscx^1|}\\
    &\leq \Probx{\tscx}{e^{\lam\eps^{\vartheta_1}}|U^1_\tau|\geq \lam \eps^{\vartheta_1-1+\theta}}=\superpoly,
\end{split}
\end{align*}
if $\theta$ is small enough to ensure $\vartheta_1-1+\theta<0$. 
Similarly, 
\begin{align*}
\begin{split}
        \Probx{\tscx}{T_1-\tau  \geq \eps^{\vartheta_1}}&\leq \Probx{\tscx}{\log\frac{|\tscx^1+\eps U^1_\tau|}{|\tscx^1|}\geq \lam\eps^{\vartheta_1}}\\ &\leq \Probx{\tscx}{|U^1_\tau|\geq \lam \eps^{\vartheta_1-1+\theta}}=\superpoly.
\end{split}
\end{align*}
In conclusion, we have
\begin{align}\label{eq:|T_1-tau|_est}
    \Probx{\tscx}{|T_1-\tau|  \geq \eps^{\vartheta_1}}=\superpoly.
\end{align}
With this estimate at hand, let us compare $Y^2_\tau $ and $Y^2_\Tone$. 
Using~\eqref{eq:SDE_near_a_saddle_point_after_duhamel2}, we have
\begin{align}\label{eq:|Y_tau-Y_Tone|_est}
    \Probx{\tscx}{|Y^2_\tau-Y^2_\Tone|\geq \tfrac{1}{2}\eps^{1+\eta}}\leq \Probx{\tscx}{|\tscx^2||e^{-\mu \tau}-e^{-\mu \Tone}|\geq \tfrac{1}{4}\eps^{1+\eta}}+ \Probx{\tscx}{|N^2_\tau - N^2_\Tone|\geq \tfrac{1}{4}\eps^{\eta}}.
\end{align}
Let us estimate the first term on the right of~\eqref{eq:|Y_tau-Y_Tone|_est}. On $\{|\Tone-\tau|< \eps^{\vartheta_1}\}$, we have 
\begin{align*}
    |e^{-\mu \tau}-e^{-\mu \Tone}| \leq \mu|\tau -\Tone|\leq \mu\eps^{\vartheta_1}.
\end{align*}
Hence, using~\eqref{eq:|T_1-tau|_est}, ~\eqref{eq:range_of_ty_tilde_y} and~\eqref{eq:parameter_condition_2}, we obtain
\begin{align}\label{eq:|e^(-mu_tau)-e^(-mu_Tzero)|_est}
\Probx{\tscx}{|\tscx^2||e^{-\mu \tau}-e^{-\mu \Tone}|\geq \tfrac{1}{4}\eps^{1+\eta}}\leq \Probx{\tscx}{\mu\eps^{\vartheta_1+\vartheta_0}\geq \tfrac{1    }{4}\eps^{1+\eta}}+\superpoly=\superpoly,
\end{align}
for sufficiently small $\eta>0$.

Then, we turn to the second term on the right of~\eqref{eq:|Y_tau-Y_Tone|_est}.
Due to~\eqref{eq:|T_1-tau|_est},
\begin{align}\label{eq:|N_tau-N_Tone|>0.5eps^eta}
    \Probx{\tscx}{|N^2_\tau - N^2_\Tone|\geq \tfrac{1}{4}\eps^{\eta}}\leq \Probx{\tscx}{|N^2_{\overline\tau} - N^2_\Tone|\geq \tfrac{1}{4}\eps^{\eta}}+\superpoly,
\end{align}
where we have set
\begin{align*}
    \overline{\tau }= (\tau \vee(\Tone -\eps^{\vartheta_1}))\wedge (\Tone +\eps^{\vartheta_1}).
\end{align*}
 Then, we write
\begin{multline}\label{eq:P(N)_split}
    \Probx{\tscx}{|N^2_{\overline\tau} - N^2_\Tone|\geq \tfrac{1}{4}\eps^{\eta}} \leq \Probx{\tscx}{e^{-\mu \Tone}|U^2_{\overline\tau} - U^2_\Tone|\geq \tfrac{1}{8}\eps^{\eta}}\\ + \Probx{\tscx}{|e^{-\mu {\overline\tau} }-e^{-\mu \Tone}||U^2_{\overline\tau}|\geq \tfrac{1}{8}\eps^{\eta}}.
\end{multline}
Due to \eqref{eq:def_V_M_U_N}, the first term on the right of \eqref{eq:P(N)_split} can be estimated as 
\begin{align*}
    & \Probx{\tscx}{e^{-\mu \Tone}|U^2_{\overline\tau} - U^2_\Tone|\geq \tfrac{1}{8}\eps^{\eta}} 
    \leq \Probx{\tscx}{e^{-\mu \Tone}|M^2_{\overline\tau} - M^2_{\Tone-\eps^{\vartheta_1}}|\geq \tfrac{1}{32}\eps^{\eta}} 
    \\
    &\qquad+ \Pp^{\tscx}\left\{e^{-\mu \Tone}|M^2_{\Tone} - M^2_{\Tone-\eps^{\vartheta_1}}|\geq \tfrac{1}{32}\eps^{\eta}\right\}
    + \Probx{\tscx}{e^{-\mu \Tone}|V^2_{\overline\tau} - V^2_\Tone|\geq \tfrac{1}{16}\eps^{\eta-1}}
    \\
    &\leq 2\Pp^{\tscx}\left\{\sup_{t\in[\Tone-\eps^{\vartheta_1},\Tone+\eps^{\vartheta_1}]}|M^2_t - M^2_{\Tone-\eps^{\vartheta_1}}|\geq \tfrac{1}{32}e^{\mu \Tone}\eps^{\eta}\right\}
    \\ &{\ }\qquad\qquad\qquad\qquad\qquad\qquad\qquad\qquad\qquad\qquad+ \Probx{\tscx}{|V^2_{\overline\tau} - V^2_\Tone|\geq \tfrac{1}{16}e^{\mu \Tone}\eps^{\eta-1}}.
\end{align*}
For $t\in[\Tone-\eps^{\vartheta_1},\Tone+\eps^{\vartheta_1}]$,
we have $\langle M^2\rangle_t- \langle M^2\rangle_{\Tone-\eps^{\vartheta_1}}\le Ce^{2\mu\Tone}\eps^{\vartheta_1}$ (see \eqref{eq:def_V_M_U_N}). Also, $|V^2_{\overline\tau} - V^2_\Tone|$ is bounded by $Ce^{\mu\Tone}\eps^{\vartheta_1}$. Using these and the exponential martingale inequality (Lemma~\ref{lem:exp-marting-ineq}), the above is $o_e(1)$ provided 
$\eta$ is small enough to guarantee
$2\eta-\vartheta_1<0$ and $\eta-1-\vartheta_1<0$.

The second term on the right of \eqref{eq:P(N)_split} can similarly be bounded from above by
\begin{align*}
    \Probx{\tscx}{\left(e^{-\mu (\Tone-\eps^{\vartheta_1}) }-e^{-\mu \Tone}\right)\sup_{t\in[\Tone-\eps^{\vartheta_1},\Tone+\eps^{\vartheta_1}]}|M^2_t|\geq \tfrac{1}{8}\eps^{\eta}}
    \\
    + \Probx{\tscx}{\left(e^{-\mu (\Tone-\eps^{\vartheta_1}) }-e^{-\mu \Tone}\right)|V^2_{\overline\tau}|\geq \tfrac{1}{8}\eps^{\eta-1}}.
\end{align*}
For $t\in[\Tone-\eps^{\vartheta_1},\Tone+\eps^{\vartheta_1}]$, $\langle M^2\rangle_t < Ce^{2\mu\Tone}$. In addition, $|V^2_{\overline\tau}|$ is bounded by $Ce^{\mu\Tone}$. Thus the the exponential martingale inequality (Lemma~\ref{lem:exp-marting-ineq}) and
\begin{align*}
    e^{-\mu (\Tone-\eps^{\vartheta_1}) }-e^{-\mu \Tone}\leq Ce^{-\mu T_1}\eps^{\vartheta_1}
\end{align*}
imply that both terms in the previous display are $o_e(1)$ 
provided $\eta>0$ is small enough to ensure $\eta-\vartheta_1<0$ and $\eta-1-\vartheta_1<0$.

In conclusion, for $\eta$ sufficiently small, we obtain that the left-hand sides in \eqref{eq:P(N)_split} and thus \eqref{eq:|N_tau-N_Tone|>0.5eps^eta} are $o_e(1)$, the latter of which combined with \eqref{eq:|e^(-mu_tau)-e^(-mu_Tzero)|_est} and \eqref{eq:|Y_tau-Y_Tone|_est} verifies \eqref{eq:Y_tau_Y_Tone_comparison}. This completes the proof.
\end{proof}

Let us now choose concrete values $\vartheta_0 = \frac{3(\beta\wedge 1)}{4}$ and $ \vartheta_1=1-\frac{\beta\wedge1}{4}$
satisfying~\eqref{eq:theta_0_condition} and~\eqref{eq:parameter_condition_2} thus making Lemmas~\ref{lem:unlikely_position_Y_htau} and~\ref{lem:approximation_starting_at_ty} applicable.
\medskip

Due to~\eqref{eq:SDE_near_a_saddle_point_after_duhamel2} and~\eqref{eq:def_Tone_T_1_range}, we have 
\begin{align*}
   Y^2_\Tone = e^{-\mu \Tone}\tscx^2+\eps N^2_\Tone=R^{-\rho }\eps^{\theta \rho} \tscx^2+\eps N^2_\Tone,\quad\text{$\Pp^\tscx$-a.s.}
\end{align*}
We define a family of sets $E^\eps_\pm(s)$ for $s\in\R$ by
\begin{align}\label{eq:def_A^eps_pm}
    E^\eps_\pm(s)= \left\{r\in\R:\ R^{-\rho }\eps^{\theta \rho} r+\eps s\in \left[\eps^\beta a\mp\tfrac{\eps^{1+\eta}}{2},\, \eps^\beta b\pm\tfrac{\eps^{1+\eta}}{2}\right] \right\},
\end{align}
which allows us to rewrite
\begin{align}\label{eq:equivalent_formulation_using_E^eps}
    \Probx{\tscx}{Y^2_\Tone\in\left[\eps^\beta a\mp\tfrac{\eps^{1+\eta}}{2},\, \eps^\beta b\pm\tfrac{\eps^{1+\eta}}{2}\right]} = \Probx{\tscx}{\tscx^2 \in E^\eps_\pm(N^2_\Tone)}.
\end{align}
We are suppressing the dependence of $E^\eps_\pm$ on $[a,b]$ in our notation.

Let us estimate the the error caused by replacing  $N^2_\Tone$ by a Gaussian r.v.\ in~\eqref{eq:equivalent_formulation_using_E^eps}. 

\begin{lemma}\label{lem:swap_by_gaussian} 
There are independent centered Gaussian r.v.'s $\NN$ and $\cZZ$ with constant variances such that \begin{align*}
\sup_{\substack{\tscx^1=\eps^\theta ,\, |\tscx^2|<\eps^{\vartheta_0}\\ [a,b]\subset \R}}\left|\Probx{\tscx}{\tscx^2 \in E^\eps_\pm(N^2_\Tone)}- \Prob{\tscx^2 \in E^\eps_\pm(\NN)}\right|
\leq \eps^{\delta} \Prob{\tscx^2 \in E^\eps_\pm(\cZZ)},
\end{align*}
for some $\delta>0$.
In addition, $\NN$ has variance $\cc_2$ given in~\eqref{eq:def_cc_2}.
\end{lemma}

\begin{proof}
Recalling the definition of $\Tone=\Tone(\eps)$ in~\eqref{eq:def_Tone_T_1_range}, we choose $\theta>0$ sufficiently small so that $\Tone \leq \bar\theta l_\eps$ for all small $\eps$, where $\bar\theta$ is given in Lemma~\ref{lem:density_est_R^2}. This allows us to apply Lemma~\ref{lem:density_est_R^2}~\eqref{item:den_est_4_R^2} with $\upsilon<\theta$ to see that there are constants $\delta,\delta',c>0$ such that, for all $\z\in \R^2$,
\begin{align}\label{eq:density_est_applied_in_aymptotic_indep}
\begin{split}
    &\sup_{\substack{\tscx^1=\eps^\theta \\ |\tscx^2|<\eps^{\vartheta_0}}}\left|\varphi^\tscx_{(U^1_\Tone,N^2_\Tone)}(\z)-\varphi^\tscx_{\bZ_\Tone}(\z)\right|
    \\
    &\leq \sup_{\substack{\tscx^1=\eps^\theta \\ |\tscx^2|<\eps^{\vartheta_0}}} C\left( |\tscx^2|+\eps^{\delta}\left(1+\eps^{-\upsilon}|\tscx^1|\right)\right)e^{-c|\z|^2} \leq \eps^{\delta'} e^{-c|\z|^2},
\end{split}
\end{align}
where $\bZ_t$ is defined in \eqref{eq:def_Z_R^2}. Note that $\bZ_t$ does not depend on $\tscx$ once we impose the constraint
$\tscx^1=\eps^\theta$, so we will write $\varphi_{\bZ_\Tone}$ instead of $\varphi^\tscx_{\bZ_\Tone}$.

Using $e^{\lam T_1}\eps^\theta = R$ (due to~\eqref{eq:def_Tone_T_1_range}) and a change of variables, we can get
\begin{align*}
    \E\left|\bZ^2_\Tone\right|^2 = \int_{-\Tone(\eps)}^0e^{2\mu s}\left|F^2(Re^{\lam s},0)\right|^2ds.
\end{align*}
Let $\NN$ be a centered Gaussian r.v.\ with variance $\cc_2$ given in~\eqref{eq:def_cc_2}. It can be easily checked that, for some $\delta'',c''>0$,
\begin{align*}
    \left|\varphi_{\bZ^2_\Tone}(s)-\varphi_{\NN}(s)\right|\leq \eps^{\delta''} e^{-c''|s|^2}, \quad s\in \R.
\end{align*}
Therefore, we conclude from this and~\eqref{eq:density_est_applied_in_aymptotic_indep} that, for some $\bar\delta,\bar c>0$,
\begin{align*}
    \sup_{\substack{\tscx^1=\eps^\theta \\ |\tscx^2|<\eps^{\vartheta_0}}}\left|\varphi^\tscx_{N^2_\Tone}(s)-\varphi_{\NN}(s)\right|\leq \eps^{\bar\delta} e^{-\bar c|s|^2}, \quad s\in \R.
\end{align*}
We emphasize that $\NN$ is independent of $\eps$, $\tscx$. Extending the probability space if necessary, we can assume that $\NN$ is independent of $Y$, and we can also take $\cZZ$ to be an independent centered Gaussian r.v.\ with density proportional to $ e^{-\bar c|\z|^2}$, $\z\in\R$. 
Then, using the above display  and integrating over the region $\{s\in\R: \tscx^2\in E^\eps_\pm(s)\}$, we obtain the desired result.
\end{proof}

Now let us combine the evolution before and after $\htau$. Recalling~\eqref{eq:Prob(Y^2_tau_in_eps[a,b])_approx_1}, we set 
\begin{align*}
    C^\eps = \{Y^1_\htau=\eps^\theta,\ |Y^2_\htau|<\eps^{\vartheta_0} \}=\{Y^1_\htau\geq0,\ |Y^2_\htau|<\eps^{\vartheta_0}\}.
\end{align*}
Hence, \eqref{eq:Prob(Y^2_tau_in_eps[a,b])_approx_1}, Lemma~\ref{lem:approximation_starting_at_ty}, \eqref{eq:equivalent_formulation_using_E^eps}, and Lemma~\ref{lem:swap_by_gaussian} imply
\begin{align}\label{eq:Prob(Y^2_tau_in_eps[a,b])_upper_and_lower_bound_with_constraint}
\begin{split}
    &\Pp^{\eps^\alpha \scx }\left\{Y^2_\tau\in\eps^\beta[a,b]\right\}
    \\
    &\eqpm \Pp^{\eps^\alpha \scx }\{Y^2_\htau \in E^\eps_\pm(\NN),\ C^\eps\}\pm \eps^{\delta} \Pp^{\eps^\alpha \scx }\{Y^2_\htau \in E^\eps_\pm(\cZZ),\ C^\eps\}\pm \superpoly.
\end{split}
\end{align}
The next result removes the constraint $\{|Y^2_\htau|<\eps^{\vartheta_0}\}$.

\begin{lemma}\label{lem:lift_constraint}
The following holds uniformly in $\scx \in K_\vk(\e)$ and $[a,b]\subset K_{\vk'}(\e)$:
\begin{align*}
    \Pp^{\eps^\alpha \scx }\left\{Y_\tau\in\{R\}\times\eps^\beta[a,b]\right\}&\eqpm \Pp^{\eps^\alpha \scx }\left\{Y^2_\htau \in E^\eps_\pm(\NN),\ Y^1_\htau\geq0\right\}\\
    &\pm \eps^{\delta} \Pp^{\eps^\alpha \scx }\left\{Y^2_\htau \in E^\eps_\pm(\cZZ),\ Y^1_\htau\geq0\right\}\pm\superpoly.
\end{align*}
\end{lemma}

\begin{proof}
Using the definition of $E^\eps_\pm$ in~\eqref{eq:def_A^eps_pm} and Gaussian tail of $\NN$, we have that, for sufficiently small $\delta'>0$,
\begin{align*}
    &\Pp^{\eps^\alpha \scx }\left\{Y^2_\htau\in E^\eps_\pm(\NN),\ Y^1_\htau\geq0,\ \left|Y^2_\htau\right|\geq \eps^{\vartheta_0}\right\}\\
    &\leq \Pp^{\eps^\alpha \scx }\left\{R^{-\rho }\eps^{\theta \rho} Y^2_\htau+\eps \NN\in\left[\eps^\beta a\mp\tfrac{\eps^{1+\eta}}{2},\, \eps^\beta b\pm\tfrac{\eps^{1+\eta}}{2}\right],\ \left|Y^2_\htau\right|\geq \eps^{\vartheta_0}\right\}\\
    &\leq \Pp^{\eps^\alpha \scx }\left\{\eps^{\theta \rho +\vartheta_0}\leq C\eps^{(\beta\wedge 1)-\delta'}\right\}+\superpoly=\superpoly.
\end{align*}
where the last equality is guaranteed by~\eqref{eq:theta_0_condition}.
Replacing $\NN$ in the above argument by $\cZZ$, we also have
\begin{align*}
    \Pp^{\eps^\alpha \scx }\left\{Y^2_\htau\in E^\eps_\pm(\cZZ),\ Y^1_\htau\geq0,\ \left|Y^2_\htau\right|\geq \eps^{\vartheta_0}\right\}=\superpoly.
\end{align*}
These two displays above together with~\eqref{eq:Prob(Y^2_tau_in_eps[a,b])_upper_and_lower_bound_with_constraint} yield the desired result.
\end{proof}

Then, we proceed to approximating $\Pp^{\eps^\alpha \scx }\{Y^2_\htau \in E^\eps_\pm(\NN),\ Y^1_\htau\geq0\}$ and \\ $\Pp^{\eps^\alpha \scx }\{Y^2_\htau \in E^\eps_\pm(\cZZ),\ Y^1_\htau\geq0\}$. The treatment is similar for both of them because they are both independent centered Gaussian r.v.'s.

The displays~\eqref{eq:SDE_near_a_saddle_point_after_duhamel2} and~\eqref{eq:htau_formula} imply that
\begin{align*}
    Y^2_\htau=e^{-\mu \htau}L+\eps N^2_\htau=L\eps^{(\alpha-\theta)\rho }\left|\scx +\eps^{1-\alpha}U^1_\htau\right|^\rho +\eps N^2_\htau,\quad\Pp^{\eps^\alpha \scx} \text{-a.s.}
\end{align*}
Using this and the definition of $E^\eps_\pm$ in \eqref{eq:def_A^eps_pm}, we have that
\begin{align}\label{eq:final_steo_asymp_indep}
\begin{split}
    &\Pp^{\eps^\alpha \scx }\left\{Y^2_\htau \in E^\eps_\pm(\NN),\ Y^1_\htau\geq0\right\}
    \\
    &=\Pp^{\eps^\alpha \scx }\left\{R^{-\rho }\eps^{\theta \rho} Y^2_\htau+\eps \NN\in \left[\eps^\beta a\mp\tfrac{\eps^{1+\eta}}{2},\, \eps^\beta b\pm\tfrac{\eps^{1+\eta}}{2}\right],\ Y^1_\htau\geq0\right\}
    \\
    &= \Pp^{\eps^\alpha \scx }\left\{L R^{-\rho }\eps^{\alpha\rho} \left|\scx +\eps^{1-\alpha}U^1_\htau\right|^\rho +R^{-\rho }\eps^{1+\theta \rho }N^2_\htau+\eps \NN\in\left[\eps^\beta a\mp\tfrac{\eps^{1+\eta}}{2},\, \eps^\beta b\pm\tfrac{\eps^{1+\eta}}{2}\right],\ Y^1_\htau\geq0\right\}.
\end{split}
\end{align}
We can apply Lemma~\ref{lem:N_tau_tail} to get that
\begin{align*}
    \sup_{\scx \in K_\vk(\e)}\Pp^{\eps^\alpha \scx }\left\{\left|R^{-\rho}\eps^{\theta \rho}N^2_\htau\right|>\tfrac{1}{2}\eps^\eta\right\}=\superpoly,
\end{align*}
for all $\eta>0$ satisfying $\eta<\theta\rho$.
This along with~\eqref{eq:final_steo_asymp_indep} yields that
\begin{align*}
    &\Pp^{\eps^\alpha \scx }\left\{Y^2_\htau \in E^\eps_\pm(\NN),\ Y^1_\htau\geq0\right\}\\
    &= \Pp^{\eps^\alpha \scx }\left\{LR^{-\rho }\eps^{\alpha\rho} \left|\scx +\eps^{1-\alpha} U^1_\htau\right|^\rho +\eps \NN\in\left[\eps^\beta a\mp\eps^{1+\eta},\, \eps^\beta b\pm\eps^{1+\eta}\right],\ Y^1_\htau\geq0\right\}+\superpoly.
\end{align*}
Lastly, due to~\eqref{eq:SDE_near_a_saddle_point_after_duhamel1},  
\[\left\{Y^1_\htau\geq 0\right\}\stackrel{\Pp^{\eps^\alpha x}}{=}\left\{\scx +\eps^{1-\alpha}U^1_\htau\geq 0\right\}.\] 
Using these estimates and Lemma~\ref{lem:lift_constraint}, we complete the proof of Lemma~\ref{lem:asymp_decouple}. \epf

\subsection{Typical exit locations}

\begin{proposition}\label{prop:typical}
Let $\alpha \in (0,1]$, $\rho>0$, $\alpha'=(\alpha\rho)\wedge 1$, $c = R^{-\rho}L$.
Let $\tau$ be defined by \eqref{eq:def_tau_exit_loc}.
Let $\frU$ and ~$\NN$ be centered independent Gaussian r.v.'s with variance $\cc_1$ and $\cc_2$ given in~\eqref{eq:def_cc_1} and~\eqref{eq:def_cc_2}, respectively. For $a,b,\scx\in\R$, set
\begin{align*}
    P^{\alpha,\rho}_\eps(\scx ,[a,b])= \Pp^{\eps^\alpha \scx }\left\{Y_\tau \in \{R\}\times \eps^{\alpha'}[a,b]\right\}.
\end{align*}
For every $\vk,\vk'>0$, the following hold for some $\delta>0$:
\begin{enumerate}
    \item \label{item:loc_lim_1} If $\rho<1$, then
    \begin{align*}
        \sup_{\substack{\scx \in K_\vk(\eps) \\ [a,b]\subset K_{\vk'}(\eps)}}\left|P^{\alpha,\rho}_\eps(\scx ,[a,b])-\Pp\left\{c\left|\scx +\eps^{1-\alpha}\frU\right|^\rho \in [a,b],\ \scx +\eps^{1-\alpha}\frU\geq 0\right\} \right|=\smallo{\eps^\delta}.
    \end{align*}
    
    \item \label{item:loc_lim_2} If $\rho= 1$, then
    \begin{align*}
        \sup_{\substack{\scx \in K_\vk(\eps) \\ [a,b]\subset K_{\vk'}(\eps)}}\Big|P^{\alpha,\rho}_\eps(\scx ,[a,b])-\Pp\left\{c\left|\scx +\eps^{1-\alpha}\frU\right|+\eps^{1-\alpha}\NN\in [a,b],\ \scx +\eps^{1-\alpha}\frU \geq 0\right\} \Big|=\smallo{\eps^\delta}.
    \end{align*}
    
    \item \label{item:loc_lim_3} If $\rho>1$ and $\alpha\rho\leq 1$, then
    \begin{align*}
        \sup_{\substack{\scx \in K_\vk(\eps) \\ [a,b]\subset K_{\vk'}(\eps)}}\left|P^{\alpha,\rho}_\eps(\scx ,[a,b])-\Pp\left\{c|\scx |^\rho+\eps^{1-\alpha\rho}\NN\in [a,b]\right\}\Pp\left\{\scx +\eps^{1-\alpha}\frU\geq 0\right\} \right|=\smallo{\eps^\delta}.
    \end{align*}
    
    \item \label{item:loc_lim_4} If $\rho>1$ and $\alpha\rho>1$, then
    \begin{align*}
        \sup_{\substack{\scx \in K_\vk(\eps) \\ [a,b]\subset K_{\vk'}(\eps)}}\left|P^{\alpha,\rho}_\eps(\scx ,[a,b])-\Pp\left\{\NN\in [a,b]\right\}\Pp\left\{\scx +\eps^{1-\alpha}\frU\geq 0\right\} \right|=\smallo{\eps^\delta}.
    \end{align*}

\end{enumerate}

\end{proposition}

\begin{remark}\label{rem:loc_lim_3}
\rm
Sometimes, it is useful to replace $|\scx |^\rho$ in part~\eqref{item:loc_lim_3} by $|\scx +\eps^{1-\alpha}\frU|^\rho$. 
We claim that
\begin{multline*}
    \Big|\Pp\big\{c|\scx |^\rho+\eps^{1-\alpha\rho}\NN\in [a,b],\ \scx +\eps^{1-\alpha}\frU\geq 0\big\} 
    \\
    - \Pp\left\{c|\scx +\eps^{1-\alpha}\frU|^\rho+\eps^{1-\alpha\rho}\NN\in [a,b],\ \scx +\eps^{1-\alpha}\frU\geq 0\right\}\Big| =\smallo{\eps^{\delta'}}
\end{multline*}
for some $\delta'>0$, uniformly in $\scx \in K_\vk(\eps)$ and $[a,b]\subset K_{\vk'}(\eps)$. 
To see this, we first  restrict $\frU$ to $[-l_\e,l_\e]$, introducing a probability error of at most $o_e(1)$. Then,
 we rewrite thus modified probabilities above as Gaussian integrals, first integrating over $\NN$ and then over $\frU$.   For a given $\frU$, the Lebesgue measure of the symmetric difference of the domains of integration for $\NN$ is bounded by $l^{p'}_\eps \eps^{(1-\alpha)-(1-\alpha\rho)}=l^{p'}_\eps \eps^{\alpha(\rho-1)}=o(\eps^{\delta'})$ for some $p',\delta'>0$, uniformly in $\scx \in K_\vk(\eps)$ and $[a,b]\subset K_{\vk'}(\eps)$. The domains of integration for $\frU$ are always the same. Hence, the above display holds, implying the following version of the estimate in \eqref{item:loc_lim_3}:
\begin{multline*}
    \sup_{\substack{\scx \in K_\vk(\eps) \\ [a,b]\subset K_{\vk'}(\eps)}}\Big|P^{\alpha,\rho}_\eps(\scx ,[a,b])
    \\
    -\Pp\left\{c|\scx +\eps^{1-\alpha}\frU|^\rho+\eps^{1-\alpha\rho}\NN\in [a,b],\ \scx +\eps^{1-\alpha}\frU\geq 0\right\} \Big|=\smallo{\eps^\delta}.
    \end{multline*}
\end{remark}

\begin{proof}[Proof of Proposition~\ref{prop:typical}]
In this proof, all statements are understood to hold uniformly in $x\in K_\vk(\eps)$ and $[a,b]\subset K_{\vk'}(\eps)$. We also shorten ``w.h.p.\ under $\Pp^{\eps^\alpha x}$ uniformly in $x\in K_\vk(\eps)$'' into ``w.h.p.'' 
For brevity, we often write $P= P^{\alpha,\rho}_\eps(\scx ,[a,b])$. 
For $a,b\in\R$, we 
define \begin{align*}
    B_{a,b} = c^{-\frac{1}{\rho}}\left[(a\vee0)^\frac{1}{\rho},\, (b\vee0)^\frac{1}{\rho}\right].
\end{align*}
Note that $x \in B_{a,b}$ is equivalent to $x\geq 0$ and $cx^\rho \in [a,b]$.
For $a,b,h \in \R$, we introduce an $h$-perturbation of $B_{a,b}$ by:
\begin{align*}
    A^{\pm h}_{a,b} = c^{-\frac{1}{\rho}}\left[(a\vee0)^\frac{1}{\rho}\mp h,\, (b\vee0)^\frac{1}{\rho}\pm h\right],
\end{align*}
 Recall $\htau$ given in~\eqref{eq:def_htau} is controlled by the parameter $\theta \in(0,\alpha)$. Later, we will choose $\theta$ to be sufficiently small.

Part~\eqref{item:loc_lim_1}.
Note that in this case, we automatically have $\alpha'=\alpha\rho<1$. Lemma~\ref{lem:hp-events-in-rectangle} implies that 
\begin{align}
    &\Pp^{\eps^\alpha x}\left\{Y_\tau \in\{R\}\times \eps^{\alpha\rho} [a,b]\right\} \notag
    \\
    &= \Pp^{\eps^\alpha x}\left\{c\left|\scx +\eps^{1-\alpha}U^1_\tau\right|^\rho +\eps^{1-\alpha\rho} N^2_\tau \in [a,b],\ \scx +\eps^{1-\alpha} U^1_\tau\geq 0\right\}+o_e(1).\label{eq:exit_loc_equiv}
\end{align}
Using Lemma~\ref{lem:N_tau_tail} (with $\theta=0$ therein), we can choose $\delta'>0$ as small as needed so that $|N^2_\tau|<\eps^{-\delta'}$ w.h.p. Hence, due to~\eqref{eq:exit_loc_equiv},
\begin{align*}
    P\eqpm \Pp^{\eps^\alpha \scx }\left\{\scx +\eps^{1-\alpha}U^1_\tau \in B_{a\mp \eps^\upsilon, b\pm\eps^\upsilon}\right\}+o_e(1),
\end{align*}
where $\upsilon = 1- \alpha\rho-\delta'$.
Here, let us use \eqref{eq:tau=zeta} to redefine $\tau$ thus introducing a probability error of order $o_e(1)$. This allows us to apply Lemma~\ref{lem:gaussian_approx} to $\tau$.
Using the above display and Lemma~\ref{lem:gaussian_approx} with $R,0,0$ substituted for $r,\theta,\xi$ therein, we have that, for all $\eta \in(0,1)$,
\begin{align*}
    P\eqpm \Pp\left\{\scx +\eps^{1-\alpha}\frU \in A^{\pm\eps^\eta}_{a\mp\eps^\upsilon, b\pm\eps^\upsilon}\right\}+\smallo{\eps^\delta}.
\end{align*}
Note that $\Leb(A^{\pm\eps^\eta}_{a\mp\eps^\upsilon, b\pm\eps^\upsilon}\triangle B_{a,b})\leq C\eps^{\frac{\upsilon}{\rho}\wedge\eta}$. Due to $\rho <1$, by choosing $\delta'$ sufficiently small and $\eta$ sufficiently close to $1$, we can ensure $\frac{\upsilon}{\rho}\wedge\eta>1-\alpha$. Using Lemma~\ref{lem:gaussian_sym_diff}, we obtain
\begin{align*}
    \Pp\left\{\scx +\eps^{1-\alpha}\frU \in A^{\pm\eps^\eta}_{a\mp\eps^\upsilon, b\pm\eps^\upsilon}\right\} = \Pp\left\{\scx +\eps^{1-\alpha}\frU \in B_{a,b}\right\}+o(\eps^{\delta''})
\end{align*}
for some $\delta''>0$ completing the proof.

Part~\eqref{item:loc_lim_2}. Applying Lemma~\ref{lem:asymp_decouple} with $\beta=\alpha\rho$, we obtain
\begin{align*}
    P\eqpm\Pp^{\eps^\alpha \scx }\left\{\scx +\eps^{1-\alpha}U^1_\htau\in B_{a-\eps^{1-\alpha }(\NN\pm\eps^{\eta}),\ b-\eps^{1-\alpha }(\NN\mp\eps^{\eta})} \right\}+\smallo{\eps^\delta}
\end{align*}
for some $\eta,\delta>0$.
Then, applying Lemma~\ref{lem:gaussian_approx} with $1, 0$ substituted for $r, \xi$ therein, we obtain
\begin{align*}
    P\eqpm \Pp^{\eps^\alpha \scx }\left\{\scx +\eps^{1-\alpha}\frU\in A^{\pm}_\e\right\}+\smallo{\eps^{\delta'}},
\end{align*}
where
\[
A^\pm_\e= A^{\pm \eps^\upsilon}_{a-\eps^{1-\alpha }(\NN\pm\eps^{\eta}),\ b-\eps^{1-\alpha }(\NN\mp\eps^{\eta})} 
\]
for $\upsilon\in(0,1-\theta)$ to be chosen and some $\delta'>0$. Due to $\rho=1$, we have $\Leb(A^\pm_\eps\triangle B_{a-\eps^{1-\alpha }\NN, b-\eps^{1-\alpha}\NN})\leq C\eps^{\upsilon \wedge (1-\alpha+\eta )}$.
Choosing $\theta$ close to zero, we can ensure that $\upsilon$ is close to $1$ to ensure that the exponent satisfies $\upsilon \wedge (1-\alpha+\eta )>1-\alpha$.
Using Lemma~\ref{lem:gaussian_sym_diff} to estimate the difference between two Gaussian integrals, we obtain the
the desired result.

Part~\eqref{item:loc_lim_3}. Note that $\alpha<1$ is necessary for this case. Using Lemma~\ref{lem:asymp_decouple} with $\beta=\alpha\rho$, we get that
\begin{align*}
    P\eqpm\Pp^{\eps^\alpha \scx }\left\{\scx +\eps^{1-\alpha}U^1_\htau\in B_{a-\eps^{1-\alpha\rho }(\NN\pm\eps^{\eta}),\ b-\eps^{1-\alpha\rho }(\NN\mp\eps^{\eta})} \right\}+\smallo{\eps^\delta}
\end{align*}
for some $\eta,\delta>0$.
Due to Lemma~\ref{lem:estimates-for-terms}~\eqref{item:tail-of_U_1}, we have $|U^1_\htau|<\eps^{-\delta'}$ w.h.p.\ for $\delta'>0$ as small as needed. Using this and the independence of $\NN$, we have
\begin{align}
\label{eq:asymp_of_P}
    P\eqpm
    \Pp\left\{c|\scx|^\rho+\eps^{1-\alpha\rho}\NN\in [a_{\pm,\eps}, b_{\pm,\eps}]\right\}
    \Pp^{\eps^\alpha \scx }\left\{\scx + \eps^{1-\alpha}U^1_\htau\geq 0\right\}+\smallo{\eps^\delta},
\end{align}
where
\begin{align*}
a_{\pm,\eps}&= a\mp\eps^{1-\alpha\rho+\eta}-c((\scx\pm \eps^{1-\alpha-\delta'})\vee0)^\rho+c|\scx|^\rho,\\
b_{\pm,\eps}&= b\pm\eps^{1-\alpha\rho+\eta}-c((\scx\mp \eps^{1-\alpha-\delta'})\vee0)^\rho+c|\scx|^\rho.
\end{align*}
If $\scx\geq -\eps^{1-\alpha-\delta'}$, then, due to $\scx \in K_\vk(\eps)$ and $\rho> 1$, the Lebesgue measure of $[a,b]\triangle[a_{\pm,\eps},b_{\pm,\eps}]$ 
is bounded by $C\eps^{(1-\alpha\rho+\eta)\wedge(1-\alpha-\delta'')}$, where $\delta''>\delta'$ still can be made as small as needed. Due to $\rho>1$, the exponent is strictly larger than $1-\alpha\rho$ for sufficiently small $\delta''$. 
Applying Lemma~\ref{lem:gaussian_sym_diff},
we obtain that the first factor on the right of~\eqref{eq:asymp_of_P} is
\begin{align*}
    \left(\Pp\{c|\scx|^\rho +\eps^{1-\alpha\rho}\NN\in[a,b]\} \pm  o(\eps^{\delta'''})\right)\ONE_{[-\eps^{1-\alpha-\delta'},\infty)}(x) + \mathcal{O}(1)\ONE_{(-\infty,-\eps^{1-\alpha-\delta'})}(x)
\end{align*}
for some $\delta'''>0$. 

For the second factor on the right of \eqref{eq:asymp_of_P}, choosing $\vk''$ sufficiently large and using Lemma~\ref{lem:estimates-for-terms}~\eqref{item:tail-of_U_1}, we have
\begin{align*}
    \Pp^{\eps^\alpha \scx }\left\{\scx +\eps^{1-\alpha}U^1_\htau \geq0\right\} & = \Pp^{\eps^\alpha \scx }\left\{\scx +\eps^{1-\alpha}U^1_\htau \in\left[0,l^{\vk''}_\eps\right]\right\} + \Pp^{\eps^\alpha \scx }\left\{\eps^{1-\alpha}U^1_\htau >l^{\vk''}_\eps-\scx \right\}\\
    &= \Pp^{\eps^\alpha \scx }\left\{\scx +\eps^{1-\alpha}U^1_\htau \in\left[0,l^{\vk''}_\eps\right]\right\}+o_e(1).
\end{align*}
Invoking Lemma~\ref{lem:gaussian_approx} 
with $1,0$ substituted for $r,\xi$ therein, we get that, for arbitrary $\upsilon\in(0,1-\theta)$ to be chosen and some $\bar\delta,\hat\delta>0$,
\begin{align*}
    \Pp^{\eps^\alpha \scx }\left\{\scx +\eps^{1-\alpha}U^1_\htau \geq0\right\}= \Pp\left\{\scx +\eps^{1-\alpha}\frU\in [0\mp\eps^\upsilon,l^{\vk''}_\eps\pm \eps^\upsilon\right\}+\smallo{\eps^{\bar\delta}}\\
    = \Pp\left\{\scx +\eps^{1-\alpha}\frU \geq 0\right\}+\smallo{\eps^{\hat\delta}},
\end{align*}
where the last equality follows, once we  choose $\upsilon$ close enough to $1-\theta$ to ensure $\upsilon >1-\alpha$, from Lemma~\ref{lem:gaussian_sym_diff}, and the Gaussian tail of $\frU$. We also have
\begin{align*}
    \Pp\left\{\scx +\eps^{1-\alpha}\frU \geq 0\right\}\ONE_{(-\infty,-\eps^{1-\alpha-\delta'})}(\scx) = o_e(1),
\end{align*}
taking into account the Gaussian tail of $\frU$.
Combining the results on both factors in \eqref{eq:asymp_of_P} completes the proof of part~\eqref{item:loc_lim_3}.

Part~\eqref{item:loc_lim_4}. Applying Lemma~\ref{lem:asymp_decouple} with $\beta=1$, we obtain that, for some $\delta>0$,
\begin{align*}
    P&\eqpm 
    \Pp^{\eps^\alpha \scx }\left\{c\eps^{\alpha\rho-1}(\scx +\eps^{1-\alpha}U^1_\htau)^\rho + \NN \in [a\mp\eps^\eta,b\pm\eps^\eta],\ \scx +\eps^{1-\alpha}U^1_\htau\geq 0\right\} + \smallo{\eps^\delta}.
\end{align*}
Due to Lemma~\ref{lem:estimates-for-terms}~\eqref{item:tail-of_U_1}, we have $|\scx +\eps^{1-\alpha}U^1_\htau|\leq l^{\vk''}_\eps$ w.h.p.\ for some sufficiently large $\vk''>0$. This along with $\alpha\rho>1$ and the independence of $\NN$ implies that, for some $\eta'>0$,
\begin{align*}
    P\eqpm \Pp\left\{\NN\in \left[a\mp\eps^{\eta'}, b\pm\eps^{\eta'}\right]\right\}\Pp^{\eps^\alpha \scx }\left\{\scx +\eps^{1-\alpha}U^1_\htau \geq0\right\}+\smallo{\eps^\delta}.
\end{align*}
Due to Lemma~\ref{lem:gaussian_sym_diff}, the first factor on the right differs from $\Pp\{\NN\in [a, b]\}$ by an error term $o(\eps^{\delta'})$ for some $\delta'$. The second one can be shown, with an argument similar to the proof of part~\eqref{item:loc_lim_3}, to be $\Pp\{\scx +\eps^{1-\alpha}\frU\geq 0\}$ up to an $o(\eps^{\delta''})$ error for some $\delta''>0$. Combining these estimates, we obtain the desired result.
\end{proof}

\section{Density estimates}\label{section:density_est}

In this section, we prove Lemma~\ref{lem:density_est_R^2}, which has been used in Sections~\ref{sec:Gaussian-approx} and~\ref{sec:LLT}. We first introduce the setting for this lemma. 

Consider the process $Y_t$ in $\R^2$ given in \eqref{eq:def_SDE_near_a_saddle_point}. Recall the associated processes $U_t$ and $N_t$ defined in \eqref{eq:def_V_M_U_N}. For $y\in \R^2$ and $t\geq 0$, we define $\R^2$-valued Gaussian vectors $Z$ and $\bZ$ by
\begin{align}\label{eq:def_Z_R^2} \begin{split}
    Z^1_t & = \int_0^te^{-\lam  s}F^1_l\left(0,e^{-\mu s}y^2\right)dW^l_s,
    \\
    Z^2_t & = e^{-\mu t}\int_0^te^{\mu s}F^2_l\left(e^{\lam s}y^1, e^{-\mu s}y^2\right)dW^l_s,
    \\
    \bZ^1_t & = \int_0^te^{-\lam  s}F^1_l\left(0,0\right)dW^l_s,
    \\
    \bZ^2_t & = e^{-\mu t}\int_0^te^{\mu s}F^2_l\left(e^{\lam s}y^1,0\right)dW^l_s,
\end{split}    
\end{align}
where we suppressed the dependence on $y$ in the notation. 

Recall that, for $y\in\R^2$, the probability measure under which $Y_0 = y$ a.s.\ is denoted by $\Pp^y$. For a random vector $\mathcal X$, we denote its probability density function (with respect to the Lebesgue measure) under $\Pp^y$ by $\varphi^y_{\mathcal X}$. Since $\bZ^1_t$ is independent of $y$, we write its density simply as $\varphi_{\bZ^1_t}$.
\begin{lemma}\label{lem:density_est_R^2} 
There is $\bar\theta >0$ such that for each $\upsilon\in(0,1)$, there are constants $C, c,\delta>0$ such that, for $\eps$ sufficiently small and all $y\in \R^2$,
    \begin{enumerate}
        \item \label{item:den_est_1_R^2} $\left|\varphi^y_{U^{1}_{T(\eps)}}(s)-\varphi^y_{Z^{1}_{T(\eps)}}(s)\right|\leq C\eps^\delta\left(1+\eps^{-\upsilon}|y^{1}|\right) e^{-c|s|^2}$ for all $s \in \R$;
        \item \label{item:den_est_2_R^2} $\left|\varphi^y_{U^{1}_{T(\eps)}}(s)-\varphi_{\bZ^{1}_{T(\eps)}}(s)\right|\leq C \left( |y^{2}|+\eps^\delta\left(1+\eps^{-\upsilon}|y^{1}|\right)\right)e^{-c|s|^2}$ for all $s \in \R$;
        \item $\left|\varphi^y_{(U^{1},N^{2})_{T(\eps)}}(\z)-\varphi^y_{Z_{T(\eps)}}(\z)\right|\leq C \eps^\delta\left(1+\eps^{-\upsilon}|y^{1}|\right) e^{-c|\z|^2}$ for all $\z\in \R^2$;
        \item \label{item:den_est_4_R^2} $\left|\varphi^y_{(U^{1},N^{2})_{T(\eps)}}(\z)-\varphi^y_{\bZ_{T(\eps)}}(\z)\right|\leq C\left( |y^{2}|+\eps^\delta\left(1+\eps^{-\upsilon}|y^{1}|\right)\right)e^{-c|\z|^2}$ for all $\z\in\R^2$
    \end{enumerate}
    hold for all deterministic functions $T(\cdot)$ satisfying $1\leq T(\eps)\leq \bar \theta l_\eps $.
\end{lemma}

This lemma is a special case of a more general result, Lemma~\ref{lem:density_est}, in higher dimensions. Our goal is to prove Lemma~\ref{lem:density_est}. We start by describing the general setting. We will deduce Lemma~\ref{lem:density_est_R^2} from Lemma~\ref{lem:density_est} in the next subsection.

\subsection{General setting and main result}
Let $\nu, d$ be positive integers satisfying $\nu < d$, and let $\lam\in \R^d$ satisfy
\begin{align}\label{eq:lam_i}
    \lam^1>\lam^2>\dots >\lam^\nu>0>\lam^{\nu+1}>\dots>\lam^d,
\end{align}
so the origin is a saddle point of the vector field $x\mapsto (\lambda^i x^i)_{1\leq i\leq d}$. The coordinates $1,\ldots,\nu$  correspond to the unstable directions near the origin, and the remaining coordinates $\nu+1,\ldots,d$ correspond to the stable directions.

We consider the following SDE
\begin{equation}
 \label{eq:Y_SDE_before_Duhamel}
dY^i_t = \lam^i Y^i_t dt + \eps F^i_j(Y_t) dW^j_t + \eps^2G^i(Y_t)dt,\quad i =1,2,\dots, d,
\end{equation}
assuming that
\begin{gather}\label{eq:modified_F}
\begin{split}
    &c_0:=\min_{|u|=1,\ u\in\R^d}|u^\intercal F(x)|^2>0, \text{ for all }x\in \R^d; 
    \\
    & F, G \text{ and their derivatives up to the third order are bounded}.
    \end{split}
\end{gather}

We consider the initial conditions, for $y\in \R^d$, 
\begin{align}\label{eq:initial_Y_0}
Y_0= y.
\end{align}
By Duhamel's principle, we can solve \eqref{eq:Y_SDE_before_Duhamel} with~\eqref{eq:initial_Y_0} by 
\begin{equation}
\label{eq:Y_after_Duhamel}
Y^j_t = 
e^{\lj t}(y^j+\eps U^j_t ) =  e^{\lj t}y^j +\eps N^j_t,
\end{equation}
where
\begin{equation}
\label{eq:U}
U^j_t=M^j_t + \eps V^j_t, \quad N^j_t=e^{\lj t}U^j_t,
\end{equation}
and
\begin{align}
\label{eq:M}
M^j_t &= \int_0^t e^{-\lj s}F^j_l(Y_s)dW^l_s,\\
\label{eq:V}
V^j_t &= \int^t_0e^{-\lj s}G^j(Y_s)ds.
\end{align}
We emphasize that $U_t$, $N_t$, $M_t$, and $V_t$ depend on $y$ and $\eps$.

For $x \in \R^d, t\in\R$, we denote
\begin{align*}\begin{split}
        x^{\leq \nu} &= (x^1,\ x^2,\dots x^\nu)\in\R^\nu,\\
        x^{>\nu} &=   (x^{\nu+1},\ x^{\nu+2},\dots x^d)\in\R^{d-\nu},\\
        e^{\lam t}x &= (e^{\lj t}x^j)_{j=1}^d.
\end{split}
\end{align*}

Define 
\begin{align}\label{eq:def_Z}
\begin{split}
    Z^i_t&=
    \begin{cases}
    \int_0^te^{-\li s}F^i_l\left(0^{\leq \nu},(e^{\lam s}y)^{>\nu}\right)dW^l_s,\  &\for i\leq \nu,\\
    e^{\li t}\int_0^te^{-\li s}F^i_l\left(e^{\lam s}y\right)dW^l_s,\  &\for i>\nu,
    \end{cases}\\
    \bZ^i_t&=\begin{cases}
    \int_0^te^{-\li s}F^i_l(0)dW^l_s,\  &\for i\leq \nu,\\
    e^{\li t}\int_0^te^{-\li s}F^i_l((e^{\lam s}y)^{\leq\nu},0^{>\nu})dW^l_s,\  &\for i>\nu.
    \end{cases}
\end{split}    
\end{align}

For a r.v.\ $\mathcal{X}$ with values in a Euclidean space, its Lebesgue density, if exists, is denoted by $\varphi_\mathcal{X}$. Since $U_t$, $N_t$, $Z_t$ and $\bZ_t$ depend on $y$, we add a superscript $y$ to the density notation to emphasize this dependence. For example, we write the density of $U_t$ as $\varphi^y_{U_t}$. Since $\bZ^{\leq\nu}_t$ is independent of $y$, we denote the density of $\bZ^{\leq\nu}_t$ by $\varphi_{\bZ^{\leq \nu}_t}$.

\begin{lemma}\label{lem:density_est} 
Consider~\eqref{eq:Y_after_Duhamel} with initial condition~\eqref{eq:initial_Y_0}. Let  
     \begin{align}\label{eq:pp}
         \pp(x)=\sum_{j,k=1}^\nu x^\frac{\lj}{\lk},\quad \text{for }x\geq 0.
     \end{align} 
    Then there is $\bar{\theta} >0$ 
     such that for each $\upsilon\in(0,1)$, there are constants $C, c,\delta>0$ such that, for $\eps$ sufficiently small and for all deterministic functions $T(\cdot)$ satisfying
    \begin{align}\label{eq:T(eps)_range}
        1\leq T(\eps)\leq \bar\theta l_\eps,
    \end{align}
    the following hold:
    \begin{enumerate}
        \item \label{item:den_est_1} $\left|\varphi^y_{U^{\leq \nu}_{T(\eps)}}(x)-\varphi^y_{Z^{\leq \nu}_{T(\eps)}}(x)\right|\leq C\eps^\delta\left(1+\pp(\eps^{-\upsilon}|y^{\leq \nu}|)\right) e^{-c|x|^2}$, for all $x \in \R^\nu$ and $y\in \R^d$;
        \item \label{item:den_est_2} $\left|\varphi^y_{U^{\leq \nu}_{T(\eps)}}(x)-\varphi_{\bZ^{\leq \nu}_{T(\eps)}}(x)\right|\leq C \left( |y^{>\nu}|+\eps^\delta\left(1+\pp(\eps^{-\upsilon}|y^{\leq \nu}|)\right)\right)e^{-c|x|^2}$, for all $x \in \R^\nu$ and $y\in \R^d$;
        \item \label{item:den_est_3} $\left|\varphi^y_{(U^{\leq \nu},N^{>\nu})_{T(\eps)}}(x)-\varphi^y_{Z_{T(\eps)}}(x)\right|\leq C \eps^\delta\left(1+\pp(\eps^{-\upsilon}|y^{\leq \nu}|)\right) e^{-c|x|^2}$, for all $x , y\in \R^d$;
        \item \label{item:den_est_4} $\left|\varphi^y_{(U^{\leq \nu},N^{>\nu})_{T(\eps)}}(x)-\varphi^y_{\bZ_{T(\eps)}}(x)\right|\leq C\left( |y^{>\nu}|+\eps^\delta\left(1+\pp(\eps^{-\upsilon}|y^{\leq \nu}|)\right)\right)e^{-c|x|^2}$, for all $x , y\in \R^d$.
    \end{enumerate}
\end{lemma}

\subsection{Preliminaries}

Let us introduce the necessary notation from the Malliavin calculus.

For any $\mathcal{T}\in(0,\infty)$, we let $\Omega_\mathcal{T}$ be the standard Wiener space for $\R^d$-valued Wiener processes on $[0,\mathcal{T}]$. We also set
\begin{align}\label{eq:H_T}
    \mathscr{H}_\mathcal{T} = L^2\left([0,\mathcal{T}];\R^d\right)
\end{align}
with the inner product denoted by $\langle\cdot,\cdot\rangle_{\mathcal{H}_\mathcal{T}}$.
Note that $\{W(h)\}_{h\in\mathscr{H}}$ given by
\begin{align*}
    W(h) = \int_0^\mathcal{T}\sum_{i=1}^dh^i(s)dW^i_s,\quad h \in \mathscr{H}_\mathcal{T},
\end{align*}
is an isonormal Gaussian process (real-valued) indexed by $\mathscr{H}_\mathcal{T}$ (meaning that $W$ is a centered Gaussian process satisfying $\E W(h)W(h') =\langle h,h'\rangle_{\mathscr{H}_\mathcal{T}}$ for all $h,h'\in \mathscr{H}_\mathcal{T}$). For $p\in[1,\infty)$, let $L^p(\Omega_\mathcal{T};\mathscr{H}_\mathcal{T})$ be the set of $\mathscr{H}_\mathcal{T}$-valued random variables with finite norm $ (\E\|\cdot\|_{\mathscr{H}_\mathcal{T}}^p)^\frac{1}{p}$. Then, the Malliavin derivative operator is an unbounded operator $\D:L^p(\Omega;\R)\to L^p([0,\mathcal{T}];\mathscr{H}_\mathcal{T})$ defined initially for ``smooth'' random variables of the form
\begin{align*}
    \mathcal{X}=f(W(h_1),\dots,W(h_m))
\end{align*}
by
\begin{align*}
    \D \mathcal{X} = \sum_{i=1}^m \partial_i f(W(h_1),\dots,W(h_m))h_i,
\end{align*}
where $f:\R^m\to\R$ is smooth and compactly supported for some $m\in\N$. It is extended to a closed operator under the graph norm
\begin{align*}
    \|\mathcal{X}\|_{1,p,\mathcal{T}} = \left(\E|\mathcal{X}|^p+ \E\|\D\mathcal{X}\|_{\mathscr{H}_\mathcal{T}}^p\right)^\frac{1}{p}.
\end{align*}
We denote the domain of $\D$ by $\mathbb{D}^{1,p}_\mathcal{T}$. For each $n\in\N$, this construction can be extended to $\D:L^p(\Omega;\mathscr{H}_\mathcal{T}^{\otimes n})\to L^p(\Omega;\mathscr{H}_\mathcal{T}^{\otimes n+1})$ with norm
\begin{align*}
    \|\mathcal{X}\|_{1,p,\mathcal{T}} = \left(\E\|\mathcal{X}\|^p_{\mathscr{H}_\mathcal{T}^{\otimes n}}+ \E\|\D\mathcal{X}\|_{\mathscr{H}_\mathcal{T}^{\otimes n+1}}^p\right)^\frac{1}{p}.
\end{align*}
Here, we omitted $n$, the degree of the tensor product, in the notation for simplicity. In the same fashion, we denote the associated domain still by $\mathbb{D}^{1,p}_\mathcal{T}$. The degree of the tensor product will be clear from the context. For $k\in \N$, the $k$-th order derivative operator $\D^{(k)}$ can be defined inductively. Its domain is denoted by $\mathbb{D}^{k,p}_\mathcal{T}$ and the associated graph norm by $\|\cdot\|_{k,p,\mathcal{T}}$. In particular, it sends an $\mathscr{H}_\mathcal{T}^{\otimes n}$-valued random variable $\mathcal{X}$ in $\mathbb{D}^{k,p}_\mathcal{T}$ to an $\mathscr{H}_\mathcal{T}^{\otimes n+k}$-valued random variable $\D^{(k)}\mathcal{X}$, for $n\in\N\cup\{0\}$ with the understanding that $\mathscr{H}_\mathcal{T}^{\otimes 0}= \R$. 
Moreover, we have
\begin{align*}
    \|\mathcal{X}\|_{k,p,\mathcal{T}}= \left(\E\|\mathcal{X}\|_{\mathscr{H}^{\otimes n}_\mathcal{T}}^p+\sum_{i=1}^k\|\D^i \mathcal{X}\|^p_{\mathscr{H}^{\otimes n+i}_\mathcal{T}}\right)^\frac{1}{p}.
\end{align*}
It is clear that $\mathbb{D}^{k,p}_\mathcal{T}\subset \mathbb{D}^{k',p'}_\mathcal{T}$ for $p'\geq p$ and $k'\geq k$. For $k\in\N$, we set $\mathbb{D}^{k,\infty}_\mathcal{T}=\cap_{p\in[0,\infty)}\mathbb{D}^{k,p}_\mathcal{T}$.

We refer to \cite[Chapter~1]{nualart} for more details on the basics of Malliavin calculus. Later, we will also need results from \cite[Chapter~2]{nualart} on the application of the Malliavin calculus to solutions of SDE.

For an $\R^m$-valued random vector $\mathcal{X}$ satisfying $\mathcal{X}^i \in \mathbb{D}^{1,1}_\mathcal{T}$ for all $i=1,2,\dots,m$, the associated \textit{Malliavin matrix} of $\mathcal{X}$ is an $m\times m$ random matrix given by
\begin{align}\label{eq:def_malliavin_matrix}
    \sigma_{\mathcal{X}}=\left(\left\langle \D \mathcal{X}^i,\,\D\mathcal{X}^j\right\rangle_{\mathscr{H}_\mathcal{T}} \right)_{1\leq i,j\leq m}.
\end{align}
If the components of $\mathcal{X}$ are in $\mathbb{D}^{k,p}_\mathcal{T}$, we write $\|\mathcal{X}\|_{k,p,\mathcal{T}}=\sum_{i=1}^m \|\mathcal{X}^i\|_{k,p,\mathcal{T}}$.

Let us recall \cite[Theorem~2.14.B]{bally2014} (see also \cite[Theorem~2.4.6]{bally}) which estimates the difference between derivatives of two densities in terms of Sobolev norms and the Malliavin matrix. For our purposes, in our statement of this result, Theorem~\ref{Thm:DensityDifference} below, we simplify the conditions of the original theorem by setting the localization random variable~$\mathbf{\Theta}$ to be $1$, the derivative order $q=0$ (i.e., we compare densities themselves, without derivatives) and using Meyer's inequality (c.f. \cite[Theorem~1.5.1]{nualart}) to bound the Ornstein--Uhlenbeck operator. We stress that, although the conditions 
of  Theorem~2.14.B as it is stated in~\cite{bally2014} do not formally allow for $q=0$, that theorem is still
valid for this value of~$q$. In fact,  in~\cite{bally2014}, Theorem~2.14 is derived from Theorem~2.1 via an approximation argument. In turn, part B of Theorem 2.1 is restated and proved in the form of 
Theorem 3.10, where~$q$ is allowed to be $0$.

\begin{theorem}[\cite{bally2014}] \label{Thm:DensityDifference}
For $i=1,2$, let $\mathcal{X}_i$ be an $\R^d$-valued random vector with components in $\mathbb{D}^{3,\infty}_\mathcal{T}$ satisfying
$\E{(\det\sigma_{\mathcal{X}_i})^{-p} }<\infty$ for every $p>1$. Then, there exist positive constants $C,a,b, \gamma$ only depending on $d$ such that for all $x\in\R^d$
\begin{align*}
    |\varphi_{\mathcal{X}_1}(x)-\varphi_{\mathcal{X}_2}(x)| \leq & C \|\mathcal{X}_1-\mathcal{X}_2 \|_{2,\gamma,\mathcal{T}}\left(\prod_{i=1,2}\left(1\vee \E{(\det\sigma_{\mathcal{X}_i})^{-\gamma} }\right)\left(1+\| \mathcal{X}_i\|_{3,\gamma,\mathcal{T}}\right) \right)^a\\
    & \cdot \left(\sum_{i=1,2}\Prob{|\mathcal{X}_i-x|<2}\right)^b.
\end{align*}
\end{theorem}

We will use this theorem to derive Lemma~\ref{lem:density_est}. Thus our goal is to estimate all the factors on the right-hand side of this bound for 
the choices of $\Xc_1$ and $\Xc_2$ relevant for Lemma~\ref{lem:density_est}.
In particular, we will need to estimate moments of Malliavin derivatives and negative moments the Malliavin covariance matrix.

The fact that $C,a,p$ do not depend on $\mathcal{T}$ is important because we will apply this estimate to times $\mathcal{T}$ given by a function of $\eps$ growing to $\infty$ as $\eps \to 0$.

Let us fix
\begin{align}\label{eq:bar_theta}
    \bar\theta = \frac{1}{8\max\{\lam^1, |\lam^d|,1\}}.
\end{align}

Hence, if $T=T(\eps)$ satisfies \eqref{eq:T(eps)_range} for all $\eps\in(0,1)$,
then
\begin{align}\label{eq:choice_of_theta}
    e^{|\lj|T}\leq \eps^{-\frac{1}{8}}\and T\leq \eps^{-\frac{1}{8}},\quad j=1,2,\dots,d, \ \ \eps\in(0,1).
\end{align}
Let us arbitrarily fix $T=T(\eps)$ satisfying~\eqref{eq:T(eps)_range} and thus~\eqref{eq:choice_of_theta}.

For a random variable $\xi$, we define 
\begin{align*}
    [\xi]_p= (\E{|\xi|^p})^\frac{2}{p}.
\end{align*}
Let us  derive a few basic inequalities.
\begin{lemma}\label{lem:pbracket_properties}
Let $p\geq 2$.
\begin{itemize}
    \item There is a positive constant $C$ depending only on $p, d$ such that, for any $t_2>t_1\geq 0$ and any adapted $\R^d$-valued process $(\mathcal{X}_{s})_{s\geq 0} = ((\mathcal{X}_{l,s})_{1\leq l\leq d})_{s\geq 0} $,
    \begin{align}\label{eq:pBracket_property_1}
        \pbracket{\int_{t_1}^{t_2}\mathcal{X}_{l,s}dW^l_s} \leq C   \int_{t_1}^{t_2}\pbracket{\mathcal{X}_{s}}ds.
    \end{align}
    \item For any $t_2>t_1\geq 0$, any measurable process $(\mathcal{X}_{s})_{s\geq 0}$,
    \begin{align}\label{eq:pBracket_property_2}
    \pbracket{\int_{t_1}^{t_2}\mathcal{X}_s ds }\leq |t_2-t_1|\int_{t_1}^{t_2}\pbracket{\mathcal{X}_s}ds.
    \end{align}
    \item For any $\mathcal{T}>0$, any $n\in \N$, and any measurable process $(\mathcal{X}_{s})_{s\in[0,\mathcal{T}]^n}$,
    \begin{align}\label{eq:pBracket_property_3}
    \pbracket{\|\mathcal{X}\|_{\mathscr{H}_\mathcal{T}^{\otimes n}}}\leq  \int_{[0,T]^n}\pbracket{\mathcal{X}_{s^1,s^2,\dots,s^n}}ds^1ds^2\cdots ds^n,
\end{align}
    where $\mathscr{H}_\mathcal{T}$ is given in \eqref{eq:H_T}.
\end{itemize}
\end{lemma}
\begin{proof}
For the reader's convenience we recall the Minkowski integral inequality:  for any $q\in[1,\infty)$, $n\geq 1$, and $[t_1,t_2]\subset \R$,
\[
    \left(\E\left|\int_{[t_1,t_2]^n}\mathcal{X}_{s^1,s^2,\dots, s^n}ds^1ds^2\dots ds^n\right|^q\right)^{\frac{1}{q}}\leq\int_{[t_1,t_2]^n}\left(\E|\mathcal{X}_{s^1,s^2,\dots,s^n}|^q\right)^{\frac{1}{q}}ds^1ds^2\dots ds^n.
\]
Using the BDG inequality, and the Minkowski integral inequality (with $q=p/2$) together with $p\geq 2$, we have
\begin{align*} 
    &\pbracket{\int_{t_1}^{t_2}\mathcal{X}_{l,s}dW^l_s} = \left(\E\left|\int_{t_1}^{t_2}\mathcal{X}_{l,s}dW^l_s\right|^p\right)^\frac{2}{p} \leq C   \left(\E\left|\int_{t_1}^{t_2}|\mathcal{X}_{s}|^2 ds\right|^\frac{p}{2}\right)^\frac{2}{p} \\
    &\leq C\int_{t_1}^{t_2}\left(\E{|\mathcal{X}_{s}|^p }\right)^\frac{2}{p} ds =  C\int_{t_1}^{t_2}\pbracket{\mathcal{X}_{s}}ds,
\end{align*}
where $C$ only depends on $p$ and $d$ due to the BDG inequality.
This is \eqref{eq:pBracket_property_1}. Using the Minkowski integral inequality and the Cauchy--Schwarz inequality, we have
\begin{align*}
    \pbracket{\int_{t_1}^{t_2}\mathcal{X}_s ds }=\left(\E\left|\int_{t_1}^{t_2}\mathcal{X}_{s}ds\right|^p\right)^\frac{2}{p} \leq \left(\int_{t_1}^{t_2}
    \left(\E{|\mathcal{X}_{s}|^p}\right)^\frac{1}{p}ds\right)^2\leq |t_2-t_1|\int_{t_1}^{t_2}\pbracket{\mathcal{X}_s}ds
\end{align*}
yielding \eqref{eq:pBracket_property_2}. Lastly, \eqref{eq:pBracket_property_3} follows from 
\begin{align*}
    \|\mathcal{X}\|^2_{\mathscr{H}_\mathcal{T}^{\otimes n}} = \int_{[0,\mathcal{T}]^n}\left|\mathcal{X}_{s^1,s^2,\dots,s^n}\right|^2 ds^1ds^2\cdots ds^n
\end{align*} 
and the Minkowski integral inequality.
\end{proof}

We set, for small $\eps$ and $T$ given in \eqref{eq:T(eps)_range},
\begin{align}\label{eq:def_Hil}
    \Hil=\Hil(\eps) = \mathscr{H}_{T(\eps)},
\end{align}
 and will apply \eqref{eq:pBracket_property_3} to processes indexed by $[0,T(\eps)]^n$.
We emphasize that $\Hil$ depends on $\eps$.

Henceforth, we fix an arbitrary $\upsilon\in(0,1)$ (as in the statement of Lemma~\ref{lem:density_est})
Most of the estimates below are obtained for all $p\geq 2$. We need this restriction in order to apply Lemma~\ref{lem:pbracket_properties} in intermediate steps. However,  it is easy to extend our results to $p\in[1,2)$ using Jensen's inequality.

 For $A,B\in\R$, we write $A\lesssim B$ if and only if there is a constant
$C>0$, only depending on $p, \lambda, F, G, \bar\theta,\nu$, such that the inequality $A\le CB$ holds. Here, $d$ is the dimension of the Euclidean space fixed at the beginning of this subsection; $\lambda,F,G$ determined the SDE given in \eqref{eq:Y_SDE_before_Duhamel}; and $\bar\theta$ has been fixed in \eqref{eq:bar_theta}. Note that, in particular, the constant $C$ is independent of~$\eps, y$. Occasionally, we stress the dependence of the constant on $p$ by writing $\lesssim_p$. 

\subsection{Derivative estimates}

In this subsection, we assume $p\geq 2$ if not otherwise specified.

We need  bounds on all the factors on the right-hand side of the estimate provided by Theorem~\ref{Thm:DensityDifference}.   The Malliavin matrix will be estimated in the next subsection.  The main goal of this subsection is to
estimate $\|\mathcal{X}_1-\mathcal{X}_2\|_{2,\gamma,\mathcal{T}}$ and $\|\mathcal{X}_i\|_{3,\gamma,\mathcal{T}}$. Thus we need to estimate
Malliavin derivatives of $ \mathcal{X}_i$ and $\mathcal{X}_1-\mathcal{X}_2$, where $\mathcal{X}_1$ is 
$(U^{\leq \nu}_{T(\eps)},N^{>\nu}_{T(\eps)})$ and $\mathcal{X}_2$ is one of the Gaussian approximations defined via~\eqref{eq:def_Z}.

To compute the Malliavin derivatives of those processes viewed as solutions of SDEs, we will use \cite[Theorems~2.2.1 and~2.2.2]{nualart}, so let us recall the notation from \cite[Section~2.2]{nualart}.
For a real-valued measurable process $(\mathcal{X}_t)_{t\in[0,T]}$, its $\Hil$-valued derivative $\mathcal{D}\mathcal{X}_t$ at any fixed $t\in[0,T]$, if it exists, can be written in (real-valued) coordinates as
\begin{align*}
     \mathcal{D}\mathcal{X}_t = \left(\mathcal{D}^{j}_{r}\mathcal{X}_t\right)_{ j\in\{1,2,\dots, d\},\, r \in [0,T]}.
\end{align*}
Similar notation applies to higher order Malliavin derivatives. For $k\in\N$, the $\Hil^{\otimes k}$-valued derivative  $\mathcal{D}^{(k)}\mathcal{X}_t$, if it exists, can be written in coordinates as
 (see, e.g., the proof of \cite[Theorem~2.2.2]{nualart})
\begin{align*}
    \mathcal{D}^{(k)}\mathcal{X}_t = \left(\mathcal{D}^{j_1,j_2,\dots,j_k}_{r_1,r_2,\dots,r_k}\mathcal{X}_t\right)_{\substack{ j_1,j_2,\dots,j_k\in\{1,2,\dots, d\} \\ r_1,r_2,\dots,r_k \in [0,T]}}.
\end{align*}
We need estimates of all these components of the Malliavin derivatives up to order~3.

\medskip

We will need to make approximations to $F(Y_s)$, and it is convenient to introduce notation for the resulting errors. For $x \in \R^d$ and $t\in \R$, we set
\begin{align*}
\begin{split}
    H^i(t,x)&=
    \begin{cases}
     F^i(x)-F^i(0^{\leq \nu}, (e^{\lam t}y)^{>\nu}), \quad & i \leq \nu,\\
     F^i(x)-F^i(e^{\lam t}y), \quad & i > \nu,\\
    \end{cases}\\
    \bH^i(t,x)&=
    \begin{cases}
     F^i(x)-F^i(0), \quad &i \leq \nu,\\
     F^i(x)-F^i((e^{\lam t}y)^{\leq\nu}, 0^{>\nu}), \quad &i > \nu.\\
    \end{cases}
\end{split}
\end{align*}
Note that we use different deterministic approximations for the unstable and stable components, which will allow for more compact formulas later on.
Using~\eqref{eq:modified_F} and~\eqref{eq:Y_after_Duhamel}, we have that, uniformly in $t$, 
\begin{align}
    |H^i(t,Y_t)|&\lesssim
        \begin{cases}
         \left(|Y^{\leq \nu}_t|+\eps|N^{>\nu}_t|\right)\wedge 1,\quad & i\leq \nu,\label{eq:H_estimate}\\
         \left(\eps |(e^{\lam t}U_t)^{\leq \nu}|+\eps|N^{>\nu}_t|\right)\wedge 1, \quad & i>\nu,
        \end{cases}\\
        |\bH^i(t,Y_t)|&\lesssim \left(|H^i(t, Y_t)|+\sum_{j>\nu}e^{\lj t}|y^j|\right)\wedge 1, \quad i\in \{1,2,\dots,d\}\label{eq:H_bar_estimate}.
\end{align}

Using the definitions~\eqref{eq:M},~\eqref{eq:V},~\eqref{eq:U}, and~\eqref{eq:modified_F}, we have that for each $q\geq 1$ there is a constant $C_q>0$ such that, for all $t$,
\begin{align}\label{eq:prelim_est_U,V,M,N}
    \begin{split}
        \E\left|M^j_t\right|^q,\ \E\left|V^j_t\right|^q,\ \E\left|U^j_t\right|^q \leq C_q, \quad &j\leq \nu;\\
        \E\left|e^{\lj t}M^j_t\right|^q,\ \E\left|e^{\lj t}V^j_t\right|^q,\ \E\left|N^j_t\right|^q \leq C_q, \quad &j> \nu.
    \end{split}
\end{align}

\subsubsection{0th order derivatives}
Our first goal is to obtain $L^p$ estimates on
\begin{equation}\label{eq:deviations}
\begin{array}{cc}
U_T^i-Z_T^i,\  U_T^i-\bar{Z}_T^i,& i\leq \nu,\\
N_T^i-Z_T^i,\  N_T^i-\bar{Z}_T^i,& i > \nu.
\end{array}
\end{equation}
Taking arbitrary $\beta\in(0,1)$ to be determined later, we define
\begin{align*}\begin{split}
        \eta^j=\inf\left\{t>0:|Y^j_t|\geq \eps^\beta\right\}, \quad j\leq \nu; \quad\quad
        \eta = \min_{1\leq j \leq \nu}\eta^j.
    \end{split}
\end{align*}
We first derive a few estimates involving $\eta^j$. The above definition implies $\eps^\beta \leq |Y^j_{\eta^j}|= e^{\lj \eta^j}|y^j+\eps U^j_{\eta^j}|$. Hence, $\eta^j \geq  \frac{1}{\lj}\log(\eps^\beta|y^j+\eps U^j_{\eta^j}|^{-1})$, which implies that 
\begin{align}\label{eq:eta_exponential_moment_est}
    \begin{split}
        \E{e^{-q \eta}}&\leq \sum_{j=1}^\nu\E{e^{-q \eta^j}}\leq \sum_{j=1}^\nu \eps^{-\beta\frac{q}{\lj}}\E{\left|y^j+\eps U^j_{\eta^j}\right|^{\frac{q}{\lj}}}\\
    &\lesssim \sum_{j=1}^\nu \left(\eps^{-\beta}\left|y^{\leq \nu}\right|\right)^{\frac{q}{\lj}}+ \sum_{j=1}^\nu \eps^{(1-\beta)\frac{q}{\lj}}\E{\left|U^j_{\eta^j}\right|^{\frac{q}{\lk}}}\\
    &\lesssim  \sum_{j=1}^\nu \left(\eps^{-\beta}\left|y^{\leq \nu}\right|\right)^{\frac{q}{\lj}}+ \sum_{j=1}^\nu \eps^{(1-\beta)\frac{q}{\lj}}, \quad q>0,
    \end{split}
\end{align}
where $\E{|U^j_{\eta^j}|^{\frac{q}{\lk}}}\lesssim 1$ follows from the definition of $U_t$ in~\eqref{eq:U} and BDG inequality. 
    
Let us consider $i\le \nu$. Recall the definition of $\pp$ from \eqref{eq:pp}. Using BDG,~\eqref{eq:H_estimate},~\eqref{eq:prelim_est_U,V,M,N} and~\eqref{eq:eta_exponential_moment_est} with $\beta = \frac{1}{2}\upsilon$ and $q=p\lam^i$, we obtain, for some $\delta'>0$, 
\begin{align}\label{eq:est_int_H_i<nu}
    \begin{split}
        &\E\left|\int_0^T\left|e^{-\li s}H^i(s, Y_s)\right|^2 ds\right|^\frac{p}{2}
        \\
        &\lesssim \sum_{j\leq \nu}\E\left|\int_0^{T\wedge \eta}\left|e^{-\li s}Y^j_s\right|^2 ds\right|^\frac{p}{2}+\sum_{j> \nu}\E\left|\int_0^{T\wedge \eta}\left|e^{-\li s}\eps N^j_s\right|^2 ds\right|^\frac{p}{2} + \E\left|\int_{T\wedge \eta}^T\left|e^{-\li s}\right|^2 ds\right|^\frac{p}{2}
        \\
        &\lesssim \eps^{\beta p}+\eps^p + \E e^{-p\li \eta} \lesssim \eps^{\beta p}+\eps^p + \pp\left(\left(\eps^{-\frac{1}{2}\upsilon}|y^{\leq \nu}|\right)^p\right) + \pp\left(\eps^{(1-\frac{\upsilon}{2})p}\right)
        \\
        &\lesssim \eps^{\delta'}\left(1+\pp\left(\eps^{-\upsilon}\left|y^{\leq \nu}\right|\right)\right)^{p},\quad i\leq \nu,
    \end{split}
\end{align}
where in the last inequality we also used that for $r\ge0$, $\pp(r^p)\lesssim_p(\pp(r))^p$. Due to~\eqref{eq:H_bar_estimate} and the fact that $\lj<0$ for $j>\nu$, \eqref{eq:est_int_H_i<nu} also implies
\begin{align}\label{eq:est_int_bH_i<nu}
\begin{split}
    &\E\left|\int_0^T\left|e^{-\li s}\bH^i(s,Y_s)\right|^2 ds\right|^\frac{p}{2}
    \\
    &\lesssim  \E\left|\int_0^T\left|e^{-\li s}H^i(s,Y_s)\right|^2 ds\right|^\frac{p}{2}+ \sum_{j>\nu}\E\left|\int_0^T\left|e^{-\li s}e^{\lj s}y^j\right|^2 ds\right|^\frac{p}{2}
    \\
    &\lesssim \eps^{\delta'}\left(1+\pp\left(\eps^{-\upsilon}\left|y^{\leq \nu}\right|\right)\right)^p + \left|y^{>\nu}\right|^p,\quad i\leq \nu.
\end{split}
\end{align}

Due to~\eqref{eq:est_int_H_i<nu} and~\eqref{eq:prelim_est_U,V,M,N}, for some $\delta_0>0$,
\begin{align}\label{eq:0th_der_U-Z_est_i<nu}
\begin{split}
    &\E\left|U^i_T-Z^i_T\right|^p\lesssim \E\left|M^i_T-Z^i_T\right|^p + \eps^p\E\left|V^i_T\right|^p\lesssim \E\left|\int_0^T\left|e^{-\li s}H^i(s, Y_s)\right|^2 ds\right|^\frac{p}{2}+\eps^p
    \\
    &\lesssim \eps^{\delta_0}\left(1+\pp\left(\eps^{-\upsilon}\left|y^{\leq \nu}\right|\right)\right)^p,\quad i\leq \nu.
\end{split}
\end{align}
Similarly, using~\eqref{eq:est_int_bH_i<nu} and~\eqref{eq:prelim_est_U,V,M,N}, we have, for some $\delta'_0>0$,
\begin{align}\label{eq:0th_der_U-bZ_est_i<nu}
    \begin{split}
        \E\left|U^i_T-\bZ^i_T\right|^p&\lesssim  \E\left|\int_0^T\left|e^{-\li s}\bH^i(s,Y_s)\right|^2 ds\right|^\frac{p}{2}+\eps^p\\
        &\lesssim \eps^{\delta'_0}\left(1+\pp\left(\eps^{-\upsilon}\left|y^{\leq \nu}\right|\right)\right)^p + \left|y^{>\nu}\right|^p,\quad i\leq \nu.
    \end{split}
\end{align}

Then, we consider $i>\nu$. Let us estimate, using Minkowski's integral inequality,~\eqref{eq:H_estimate} and~\eqref{eq:prelim_est_U,V,M,N},\begin{align}\label{eq:est_int_H_i>nu}
    \begin{split}
        &\E\left|e^{2\li T}\int_0^T\left|e^{-\li s}H^i(s,Y_s)\right|^2 ds\right|^\frac{p}{2}
        \\
        &\lesssim  \sum_{j\leq \nu}\E\left|e^{2\li T}\int_0^{T}\left|e^{-\li s}\eps e^{\lj s} U^j_s\right|^2 ds\right|^\frac{p}{2}+\sum_{j> \nu}\E\left|e^{2\li T}\int_0^{T}\left|e^{-\li s}\eps N^j_s\right|^2 ds\right|^\frac{p}{2}
        \\
        &\lesssim \sum_{j\leq \nu}\eps^{ p}e^{p\li T}+\eps^p \lesssim \eps^{\frac{p}{2}},\quad i>\nu.
    \end{split}
\end{align}
Since $\lj<0$ for $j>\nu$, due to~\eqref{eq:H_bar_estimate} and~\eqref{eq:est_int_H_i>nu}, similar to the derivation of~\eqref{eq:est_int_bH_i<nu}, one can see
\begin{align}\label{eq:est_int_bH_i>nu}
    \begin{split}
        &\E\left|e^{2\li T}\int_0^T\left|e^{-\li s}\bH^i(s,Y_s)\right|^2 ds\right|^\frac{p}{2} \lesssim \eps^{\frac{p}{2}}+\left|y^{>\nu}\right|^p,\quad i>\nu.
    \end{split}
\end{align}

Using~\eqref{eq:est_int_H_i>nu} and~\eqref{eq:prelim_est_U,V,M,N}, we obtain 
\begin{align}\label{eq:0th_der_U-Z_est_i>nu}
\begin{split}
        \E\left|N^i_T-Z^i_T\right|^p&\lesssim \E\left|e^{\li T}M^i_T-Z^i_T\right|+\eps^p\E\left|e^{\li T}V^i_T\right|
        \\
        &\lesssim \E\left|e^{2\li T}\int_0^T\left|e^{-\li s}H^i(s,Y_s)\right|^2 ds\right|^\frac{p}{2}+\eps^p \lesssim \eps^{\frac{p}{2}},\quad i>\nu.
\end{split}
\end{align}

From~\eqref{eq:est_int_bH_i>nu} and~\eqref{eq:H_bar_estimate} it can be derived that
\begin{align}\label{eq:0th_der_U-bZ_est_i>nu}
    \E\left|N^i_T-\bZ^i_T\right|^p\lesssim \eps^{\frac{p}{2}} + \left|y^{>\nu}\right|^p,\quad i> \nu.
\end{align}

\subsubsection{1st order derivatives}
In order to estimate the Malliavin derivatives of the r.v.'s in~\eqref{eq:deviations}, we need to estimate the derivatives 
of $U_t^i$ and $N_t^i$. These derivatives are, in turn, related to those of  $Y_t^i$ due to ~\eqref{eq:Y_after_Duhamel} and~\eqref{eq:U}:
\begin{align}\label{eq:DY_formula}
    \D^{(k)}Y^i_t = 
    \eps e^{\li t }\D^{(k)}U^i_t = \eps \D^{(k)}N^i_t,
\end{align}
where the superscript in parentheses indicates the order of differentiation. For $j=1,2,\dots, d$, \cite[Theorem 2.2.1]{nualart} implies
\begin{align}\label{eq:formula_of_DU}
    \D^j_rU^i_t = e^{-\li r}F^i_j(Y_r) + \int_r^te^{ -\li s}\partial_k F^i_l(Y_s)\D^j_rY^k_sdW^l_s+\eps\int_r^t e^{-\li s}\partial_k G^i(Y_s)\D^j_r Y^k_s ds.
\end{align}

Let $0\leq r\leq t\leq T$. We use \eqref{eq:formula_of_DU},~\eqref{eq:pBracket_property_1}, and~\eqref{eq:pBracket_property_2} to obtain the first inequality below; we plug in~\eqref{eq:DY_formula} to obtain the second inequality; and use estimates~\eqref{eq:modified_F} and~\eqref{eq:choice_of_theta} to obtain that, uniformly in $r,t$,
\begin{multline}
    \label{eq:estim-1st-der-of-U}
        \pbracket{\D^j_r U^i_t}  \lesssim \pbracket{e^{-\li r}F^i_j(Y_r)} +  \sum_{k=1}^d\left(\int_r^te^{ -2\li s}\pbracket{\D^j_rY^k_s}ds+\eps^2 T\int_r^te^{-2\li s}\pbracket{ \D^j_r Y^k_s }ds\right)
        \\ \lesssim e^{-2\li r} + \left(\eps^2+\eps^4 T\right) \sum_{k\leq \nu}\int_r^te^{2(\lk -\li)s}\pbracket{\D^j_rU^k_s}ds+\left(\eps^2+\eps^4T\right)\sum_{k>\nu}\int_r^te^{-2\li s}\pbracket{\D^j_rN^k_s}ds
    \\
    \lesssim e^{-2\li r} + \eps \int_r^t \sum_{k\leq \nu}\pbracket{\D^j_rU^k_s}ds+\sum_{k>\nu}\pbracket{\D^j_rN^k_s}ds, \quad i\leq \nu,
\end{multline}
Similarly, we have that, uniformly in $r,t$,
\begin{multline}
    \label{eq:estim-1st-der-of-N}
        \pbracket{\D^j_r N^i_t}  \lesssim \pbracket{e^{\li(t- r)}F^i_j(Y_r)} +  \sum_{k=1}^de^{2\li t}\left(\int_r^te^{ -2\li s}\pbracket{\D^j_rY^k_s}ds+\eps^2 T\int_r^te^{-2\li s}\pbracket{ \D^j_r Y^k_s }ds\right)\\
    \lesssim e^{2\li(t-r)} + \left(\eps^2+\eps^4 T\right) \sum_{k\leq \nu}\int_r^te^{2\lk s}\pbracket{\D^j_rU^k_s}ds+\left(\eps^2+\eps^4T\right)\sum_{k>\nu}\int_r^t\pbracket{\D^j_rN^k_s}ds\\
    \lesssim e^{2\li(t-r)} + \eps \int_r^t \sum_{k\leq \nu}\pbracket{\D^j_rU^k_s}ds+\sum_{k>\nu}\pbracket{\D^j_rN^k_s}ds, \quad i> \nu.
\end{multline}

We need the following lemma.
\begin{lemma} \label{lem:solve_inequality_system}
Let $d,l\in\N$ and $m\geq 0$. Write $\mathbf{r}=(r_i)_{i=1}^l$ with all $r_i$ nonnegative, and $\hat{\mathbf{r}}=\max_{1\leq i\leq d}r_i$. Suppose $c(\mathbf{r},t)\geq 0$ for all $\mathbf{r}, t$ and that $t\mapsto c(\mathbf{r},t)$ is nondecreasing for each fixed $\mathbf{r}$. Then, the system of inequalities
\begin{align}\label{eq:system_of_ineq}
    0\leq a^i(\mathbf{r},t) \leq C\left(  \eps^m c^i(\mathbf{r},t) +  \eps \sum_{k=1}^d\int_{\hat{\mathbf{r}}}^ta^k(\mathbf{r},s)ds\right), \quad  \hat{\mathbf{r}}\leq t\leq T,\ i=1,2,...,d,
\end{align}
with $T$ satisfying \eqref{eq:choice_of_theta}, implies that there is a constant $C$ independent of $\eps,\ T$ such that $a^i(\mathbf{r},t)\le C\eps^m\sum_{k=1}^d c^k(\mathbf{r},t)$ for all
$t\in [0,T]$, $\mathbf{r}$ satisfying $\hat{\mathbf{r}}\leq t$, and $i=1,2,\dots,d$.
\end{lemma}
\begin{proof}[Proof of Lemma~\ref{lem:solve_inequality_system}]
Let $b(\mathbf{r},t)= \sum_{i=1}^d a^i(\mathbf{r},t)$. We sum up the inequalities~\eqref{eq:system_of_ineq} in $i$ to obtain
\begin{align*}
    0\leq b(\mathbf{r},t)\leq C\left(\eps^m \sum_{k=1}^d c^k(\mathbf{r},t) +\eps \int_0^t b(\mathbf{r},s)ds\right).
\end{align*}
Gronwall's inequality implies that for some constant $C$ independent of $\eps$ 
\begin{align*}
    0\leq b(\mathbf{r},t) \leq C \eps^m\sum_{k=1}^d c^k(\mathbf{r},t) e^{C\eps T} .
\end{align*}
Due to~\eqref{eq:choice_of_theta}, $ e^{C\eps T}$ is bounded. Using $a^i(\mathbf{r},t)\leq b(\mathbf{r},t)$, we complete the proof.
\end{proof}

Applying this result
with $l=1$, $m=0$, $c^i(r,t)=e^{-2\li r} $ and $a^i(r,t)=\pbracket{\D^j_rU^i_t} $ for $i\le \nu$, 
$c^i(r,t)=e^{2\li (t-r)} $ 
and $a^i(r,t)= \pbracket{\D^j_rN^i_t}$
for $i>\nu$
to~\eqref{eq:estim-1st-der-of-U} and~\eqref{eq:estim-1st-der-of-N}, we obtain, for $i\leq \nu,$ $m>\nu$, $1\leq j\leq d$
\
\begin{align}\label{eq:DN,DU_pbracket_est}
    \pbracket{\D^j_rU^i_t},\ \pbracket{\D^j_rN^m_t}\lesssim \sum_{k\leq\nu} e^{-2\lk r} +\sum_{k>\nu} e^{2\lk(t-r)} \lesssim 1 , \quad   r\leq t\leq T,
\end{align}
which implies due to~\eqref{eq:DY_formula} and~\eqref{eq:choice_of_theta} that
\begin{align}\label{eq:DY_est}
    \pbracket{\D^j_rY^i_t}\lesssim \eps, \quad r\leq t\leq T,\ 1\leq i,j\leq d.
\end{align}
The estimates~\eqref{eq:DN,DU_pbracket_est} together with~\eqref{eq:pBracket_property_3} give
\begin{align}\label{eq:DU_est}
\pbracket{\left\|\D (U^{\leq \nu}_T, N^{>\nu}_T) \right\|_{\Hil}} \lesssim 1.
\end{align}

For $i\leq \nu$, we have that, due to \eqref{eq:def_Z},
\begin{align*}
    \D^j_r Z^i_t = e^{-\li r}F^i_j\left(0^{\leq \nu},(e^{\lambda s}y)^{>\nu}\right),
\end{align*}
which along with \eqref{eq:formula_of_DU} yields that
\begin{align}\label{eq:D(U-Z)_formula}
    \D^j_r(U^i_t-Z^i_t)=e^{-\li r}H^i_j(r,Y_r)+\int_r^te^{-\li s}\partial_kF^i_l(Y_s)\D^j_rY^k_sdW^l_s+\eps\int_r^te^{-\li s}\partial_kG^i(Y_s)\D^j_rY^k_sds.
\end{align}
Hence, we obtain 
\begin{align}\label{eq:D(U-Z)_hil_norm_est}
\begin{split}
        &\pbracket{\|\D(U^i_T-Z^i_T)\|_\Hil}\lesssim \left( \E\left|\int_0^T\left|e^{-\li r}H^i(r,Y_r)\right|^2 dr \right|^\frac{p}{2}\right)^\frac{2}{p} \\
    & + \sum_{j, k = 1}^d \left[\left\|\int_\cdot^T e^{ -\li s}\partial_k F^i_l(Y_s)\D^j_\cdot Y^k_sdW^l_s\right\|_\Hil\right]_p \\
    & + \sum_{j, k = 1}^d \eps^2\left[\left\| \int_\cdot^T e^{ -\li s}\partial_k G^i(Y_s)\D^j_\cdot Y^k_sds\right\|_\Hil\right]_p.
\end{split}
\end{align}
Due to~\eqref{eq:est_int_H_i<nu}, the first term on the right is $\lesssim \eps^{2\delta'/p}(1+\pp(\eps^{-\upsilon}|y^{\leq \nu}|))^2$. For the next two terms, we first invoke properties~\eqref{eq:pBracket_property_1},~\eqref{eq:pBracket_property_2} and~\eqref{eq:pBracket_property_3}, and then apply the boundedness of derivatives of $F$,~\eqref{eq:DY_est},~\eqref{eq:choice_of_theta} to get 
\begin{align}\label{eq:D(U-Z)_hil_norm_est_step_1}
\begin{split}
        &\left[\left\|\int_\cdot^T e^{ -\li s}\partial_k F^i_l(Y_s)\D^j_\cdot Y^k_sdW^l_s\right\|_\Hil\right]_p \leq \int_0^T \left[\int_r^T e^{ -\li s}\partial_k F^i_l(Y_s)\D^j_r Y^k_sdW^l_s\right]_pdr\\
    &\lesssim \int_0^T \int_r^T \sum_{l}\left[e^{ -\li s}\partial_k F^i_l(Y_s)\D^j_r Y^k_s\right]_p ds dr \lesssim \int_0^T \int_r^T \pbracket{\D^j_r Y^k_s} ds\, dr \lesssim \eps T^2 \leq \eps^\frac{1}{2}
\end{split}
\end{align}
and, similarly,
\begin{align}\label{eq:D(U-Z)_hil_norm_est_step_2}
        \eps^2\left[\left\| \int_\cdot^T e^{ -\li s}\partial_k G^i(Y_s)\D^j_\cdot Y^k_sds\right\|_\Hil\right]_p  \lesssim \eps^2 T\int_0^T \int_r^T \pbracket{\D^j_r Y^k_s} ds\, dr \leq \eps^2.
\end{align}
Therefore, these yield, for some $\delta_1>0$,
\begin{align}\label{eq:D(U-Z)_<_nu_est}
    \pbracket{\left\|\D(U^{\leq \nu}_T-Z^{\leq \nu}_T)\right\|_\Hil} \lesssim \eps^{\delta_1}\left(1+\pp\left(\eps^{-\upsilon}\left|y^{\leq \nu}\right|\right)\right)^2.
\end{align}

For $i>\nu$, we can compute 
\begin{align}\label{eq:D(N-Z)_formula}
\begin{split}
       \D^j_r(N^i_t-Z^i_t)=e^{\li(t- r)}H^i_j(r,Y_r)+e^{\li t}\int_r^te^{-\li s}\partial_kF^i_l(Y_s)\D^j_rY^k_sdW^l_s\\
    +\eps e^{\li t}\int_r^te^{-\li s}\partial_kG^i(Y_s)\D^j_rY^k_sds.
\end{split}
\end{align}
Note that now $\li <0$. To bound $\pbracket{\|\D(N^i_T-Z^i_T)\|_\Hil}$, 
we first estimate it  similarly to~\eqref{eq:D(U-Z)_hil_norm_est}, and then apply ~\eqref{eq:est_int_H_i>nu} and estimates analogous to~\eqref{eq:D(U-Z)_hil_norm_est_step_1} and~\eqref{eq:D(U-Z)_hil_norm_est_step_2} to obtain, for some $\delta'_1>0$,
\begin{align}\label{eq:D(U-Z)_>_nu_est}
    \pbracket{\left\|\D(N^{> \nu}_T-Z^{> \nu}_T)\right\|_\Hil} \lesssim \eps^{\delta'_1}.
\end{align}

To compare with $\bZ_T$, we note that $\D^j_r(U^i_t-\bZ^i_t)$ and $\D^j_r(N^i_t-\bZ^i_t)$ have representations similar to~\eqref{eq:D(U-Z)_formula} and~\eqref{eq:D(N-Z)_formula}, respectively, only with $H$ replaced by~$\bH$. Repeating estimations ~\eqref{eq:D(U-Z)_formula}---\eqref{eq:D(U-Z)_>_nu_est} and using~\eqref{eq:est_int_bH_i<nu} and~\eqref{eq:est_int_bH_i>nu} in place of~\eqref{eq:est_int_H_i<nu} and~\eqref{eq:est_int_H_i>nu}, we obtain
\begin{align}
    &\pbracket{\left\|\D(U^{\leq \nu}_T-\bZ^{\leq \nu}_T)\right\|_\Hil} \lesssim \eps^{\delta_1}\left(1+\pp\left(\eps^{-\upsilon}\left|y^{\leq \nu}\right|\right)\right)^2+\left|y^{>\nu}\right|^2,\label{eq:D(U-bZ)_<_nu_est} 
    \\
    &\pbracket{\left\|\D(N^{>\nu}_T-\bZ^{> \nu}_T)\right\|_\Hil} \lesssim \eps^{\delta'_1}+\left|y^{>\nu}\right|^2.\label{eq:D(U-bZ)_>_nu_est} 
\end{align}

\subsubsection{2nd order derivatives}
Note that~\eqref{eq:def_Z} implies that $Z_t$ and $\overline Z_t$ are integrals of deterministic functions and thus
\begin{align}\label{eq:higher_order_der_Z_bZ}
    \D^{(k)}Z_t=\D^{(k)}\bZ_t=0, \quad k\geq 2,\ t\geq 0.
\end{align}

To compute higher order derivatives of $U_t^i$, $i=1,\ldots,d$, let us rewrite~\eqref{eq:Y_after_Duhamel} as
\begin{align*}U^i_t&=\int_0^t e^{-\li s}F^i_l(Y_s)dW^l_s+\eps\int^t_0e^{-\li s}G^i(Y_s)ds\\
\notag  
   &=\int_0^t e^{-\li s}F^i_l\left( e^{\lam s}(y+\eps U_s)\right)dW^l_s+\eps\int^t_0e^{-\li s}G^i\left( e^{\lam s }(y+\eps U_s)\right)ds   
\end{align*}
and apply formula (2.54) in~\cite[Section 2.2]{nualart} to this equation which plays the role of 
of equation~(2.37) therein.  For $r_1, r_2\leq t \leq T$, we obtain
\begin{align}
    \begin{split}\label{eq:D^2U}
    &\D^{j_1,j_2}_{r_1,r_2}U^i_t \quad =\quad  e^{-\li r_1}\partial_k F^i_{j_1}(Y_{r_1})\D^{j_2}_{r_2}Y^k_{r_1}+e^{-\li r_2}\partial_k F^i_{j_2}(Y_{r_2})\D^{j_1}_{r_1}Y^k_{r_2}
    \\
    & + \int_{r_1\vee r_2}^t e^{-\li s }\left(\partial^2_{k_1,k_2} F^i_l(Y_s)\right)\left(\D^{j_1}_{r_1}Y^{k_1}_s\right)\left(\D^{j_2}_{r_2}Y^{k_2}_s\right)dW^l_s +\int_{r_1\vee r_2}^t e^{-\li s }\partial_{k} F^i_l(Y_s)\D^{j_1,j_2}_{r_1,r_2}Y^k_s  dW^l_s
    \\
    & +  \eps\int_{r_1\vee r_2}^t e^{-\li s }\left(\partial^2_{k_1,k_2} G^i(Y_s)\right)\left(\D^{j_1}_{r_1}Y^{k_1}_s\right)\left(\D^{j_2}_{r_2}Y^{k_2}_s\right)ds + \eps\int_{r_1\vee r_2}^t e^{-\li s }\partial_{k} G^i(Y_s)\D^{j_1,j_2}_{r_1,r_2}Y^k_s  ds.
        \end{split}
\end{align}
We can also derive this formula directly from~\eqref{eq:formula_of_DU}.

Let us use this identity to estimate $  \pbracket{\D^{j_1,j_2}_{r_1,r_2}U^i_t} $ for $i\leq\nu$. In this case, we have $e^{-\li s}\leq 1$ for all $s\geq 0$.
We use properties~\eqref{eq:pBracket_property_1} and ~\eqref{eq:pBracket_property_2}, the boundedness of the derivatives of $F$ and $G$
assumed in \eqref{eq:modified_F}, 
the estimate $\pbracket{(\D^{j_1}_{r_1}Y^{k_1}_s)(\D^{j_2}_{r_2}Y^{k_2}_s)}\leq \pbracketX{\D^{j_1}_{r_1}Y^{k_1}_s}{2p} \pbracketX{\D^{j_2}_{r_2}Y^{k_2}_s}{2p} $ implied by the Cauchy--Schwarz inequality, and 
~\eqref{eq:DY_formula} to obtain that, uniformly in $r_1,r_2\leq t\leq T$, 
\begin{align*}
\begin{split}
    \pbracket{\D^{j_1,j_2}_{r_1,r_2}U^i_t} & \lesssim \pbracket{\D^{j_2}_{r_2}Y^k_{r_1}} + \pbracket{\D^{j_1}_{r_1}Y^k_{r_2}}
    \\
    &+ \left(1+\eps^2T\right)\sum_{k_1,k_2 = 1}^d\int_{r_1\vee r_2}^t  \pbracketX{\D^{j_1}_{r_1}Y^{k_1}_s}{2p} \pbracketX{\D^{j_2}_{r_2}Y^{k_2}_s}{2p} ds
    \\
    &+\left(1+\eps^2T\right) \sum_{k=1}^d \int_{r_1\vee r_2}^t \pbracket{\D^{j_1,j_2}_{r_1,r_2}Y^k_s}  ds,\quad i \leq\nu.
\end{split}
\end{align*}
Similarly, using \eqref{eq:D^2U}, the relation in \eqref{eq:DY_formula}, and that $e^{\li(t-s)}\leq 1$ for all $s\leq t$ when $i>\nu$, we have exactly the same bound for $\pbracket{\D^{j_1,j_2}_{r_1,r_2}N^i_t}$, $i>\nu$, uniformly in $r_1,r_2\leq t\leq T$: 
\begin{align*}
    \pbracket{\D^{j_1,j_2}_{r_1,r_2}N^i_t} & \lesssim \pbracket{\D^{j_2}_{r_2}Y^k_{r_1}} + \pbracket{\D^{j_1}_{r_1}Y^k_{r_2}}
    \\
    &+ \left(1+\eps^2T\right)\sum_{k_1,k_2 = 1}^d\int_{r_1\vee r_2}^t  \pbracketX{\D^{j_1}_{r_1}Y^{k_1}_s}{2p} \pbracketX{\D^{j_2}_{r_2}Y^{k_2}_s}{2p} ds
    \\
    &+\left(1+\eps^2T\right) \sum_{k=1}^d \int_{r_1\vee r_2}^t \pbracket{\D^{j_1,j_2}_{r_1,r_2}Y^k_s}  ds,\quad i >\nu.
\end{align*}
Applying~\eqref{eq:DY_est} to bound the first order derivatives of $Y$, using \eqref{eq:DY_formula} to rewrite the second order derivatives of $Y$ in terms of $U$ for $k\leq \nu$ and in terms of $N$ for $k>\nu$, and then applying~\eqref{eq:choice_of_theta} to bound $T$ and $e^{2\lk s}$ for $k\leq \nu$, one can see that, uniformly in $r_1,r_2\leq t\leq T$, $i\leq \nu$ and $m>\nu$, 
\begin{align}\label{eq:methods_treating_2nd_der} 
    \pbracket{\D^{j_1,j_2}_{r_1,r_2}U^i_t}, \ \pbracket{\D^{j_1,j_2}_{r_1,r_2}N^m_t} \lesssim \eps +\eps  \int_{r_1\vee r_2}^t \bigg(\sum_{k\leq \nu}\pbracket{\D^{j_1,j_2}_{r_1,r_2}U^k_s} + \sum_{k > \nu} \pbracket{\D^{j_1,j_2}_{r_1,r_2}N^k_s}\bigg)ds.
\end{align}
Let us momentarily fix $j_1,j_2$, and set
\begin{align*}
    a^i(r_1,r_2,t)=
    \begin{cases}
     \pbracket{\D^{j_1,j_2}_{r_1,r_2}U^i_t}, \quad i\leq \nu,\\
     \pbracket{\D^{j_1,j_2}_{r_1,r_2}N^i_t}=e^{2\lj t}\pbracket{\D^{j_1,j_2}_{r_1,r_2}U^i_t}, \quad i> \nu.
    \end{cases}
\end{align*}
Plug this into~\eqref{eq:methods_treating_2nd_der} to obtain that, uniformly in $r_1,r_2\leq t\leq T$,
\begin{align*}
    a^i(r_1,r_2,t)\lesssim \eps + \eps \sum_{k=1}^d\int_{r_1\vee r_2}^t a^k(r_1,r_2,s)ds, \quad i\in\{1,2,\dots,d\}.
\end{align*}

Lemma~\ref{lem:solve_inequality_system} implies
\begin{align}\label{eq:D^2U,D^2_est}
        \pbracket{\D^{j_1,j_2}_{r_1,r_2}U^i_t},\ \pbracket{\D^{j_1,j_2}_{r_1,r_2}N^m_t} \lesssim \eps, \quad \for i\leq \nu,\ m>\nu; \ r_1, r_2, t\leq T,\ 1\leq j_1,j_2\leq d,
\end{align}
This result, due to~\eqref{eq:DY_formula} and~\eqref{eq:choice_of_theta}, yields
\begin{align}\label{eq:D^2Y_est}
    \pbracket{\D^{j_1,j_2}_{r_1,r_2}Y^i_t} \lesssim \eps^2, \quad \for 1\leq i\leq d; \ r_1, r_2, t\leq T,\ 1\leq j_1,j_2\leq d,
\end{align}
which is for later use.

From~\eqref{eq:pBracket_property_3} and~\eqref{eq:D^2U,D^2_est}, we obtain
\begin{align*}
    \pbracket{\left\|\D^{(2)} U^i_T \right\|_{\Hil^{\otimes  2}}} \lesssim  \sum_{j_1,j_2=1}^d\int_{[0,T]^2}\pbracket{\D^{j_1,j_2}_{r_1,r_2}U^i_T}dr_1dr_2 \lesssim \eps T^2,\quad i\le \nu,
\end{align*}
and a similar bound for $\pbracket{\|\D^{(2)} N^i_T \|_{\Hil^{\otimes  2}}}$, $i\ge \nu$. This along with~\eqref{eq:choice_of_theta},~\eqref{eq:D^2U,D^2_est} and~\eqref{eq:higher_order_der_Z_bZ} implies
\begin{align}\label{eq:D^2(U-Z)_est}
\begin{split}
        &
        \pbracket{\left\|\D^{(2)} \left(U^{\leq \nu}_T, N^{>\nu}_T\right) \right\|_{\Hil^{\otimes  2}}} \lesssim \eps^\frac{1}{2}, 
        \\
        &\pbracket{\left\|\D^{(2)}\left( \left(U^{\leq \nu}_T, N^{>\nu}_T\right)-Z_T \right)\right\|_{\Hil^{\otimes  2}}}\lesssim\eps^\frac{1}{2},
        \\
        &\pbracket{\left\|\D^{(2)}\left( \left(U^{\leq \nu}_T, N^{>\nu}_T\right)-\bZ_T\right)\right \|_{\Hil^{\otimes  2}}}\lesssim \eps^\frac{1}{2}.
\end{split}
\end{align}

\subsubsection{3rd order derivatives} Similarly to the above argument for second order derivatives, we apply
~(2.54) from~\cite[Section~2.2]{nualart} to obtain that
for $r_1,r_2,r_3\leq t\leq T$,
\begin{align*}
    &\D^{j_1,j_2,j_3}_{r_1,r_2,r_3}U^i_t 
    = A^{j_1,j_2,j_3}_{r_1,r_2,r_3}+ 
    \int_{r_1\vee r_2\vee r_3}^t e^{-\lambda^i s}B^{j_1,j_2,j_3}_{r_1,r_2,r_3,l}(s)dW_s^l+ \eps \int_{r_1\vee r_2\vee r_3}^t e^{-\lambda^i s}C^{j_1,j_2,j_3}_{r_1,r_2,r_3}(s)ds,
\end{align*}
where: $A^{j_1,j_2,j_3}_{r_1,r_2,r_3}$ is a linear combination of terms 
\begin{align*}
e^{-\li r_{n_0}}\partial^2_{k_1,k_2}F^i_{j_{n_0}}(Y_{r_{n_0}})\prod_{m=1}^2\D^{j_{n_m}}_{r_{n_m}}Y^{k_m}_{r_{n_0}},\quad 
e^{-\li r_{n_0}} \partial_{k}F^i_{j_{n_0}}(Y_{r_{n_0}})\D^{j_{n_1},j_{n_2}}_{r_{n_1},r_{n_2}}Y^{k}_{r_{n_0}}, 
\end{align*}
$B^{j_1,j_2,j_3}_{r_1,r_2,r_3,l}(s)$ is a linear combination of terms 
\begin{align*}
\partial^3_{k_1,k_2,k_3}F^i_l(Y_s)\prod_{m=1}^3\D^{j_{m}}_{r_{m}}Y^{k_m}_s,\quad
\partial^2_{k_1,k_2}F^i_l(Y_s)\left(\D^{j_{n_1},j_{n_2}}_{r_{n_1},r_{n_2}}Y^{k_1}_s\right)\left(\D^{j_{n_3}}_{r_{n_3}}Y^{k_2}_s\right),\quad 
\partial_{k}F^i_l(Y_s)\D^{j_{1},j_{2},j_{3}}_{r_{1},r_{2},r_{3}}Y^{k}_s, 
\end{align*}
$C^{j_1,j_2,j_3}_{r_1,r_2,r_3}(s)$ is a linear combination of terms 
\begin{align*}
\partial^3_{k_1,k_2,k_3}G^i(Y_s)\prod_{m=1}^3\D^{j_{m}}_{r_{m}}Y^{k_m}_s,\quad
\partial^2_{k_1,k_2}G^i(Y_s)\left(\D^{j_{n_1},j_{n_2}}_{r_{n_1},r_{n_2}}Y^{k_1}_s\right)\left(\D^{j_{n_3}}_{r_{n_3}}Y^{k_2}_s\right),\quad 
\partial_{k}G^i(Y_s)\D^{j_{1},j_{2},j_{3}}_{r_{1},r_{2},r_{3}}Y^{k}_s.
\end{align*}
In all these terms, $\{n_0,n_1,n_2\} = \{1,2,3 \}$.

Following the same steps as in the analysis of~\eqref{eq:D^2U}, 
applying the estimates of first and second derivatives that we already have established in~\eqref{eq:DY_est} and~\eqref{eq:D^2Y_est}  we obtain, for $i\leq \nu$, $m>\nu$, uniformly in $r_1,r_2,r_3\leq t\leq T$,
\begin{align*}
    \pbracket{\D^{j_1,j_2,j_3}_{r_1,r_2,r_3}U^i_t},\ \pbracket{\D^{j_1,j_2,j_3}_{r_1,r_2,r_3}N^m_t} \lesssim \eps^2+ \eps \left(\int_r^t\sum_{k\leq \nu} \pbracket{\D^{j_1,j_2,j_3}_{r_1,r_2,r_3}U^k_s} +\sum_{k>\nu} \pbracket{\D^{j_1,j_2,j_3}_{r_1,r_2,r_3}N^k_s} ds\right),
\end{align*}
where $r=r_1\vee r_2\vee r_3$.
Lemma~\ref{lem:solve_inequality_system} then implies that for $i\leq \nu$, $m>\nu$, and $1\leq j_1,j_2,j_3\leq d$,
\begin{align*}
        \pbracket{\D^{j_1,j_2,j_3}_{r_1,r_2,r_3}U^i_t},\ \pbracket{\D^{j_1,j_2,j_3}_{r_1,r_2,r_3}N^m_t} \lesssim \eps^2, \quad  \ r_1, r_2, r_2, t\leq T.
\end{align*}
This along with~\eqref{eq:choice_of_theta} and~\eqref{eq:pBracket_property_3} implies
\begin{align}\label{eq:D^3U_est}
        \pbracket{\left\|\D^{(3)} \left(U^{\leq \nu}_T, N^{>\nu}_T\right) \right\|_{\Hil^{\otimes  3}}} \lesssim \eps.
\end{align}

\subsubsection{Sobolev norms} Note that estimates above are obtained for an arbitrarily fixed $p\geq 2$. 
Then, 
0th order derivative estimates~\eqref{eq:prelim_est_U,V,M,N}, \eqref{eq:0th_der_U-Z_est_i<nu},~\eqref{eq:0th_der_U-Z_est_i>nu},~\eqref{eq:0th_der_U-bZ_est_i<nu},~\eqref{eq:0th_der_U-bZ_est_i>nu}, 1st order derivative estimates~\eqref{eq:DU_est},~\eqref{eq:D(U-Z)_<_nu_est},~\eqref{eq:D(U-Z)_>_nu_est},~\eqref{eq:D(U-bZ)_<_nu_est},~\eqref{eq:D(U-bZ)_>_nu_est}, 2nd order derivative estimates
~\eqref{eq:D^2(U-Z)_est} and 3rd order derivative estimates~\eqref{eq:D^3U_est} along with Jensen's inequality yield the following bounds on Sobolev norms: for every $p\geq 1$, 
there is $\delta>0$ such that
\begin{align}\label{eq:derivative_est_conclusion_result_U_N}
    \begin{split}
        \left\|\left(U^{\leq \nu}_T,N^{>\nu}_T\right)\right\|_{3, p}
        &\lesssim_p 1,
        \\
        \left\|\left(U^{\leq \nu}_T,N^{>\nu}_T\right)- Z_T\right\|_{2, p}&\lesssim_p \eps^\delta \left(1+\pp\left(\eps^{-\upsilon}\left|y^{\leq\nu}\right|\right)\right),
        \\
        \left\|\left(U^{\leq \nu}_T,N^{>\nu}_T\right)- \bZ_T\right\|_{2, p}&\lesssim_p \eps^\delta \left(1+\pp\left(\eps^{-\upsilon}\left|y^{\leq\nu}\right|\right)\right)+\left|y^{>\nu}\right|.
    \end{split}
\end{align}
Since $Z_T$ and $\bZ_T$ are linear in $W$, it is easy to compute
\begin{align}\label{eq:derivative_est_conclusion_result_Z}
\left\|Z_T \right\|_{3,p},\quad 
\left\|\bZ_T \right\|_{3,p} \lesssim_p 1, \quad p\geq 1.
\end{align}

\subsection{Malliavin matrix estimates}
We recall the definition of Malliavin matrices given in \eqref{eq:def_malliavin_matrix}. We replace $\mathcal{T}$ therein by $T$ given in \eqref{eq:t_k}, or equivalently, replace $\mathscr{H}_\mathcal{T}$ therein by $\Hil$ given in \eqref{eq:def_Hil}. We want to show that for each $p\geq 1$ there is a constant $C_p$ such that
\begin{align}\label{eq:goal_of_malliavin_matrix_est}
   \E\left|\det\sigma_{(U^{\leq \nu}_T, N^{>\nu}_T)}\right|^{-p},\
        \E\left|\det\sigma_{Z_T}\right|^{-p},\ \ \E\left|\det\sigma_{\bZ_T}\right|^{-p}&\leq C_p, \quad \eps\in (0,1).
\end{align}

 Since the Malliavin matrices of $Z_T,Z^{\leq \nu},\bZ^{\leq \nu}_T,\bZ_T$ are deterministic,  the corresponding bounds are, in fact, trivial, and we need to 
consider only the negative moments of and $\det\sigma_{(U^{\leq \nu}_T, N^{>\nu}_T)}$.

\subsubsection{Boundedness of  $\E|\det\sigma_{(U^{\leq \nu}_T, N^{>\nu}_T)}|^{-p}$} 
We express $Y_s$ in terms of $U_s$ using~\eqref{eq:Y_after_Duhamel}, and rewrite~\eqref{eq:formula_of_DU} as 
\begin{align*}
    \D^j_rU^i_t = A^i_j(r) + \int_r^t \bA^i_{k,l}(s)\D^j_rU^k_sdW^l_s +\int_r^t \bB^i_k(s)\D^j_rU^k_s ds,
\end{align*}
where 
\begin{align}\label{eq:def_A_bA_bB}
        A^i_j(r)=e^{-\li r}F^i_j(Y_r), \quad \bA^i_{k,l}(s)=\eps e^{(\lk-\li)s}\partial_k F^i_l(Y_s),\quad 
    \bB^i_k(s)=\eps^2 e^{(\lk-\li)s}\partial_k G^i(Y_s).
\end{align}
Due to~\eqref{eq:modified_F} and~\eqref{eq:choice_of_theta}, for all $i,\ j ,\ k,\ l$, we have
\begin{align}\label{eq:A,bA,bB_est}
    \left|A^i_j(r)\right|\lesssim e^{-\li r},\ r\le T;\qquad \sup_{s\le T} \left|\bA^i_{k,l}(s)\right|\lesssim \eps^\frac{3}{4},\ \sup_{s\le T} \left|\bB^i_k(s)\right|\lesssim \eps^\frac{7}{4}.
\end{align}

Two useful $d\times d$-matrix-valued processes are given by 
\begin{align}\label{eq:def_YY_ZZ}
    \begin{split}
        \YY^i_j(t)=\delta^i_j &+ \int_0^t\bA^i_{k,l}(s)\YY^k_j(s)dW^l_s +\int_0^t\bB^i_k(s)\YY^k_j(s)ds,\\
        \ZZ^i_j(t)= \delta^i_j &-\int_0^t \bA^k_{j,l}(s)\ZZ^i_k(s)dW^l_s -\int_0^t\left(\bB^k_j(s)-\sum_{l=1}^d\bA^k_{m,l}(s)\bA^m_{j,l}(s)\right)\ZZ^i_k(s)ds,
    \end{split}
\end{align}
where $\delta^i_j$ is the Kronecker symbol.
They correspond to (2.57) and (2.58) in \cite[Section 2.3.1]{nualart}.
Using the It\^o's formula, one can check that (see the computations below (2.58) in \cite[Section 2.3.1]{nualart})
\begin{align}\label{eq:ZY=YZ=I}
    \ZZ(t)\YY(t)=\YY(t)\ZZ(t)=I,
\end{align}
where $I$ the identity matrix . Furthermore, (2.60) and (2.61) from \cite[Section 2.3.1]{nualart} show that 
\begin{align}\label{eq:sigma_U_T_formula}
    \begin{split}
        \sigma_{U_t}=\YY(t)\CC_t\YY(t)^\intercal
    \end{split}
\end{align}
where $\intercal$ denotes the matrix transpose operation and 
\begin{align}\label{eq:def_of_C_t}
    \CC^{ij}_t=\sum_{l=1}^d \int_0^t \ZZ^i_k(s)A^k_l(s)\ZZ^j_m(s)A^m_l(s)ds.
\end{align}

Let $\Lambda=\Lambda(T(\eps))$ be a $d\times d$ diagonal matrix with diagonal entries
\begin{align}\label{eq:def_of_Lambda}
    \Lambda^i_i=
    \begin{cases}
    1 \quad &\for i\leq \nu,\\
    e^{\li T}\leq 1 \quad &\for i>\nu.
    \end{cases}
\end{align}

Due to~\eqref{eq:U}, we have $\D^j_rN^i_T=e^{\li T}\D^j_rU^i_T=\Lambda^i_i\D^j_rU^i$ for $i>\nu$, which together with~\eqref{eq:sigma_U_T_formula} implies that 
\begin{align*}\sigmaUN=\Lambda \sigma_{U_T}\Lambda^\intercal= \Lambda \YY(T)\CC_T\YY(T)^\intercal\Lambda^\intercal.
\end{align*}

Let us define a $d\times d$-matrix valued process $\tYY(t)$ by (no summation over repeated indices is involved)
\begin{align*}\tYY^i_j(t)= \frac{\Lambda^i_i}{\Lambda^j_j}\YY^i_j(t),\quad 1\leq i, j\leq d,
\end{align*}
which satisfies
\begin{align*}\Lambda\YY(T)= \tYY(T)\Lambda,
\end{align*}
which, due to~\eqref{eq:ZY=YZ=I}, implies that
\begin{align}
    \det \ZZ(T)&=(\det \YY(T))^{-1}=(\det \tYY(T))^{-1},\label{eq:detZZ=dettYY^-1} \\
    \sigmaUN &= \tYY(T)\Lambda \CC_T\Lambda^\intercal\tYY(T)^\intercal.\label{eq:sigmaUN_under_tYY}
\end{align}
Then,~\eqref{eq:detZZ=dettYY^-1},~\eqref{eq:sigmaUN_under_tYY} and the Cauchy--Schwarz inequality yield
\begin{align}\label{eq:Cauchy--Schwarz_ineq_sigmaUN}
    \E\left|\det\sigmaUN\right|^{-p}\leq \left(\E\left|\det\Lambda \CC_T\Lambda^\intercal\right|^{-2p}\right)^\frac{1}{2}\left(\E\left|\det \ZZ(T)\right|^{4p}\right)^\frac{1}{2}.
\end{align}
To estimate $\E|\det \ZZ(T)|^{p}$ for $p\geq 1$, we study objects related to $\ZZ(t)$, which will be needed later. Let us define 
\begin{align}\label{eq:def_bZZ_hZZ}
    \bZZ^i_j(t)=\sup_{0\leq s\leq t}\left|\ZZ^i_j(s)\right|, \quad \hZZ^i_j(t)=\ZZ^i_j(t)-\delta^i_j.
\end{align}
Displays~\eqref{eq:def_YY_ZZ} and~\eqref{eq:A,bA,bB_est} imply that
\begin{align*}
    \bZZ^i_j(T) \lesssim \delta^i_j + \sup_{0\leq r\leq T}\left|\int_0^r \bA^k_{j,l}\ZZ^i_k(s)dW^l_s\right|+\int_0^T \eps^\frac{3}{2} \sum_{k=1}^d\bZZ^i_k(s)ds.
\end{align*}
We take $[\,\cdot\,]_p$ of both sides and use~\eqref{eq:A,bA,bB_est} and~\eqref{eq:choice_of_theta} to obtain, for $\eps \in (0,1)$,
\begin{align*}
    \pbracket{\bZZ^i_j(T)}&\lesssim_p \delta^i_j+\sum_{k=1}^d\int_0^T\left(\eps^\frac{3}{2} + \eps^3T\right)\pbracket{\bZZ^i_k(s)}ds \lesssim_p \delta^i_j+\eps^\frac{3}{2} \sum_{k=1}^d\int_0^T\pbracket{\bZZ^i_k(s)}ds.
\end{align*}
Lemma~\ref{lem:solve_inequality_system} implies now that for each $p\geq 2$,
\begin{align}\label{eq:bZZ_est}
    \pbracket{\bZZ^i_j(T)} \lesssim_p 1, \quad \eps \in (0,1).
\end{align}
A similar calculation reveals that
\begin{align*}
    \pbracket{\sup_{0\leq t\leq T}\left|\hZZ^i_j(t)\right|}
    \lesssim_p \eps^\frac{3}{2} \sum_{k=1}^d\int_0^T\pbracket{\bZZ^i_k(s)}ds.
\end{align*}
Plugging~\eqref{eq:bZZ_est} into the above display we obtain, for each $p\geq 2$,
\begin{align}\label{eq:hZZ_est}
    \pbracket{\sup_{0\leq t\leq T}\left|\hZZ^i_j(t)\right|}\lesssim_p \eps, \quad \eps \in(0,1). 
\end{align}

Expressing $\det \ZZ(T)$ as a polynomial of the matrix entries, applying~\eqref{eq:bZZ_est} and H\"older's inequality, we see that for each $p\geq 1$, there is $C_p>0$ such that
\begin{align*}
    \E\left|\det \ZZ(T)\right|^p\leq C_p, \quad \eps \in (0,1).
\end{align*}

In view of the above display and~\eqref{eq:Cauchy--Schwarz_ineq_sigmaUN}, to bound $\E|\det\sigmaUN|^{-p}$, it remains to show that $\E|\det\Lambda \CC_T\Lambda^\intercal|^{-2p}$ is bounded.

Let $\mu_{\Lambda \CC_T\Lambda^\intercal}$ be the smallest eigenvalue of $\Lambda \CC_T\Lambda^\intercal$, which is nonnegative since $\Lambda \CC_T\Lambda^\intercal$ is positive semi-definite. Then, it suffices to show, for each $p\geq 1$,
 there is $C_p>0$ such that
\begin{align}\label{eq:smalles_eigenvalue_for_Lambda_C_Lambda}
    \Prob{\mu_{\Lambda \CC_T\Lambda^\intercal}\leq \zeta}\leq C_p\zeta^p,\quad \zeta\geq 0.
\end{align}
To this end, we will use the following lemma  (\cite[Lemma 5.4]{YB-and-HBC:10.1214/20-AAP1599}):
\begin{lemma}\label{lem:smallest_eigenvalue}
Let $\mathcal{A}$ be a symmetric positive semi-definite
random $d\times d$ matrix. Let~$\mu$ be its smallest eigenvalue. Then for each $p\geq 1$, there is $C_{p,d}>0$ such that
\begin{align*}
    \Prob{\mu \leq \zeta}\leq C_{p,d} \left(\sup_{|v|=1} \E{| v\cdot \mathcal{A}v|^{-(p+2d)} }+ \E\left|\sum_{i,j =1}^d |\mathcal{A}^{ij}|^2\right|^\frac{p}{2}\right)\zeta^p,\quad \zeta\geq 0.
\end{align*}
\end{lemma}
We will apply this lemma to $\mathcal{A}=\Lambda \CC_T\Lambda^\intercal$.

For the second term in the parentheses, it suffices to fix arbitrary $p\geq 1$ and estimate $\E|(\Lambda C_T\Lambda^\intercal)^{ij}|^p$.
Note that, due to~\eqref{eq:def_of_C_t} and~\eqref{eq:def_of_Lambda}, 
\begin{align*}
    (\Lambda \CC_T\Lambda^\intercal)^{ij} = \sum_{1\leq k,l,m\leq d}\int_0^T \Lambda^i_i\ZZ^i_k(s)A^k_l(s)\Lambda^j_j\ZZ^j_m(s)A^m_l(s)ds.
\end{align*}
We split terms on the right of the above display into three cases.

The first case is where $k\neq i$ and $m\neq j$, in which $\ZZ^i_k(s)=\hZZ^i_k(s)$ and $\ZZ^j_m(s)=\hZZ^j_m(s)$ (recall the definition of $\hZZ$ in~\eqref{eq:def_bZZ_hZZ}). Using~\eqref{eq:pBracket_property_2},~\eqref{eq:A,bA,bB_est},~\eqref{eq:def_of_Lambda},~\eqref{eq:hZZ_est}, and~\eqref{eq:choice_of_theta}, we obtain (with no summation over repeated indices) by the Cauchy--Schwartz inequality 
\begin{align*}
    &\pbracket{\sum_{l=1}^d\int_0^T \Lambda^i_i\ZZ^i_k(s)A^k_l(s)\Lambda^j_j\ZZ^j_m(s)A^m_l(s)ds}
    \\&\lesssim_p T\int_0^T e^{2|\lam^k|T}e^{2|\lam^m|T}\pbracketX{\hZZ^i_k(s)}{2p}\pbracketX{\hZZ^j_m(s)}{2p}ds\\
    &\lesssim_p T\int_0^T \eps^{-\frac{1}{4}}\eps^{-\frac{1}{4}}\eps^2ds\leq T^2\eps^\frac{3}{2}\leq 1.
\end{align*}
The second case is where $k=i$ and $m=j$. Applying the same estimates but with~\eqref{eq:bZZ_est} in place of~\eqref{eq:hZZ_est}, we obtain 
\begin{align*}
    &\E\left|\sum_{l=1}^d\int_0^T \Lambda^i_i\ZZ^i_i(s)A^i_l(s)\Lambda^j_j\ZZ^j_j(s)A^j_l(s)ds\right|^p\\
    &\lesssim_p \sum_{l=1}^d\left(\int_0^T \left(\E\left| \Lambda^i_i\ZZ^i_i(s)A^i_l(s)\Lambda^j_j\ZZ^j_j(s)A^j_l(s)\right|^p\right)^\frac{1}{p}ds\right)^p \\ &\lesssim_p \left(\int_0^T \Lambda^i_ie^{-\li s }\Lambda^j_je^{-\lj s}  ds\right)^p\lesssim_p 1,
\end{align*}
where the last $\lesssim$ follows from~\eqref{eq:lam_i} and~\eqref{eq:def_of_Lambda}.

The third case is where either $k=i$ and $m\neq j$, or $k\neq i$ and $m=j$. It can be treated using a combination of above arguments.

Therefore, we conclude that $\E|(\Lambda C_T\Lambda^\intercal)^{ij}|^p\lesssim_p 1$, for each $p\geq 1$. 
Thus to derive~\eqref{eq:smalles_eigenvalue_for_Lambda_C_Lambda} from Lemma~\ref{lem:smallest_eigenvalue}, we only need to verify that for each $p\geq1$ there is $C_p$ such that
\begin{align}\label{eq:Prob<v,Lambda_C_T_Lambda_v>leq_zeta}
    \Prob{ v\cdot \left( \Lambda C_T\Lambda^\intercal v\right)\leq \zeta}\leq C_p \zeta^p, \quad \zeta>0,\ v\in \Sph^{d-1},\  \eps\in(0,1),
\end{align}
where $\Sph^{d-1}=\{x\in\R^d:\ |x|=1\}$ is the unit sphere.

\begin{proof}[Proof of~\eqref{eq:Prob<v,Lambda_C_T_Lambda_v>leq_zeta}] Due to~\eqref{eq:def_of_C_t}, one can see
\begin{align*}
    v\cdot( \Lambda C_T\Lambda^\intercal v) =\int_0^T|A(s)^\intercal \ZZ(s)^\intercal\Lambda v|^2ds.
\end{align*}
Using~\eqref{eq:def_A_bA_bB} and~\eqref{eq:modified_F}, we have
\begin{align}\label{eq:<v,C_Tv>_geq_c_0_integral}
    v\cdot( \Lambda C_T\Lambda^\intercal v) \geq 
    \int_0^T |R_s|^2ds,
\end{align}
where $R_t=(R_t^1,\ldots,R_t^d)$ is defined by
\begin{align}\label{eq:R^j_t}
    R^j_t= \sqrt{c_0}\sum_{i=1}^d  e^{-\lj t}\ZZ^i_j(t)\Lambda^i_iv^i,\quad j=1,2,\ldots,d,\ t\in[0,T],
\end{align}
with $c_0$ introduced in~\eqref{eq:modified_F} and the dependence on $v\in \Sph^{d-1}$ suppressed.
In this notation,~\eqref{eq:<v,C_Tv>_geq_c_0_integral} and~\eqref{eq:def_R^j_t} imply that
\begin{align*}
    \Prob{  v\cdot( \Lambda C_T\Lambda^\intercal v ) \leq \zeta }\leq \Prob{\int_0^T|R_s|^2ds\leq \zeta}.
\end{align*}
The desired result~\eqref{eq:Prob<v,Lambda_C_T_Lambda_v>leq_zeta} follows from the next lemma.
\end{proof}

\begin{lemma}\label{lem:near_zero_probability_est}
Let $R_s$ be given in~\eqref{eq:R^j_t} which depends on the choice of $v\in \Sph^{d-1}$. 
For each $p\geq 1$, there is $C_p>0$ independent of $v$ such that
\begin{align}\label{eq:desired_result_lemma_near_zero_probability_est}
    \Prob{\int_0^T|R_s|^2ds \leq \zeta}\leq C_p \zeta^{\frac{1}{16}p}, \quad \zeta>0,\ \eps\in(0,1).
\end{align}
\end{lemma}

\begin{proof}[Proof of Lemma~\ref{lem:near_zero_probability_est}]
We can rewrite
\begin{align}\label{eq:def_R^j_t}
    \begin{split}
        R^j_t&= R^j_0+M^j_t+A^j_t +B^j_t \\
        &= R^j_0 + \int_0^t u^j_l(s)dW^l_s+\int_0^t a^j(s)ds+\int_0^t b^j(s)ds, \quad j=1,2,\dots,d,
    \end{split}
\end{align}
where  $R_0$, $u(s)$, $a(s), b(s)$ are obtained as follows: we first apply It\^o's formula using the definition of $\ZZ(s)$ given in~\eqref{eq:def_YY_ZZ}, which determines $R_0, u(s)$ and $a(s)+b(s)$; then we write $\ZZ^i_j(s)=\delta^i_j+\hZZ^i_j(s)$ (see~\eqref{eq:def_bZZ_hZZ}) in one of the summations in $a(s)+b(s)$; finally, we collect the terms with $\delta^i_j$ to be $b(s)$ and all the rest to be $a(s)$. Thus
\begin{align}\label{eq:_def_R_0,u,a,b}
    \begin{split}
        R^j_0&=\sqrt{c_0}\Lambda^j_jv_j = 
        \begin{cases}
        \sqrt{c_0}v_j, &\quad  j\leq \nu,\\
        \sqrt{c_0}e^{\lj T}v_j, &\quad  j>\nu,
        \end{cases}\\
        u^j_l(s)&=-\sqrt{c_0}\sum_{i,k=1}^d\Lambda^i_iv_i e^{-\lj s}\bA_{j,l}^k(s)\ZZ^i_k(s),\\
        a^j(s)&=-\left(\sqrt{c_0}\sum_{i=1}^d \Lambda^i_iv_i \lj e^{-\lj s}\hZZ^i_j(s)\right)\\
        &\quad -\left(\sqrt{c_0}\sum_{i,k,m} \Lambda^i_iv_ie^{-\lj s}\left(\bB^k_j(s)-\sum_l\bA^k_{m,l}(s)\bA^m_{j,l}(s)\right)\ZZ^i_k(s)\right)\\
        b^j(s)&=-\sqrt{c_0}\Lambda^j_jv_j\lam^j e^{-\lj s}=
        \begin{cases}
        -\sqrt{c_0}v_j\lj e^{-\lj s}, \quad &  j\leq \nu,\\
        -\sqrt{c_0}v_j\lj e^{\lj (T- s)}, \quad &  j>\nu.
        \end{cases}
    \end{split}
\end{align}

We estimate
\begin{align*}
    \E\sup_{0\leq s\leq T}|u(s)|^p\lesssim_p \sum_{i,j,k,l}e^{p|\lj|T}\E \left|\sup_{0\leq s\leq T}\bA^k_{j,l}(s)\bZZ^i_k(T)\right|^p\lesssim_p \eps^{-\frac{p}{8}}\eps^\frac{3p}{4}\leq \eps^\frac{p}{2}, \quad\eps \in (0,1),
\end{align*}
where the first inequality follows from the expression of $u(s)$ in~\eqref{eq:_def_R_0,u,a,b}, and the second inequality is due to~\eqref{eq:choice_of_theta},~\eqref{eq:A,bA,bB_est}, and~\eqref{eq:bZZ_est}. Similarly, first use the definition of $a(s)$ in~\eqref{eq:_def_R_0,u,a,b} and then estimate terms according to~\eqref{eq:A,bA,bB_est},~\eqref{eq:hZZ_est}, and~\eqref{eq:choice_of_theta} to see
\begin{align*}
\begin{split}
    \E\left(\int_0^T|a(s)|^2ds\right)^p&\lesssim_p \sum_{i,j}\E\left(\left|\sup_{0\leq s\leq T}|\hZZ^i_j(s)|\right|^2\int_0^Te^{2|\lj|s}ds\right)^p
    \\
    &+\sum_{i,j,k,l}\E\left(\left(\eps^\frac{7}{4}+\eps^\frac{3}{2}\right)^{2p}\left|\bZZ^i_j(T)\right|^{2p}\left(\int_0^Te^{2|\lj|s}ds\right)^p\right)
    \\
    &\lesssim_p \eps^{3p}\eps^{-\frac{2p}{8}}\le \eps^\frac{p}{2},\quad \eps\in(0,1).
\end{split}
\end{align*}

The above two estimates and Markov's inequality imply that for some $C_p>0$ independent of $v\in \Sph^{d-1}$,
\begin{align}\label{eq:tail_est_sup_theta_s}
     \Prob{\sup_{0\leq s\leq T}\left(|u(s)|+\int_0^s|a(r)|^2dr\right) >\eps^\frac{1}{2} \zeta^{-\frac{1}{16}}}\leq C_p\zeta^{\frac{1}{16}p},\quad \zeta>0,\  \eps\in(0,1).
\end{align}

Let $j$ be the index that satisfies $|v_j|=\max_{1\leq j\leq  d}|v_i|$. Since $v\in\Sph^{d-1}$, we have
\begin{align}\label{eq:lower_bound_|v_j|}
    |v_j|\geq d^{-\frac{1}{2}}.
\end{align}

In addition, let
\begin{align}\label{eq:def_blam_ulam}
    \blam = \max_{0\leq i\leq d}|\li|, \quad \ulam =\min_{0\leq i\leq d}|\li|.
\end{align}
Recalling the definition of $M_t$ in~\eqref{eq:def_R^j_t}, introducing one more auxiliary process
\begin{align*}N^j_t=\int_0^t R^j(s)u^j_l(s)dW^l_s, \quad j=1,2,\dots, d,
\end{align*}
we define, for each $\zeta>0$ and each $\eps \in(0,1)$,
\begin{align}\label{eq:def_of_B^zeta_eps}
\begin{split}
    B_{0}^{\zeta,\eps}&=\left\{\int_0^T\left|R_s\right|^2ds \leq \zeta,\ \sup_{0\leq s\leq T}\left(|u(s)|+\int_0^s|a(r)|^2dr\right)\leq \eps^\frac{1}{2}\zeta^{-\frac{1}{16}}\right\},
    \\
    B^{\zeta, \eps}_1&=\left\{\qd{M^j}_T \leq (c_1+1)\zeta^\frac{1}{8}, \sup_{0\leq t\leq T}\left|M^j_t\right|\geq \zeta^\frac{1}{32}\right\},
    \\
    B^{\zeta, \eps}_2&=\left\{\qd{N^j}_T \leq\eps \zeta^\frac{7}{8}, \sup_{0\leq t\leq T}\left|N^j_t\right|\geq \eps^\frac{1}{2}\zeta^\frac{3}{8}\right\},
\end{split}
\end{align}
where 
\begin{align}\label{eq:def_c_1}
    c_1=\sqrt{2c_0\ulam^{-1}}+5.
\end{align}
These sets depend on $\eps$ since $T=T(\eps)$, $R_s$, $u(s)$, $a(s)$, $M_s$, and $N_s$ do.
The exponential martingale inequality implies that, for some $C_p>0$ 
independent of $v\in \Sph^{d-1}$,
\begin{align*}
    \Prob{ B^{\zeta,\eps}_1\cup  B^{\zeta,\eps}_2} \leq 2 \exp\left(-\tfrac{\zeta^{-\frac{1}{16}}}{2(c_1+1)}\right)+ 2 \exp\left(-\tfrac{\zeta^{-\frac{1}{8}}}{2}\right)\leq C_p \zeta^{\frac{1}{16}p}, \quad \zeta>0,\ \eps\in(0,1).
\end{align*}

This and~\eqref{eq:tail_est_sup_theta_s} imply that to derive the desired result~\eqref{eq:desired_result_lemma_near_zero_probability_est} it remains
to show that there is $\zeta_0>0$ such that
\begin{align}\label{eq:key_set_inclusion}
    B^{\zeta,\eps}_{0}\subset   B^{\zeta,\eps}_1\cup  B^{\zeta,\eps}_2, \quad \zeta \in (0, \zeta_0),\ \eps\in(0,1).
\end{align}
Let us fix the following two constants
\begin{align}\label{eq:def_c_2_c_3}
    c_2=2+\sqrt{2}+\sqrt{c_0\blam}, \quad \quad c_3=\left(\frac{c_2}{\sqrt{c_0d^{-1}}}\right)^3,
\end{align}
and derive~\eqref{eq:key_set_inclusion} for $\zeta_0$ chosen small enough to ensure 
\begin{align}\label{eq:zeta_0_condition}
  \sqrt{2c_0\blam}\zeta_0^\frac{1}{2}< c_0d^{-1},\quad  \zeta_0^\frac{1}{3}<\frac{1}{2}, \quad c_2\zeta_0^\frac{1}{32}< \sqrt{c_0d^{-1}},\quad (c_2+\sqrt{c_3})\zeta_0^\frac{1}{64}< \sqrt{c_0d^{-1}} .
\end{align}
Suppose~\eqref{eq:key_set_inclusion} is false. Then we can choose $\zeta \in (0,\zeta_0)$, $\eps\in(0,1)$ and $\omega$ such that
\begin{align}\label{eq:omega_condition}
    \omega \in  B^{\zeta,\eps}_{0} \setminus \left( B^{\zeta,\eps}_1\cup  B^{\zeta,\eps}_2\right).
\end{align}
Since $\omega \in  B^{\zeta,\eps}_{0}$ due to~\eqref{eq:omega_condition}, we have
\begin{align*}
    \qd{N^j}_T \leq \int_0^T\left|R^j_s u^j(s)\right|^2 ds \leq \left(\sup_{0\leq s\leq T}|u(s)|^2\right)\int_0^T|R_s|^2ds \leq \left(\eps^\frac{1}{2}\zeta^{-\frac{1}{16}}\right)^2 \zeta=\eps\zeta^\frac{7}{8}.
\end{align*}
Since $\omega \not\in  B^{\zeta,\eps}_2$, this implies
\begin{align}\label{eq:int_R_s_u(s)ds_est}
    \sup_{0\leq t\leq T}\left|\int_0^tR^j_su^j_l(s)dW^l_s\right|=\sup_{0\leq t\leq T}\left|N^j_t\right|<\eps^\frac{1}{2}\zeta^\frac{3}{8}.
\end{align}
Since $\omega \in  B^{\zeta,\eps}_{0}$,  the Cauchy--Schwarz inequality implies
\begin{align}
    \sup_{0\leq t\leq T}\left|\int_0^tR^j_sa^j(s)ds\right|&\leq \left(\int_0^T\left|R^j_s\right|^2ds\right)^\frac{1}{2}\left(\int_0^T\left|a^j(s)\right|^2ds\right)^\frac{1}{2}
    \notag
    \\
    &\leq \zeta^{\frac{1}{2}}\left(\eps^\frac{1}{2}\zeta^{-\frac{1}{16}}\right)^\frac{1}{2}=\eps^\frac{1}{4}\zeta^\frac{15}{32}.
    \label{eq:int_R_s_a(s)ds_est}
\end{align}
We 
recall $b(s)$ defined in~\eqref{eq:_def_R_0,u,a,b}. 

The It\^o formula applied to~\eqref{eq:def_R^j_t} gives
\begin{align*}
    \left|R^j_t\right|^2 &=\left|R^j_0\right|^2 + 2\int_0^t R_s^j dR^j_s +\qd{M^j}_t\\
    &= \left|R^j_0\right|^2 + 2\left(\int_0^t R_s^j u^j_ldW^l_s + \int_0^tR^j_sa^j(s)ds+\int_0^tR^j_sb^j(s)ds\right)+\qd{M^j}_t.
\end{align*}
This together with~\eqref{eq:int_R_s_u(s)ds_est},~\eqref{eq:int_R_s_a(s)ds_est} and $\omega \in  B^{\zeta,\eps}_{0}$ due to~\eqref{eq:omega_condition} implies
\begin{align*}
    \int_0^T\qd{M^j}_tdt &= \int_0^T\left|R^j_s\right|^2 dt-T\left|R^j_0\right|^2  - 2\int_0^T \int_0^t R_s^j dR^j_s dt \\
    &\leq \zeta - T\left|R_0\right|^2 + 2\int_0^T\left|\int_0^tR^j_sb^j(s)ds\right|dt+2T\left(\eps^\frac{1}{2}\zeta^\frac{3}{8}+\eps^\frac{1}{4}\zeta^\frac{15}{32}\right).
\end{align*}
We treat cases $j\leq \nu$ and $j> \nu$ separately. 

If $j\leq \nu$, i.e.,  $\lj>0$,  we 
use the definition of $R^j_0$ in~\eqref{eq:_def_R_0,u,a,b} and~\eqref{eq:lower_bound_|v_j|} to bound $|R^j_0|$ from below, 
use $\omega\in B^{\zeta,\eps}_0$ 
to estimate 
the iterated integral term by
\begin{align*}\begin{split}
    \left|\int_0^t R^j_sb^j(s)ds\right|&\leq \left(\int_0^T\left|R^j_s\right|^2ds\right)^\frac{1}{2}\left(\int_0^t\left|b^j(s)\right|^2ds\right)^\frac{1}{2}\leq \zeta^\frac{1}{2}v_j\sqrt{c_0}\lj \left(\int_0^te^{-2\lj s}ds\right)^\frac{1}{2}\\
    &\leq\zeta^\frac{1}{2}\sqrt{c_0}\lj\frac{1}{\sqrt{2\lj}}= \sqrt{\frac{c_0|\lj|}{2}}\zeta^\frac{1}{2},\quad t\leq T,
\end{split}
\end{align*}
 and use the first condition in~\eqref{eq:zeta_0_condition} to deduce
\begin{align*}
    - T\left|R^j_0\right|^2 + 2\int_0^T\left|\int_0^tR^j_sb^j(s)ds\right|dt &\leq -T\frac{c_0}{d}+2T\sqrt{\frac{c_0\lj}{2}}\zeta^\frac{1}{2}
    \\
    &\leq T\left(\sqrt{2c_0\blam}\zeta^\frac{1}{2}-c_0d^{-1}\right)\leq 0.
\end{align*}
where $\blam$ was defined in~\eqref{eq:def_blam_ulam}. 
If $j>\nu$, i.e., $\lj<0$, we use
\begin{align*}\begin{split}
    \left|\int_0^t R^j_sb^j(s)ds\right|&\leq \zeta^\frac{1}{2}\sqrt{c_0}\left|\lj\right|e^{\lj T} \left(\int_0^te^{-2\lj s}ds\right)^\frac{1}{2}
    \\
    &\leq \zeta^\frac{1}{2}\sqrt{c_0}\left|\lj\right|e^{\lj T}\frac{e^{-\lj t}}{\sqrt{2\left|\lj\right|}}= \sqrt{\frac{c_0\left|\lj\right|}{2}}\zeta^\frac{1}{2}e^{\lj (T-t)},\quad t\leq T,
\end{split}
\end{align*}
to obtain
\begin{align*}
    - T\left|R^j_0\right|^2 + 2\int_0^T\left|\int_0^tR^j_sb^j(s)ds\right|dt &\leq 0+ 2\int_0^T\sqrt{\frac{c_0|\lj|}{2}}\zeta^\frac{1}{2}e^{\lj(T-t)}dt\\
    &\leq \sqrt{2c_0|\lj|}e^{\lj T}\frac{e^{-\lj T}-1}{|\lj|}\zeta^\frac{1}{2}\leq \sqrt{2c_0\ulam^{-1}}\zeta^\frac{1}{2},
\end{align*}
where $\ulam$ was defined in~\eqref{eq:def_blam_ulam}. Recall $c_1$ given in~\eqref{eq:def_c_1}. These estimates along with~\eqref{eq:choice_of_theta} show that, in both cases,
\begin{align*}
    \int_0^T\qd{M^j}_tdt \leq \zeta + \sqrt{2c_0\ulam^{-1}}\zeta^\frac{1}{2} + 2\left(\zeta^\frac{3}{8}+\zeta^\frac{3}{8}\right)\leq c_1\zeta^\frac{3}{8}.
\end{align*}
Since $t\mapsto \qd{M^j}_t$ is nondecreasing, we conclude that
\begin{align*}
    \gamma \qd{M^j}_{T-\gamma}\leq c_1\zeta^\frac{3}{8}, \quad 0< \gamma \leq T.
\end{align*}
Since $\omega \in  B^{\zeta,\eps}_{0}$ implies $\sup_{0\leq s\leq T}|u(s)|\leq \eps^\frac{1}{2} \zeta^{-\frac{1}{16}}\leq \zeta^{-\frac{1}{16}}$, using the definition of~$M_t$ in~\eqref{eq:def_R^j_t}, we get
\begin{align*}
    \qd{M^j}_T-\qd{M^j}_{T-\gamma}\leq \gamma \zeta^{-\frac{1}{8}}.
\end{align*}
The above two displays yield $\qd{M^j}_T\leq c_1\gamma^{-1}\zeta^\frac{3}{8} + \gamma \zeta^{-\frac{1}{8}}$. Recall that in the statement of Lemma~\ref{lem:density_est}, it is required that $T\geq 1$. The second condition in~\eqref{eq:zeta_0_condition} thus guarantees that $\zeta^\frac{1}{4} <\zeta_0^\frac{1}{4}<(\frac{1}{2})^\frac{3}{4}<1\leq T$.
Therefore, we can set $\gamma = \zeta^\frac{1}{4}$ and obtain
\begin{align*}
    \qd{M^j}_T \leq c_1\zeta^{-\frac{1}{4}+\frac{3}{8}}+\zeta^{\frac{1}{4}-\frac{1}{8}}\leq (c_1+1)\zeta^\frac{1}{8}.
\end{align*}
Since $\omega \not \in B^{\zeta,\eps}_1$ due to~\eqref{eq:omega_condition}, the definition of $B^{\zeta,\eps}_1$ in~\eqref{eq:def_of_B^zeta_eps} indicates that
\begin{align}\label{eq:sup_|M_t|_est}
    \sup_{0\leq t\leq T}\left|M_t^j\right| < \zeta^\frac{1}{32}.
\end{align}
On the other hand, Markov's inequality and $\omega \in B^{\zeta,\eps}_{0}$ imply that
\begin{align*}
    \Leb\left\{t\in[0,T]:\left|R^j_t\right|\geq \zeta^\frac{1}{3}\right\}\leq \frac{1}{\zeta^\frac{2}{3}}\int_0^T\left|R^j_t\right|^2dt \leq \zeta^\frac{1}{3}.
\end{align*}
Using~\eqref{eq:sup_|M_t|_est} and~\eqref{eq:def_R^j_t}, we thus have
\begin{align*}
    \Leb\left\{t\in[0,T]:\left|R^j_0+A^j_t+B^j_t\right|\geq \zeta^\frac{1}{3}+\zeta^\frac{1}{32}\right\}\leq \zeta^\frac{1}{3}.
\end{align*}
Note that $\zeta^\frac{1}{3}<\zeta_0^\frac{1}{3}\leq \frac{1}{2} \leq \frac{1}{2}T$ due to the second condition in~\eqref{eq:zeta_0_condition} and $T\geq 1$. Hence, for each $t\in[0,T]$, there is $t'\in[0,T]$ satisfying $|t-t'|\leq 2\zeta^\frac{1}{3}$ and $|R^j_0+A^j_{t'}+B^j_{t'}|<\zeta^\frac{1}{3}+\zeta^\frac{1}{32}$. Recall the definitions $A^j_t$ and $B^j_t$ in~\eqref{eq:def_R^j_t} and $b(s)$ in~\eqref{eq:_def_R_0,u,a,b}. Then, for each $t\in[0,T]$, we obtain, regardless of whether $j\leq \nu$ or $j>\nu$,
\begin{align}\label{eq:|R_0+A_t+B_t|_est}
\begin{split}
    &\left|R^j_0+A^j_t+B^j_t\right|\leq \left|R^j_0+A^j_{t'}+B^j_{t'}\right|+\left|\int_{t'}^ta^j(s)ds\right| +\left|\int_{t'}^tb^j(s)ds\right|
    \\
    &< \zeta^\frac{1}{3}+\zeta^\frac{1}{32} + \left|\int_{t'}^t|a(s)|^2ds\right|^\frac{1}{2}\left|t-t'\right|^\frac{1}{2}+\left|\int_{t'}^t|b^j(s)|^2ds\right|^\frac{1}{2}\left|t-t'\right|^\frac{1}{2}
    \\
    &\leq \zeta^\frac{1}{3}+\zeta^\frac{1}{32}+ \eps^\frac{1}{4}\zeta^{-\frac{1}{32}}\sqrt{2}\zeta^{\frac{1}{6}}+\sqrt{c_0}\left|\lj\right| \left(2\lj\right)^{-\frac{1}{2}}\sqrt{2}\zeta^\frac{1}{6}
    \\
    &\leq \left(2+\sqrt{2}+\sqrt{c_0\blam}\right)\zeta^\frac{1}{32}= c_2 \zeta^\frac{1}{32}, \quad t \leq T.
\end{split}
\end{align}
where $c_2$ was given in~\eqref{eq:def_c_2_c_3}. We used the assumption $\omega \in B^{\zeta,\eps}_0$ to bound the integral of $|a(s)|^2$
and the definition of $b^j$ to bound the integral of $|b^j(s)|^2$.

Setting $t=0$ in the above display we obtain
\begin{align}\label{eq:|R^j_0|_bound}
    \left|R^j_0\right|<c_2\zeta^\frac{1}{32}.
\end{align}

If $j\leq \nu$, then, using the expression for $R^j_0$ in~\eqref{eq:_def_R_0,u,a,b} and~\eqref{eq:lower_bound_|v_j|}, we obtain $|R^j_0|\geq \sqrt{c_0d^{-1}}$, which along with~\eqref{eq:|R^j_0|_bound} and the third condition in~\eqref{eq:zeta_0_condition} implies
\begin{align*}
    \sqrt{c_0d^{-1}}<  c_2\zeta^\frac{1}{32} <c_2\zeta_0^\frac{1}{32}\leq \sqrt{c_0d^{-1}},
\end{align*}
a contradiction.

If $j>\nu$,  then, due to~\eqref{eq:choice_of_theta}, we have $e^{\lj T}=e^{-|\lj|T}\geq \eps^\frac{1}{8}$. Due to the formula for $R^j_0$ in~\eqref{eq:_def_R_0,u,a,b},~\eqref{eq:lower_bound_|v_j|} and~\eqref{eq:|R^j_0|_bound}, we have
\begin{align}\label{eq:bound_for_R^j_0}
    \sqrt{c_0d^{-1}}\eps^\frac{1}{8}\leq \left|R^j_0\right|< c_2 \zeta^\frac{1}{32}.
\end{align}
Since~\eqref{eq:choice_of_theta} gives $T\leq \eps^{-\frac{1}{8}}$ and $\omega\in B^{\zeta,\eps}_0$,~\eqref{eq:bound_for_R^j_0} implies
\begin{align*}
    T\int_0^T\left|a^j(s)\right|^2ds\leq T\eps^\frac{1}{2}\zeta^{-\frac{1}{16}}\leq \eps^\frac{3}{8}\zeta^{-\frac{1}{16}}
    \leq \left(\frac{c_2\zeta^\frac{1}{32}}{\sqrt{c_0d^{-1}}}\right)^3\zeta^{-\frac{1}{16}} = c_3 \zeta^\frac{1}{32},
\end{align*}
where $c_3$ was given in~\eqref{eq:def_c_2_c_3}. Setting $t=T$ in~\eqref{eq:|R_0+A_t+B_t|_est} and recalling that $A^j_t$ is defined in~\eqref{eq:def_R^j_t}, we see that the above display implies:
\begin{align}\label{eq:contradiction_ingredient_1}
    \left|R^j_0+B^j_T\right|\leq c_2\zeta^\frac{1}{32}+\left|A^j_T\right|\leq c_2\zeta^\frac{1}{32}+T^\frac{1}{2}\left|\int_0^T\left|a^j(s)\right|^2ds\right|^\frac{1}{2}\leq \left(c_2+\sqrt{c_3}\right)\zeta^\frac{1}{64}.
\end{align}
On the other hand, expressions for $R^j_0$, $B^j_t$ in~\eqref{eq:def_R^j_t},~\eqref{eq:_def_R_0,u,a,b} show that
\begin{align}\label{eq:contradiction_ingredient_2}
    R^j_0+B^j_T=\sqrt{c_0}e^{\lj T}v_j - \int_0^T\sqrt{c_0}v_j\lj e^{\lj (T-s)}ds = \sqrt{c_0}v_j.
\end{align}
Lastly, we have
\begin{align*}
    \sqrt{c_0d^{-1}}\leq \sqrt{c_0}\left|v_j\right|\leq  \left(c_2+\sqrt{c_3}\right)\zeta^\frac{1}{64}<\left(c_2+\sqrt{c_3}\right)\zeta_0^\frac{1}{64}<\sqrt{c_0d^{-1}},
\end{align*}
where the first inequality follows from~\eqref{eq:lower_bound_|v_j|}, the second one from~\eqref{eq:contradiction_ingredient_1} and~\eqref{eq:contradiction_ingredient_2}, the last one from the fourth condition in~\eqref{eq:zeta_0_condition}. But, the above display is absurd.

By contradiction, $\eqref{eq:key_set_inclusion}$ holds for $\zeta_0$ satisfying~\eqref{eq:zeta_0_condition}. This completes the proof of
\eqref{eq:key_set_inclusion} and thus 
Lemma~\ref{lem:near_zero_probability_est}.
\end{proof}

In conclusion, we have shown that for each $p\geq 1$ there is $C_p>0$ such that
\begin{align*}
    \E\left|\det\sigma_{(U^{\leq \nu}_T, N^{>\nu}_T)}\right|^{-p}\leq C_p, \quad \eps\in(0,1).
\end{align*}

\subsection{Proof of Lemma~\ref{lem:density_est}}
Using the exponential martingale inequality, the boundedness of $V^i_t$ for $i\leq \nu$ and that of $e^{\li t}V^t_t$ for $i>\nu$, one can see that there are constants $C,\ c>0$ independent of $y$, $\eps$, $\theta$, and any particular choice of $T=T(\eps)$ such that, uniformly in $y\in\R^d$,
\begin{align*}
    \Prob{\left|(U^{\leq \nu}_T,N^{>\nu}_T)-x\right|<2},\ \Prob{\left|Z_T-x\right|<2},\ \Prob{\left|\bZ_T-x\right|<2}&\leq Ce^{-c|x|^2}, \quad x\in\R^d.
\end{align*}
This display, along with~\eqref{eq:derivative_est_conclusion_result_U_N}, ~\eqref{eq:derivative_est_conclusion_result_Z},  ~\eqref{eq:goal_of_malliavin_matrix_est} and Theorem \ref{Thm:DensityDifference} implies parts~\eqref{item:den_est_3} 
and~\eqref{item:den_est_4} of Lemma~\ref{lem:density_est}. Parts~\eqref{item:den_est_1} and ~\eqref{item:den_est_2} follow then straightforwardly. \epf

\bibliographystyle{alpha}
\newcommand{\etalchar}[1]{$^{#1}$}

\end{document}